\documentclass[12pt]{article}
\usepackage{geometry}
\usepackage{mathrsfs,amsthm,graphicx,color,verbatim,bm,bbm,amsmath,amsfonts,amssymb,newclude,nicefrac,amsfonts,graphicx,enumerate,hyperref}
\usepackage[latin1]{inputenc}

\theoremstyle{plain}
\newtheorem{theorem} {Theorem}[section]
\newtheorem{lemma} [theorem]{Lemma}
\newtheorem{corollary} [theorem]{Corollary}
\newtheorem{proposition} [theorem]{Proposition}

\theoremstyle{definition}

\begin{document}

\newcommand{\E}{\mathbb{E}\!}
\newcommand{\ES}{\mathbb{E}}
\renewcommand{\P}{\mathbb{P}}
\newcommand{\R}{\mathbb{R}}
\newcommand{\N}{\mathbb{N}}
\newcommand{\smallsum}{\textstyle\sum}
\newcommand{\tr}{\operatorname{trace}}
\newcommand{\citationand}{\&}
\newcommand{\dt}[1][t]{\, \mathrm{d} #1}
\newcommand{\grid}{\,\mathcal{P}}
\newcommand{\euler}{Z^{\,N}}
\newcommand{\leuler}{\tilde{Z}^{\,N}}
\newcommand{\exteuler}{\bar{Z}^{\,N}}

\newcommand{\floor}[3]{\lfloor #1\rfloor_{#2,#3}}
\newcommand{\tstep}[2]{h_{#1,#2}}
\newcommand{\mesh}[2]{\|#1\|^{#2}_{\grid_T}}
\newcommand{\bigbrack}[2]{\big( #1\big)^{#2}}
\newcommand{\bignorm}[3]{\bnl #1\bnr^{#2}_{#3}}
\newcommand{\bigsharp}[2]{\big[ #1\big]^{#2}}
\newcommand{\lpn}[3]{L^{#1}(#2;#3)}
\newcommand{\eulerm}[1]{Z^{\,\Theta_{#1}}}
\newcommand{\eulerpart}[1]{Z^{#1}}

\newcommand{\expeuler}[1]{Z^{\,\text{exp},#1}}
\newcommand{\impeuler}[1]{Z^{\,\text{imp},#1}}

\newcommand{\embed}[2]{\kappa^I_{#1,\,#2}}

\newcommand{\groupC}[1]{C^{#1}}
\newcommand{\deltaset}[1]{\mathbb{D}_{#1}}
\newcommand{\set}{{\mathbb{H}}}
\newcommand{\Co}{\chi}
\newcommand{\Cr}[1]{\rho_{#1}}

\title{Weak convergence rates of spectral\\ Galerkin
approximations for SPDEs\\ with nonlinear diffusion coefficients}

\author{Daniel Conus,
Arnulf Jentzen, and 
Ryan Kurniawan\\[2mm]
\emph{Lehigh University, USA
and ETH Zurich, Switzerland}}

\maketitle

\begin{abstract}
Strong convergence rates for (temporal, spatial, and noise) numerical approximations 
of semilinear stochastic evolution equations (SEEs) with smooth and regular nonlinearities are well understood in 
the scientific literature. Weak convergence rates for numerical approximations of 
such SEEs have been investigated for about two decades and are far away from being well 
understood: roughly speaking, no essentially sharp weak convergence rates are known 
for parabolic SEEs with nonlinear diffusion coefficient functions; see 
Remark~2.3 in [A.\ Debussche, Weak approximation of stochastic partial differential equations: the 
nonlinear case, Math.\ Comp.\ 80 (2011), no.\ 273, 89--117] for details. In this article
we solve the weak convergence problem emerged from Debussche's article in the case of
spectral Galerkin approximations and establish essentially sharp
weak convergence rates for spatial spectral Galerkin approximations of semilinear SEEs
with nonlinear diffusion coefficient functions. 
Our solution to the weak convergence problem 
does not
use Malliavin calculus.
Rather, key ingredients in our 
solution to the weak convergence problem emerged from Debussche's article 
are the use of appropriately modified versions of the spatial Galerkin approximation processes
and applications of a mild It\^{o} type formula 
for solutions and numerical approximations of semilinear SEEs. 
This article solves the weak convergence problem emerged from 
Debussche's article merely in the case of spatial spectral Galerkin approximations
instead of other more complicated numerical approximations.
Our method of proof extends, however, to a number of other kinds of spatial and temporal numerical approximations for semilinear SEEs.
\end{abstract}

 \tableofcontents


\section{Introduction}
\label{sec:intro}

Both strong and numerically weak convergence rates for numerical approximations of 
finite dimensional stochastic ordinary differential equations (SODEs) with smooth 
and regular 
nonlinearities are 
well understood in the literature; see, e.g., the monographs 
Kloeden \& Platen~\cite{kp92} and Milstein~\cite{m95}.
The situation is different in the case of 
possibly infinite dimensional semilinear stochastic evoluation
equations (SEEs).
While strong convergence rates for (temporal, spatial, and noise) numerical approximations 
of semilinear SEEs with smooth and regular nonlinearities are well understood in 
the scientific literature, 
weak convergence rates for numerical approximations of 
such SEEs have been investigated since about 14 years ago and are far away from being well 
understood: roughly speaking, no essentially sharp weak convergence rates are known 
for parabolic SEEs with nonlinear diffusion coefficient functions (see Remark~2.3 in 
Debussche~\cite{Debussche2011} for details).
In this article
we solve 
the weak convergence problem 
emerged from Debussche's article 
in the case of
spectral Galerkin approximations and establish essentially
sharp weak convergence rates for spatial spectral Galerkin approximations of semilinear SEEs
with nonlinear diffusion coefficient functions. 
To illustrate the weak convergence problem emerged from 
Debussche's article and our solution to the problem we consider the following setting as a special case of
our general setting in Section~\ref{sec:weak_irregular} below.
Let 
$ ( H, \left< \cdot, \cdot \right>_H, \left\| \cdot \right\|_H ) $
and
$ ( U, \left< \cdot, \cdot \right>_U, \left\| \cdot \right\|_U ) $
be separable $ \R $-Hilbert spaces, let 
$ T \in (0,\infty) $, 
let $ ( \Omega, \mathcal{F}, \P, ( \mathcal{F}_t )_{ t \in [0,T] } ) $
be a stochastic basis,
let $ ( W_t )_{ t \in [0,T] } $ be an $ \operatorname{Id}_U $-cylindrical
$ ( \Omega, \mathcal{F}, \P, ( \mathcal{F}_t )_{ t \in [0,T] } ) $-Wiener process,
let $ (e_n)_{ n \in \N } \subseteq H $ be an orthonormal basis of $ H $,
let $ ( \lambda_n )_{ n \in \N } \subseteq ( 0, \infty ) $ 
be an increasing sequence,
let 
$ A \colon D(A) \subseteq H \to H $ be a closed linear operator 
such that 
$
  D( A ) 
  = 
  \{ 
    v \in H 
  \colon 
    \sum_{ n \in \N }
    \left| 
      \lambda_n
      \langle e_n, v \rangle_H 
    \right|^2
  < \infty
  \}
$
and 
$ 
  \forall \, n \in \N \colon A e_n = - \lambda_n e_n $,
let
$  
  ( 
    H_r 
    ,
    \langle \cdot, \cdot \rangle_{ H_r }
    ,
    \left\| \cdot \right\|_{ H_r } 
  )
$,
$ r \in \R $,
be a family of interpolation spaces associated to $ - A $
(cf., e.g., \cite[Section~3.7]{sy02}),
let
$ 
  \iota \in [ 0, \nicefrac{ 1 }{ 4 } ]
$, 
$ \xi \in H_{ \iota } $,
$ 
\gamma \in [ 0, \nicefrac{ 1 }{ 2 } ] 
$,
and let
$ 
  F \in 
  \cap_{ r < \iota - \gamma }
  C^4_b( H_{ \iota } , H_r ) 
$, 
$
  B \in 
  \cap_{ r < \iota - \nicefrac{ \gamma }{ 2 } }
  C^4_b( H_{ \iota } , HS( U, H_r ) )
$.  
For two $ \R $-Hilbert spaces 
$ ( V_1, \langle\cdot,\cdot\rangle_{V_1}, \left\| \cdot \right\|_{ V_1 } ) $ 
and 
$ ( V_2, \langle\cdot,\cdot\rangle_{V_2}, \left\| \cdot \right\|_{ V_2 } ) $ 
we denote by $ C^4_b( V_1 , V_2 ) $ the 
$ \R $-vector space of all four times continuously Fr\'{e}chet differentiable
functions from $V_1$ to $V_2$ with globally bounded derivatives
and by $HS(V_1,V_2)$ the 
$\R$-Hilbert space of all Hilbert-Schmidt operators from $V_1$ to $V_2$.
We also note that the hypothesis that 
$  
( 
H_r 
,
\langle \cdot, \cdot \rangle_{ H_r }
,
\left\| \cdot \right\|_{ H_r } 
)
$,
$ r \in \R $,
is a family of interpolation spaces associated to $ - A $
ensures for all $r\in[0,\infty)$
that $H_r=D((-A)^r)$
and
$
  H_1 = D(A) \subseteq H = H_0
$.
The above 
assumptions 
imply (cf., e.g.,  Da Prato et al.~\cite[Proposition~3]{DaPratoJentzenRoeckner2012}, 
Brze{\'z}niak~\cite[Theorem~4.3]{b97b}, 
Van Neerven et al.~\cite[Theorem~6.2]{vvw08}) 
the existence of a continuous mild solution 
process 
$ X \colon [0,T] \times \Omega \to H_{ \iota } $
of the SEE
\begin{equation}
\label{eq:SEE_intro}
dX_t= 
\left[ AX_t 
+ 
F(X_t) 
\right] 
\,dt 
+ 
B(X_t) 
\,dW_t, 
\qquad 
t\in[0,T]
,
\qquad 
X_0=\xi
.
\end{equation}
{As an example for~\eqref{eq:SEE_intro},
we think of $ H = U = L^2( (0,1) ; \R ) $
being the $ \R $-Hilbert space of equivalence classes
of Lebesgue-Borel square integrable functions from $ ( 0,1 ) $ to $ \R $
and $ A $ being an appropriate linear differential operator on $ H $.
In particular, in Subsection~\ref{sec:anderson} we formulate the continuous version of 
\emph{the one-dimensional parabolic Anderson model} as an example for~\eqref{eq:SEE_intro}
(in this example the parameter $ \gamma $, which controls the regularity of the operators
$ F $ and $ B $, satisfies $ \gamma = \nicefrac{ 1 }{ 2 } $)
and in Subsection~\ref{sec:cahn} we formulate \emph{a fourth-order stochastic partial
differential equation} as an example for~\eqref{eq:SEE_intro}
(in this second example the parameter $\gamma$ satisfies $ \gamma = \nicefrac{ 1 }{ 4 } $).}
%
%
%
%
%

Strong convergence rates for (temporal, spatial, and noise) numerical approximations
for SEEs of the form~\eqref{eq:SEE_intro} 
are well understood.
Weak convergence rates for numerical approximations of SEEs 
of the form~\eqref{eq:SEE_intro}
have been investigated for about two decades; cf., e.g., 
\cite{s03, 
h03b, 
dd06, 
dp09, 
GeissertKovacsLarsson2009, 
h10c, 
Debussche2011, 
kll11, 
Doersek2012, 
Lindgren2012PhDThesis,
LindnerSchilling2013, 
KovacsLarssonLindgren2013BIT,
WangGan2013WeakHeatAdditiveNoise, 
Kruse2014_PhD_Thesis, 
Brehier2014,
Wang2014b, 
AnderssonKruseLarsson2016,
AnderssonLarsson2016,
BrehierKopec2016,
Wang2016481}.
Except for Debussche \& De Bouard~\cite{dd06}, 
Debussche~\cite{Debussche2011},
and Andersson \& Larsson~\cite{AnderssonLarsson2016}, all of the above mentioned references assume, 
beside further assumptions, that the considered SEE is driven by additive noise. 
In Debussche \& De Bouard~\cite{dd06} 
weak convergence rates for the nonlinear Schr\"{o}dinger equation,
{which} dominant linear operator generates a group (see~\cite[Section~2]{dd06}) instead of 
only a semigroup as in the general setting of the SEE~\eqref{eq:SEE_intro}, are analyzed.
The method of proof in Debussche \& De Bouard~\cite{dd06} 
strongly exploits this property of 
the nonlinear Schr\"{o}dinger equation
(see~\cite[Section~5.2]{dd06}).
Therefore, 
the method of proof in \cite{dd06}
can, in general, not be used to establish weak convergence
rates for the SEE~\eqref{eq:SEE_intro}.
In Debussche's seminal article \cite{Debussche2011} (see also Andersson \& Larsson~\cite{AnderssonLarsson2016}), essentially 
sharp weak convergence rates for SEEs of the form~\eqref{eq:SEE_intro} are established 
under the hypothesis that the second derivative of the
diffusion coefficient $ B $ satisfies the smoothing property
that there exists a real number $ L \in [0,\infty) $ 
such that for all $ x, v, w \in H $ it holds 
that\footnote{Assumption~\eqref{intro_eq:affine} above slightly differs from the original assumption in \cite{Debussche2011} 
as we believe that there is a small typo in equation (2.5) in \cite{Debussche2011};
see inequality (4.3) in the proof of Lemma~4.5 in \cite{Debussche2011} for details.}
\begin{equation}
\label{intro_eq:affine}
  \left\|
    B''( x )( v , w )
  \right\|_{ L(H) 
  }
  \leq 
  L
  \left\| v \right\|_{
    H_{ - \nicefrac{ 1 }{ 4 } } 
  }
  \left\| w 
  \right\|_{
    H_{ - \nicefrac{ 1 }{ 4 } }
  }
  .
\end{equation}
As pointed out in Remark~2.3 in Debussche~\cite{Debussche2011},
assumption~\eqref{intro_eq:affine}
is a serious restriction for SEEs of the form~\eqref{eq:SEE_intro}.
Roughly speaking, assumption~\eqref{intro_eq:affine}
imposes that the second derivative of the diffusion
coefficient function vanishes and thus that the diffusion coefficient function
is affine linear.
Remark~2.3 
in Debussche~\cite{Debussche2011} 
also asserts that
assumption~\eqref{intro_eq:affine} is crucial in the weak convergence proof 
in \cite{Debussche2011}, 
that assumption~\eqref{intro_eq:affine} is used in 
an essential way
in Lemma~4.5 in \cite{Debussche2011}, and that
Lemma~4.5 in \cite{Debussche2011}, in turn, 
is used at many points in the weak convergence proof in \cite{Debussche2011}.
To the best of our knowledge, 
it remained an open problem to establish essentially sharp weak convergence rates
for any type of temporal, spatial, or noise numerical approximation of the 
SEE~\eqref{eq:SEE_intro} without imposing Debussche's assumption~\eqref{intro_eq:affine}.
In this article we solve this problem in the case of spatial spectral
Galerkin approximations for the SEE~\eqref{eq:SEE_intro}.
This is the subject of the following theorem (Theorem~\ref{intro:theorem}), which follows immediately from Corollary~\ref{cor:sharp.weak.rates} below.

\begin{theorem}
\label{intro:theorem}
Assume the setting in the first paragraph 
of Section~\ref{sec:intro},
let $ \varphi \in C^4_b( H_{ \iota } , \R ) $, 
let
$
  ( P_N )_{ N \in \N } \subseteq L( H_{ - 1 } )
$ 
satisfy for all 
$ N \in \N $, 
$ v \in H $
that
$
  P_N(v) = \sum^N_{n=1} \left< e_n, v \right>_H e_n
$, 
and for every $ N \in \N $ let 
$
  X^N \colon [0,T] \times \Omega \rightarrow P_N( H )
$
be a continuous mild solution of the SEE
\begin{equation}
\label{eq:Galerin_SEEs}
  d X^N_t = 
  \left[ P_N A X^N_t + P_N F( X^N_t ) \right] dt
  +
  P_N B( X^N_t ) \, dW_t
  ,
  \quad 
  t \in [0,T] 
  ,
  \quad 
  X^N_0 = P_N( \xi )
  .
\end{equation}
Then for every $ \varepsilon \in ( 0, \infty ) $ 
there exists a real number 
$ C_{ \varepsilon } \in [0,\infty) $
such that for all $ N \in \N $
it holds that
\begin{equation}
  \left|
    \ES\big[ 
      \varphi( X_T )
    \big]
    -
    \ES\big[ 
      \varphi( X^N_T )
    \big]
  \right|
\leq
    C_{ \varepsilon } \cdot
    \left( \lambda_N \right)^{
      - ( 1 - \gamma - \varepsilon )
    }
  .
\end{equation}
\end{theorem}

Let us add a few comments regarding 
Theorem~\ref{intro:theorem}.
First, we would like to emphasize that 
in the general setting of Theorem~\ref{intro:theorem},
the weak convergence rate established
in Theorem~\ref{intro:theorem}
can \emph{essentially not be improved}.
More specifically, in 
Corollary~\ref{cor:lower_bound}
in Section~\ref{sec:lower_bound} below
we give for every 
$ \iota \in [ 0, \nicefrac{ 1 }{ 4 } ] $
and every
$ \gamma \in [ 0 , \nicefrac{ 1 }{ 2 } ] $
examples of
$ A \colon D(A) \subseteq H \to H $,
$ \xi \in H_\iota $, 
$ 
  F \in 
  \cap_{ r < \iota - \gamma }
  C^4_b( H_{ \iota } , H_{ r } ) 
$,
$ 
  ( U, \left< \cdot, \cdot \right>_U , \left\| \cdot \right\|_U ) 
$, 
$
  B \in 
  \cap_{ r < \iota - \nicefrac{ \gamma }{ 2 } }
  C^4_b( H_{ \iota } , HS( U, H_r ) )
$,
and
$ \varphi \in C^4_b( H_{ \iota } , \R ) $
such that 
there exists a real number $ C \in (0,\infty) $
such that
for all $ N \in \N $
it holds that
\begin{equation}
  \left|
    \ES\big[ 
      \varphi( X_T )
    \big]
    -
    \ES\big[ 
      \varphi( X^N_T )
    \big]
  \right|
\geq
    C \cdot
    \left( \lambda_N \right)^{
      - ( 1 - \gamma )
    }
    .
\end{equation}
In addition, we emphasize that in the setting 
of Theorem~\ref{intro:theorem} it is well known 
(cf., e.g., Cox et al.~\cite[Corollary~3.3]{CoxHutzenthalerJentzenWelti2017})
that for every $ \varepsilon \in ( 0 , \infty ) $
there exists a real number $ C_{ \varepsilon } \in [0,\infty) $
such that
for all $ N \in \N $ it holds that
\begin{equation}
\label{eq:strong_rate}
  \big(
    \ES\big[ 
      \| X_T - X^N_T \|^2_{ H_{ \iota } }
    \big]
  \big)^{ \nicefrac{ 1 }{ 2 } }
\leq
  C_{ \varepsilon }
  \cdot
  \left( 
    \lambda_N
  \right)^{
    - \left( \frac{ 1 - \gamma }{ 2 } - \varepsilon \right)
  }
  .
\end{equation}
The weak convergence rate 
$ 1 - \gamma - \varepsilon $
established
in Theorem~\ref{intro:theorem}
is thus \emph{twice} the well-known strong convergence
rate $ \frac{ 1 - \gamma - \varepsilon }{ 2 } $ in~\eqref{eq:strong_rate}.
{Moreover, Theorem~\ref{intro:theorem} is -- to the best of our knowledge -- the first result in the scientific literature
which establishes an essentially sharp weak convergence rate for numerical
approximations of the continuous version of the one-dimensional parabolic Anderson model (see Subsection~\ref{sec:anderson} for details).}
%
%
%
%
We also would like to point out that the weak convergence result in Theorem~2.2 in Debussche~\cite{Debussche2011} assumes that~\eqref{intro_eq:affine} holds (see~(2.5) in~\cite{Debussche2011}),
that $B$ maps from $H$ to $L(H)$
(instead of from $H$ to $HS(U,H_r)$ for $r\in(-\infty,-\nicefrac{1}{2})$),
and that $\varphi$, $F$, and $B$ are three times continuously Fr\'{e}chet differentiable with globally bounded derivatives 
(instead of four times continuously Fr\'{e}chet differentiable as in Theorem~\ref{intro:theorem} above) but restricts to the irregular case $\gamma=\nicefrac{1}{2}$ in the above framework.
The weak convergence result in Theorem~\ref{intro:theorem} above does not assume~\eqref{intro_eq:affine} and does assume that $\varphi$, $F$, and $B$ are four times continuously Fr\'{e}chet differentiable but also establishes essentially sharp weak convergence rates in the more regular cases $\gamma\in[0,\nicefrac{1}{2})$ such as in several cases of trace class noise. 
In the very regular case of finite dimensional SEEs it is typically assumed that $F$ and $B$ (and $\varphi$) are four times continuously differentiable (cf., e.g., Kloeden \& Platen~\cite[Theorem~9.7.4]{kp92}). 
Next we add that the proof of Theorem~\ref{intro:theorem} can in a straightforward way be extended
to the case where $ \varphi $ has at most polynomially growing derivatives.
It is, however, not clear to us how to treat the case where $ F $ and $ B $ are globally 
Lipschitz continuous but with the first four derivatives growing polynomially.
Furthermore, we emphasize that Theorem~\ref{intro:theorem}
solves the weak convergence problem emerged from 
Debussche's article (see (2.5) and Remark~2.3 in Debussche~\cite{Debussche2011})
merely in the case of spatial spectral Galerkin approximations
instead of other more complicated numerical approximations
for the SEE~\eqref{eq:SEE_intro}.
The method of proof of our weak convergence results, however, 
\emph{can be extended
to a number of other kind of spatial and temporal numerical approximations}
for SEEs of the form~\eqref{eq:SEE_intro}.
In particular, in our proceeding article~\cite{JentzenKurniawan2015arXiv} we extend the method of proof developed here to establish essentially sharp weak convergence rates for different types of temporal numerical approximations (such as exponential Euler \{see~\cite[Subsection~1.5.1]{JentzenKurniawan2015arXiv}\} and linear-implicit Euler \{see~\cite[Subsection~1.5.2]{JentzenKurniawan2015arXiv}\} approximations for SPDEs) for SPDEs with possibly non-constant diffusion coefficients without neither assuming~\eqref{intro_eq:affine} nor that $B$ maps from $H$ to $L(H)$.
Next we point out that the proof in Debussche's article~\cite{Debussche2011}
as well as many other proofs in the above mentioned weak convergence articles
use \emph{Malliavin calculus}.
Our method of proof does not use Malliavin calculus but uses -- in some sense -- merely
elementary arguments as well as the mild It\^{o} formula in 
Da Prato et al.~\cite{DaPratoJentzenRoeckner2012}. 

The paper is organized as follows. In Section~\ref{sec:sketch_proof} below we give a rough sketch of the proof of Theorem~\ref{intro:theorem} without technical details. However, the main ideas that we use to obtain an 
essentially sharp rate of convergence are highlighted in Section~\ref{sec:sketch_proof} below. 
In Section~\ref{sec:examples} we illustrate Theorem~\ref{intro:theorem} and Corollary~\ref{cor:sharp.weak.rates}, respectively, 
by two simple examples.
Sections~\ref{sec:notation} and~\ref{sec:general_setting} present 
the notation and the framework used in this paper.
Section \ref{sec:galerkin_project} studies weak convergence rates for the spectral Galerkin projections $P_N(X_T)$, $N\in\N$, associated to the solution process $X_t$, $t\in[0,T]$, of the SEE~\eqref{eq:SEE_intro}. 
The result of this section is then used in Section~\ref{sec:weak_reg} to obtain the weak convergence 
of the Galerkin approximation~\eqref{eq:Galerin_SEEs} to the solution of~\eqref{eq:SEE_intro} 
in the case where the drift operator $F$, the diffusion operator $B$, and the initial 
condition are mollified in an appropriate sense. This provides a less general version of Theorem \ref{intro:theorem}. 
Section~\ref{sec:strong_convergence} is devoted to the proof of an elementary strong convergence result. 
In Section~\ref{sec:weak_irregular} the weak convergence result from Section~\ref{sec:weak_reg} and the elementary strong convergence result from Section~\ref{sec:strong_convergence} are used to establish weak convergence 
(see Corollary~\ref{weak_irreg_B}) for general drift and diffusion operators. 
Section~\ref{sec:sharp.weak.rates} specializes the weak convergence result from Section~\ref{sec:weak_irregular} to the framework of this introductory section.
Finally, 
in Section \ref{sec:lower_bound} we consider the case $ F = 0 $ and provide examples of 
constant (additive noise) functions $ B $ which show that the weak convergence rate established 
in Theorem \ref{intro:theorem} can essentially not be improved.

\subsection{Sketch of the proof of the main weak convergence result} 
\label{sec:sketch_proof}

In the following we give a brief sketch of our method of proof of Theorem~\ref{intro:theorem} and Corollary~\ref{cor:sharp.weak.rates}, respectively, in the case where $ \xi \in H_{ \iota + 2 } $ (the case where $ \xi \in H_{ \iota } $ then follows
from a standard mollification procedure; see~\eqref{irregeq:mollified} in the proof 
of Proposition~\ref{weak_irreg} in Section~\ref{sec:weak_irregular} for details).
In our weak convergence proof we intend to work 
(as it is often the case in the case of weak convergence for S(P)DEs;
see, e.g., R\"{o}\ss ler~\cite{r03} and Debussche~\cite{Debussche2011})
with the Kolmogorov backward equation
associated to~\eqref{eq:SEE_intro}.
In the case of an SEE with a general nonlinear diffusion coefficient
it is, however, not clear whether the solutions of the SEE~\eqref{eq:SEE_intro} also provide strong solutions of
the Kolmogorov backward equation associated
to~\eqref{eq:SEE_intro};
cf.\ \cite[item~(iv) of Theorem~1.1]{AnderssonHefterJentzenKurniawan2016}, \cite[Corollary~1.2]{HefterJentzenKurniawan2017}, and~\cite[pages~249--251]{DaPrato2007}.
We therefore work with suitable mollified versions of~\eqref{eq:SEE_intro}
and~\eqref{eq:Galerin_SEEs}.
More formally,
for every $ \kappa \in ( 0, \infty ) $
let $ F_{ \kappa } \colon H_{ \iota } \to H_{\iota+2} $
and
$ B_{ \kappa } \colon H_{ \iota } \to HS( U, H_{\iota+2} ) $
be the functions which satisfy for all $ x \in H_{ \iota } $ that 
$
  F_{ \kappa }( x ) = e^{ \kappa A } F( x )
$
and
$
  B_{ \kappa }( x ) = e^{ \kappa A } B( x )
$.
For every $ \kappa \in ( 0, \infty) $, $ x \in H_{ \iota } $
let
{$
  \hat{X}^{ x, \kappa } \colon [0,T] \times \Omega \to H_{ \iota }
$}
be a continuous mild solution of the SEE
\begin{equation}
\label{eq:SEE_intro_mol}
  d \hat{X}^{ x, \kappa }_t 
  =
  \big[ 
    A \hat{X}^{ x, \kappa }_t
    +
    F_{ \kappa }( \hat{X}^{ x, \kappa }_t )
  \big] 
  \,
  dt
  +
  B_{ \kappa }( \hat{X}^{ x, \kappa }_t ) \, dW_t
  ,
  \qquad 
  t \in [0,T] 
  ,
  \qquad 
  \hat{X}_0^{ x, \kappa } = x.
\end{equation}
For every $ \kappa \in ( 0, \infty ) $
let 
$ u_{ \kappa } \colon [0,T] \times H_{ \iota } \to \R $
be the function which satisfies 
for all $ (t,x) \in [0,T] \times H_{ \iota } $ that
$
  u_{ \kappa }( t, x ) 
  =
  \ES\big[ 
    \varphi( \hat{X}^{ x, \kappa }_{ T - t } )
  \big]
$.
In particular, notice that for all $ \kappa \in (0,\infty) $ and all nonrandom $x \in H_{\iota}$
it holds that $u_{\kappa}(T,x) = \varphi(x)$.
Then, for every $ \kappa \in ( 0, \infty ) $,
$ N \in \N $
let
$ 
  X^{ N, \kappa } \colon [0,T] \times \Omega \to H_{ \iota } 
$
be a continuous mild solution of the SEE
\begin{equation}
\label{eq:Galerin_SEEs_mol}
  d X^{ N, \kappa }_t = 
  \big[ 
    P_N A X^{ N, \kappa }_t + P_N F_{ \kappa }( X^{ N, \kappa }_t ) 
  \big]
  \, dt
  +
  P_N B_{ \kappa }( X^{ N, \kappa }_t ) \, dW_t
  ,
  \;\;
  t \in [0,T] 
  ,
  \;\;
  X^{ N, \kappa }_0 = P_N( \xi )
  .
\end{equation}
The \emph{first key idea} in our proof is then to bring 
\emph{certain modified versions} of the SEEs~\eqref{eq:Galerin_SEEs}
and~\eqref{eq:Galerin_SEEs_mol} respectively into play
to analyze the weak approximation errors 
$
  \big|
    \ES\big[ 
      \varphi( \hat{X}_T^{ \xi, \kappa } )
    \big]
    -
    \ES\big[ 
      \varphi( X^{ N, \kappa }_T )
    \big]
  \big|
$
for $ N \in \N $, $ \kappa \in (0,\infty) $. 
More specifically, for every $ \kappa \in ( 0,\infty) $, $ N \in \N $ 
let $ Y^{ N, \kappa } \colon [0,T] \times \Omega \to H_{ \iota + 2 } $
be a continuous mild solution of the SEE
\begin{equation}
\label{eq:modifiedSEE}
  d Y^{ N, \kappa }_t =
  \big[ 
    A Y^{ N, \kappa }_t
    +
    F_{ \kappa }\big( 
      P_N( Y^{ N, \kappa }_t ) 
    \big)
  \big]
  \,
  dt
  +
  B_{ \kappa }\big( 
    P_N( Y^{ N, \kappa }_t ) 
  \big)
  \, dW_t
  ,
  \quad 
  t \in [0,T] 
  ,
  \quad
  Y^{N,\kappa}_0 = \xi
  .
\end{equation}
It is crucial in~\eqref{eq:modifiedSEE}
that $ P_N( \cdot ) $ appears inside the arguments of $ F_{ \kappa } $
and $ B_{ \kappa } $ instead of in front of $ F_{ \kappa } $ and $ B_{ \kappa } $ as in~\eqref{eq:Galerin_SEEs_mol} (and~\eqref{eq:Galerin_SEEs}). 
Moreover, notice the projection 
$ P_N(Y^{N,\kappa}_t) = X^{N,\kappa}_t $
$ \P $-a.s.\ for all $N \in \N$, $\kappa \in (0,\infty)$, $t \in [0,T]$.
To estimate the weak approximation errors
$
  \big|
    \ES\big[ 
      \varphi( \hat{X}_T^{ \xi, \kappa } )
    \big]
    -
    \ES\big[ 
      \varphi( X^{ N, \kappa }_T )
    \big]
  \big|
$
for $ N \in \N $, $ \kappa \in ( 0, \infty ) $
we then apply the triangle inequality to obtain 
that for all $ \kappa \in ( 0,\infty) $, $ N \in \N $ it holds that
\begin{equation}
\label{eq:triangle_intro}
\begin{split}
&
  \big|
    \ES\big[ 
      \varphi( \hat{X}_T^{ \xi, \kappa } )
    \big]
    -
    \ES\big[ 
      \varphi( X^{ N, \kappa }_T )
    \big]
  \big|
\\ & \quad 
  \leq 
  \big|
    \ES\big[ 
      \varphi( \hat{X}_T^{ \xi, \kappa } )
    \big]
    -
    \ES\big[ 
      \varphi( Y^{ N, \kappa }_T )
    \big]
  \big|
  +
  \big|
    \ES\big[ 
      \varphi( Y^{ N, \kappa }_T )
    \big]
    -
    \ES\big[ 
      \varphi( X_T^{ N , \kappa } )
    \big]
  \big|
\\ & \quad 
  =
  \big|
    u_{ \kappa }( 0, \xi )
    -
    \ES\big[
      u_{ \kappa }( T, Y^{ N, \kappa }_T )
    \big]
  \big|
  +
  \big|
    \ES\big[ 
      \varphi( Y^{ N, \kappa }_T )
    \big]
    -
    \ES\big[ 
      \varphi( P_N( Y_T^{ N , \kappa } ) )
    \big]
  \big|
\\ & \quad 
  =
  \big|
    \ES\big[
      u_{ \kappa }( T, Y^{ N, \kappa }_T )
      -
      u_{ \kappa }( 0, Y^{ N, \kappa }_0 )
    \big]
  \big|
  +
  \big|
    \ES\big[ 
      \varphi( Y^{ N, \kappa }_T )
    \big]
    -
    \ES\big[ 
      \varphi( P_N( Y_T^{ N , \kappa } ) )
    \big]
  \big|
  .
\end{split}
\end{equation}
Roughly speaking, 
the processes $ Y^{ N, \kappa } $, $ N \in \N $,
$ \kappa \in ( 0, \infty ) $,
are chosen in such a way so that 
it is not so difficult anymore to estimate
$
  \big|
    \ES\big[
      u_{ \kappa }( T, Y^{ N, \kappa }_T )
      -
      u_{ \kappa }( 0, Y^{ N, \kappa }_0 )
    \big]
  \big|
$
and
$
  \big|
    \ES\big[ 
      \varphi( Y^{ N, \kappa }_T )
    \big]
    -
    \ES\big[ 
      \varphi( P_N( Y_T^{ N , \kappa } ) )
    \big]
  \big|
$
on the right hand side of~\eqref{eq:triangle_intro}.
More formally,
to estimate the term
$
  \big|
    \ES\big[ 
      \varphi( Y^{ N, \kappa }_T )
    \big]
    -
    \ES\big[ 
      \varphi( P_N( Y_T^{ N , \kappa } ) )
    \big]
  \big|
$
on the right hand side of~\eqref{eq:triangle_intro}
(see Section~\ref{sec:galerkin_project} 
and Lemma~\ref{weak_regle0} in Section~\ref{sec:weak_reg})
we apply 
the mild It\^{o} formula in 
Corollary~2 in 
Da Prato et al.~\cite{DaPratoJentzenRoeckner2012}
to 
$
    \ES\big[ 
      \varphi( Y^{ N, \kappa }_t )
    \big]
$,
$ t \in [0,T] $,
and to
$
    \ES\big[ 
      \varphi( P_N( Y_T^{ N , \kappa } ) )
    \big]
$,
$ t \in [0,T] $,
and then estimate the difference of the resulting 
terms in a straightforward way (see the proof 
of Proposition~\ref{prop:weak_regle0} 
in Section~\ref{sec:galerkin_project} below for details).
This allows us to prove 
(see Proposition~\ref{prop:weak_regle0} below) 
that there exist
real numbers $ C^{ (1) }_{ \varepsilon } \in [0,\infty) $, $ \varepsilon \in (0,\infty) $,
such that for all 
$ \varepsilon, \kappa \in (0,\infty) $, $ N \in \N $
it holds that
\begin{equation}
\label{eq:intro_estimate1}
  \big|
    \ES\big[ 
      \varphi( Y^{ N, \kappa }_T )
    \big]
    -
    \ES\big[ 
      \varphi( X_T^{ N , \kappa } )
    \big]
  \big|
  =
  \big|
    \ES\big[ 
      \varphi( Y^{ N, \kappa }_T )
    \big]
    -
    \ES\big[ 
      \varphi( P_N( Y_T^{ N , \kappa } ) )
    \big]
  \big|
\leq
    C_{ \varepsilon }^{ (1) }
    \left( 
      \lambda_N
    \right)^{
      - ( 1 - \gamma - \varepsilon )
    }
  .
\end{equation}
To estimate the term
$
  \big|
    \ES\big[
      u_{ \kappa }( T, Y^{ N, \kappa }_T )
      -
      u_{ \kappa }( 0, Y^{ N, \kappa }_0 )
    \big]
  \big|
$
on the right hand side of~\eqref{eq:triangle_intro}
we apply the standard It\^{o} formula
to the stochastic processes
$
  \big(
    u_{ \kappa }( t, Y^{ N, \kappa }_t )
  \big)_{ 
    t \in [0,T] 
  }
$,
$ \kappa \in (0,\infty) $,
and use the fact that the functions
$ u_{ \kappa } $,
$ \kappa \in (0,\infty) $,
solve the Kolmogorov backward equation
associated to~\eqref{eq:SEE_intro_mol}
to obtain that
for all 
$ \kappa \in ( 0, \infty) $, $ N \in \N $
it holds that
\begin{align}
\label{eq:intro_Kolmogorov}
&
  \big|
    \ES\big[
      u_{ \kappa }( T, Y^{ N, \kappa }_T )
      -
      u_{ \kappa }( 0, Y^{ N, \kappa }_0 )
    \big]
  \big|
\leq
  \textstyle
  \int\limits_0^T
  \displaystyle
  \big|
  \ES\big[
    \big( 
      \tfrac{ \partial }{ \partial x } u_{ \kappa }
    \big)( s, Y^{ N, \kappa }_s )
    \big(
      F_{ \kappa }(
        P_N( Y^{ N, \kappa }_s )
      )
      -
      F_{ \kappa }(
        Y^{ N, \kappa }_s
      )
    \big)
  \big]
  \big|
  \,
  ds
\nonumber
\\ & 
  +
  \textstyle
  \sum\limits_{ b \in \mathbb{U} }
  \int\limits_0^T
  \displaystyle
  \tfrac{
  \left|
  \ES\left[
  \big( 
    \frac{ \partial^2 }{ \partial x^2 } u_{ \kappa } 
  \big)( s, Y^{ N, \kappa }_s 
  )\left(
    \big[
      B_{ \kappa }(
        P_N( Y^{ N, \kappa }_s )
      )
      +
      B_{ \kappa }(
        Y^{ N, \kappa }_s 
      )
    \big]
    b
    ,
    \big[
      B_{ \kappa }(
        P_N( Y^{ N, \kappa }_s )
      )
      -
      B_{ \kappa }(
        Y^{ N, \kappa }_s 
      )
    \big]
    b
  \right)
  \right]
  \right|
  }{ 2 }
  \,
  ds
\end{align}
where $ \mathbb{U} \subseteq U $
is an arbitrary orthonormal basis of $ U $;
cf.\ \eqref{eq:Kolmogorov} in Section~\ref{sec:weak_reg} below.
The next key idea in our weak convergence proof 
is then to again apply the mild It\^{o} formula
(see Da Prato et al.~\cite{DaPratoJentzenRoeckner2012})
to the terms appearing on the right hand side of~\eqref{eq:intro_Kolmogorov}.
After applying the mild It\^{o} formula, 
the resulting terms can
be estimated in a straightforward way 
by using the estimates 
for the functions
$
  u_{ \kappa }
$,
$ \kappa \in (0,T] $,
from Andersson et al.~\cite{AnderssonHefterJentzenKurniawan2016}.
This allows us 
(cf.\ \eqref{eq:weak_regle} in Lemma~\ref{weak_regle} and~\eqref{irreg_2}--\eqref{irreg_5} in the proof of Proposition~\ref{weak_irreg})
to prove that for all 
$ \varepsilon \in ( 0, \infty ) $
there exists a real number 
$ C^{ (2) }_{ \varepsilon } \in [0,\infty) $
such that for all 
$ \kappa \in (0,T] $, 
$ N \in \N $
it holds that
\begin{equation}
\label{eq:intro_estimate2}
  \big|
    \ES\big[ 
      \varphi( \hat{X}_T^{ \xi, \kappa } )
    \big]
    -
    \ES\big[ 
      \varphi( Y^{ N, \kappa }_T )
    \big]
  \big|
=
  \big|
    \ES\big[
      u_{ \kappa }( T, Y^{ N, \kappa }_T )
      -
      u_{ \kappa }( 0, Y^{ N, \kappa }_0 )
    \big]
  \big|
\leq
  \tfrac{
    C^{ (2) }_{ \varepsilon }
  }{
    \kappa^{ \varepsilon }
    \,
    \left( \lambda_N \right)^{ ( 1 - \gamma - \varepsilon ) }
  }
  .
\end{equation}
Putting~\eqref{eq:intro_estimate2}
and~\eqref{eq:intro_estimate1}
into~\eqref{eq:triangle_intro}
then proves that for all 
$ \varepsilon \in (0,\infty) $, 
$ \kappa \in (0,T] $, $ N \in \N $
it holds that
\begin{equation}
\label{eq:estimate_no_mollifying}
  \big|
    \ES\big[ 
      \varphi( \hat{X}_T^{ \xi, \kappa } )
    \big]
    -
    \ES\big[ 
      \varphi( X^{ N, \kappa }_T )
    \big]
  \big|
\leq  
  C^{ (2) }_{ \varepsilon }
  \,
  \kappa^{ - \varepsilon }
  \left( \lambda_N \right)^{ - ( 1 - \gamma - \varepsilon ) }
  +
  C_{ \varepsilon }^{ (1) }
  \left( 
    \lambda_N
  \right)^{
    - ( 1 - \gamma - \varepsilon )
  }
  .
\end{equation}
Estimates~\eqref{eq:intro_estimate2} and~\eqref{eq:estimate_no_mollifying}
illustrate that we cannot simply let the 
mollifying parameter $ \kappa $ tend 
to $ 0 $ because the right hand side 
of~\eqref{eq:estimate_no_mollifying}
diverges as
$ \kappa $ tends to $ 0 $.
The last key idea in our proof is then to make use of the following 
\emph{-- somehow nonstandard -- mollification procedure} to overcome this problem.
For this mollification procedure we first 
use well-known \emph{strong convergence analysis}
to prove 
(cf.\ Proposition~\ref{strong} and Corollary~\ref{cor:strong.convergence.simplified} in Section~\ref{sec:strong_convergence})
that for all $ \varepsilon \in ( 0, \infty ) $
there exists 
a real number 
$ C^{ (3) }_{ \varepsilon } \in [0,\infty) $
such that
for all 
$ \kappa \in ( 0, T] $, $ N \in \N $
it holds that
\begin{equation}
\label{eq:intro_estimate3}
  \big|
    \ES\big[ 
      \varphi( X_T )
    \big]
    -
    \ES\big[ 
      \varphi( \hat{X}_T^{ \xi, \kappa } )
    \big]
  \big|
  +
  \big|
    \ES\big[ 
      \varphi( X_T^N )
    \big]
    -
    \ES\big[ 
      \varphi( X_T^{ N, \kappa } )
    \big]
  \big|
\leq
  C^{ (3) }_{ \varepsilon }
  \,
  \kappa^{ 
    \left(
      \frac{ 1 - \gamma }{ 2 } - \varepsilon 
    \right)
  }
  .
\end{equation}
Combining~\eqref{eq:intro_estimate3}
with~\eqref{eq:estimate_no_mollifying} 
then shows that
for all
$ \varepsilon \in (0,\infty) $,
$ \kappa \in (0,T] $,
$ N \in \N $
it holds that
\begin{equation}
\label{eq:intro_estimate4}
\begin{split}
  \big|
    \ES\big[ 
      \varphi( X_T )
    \big]
    -
    \ES\big[ 
      \varphi( X^N_T )
    \big]
  \big|
& \leq
  \frac{
    C_{ \varepsilon }^{ (1) }
  }{
    \left( 
      \lambda_N
    \right)^{
      ( 1 - \gamma - \varepsilon )
    }
  }
  +
  \frac{
    C^{ (2) }_{ \varepsilon }
  }{
    \kappa^{ \varepsilon }
    \left( \lambda_N 
    \right)^{ ( 1 - \gamma - \varepsilon ) }
  }
  +
  C^{ (3) }_{ \varepsilon }
  \,
  \kappa^{ 
    \left(
      \frac{ 1 - \gamma }{ 2 } - \varepsilon 
    \right)
  }
  .
\end{split}
\end{equation}
As the left hand side of~\eqref{eq:intro_estimate4}
is independent of $ \kappa \in (0,T] $,
we can minimize the right hand side of~\eqref{eq:intro_estimate4} over 
$ \kappa \in (0,T] $
(instead of letting $ \kappa $ tend to $ 0 $)
and this will allow us to complete the proof
of Theorem~\ref{intro:theorem}; see~\eqref{irreg_6}
and~\eqref{irreg_6b} in the proof 
of Proposition~\ref{weak_irreg}
in Section~\ref{sec:weak_irregular}
below for details.

\subsection{Examples}
\label{sec:examples}

In this section we illustrate Theorem~\ref{intro:theorem} and Corollary~\ref{cor:sharp.weak.rates}, respectively, 
by two simple examples. In Subsection~\ref{sec:anderson}
we apply Theorem~\ref{intro:theorem} to
the continuous version of the one-dimensional parabolic Anderson model
and in Subsection~\ref{sec:cahn} we apply Theorem~\ref{intro:theorem}
to a Cahn-Hilliard-Cook type equation.

\subsubsection{Parabolic Anderson model and nonlinear heat-type SPDEs}
\label{sec:anderson}

Let $ H = L^2( (0,1) ; \R ) $
be the $ \R $-Hilbert space of equivalence classes
of Lebesgue-Borel square integrable functions from $ ( 0,1 ) $ to $ \R $, 
let 
$ T, \kappa, \delta, \nu \in (0,\infty) $, 
$ \xi \in H $, 
let 
$
  ( \Omega, \mathcal{F}, \P , 
$
$  
  ( \mathcal{F}_t )_{ t \in [0,T] } )
$
be a stochastic basis, 
let 
$
  ( W_t )_{ t \in [0,T] }
$
be an $ \operatorname{Id}_H $-cylindrical $
( \Omega, \mathcal{F}, \P , ( \mathcal{F}_t )_{ t \in [0,T] } )
$-Wiener process, 
let 
$
  e_n \in H
$, 
$ n \in \N $, 
be the orthonormal basis of $ H $ 
which satisfies for all $ n \in \N $ that 
$
  e_n
  =
  \sqrt{2} \, \sin( n \pi (\cdot) )
$, 
let $ A \colon D(A) \subseteq H \rightarrow H $ 
be the linear operator which satisfies
$
  D(A)
  =
  \{
    v \in H
    \colon
    \sum^\infty_{ n=1 }
    n^4
    |\langle
    e_n, v
    \rangle_H|^2
    < \infty
  \}
$ 
and 
$ 
  \forall \, v \in D(A) \colon
  Av
  =
  \sum^\infty_{ n=1 }
  -\nu n^2 \pi^2
  \langle
  e_n, v
  \rangle_H
  \,
  e_n
$, 
let
$  
  ( 
    H_r 
    ,
    \left< \cdot, \cdot \right>_{ H_r }
    ,$ $
    \left\| \cdot \right\|_{ H_r } 
  )
$, 
$ r \in \R $, 
be a family of interpolation spaces associated to $ - A $, 
let
$
  ( P_N )_{ N \in \N } \subseteq L( H_{ - 1 } )
$ 
satisfy for all 
$ N \in \N $, $ v \in H $ 
that 
$
  P_N(v) = \sum^N_{n=1} \left< e_n, v \right>_H e_n
$, 
let 
$\psi\colon H\to H$
be a four times continuously Fr\'{e}chet differentiable function with globally bounded derivatives,
and let 
$
  B \colon H \to HS( H, H_{ - 1 / 4 - \delta } )
$
be the function which satisfies for all 
$ v \in H $
and all uniformly continuous functions  
$ u \colon (0,1) \to \R $
that 
$
  B(v)u
  =
  \psi(v) \cdot u
$. 
The above 
assumptions 
ensure the existence of 
$
  ( \mathcal{F}_t )_{ t \in [0,T] }
$-adapted continuous stochastic processes
$ X \colon [0,T] \times \Omega \to H $ 
and
$
X^N \colon [0,T] \times \Omega \rightarrow P_N( H )
$,
$ N \in \N $,  
which satisfy that for all 
$N\in\N$, $ t \in [0,T] $ 
it holds $\P$-a.s.\ that
\begin{equation}
\label{eq:SEE.heat}
  X_t = e^{ At } \xi + \int_0^t e^{ A( t - s ) } B( X_s ) \, dW_s
\end{equation}
and
$
X^N_t = e^{ At } P_N(\xi) + \int_0^t e^{ A( t - s ) } P_N B( X^N_s ) \, dW_s
$.
In the case where 
$
  \forall \, v\in H \colon 
  \psi(v)=
  \big[1+\|v\|^2_H\big]^{-1} \, v
$
the stochastic process $ X $ is a mild solution process of 
\begin{align}
\label{eq:nonlinear.heat}
&
  dX_t(x)
  = 
  \nu
  \,
  \tfrac{ \partial^2 }{ \partial x^2 } 
  X_t(x)
  \, dt 
  + 
  \frac{X_t(x)}{
  1
  +
  \int^1_0
  |X_t(y)|^2
  \, dy
  }
  \, dW_t(x)
\end{align}
with 
$  X_t(0) = X_t(1) = 0$
and
$ X_0( x ) = \xi( x ) $
for $ x \in ( 0, 1 ) $, $ t \in [ 0, T ] $
and the stochastic processes $X^N$, $N\in\N$, 
are spatial spectral Galerkin approximations of~\eqref{eq:nonlinear.heat}.
In the case where 
$
  \forall \, v\in H \colon 
  \psi(v)=\kappa\cdot v
$
the stochastic process $ X $ is a mild solution process of the 
continuous version of the one-dimensional parabolic Anderson model
\begin{align}
\label{eq:parabolic_Anderson}
&
  dX_t(x)
  = 
  \nu
  \,
  \tfrac{ \partial^2 }{ \partial x^2 } 
  X_t(x)
  \, dt 
  + 
  \kappa \,
  X_t(x)  
  \, dW_t(x)
\end{align}
with 
$  X_t(0) = X_t(1) = 0$
and
$ X_0( x ) = \xi( x ) $
for $ x \in ( 0, 1 ) $, $ t \in [ 0, T ] $
(cf., e.g., Carmona \& Molchanov~\cite{CarmonaMolchanov1994})
and the stochastic processes $X^N$, $N\in\N$, 
are spatial spectral Galerkin approximations of~\eqref{eq:parabolic_Anderson}.
Theorem~\ref{intro:theorem} and Corollary~\ref{cor:sharp.weak.rates}, respectively, apply here with $ \gamma = \nicefrac{ 1 }{ 2 } $, 
that is, Theorem~\ref{intro:theorem} and Corollary~\ref{cor:sharp.weak.rates}, respectively, ensure that
for all 
$
  \varphi \in C^4_b( H, \R )
$,
$ \varepsilon \in ( 0, \infty ) $ 
it holds that there exists a real number
$ C \in \R $
such that for all $ N \in \N $
it holds that
\begin{equation}
  \left|
    \ES\big[ 
      \varphi( X_T )
    \big]
    -
    \ES\big[ 
      \varphi( X^N_T )
    \big]
  \right|
\leq
    C
    \cdot
    N^{
      - ( 1 - \varepsilon )
    }
  .
\end{equation}
Theorem~\ref{intro:theorem} and Corollary~\ref{cor:sharp.weak.rates}, respectively, thus demonstrate that the spatial spectral Galerkin approximations
$X^N$, $N\in\N$, 
of~\eqref{eq:SEE.heat}, \eqref{eq:nonlinear.heat}, and~\eqref{eq:parabolic_Anderson}, respectively, converge with rate
$1-\varepsilon$ 
to the stochastic process $X$ of~\eqref{eq:SEE.heat}, \eqref{eq:nonlinear.heat}, and~\eqref{eq:parabolic_Anderson}.
To the best of our knowledge, Theorem~\ref{intro:theorem} and Corollary~\ref{cor:sharp.weak.rates}, respectively, are the first results in the scientific literature which establish essentially sharp weak convergence rates for numerical approximations of~\eqref{eq:nonlinear.heat} and~\eqref{eq:parabolic_Anderson}, respectively.

\subsubsection{A Cahn-Hilliard-Cook type equation}
\label{sec:cahn}

Let $ H = L^2( (0,1) ; \R ) $
be the $ \R $-Hilbert space of equivalence classes
of Lebesgue-Borel square integrable functions from $ ( 0,1 ) $ to $ \R $, 
let 
$ T, \kappa, \delta \in (0,\infty) $, 
$ \xi \in H $, 
let 
$
  ( \Omega, \mathcal{F}, \P , 
$
$
  ( \mathcal{F}_t )_{ t \in [0,T] } )
$
be a stochastic basis, 
let 
$
( W_t )_{ t \in [0,T] }
$
be an $ \operatorname{Id}_H $-cylindrical $
( \Omega, \mathcal{F}, \P , ( \mathcal{F}_t )_{ t \in [0,T] } )
$-Wiener process, 
let 
$
  e_n \in H
$, 
$ n \in \N_0 $, 
be the orthonormal basis of $ H $ which satisfies 
for all $ n \in \N $ that 
$
  e_0
  =
  1
$
and 
$
  e_n
  =
  \sqrt{2} \, \cos( n \pi (\cdot) )
$, 
let $ A \colon D(A) \subseteq H \rightarrow H $ 
be the linear operator which satisfies
$
  D(A)
  =
  \{
    v \in H
    \colon
    \sum_{ n \in \N }
    n^8
    |\langle
    e_n, v
    \rangle_H|^2
    < \infty
  \}
$ 
and 
$ 
  \forall \, v \in D(A) \colon
  Av
  =
  \sum^\infty_{ n = 0 }
  ( n^2 \pi^2 - n^4 \pi^4 - 1 )
  \langle
  e_n, v
  \rangle_H
  \,
  e_n
$, 
let
$  
  ( 
    H_r 
    ,
    \left< \cdot, \cdot \right>_{ H_r }
    ,$ $
    \left\| \cdot \right\|_{ H_r } 
  )
$, 
$ r \in \R $, 
be a family of interpolation spaces associated to $ - A $, 
let
$
  ( P_N )_{ N \in \N } \subseteq L( H_{ - 1 } )
$ 
satisfy for all 
$ N \in \N $, $ v \in H $ 
that 
$
  P_N(v) = \sum^N_{n=0} \left< e_n, v \right>_H e_n
$, 
and let 
$
  F \colon H \to H_{ -1/4 - 2\delta }
$
and 
$
  B \colon H \to HS( H, H_{ -1/8 - \delta } ) 
$
satisfy for all 
$ v \in H $ 
and all uniformly continuous functions
$ u \colon (0,1) \to \R $
that 
$F(v)=v$ 
and 
$
  B(v)u
  =
  \kappa \cdot v \cdot u
$. 

The above assumptions 
ensure the existence of
$
  ( \mathcal{F}_t )_{ t \in [0,T] }
$-adapted continuous stochastic processes
$ X \colon [0,T] \times \Omega \to H $ 
and
$
X^N \colon [0,T] \times \Omega \rightarrow P_N( H )
$,
$ N \in \N $,  
which satisfy that for all 
$N\in\N$, $ t \in [0,T] $ it holds $\P$-a.s.\ that
\begin{equation}
X_t = e^{ At } \xi
+
\int_0^t e^{ A( t - s ) } F(X_s) \, ds
+
\int_0^t e^{ A( t - s ) } B( X_s ) \, dW_s
\end{equation}
and 
$
X^N_t = e^{ At } P_N(\xi) 
+
\int_0^t e^{ A( t - s ) } P_N F(X^N_s) \, ds
+
\int_0^t e^{ A( t - s ) } P_N B( X^N_s ) \, dW_s
$.
The stochastic process $ X $ is thus a solution process of the 
\textit{Cahn-Hilliard-Cook type equation}
\begin{equation}
\label{eq:linear_Cahn-Hilliard-Cook}
\begin{split}
&
dX_t(x)= 
\big[
-\tfrac{\partial^4}{\partial x^4} \,
X_t(x)
-
\tfrac{\partial^2}{\partial x^2} \,
X_t(x)
\big]
\,dt 
+ 
\kappa \,
X_t(x) 
\,dW_t(x) 
\end{split}
\end{equation}
with 
$
  X'_t(0) = X'_t(1) = X^{(3)}_t(0) = X^{(3)}_t(1) = 0
$
and 
$
  X_0(x)=\xi(x)
$
for $ x \in ( 0, 1 ) $, $ t \in [ 0, T ] $
and the stochastic processes $X^N$, $N\in\N$, 
are spatial spectral Galerkin approximations of~\eqref{eq:linear_Cahn-Hilliard-Cook}.
Theorem~\ref{intro:theorem} and Corollary~\ref{cor:sharp.weak.rates}, respectively, apply here with $ \gamma = \nicefrac{ 1 }{ 4 } $, that is,
Theorem~\ref{intro:theorem} and Corollary~\ref{cor:sharp.weak.rates}, respectively, ensure that 
for all 
$
  \varphi \in C^4_b( H, \R )
$,
$ \varepsilon \in ( 0, \infty ) $ 
it holds that there exists a real number
$ C \in \R $
such that for all $ N \in \N $
it holds that
\begin{equation}
  \left|
    \ES\big[ 
      \varphi( X_T )
    \big]
    -
    \ES\big[ 
      \varphi( X^N_T )
    \big]
  \right|
\leq
    C
    \cdot
    N^{
      - ( 3 - \varepsilon )
    }
  .
\end{equation}
Theorem~\ref{intro:theorem} and Corollary~\ref{cor:sharp.weak.rates}, respectively, thus demonstrate that the spatial spectral Galerkin approximations
$X^N$, $N\in\N$, 
of~\eqref{eq:linear_Cahn-Hilliard-Cook} converge with rate
$3-\varepsilon$ 
to the solution process $X$ of~\eqref{eq:linear_Cahn-Hilliard-Cook}.
To the best of our knowledge, Theorem~\ref{intro:theorem} and Corollary~\ref{cor:sharp.weak.rates}, respectively, are the first results in the scientific literature which establish essentially sharp weak convergence rates for numerical approximations of~\eqref{eq:linear_Cahn-Hilliard-Cook}.

\subsection{Notation}
\label{sec:notation}

Throughout this article the following notation is used. 
For every set $ S $ we denote by 
$
  \operatorname{Id}_S \colon S \rightarrow S
$ 
the identity mapping on $ S $. 
For every set $ S $ 
we denote by 
$ \mathcal{P}(S) $ the power set of $ S $.
We denote by
$
  \mathcal{E}_r \colon [0,\infty) \rightarrow [0,\infty)
$,
$
  r \in (0,\infty)
$,
the functions which satisfy for all 
$ r \in (0,\infty) $, 
$ x \in [0,\infty) $
that 
$
  \mathcal{E}_r(x)
  =
  \big[
    \sum^{ \infty }_{ n = 0 }
    \frac{ 
        x^{ 2 n }
        \,
        \Gamma(r)^n
    }{
      \Gamma( n r + 1 ) 
    }
  \big]^{
    \nicefrac{ 1 }{ 2 }
  }
$
(generalized exponential function; cf., e.g., Exercise~3 in Chapter~7 in Henry~\cite{h81},
(1.0.1) in Chapter~1 in Gorenflo et al.~\cite{Gorenfloetal2014}, 
and~(16) in Andersson et al.~\cite{AnderssonJentzenKurniawan2016arXiv}).
For all
normed $ \R $-vector spaces 
$
  ( E_1, \left\|\cdot\right\|_{ E_1 } )
$
and
$
  ( E_2, \left\|\cdot\right\|_{ E_2 } )
$
and every nonnegative integer $ k \in \N_0 $
we denote by
$ 
  \left| \cdot \right|_{ 
    \operatorname{Lip}^k( E_1, E_2 ) 
  }
  ,
  \left\| \cdot \right\|_{ 
    \operatorname{Lip}^k( E_1, E_2 ) 
  }
  \colon
$
$
  C^k( E_1, 
  E_2 )
  \to [0,\infty]
$
the functions which satisfy 
for all $ f \in C^k( E_1, E_2 ) $ that 
\begin{equation}
  \left| f \right|_{
    \operatorname{Lip}^k( E_1, E_2 ) 
  }
  =
  \begin{cases}
    \sup_{ 
      \substack{
        x, y \in E_1 ,
      \\
        x \neq y
      }
    }
    \tfrac{
      \| f( x ) - f( y ) \|_{ E_2 }
    }{
      \left\| x - y \right\|_{ E_1 }
    }
  &
    \colon
    k = 0
  \\
    \sup_{ 
      \substack{
        x, y \in E_1 ,
      \\
        x \neq y
      }
    }
    \tfrac{
      \| f^{ (k) }( x ) - f^{ (k) }( y ) \|_{ L^{ (k) }( E_1, E_2 ) }
    }{
      \left\| x - y \right\|_{ E_1 }
    }
  &
    \colon
    k \in \N
  \end{cases}
\end{equation}
and 
$
  \left\| f \right\|_{
    \operatorname{Lip}^k( E_1, E_2 )
  }
  =
  \left\| f( 0 ) \right\|_{ E_2 }
  +
  \sum_{ l = 0 }^k
  \left|
    f
  \right|_{
    \operatorname{Lip}^l( E_1, E_2 )
  }
$
and we denote by 
$
  \operatorname{Lip}^k( E_1, 
  E_2 ) 
$
the set 
given by
$
  \operatorname{Lip}^k( E_1, E_2 ) 
  = 
  \{ 
    f \in C^k( E_1, E_2 ) \colon 
    \left\| f \right\|_{
      \operatorname{Lip}^k( E_1, E_2 )
    }
    < \infty
  \} 
$.
For all normed $ \R $-vector spaces 
$
  ( E_1, \left\|\cdot\right\|_{ E_1 } )
$
and
$
  ( E_2, \left\|\cdot\right\|_{ E_2 } )
$
and every natural number $ k \in \N $
we denote by 
$
  \left|
    \cdot
  \right|_{
    C^k_b(E_1,E_2)
  },
  \left\|
    \cdot
  \right\|_{ C^k_b(E_1,E_2) }
  \colon 
  C^k(E_1,E_2)\rightarrow[0,\infty]
$
the functions which satisfy 
for all 
$
  f \in C^k(E_1,E_2)
$
that 
$
  | f |_{ 
    C^k_b( E_1, E_2)
  }
=
  \sup_{ x \in E_1 }
  \|
    f^{ ( k ) }( x ) 
  \|_{
    L^{ (k) }( E_1, E_2 ) 
  }
$
and
$
  \| f \|_{
    C^k_b(E_1,E_2)
  }
=
  \| f(0) \|_{ E_2 } +
  \sum^k_{ l = 1 }
  | f |_{ C^l_b(E_1,E_2) }
$
and we denote by 
$
  C^k_b( E_1, E_2 ) 
$
the set given by
$
  C^k_b( E_1, E_2 ) 
  = 
  \{ 
    f \in C^k( E_1, E_2 ) 
    \colon 
    \| f \|_{
      C^k_b( E_1, E_2 )
    }
    < \infty
  \} 
$.


\subsection{Setting}
\label{sec:general_setting}

Throughout this article the following setting 
is frequently used.
Consider the notation in Section~\ref{sec:notation}, 
let 
$ ( H, \langle \cdot, \cdot \rangle_H, \left\| \cdot \right\|_H ) $
and
$ ( U, \langle \cdot, \cdot \rangle_U, \left\| \cdot \right\|_U ) $
be separable $ \R $-Hilbert spaces,
let $ T \in (0,\infty) $, 
let 
$ ( \Omega, \mathcal{F}, \P, ( \mathcal{F}_t )_{ t \in [0,T] } ) $
be a stochastic basis,
let $ ( W_t )_{ t \in [0,T] } $ be an $ \operatorname{Id}_U $-cylindrical
$ ( \Omega, \mathcal{F}, \P, ( \mathcal{F}_t )_{ t \in [0,T] } ) $-Wiener process,
%
%
let $ \set\subseteq H $ be a nonempty orthonormal basis,
let $ \lambda\colon\set\rightarrow\R $ be a function
satisfying
$ 
  \sup_{ b \in \set } \lambda_b < 0
$,
let 
$ A \colon D(A) \subseteq H \to H $ be a linear operator 
which satisfies
$
  D( A ) 
  = 
  \{ 
    v \in H 
  \colon 
    \sum_{ b \in \set }
    \left| 
      \lambda_b
      \langle b, v \rangle_H 
    \right|^2
  < \infty
  \}
$
and 
$ 
  \forall \, v \in D(A) \colon
  A v = 
  \sum_{ b \in \set }
  \lambda_b
  \langle b, v \rangle_H
  b
$,
let
$  
  ( 
    H_r 
    ,
    \langle \cdot, \cdot \rangle_{ H_r }
    ,$ $
    \left\| \cdot \right\|_{ H_r } 
  )
$,
$ r \in \R $,
be a family of interpolation spaces associated to $ - A $,
and let
$ 
  ( P_I )_{ I \in \mathcal{P}( \set ) } \subseteq L( H_{ - 1 } ) 
$
satisfy for all
$ v \in H $,
$ I \in \mathcal{P}( \set ) $
that
$ 
  P_I( v ) 
  = \sum_{ b \in I } \left< b, v \right>_H b
$.

\subsection{Auxiliary lemmas}
Throughout this article we frequently use the following well-known lemmas.

\begin{lemma}
\label{lem:aux}
Assume the setting in Section~\ref{sec:general_setting}. 
Then it holds for all $ r \in [0, 1] $ that
$
\sup_{ t \in [0,\infty) } 
\left\|
( - t A )^r
e^{ A t }
\right\|_{ L(H) }
\leq 
\sup_{ x \in (0,\infty) 
}
\left[
\frac{
	x^r
}{
e^x
}
\right]
\leq
\left[ 
\frac{ r }{ e }
\right]^r
\leq
1
$.
\end{lemma}

\begin{lemma}[See, e.g., Lemma~2.2 in Andersson et al.~\cite{AnderssonJentzenKurniawan2016arXiv}]
\label{lem:Kuratowski}
Let 
$ (V_k, \left\|\cdot\right\|_{V_k}) $, $ k \in \{0,1\} $, 
be separable $\R$-Banach spaces with 
$ V_1 \subseteq V_0 $ continuously. 
Then 
\begin{equation}
\label{eq:Kuratowski}
  \mathcal{B}(V_1) =
  \{
    B \in \mathcal{P}(V_1) \colon
    (
      \exists \, A \in \mathcal{B}(V_0) 
      \colon
      B = A \cap V_1
    )
  \} 
  \subseteq \mathcal{B}(V_0)
  .
\end{equation}
\end{lemma}

\section{Weak convergence for Galerkin projections of SEEs}
\label{sec:galerkin_project}

In this section we establish weak convergence rates for \emph{Galerkin projections} 
of SEEs (see Proposition~\ref{prop:weak_regle0} below). 
More specifically, in the framework of Section~\ref{sec:general_setting} we establish in Proposition~\ref{prop:weak_regle0} below an explicit upper bound for the weak approximation error
\begin{equation}
  \left|
    \ES\big[ 
      \varphi( X_T )
    \big]
    -
    \ES\big[ 
      \varphi\big( 
        P_I( X_T ) 
      \big)
    \big]
  \right|,
\end{equation}
where 
$
  I \subseteq \mathbb{H}
$ 
is a set, 
where $\varphi\colon H\to\R$ is a twice continuously Fr\'{e}chet differentiable function with globally bounded and globally Lipschitz continuous derivatives, and where 
$
  X \colon [0,T] \times \Omega \to H
$
is a suitable mild solution process of the SEE~\eqref{eq:SEE.lip}.
In this section the nonlinearities in the SEE~\eqref{eq:SEE.lip} are not mollified and may take values in appropriate negative interpolation spaces.
Proposition~\ref{prop:weak_regle0}, in particular, proves inequality~\eqref{eq:intro_estimate1} in 
Section~\ref{sec:sketch_proof}.
In Corollary~\ref{cor:weak_reg} in Section~\ref{sec:weak_reg} below we will use Proposition~\ref{prop:weak_regle0}
to establish weak convergence rates for \emph{Galerkin approximations} of SEEs with mollified nonlinearities.
In particular, in Section~\ref{sec:weak_reg} we establish upper error bounds for the first summand on the right hand side of~\eqref{eq:triangle_intro} (see Lemma~\ref{weak_regle} in Subsection~\ref{sec:weak_regular} below) and we use these upper error bounds together with Proposition~\ref{prop:weak_regle0} in this section to obtain upper error bounds for the left hand side of~\eqref{eq:triangle_intro}.
Proposition~\ref{prop:weak_regle0} is a slightly modified version of Corollary~8 in 
Da Prato et al.~\cite{DaPratoJentzenRoeckner2012}.

\subsection{Setting}
\label{sec:setting_galerkin_project}

Assume the setting in Section~\ref{sec:general_setting}
and let 
$ \vartheta \in [0,1) $,
$
  F \in 
  \operatorname{Lip}^0( H , H_{ - \vartheta } ) 
$, 
$
  B \in 
  \operatorname{Lip}^0( 
    H, 
    HS( 
      U, 
      $
      $
      H_{ 
        - \nicefrac{ \vartheta }{ 2 } 
      } 
    ) 
  ) 
$, 
$
  \varphi \in \operatorname{Lip}^2( H, \R)
$, 
$
  \xi \in \lpn{3}{\P|_{ \mathcal{F}_0 } }{ H }
$.

The above assumptions 
ensure that there exists an up-to-modifications unique 
$
  ( \mathcal{F}_t )_{ t \in [0,T] }
$-predictable stochastic process
$ 
  X \colon [0,T] \times \Omega \to H
$
which satisfies 
$
  \sup_{ t \in [0,T] }
  \| X_t \|_{ \lpn{3}{\P}{H} } 
  < \infty
$ 
and which satisfies that
for all $ t \in [0,T] $
it holds $ \P $-a.s.\ that
\begin{equation}
\label{eq:SEE.lip}
  X_t
  = 
    e^{ A t } \xi 
  + 
    \int_0^t e^{ A ( t - s ) } F( X_s ) \, ds
  + 
    \int_0^t e^{ A ( t - s ) } B( X_s ) \, dW_s 
  .
\end{equation}

\subsection{A weak convergence result}

\begin{proposition}
\label{prop:weak_regle0}
Assume the setting in Section~\ref{sec:setting_galerkin_project} and let $\rho\in[0,1-\vartheta)$, $I \in \grid(\set)$. Then
\begin{equation}
\begin{split}
&
  \left|
    \ES\big[ 
      \varphi( X_T )
    \big]
    -
    \ES\big[ 
      \varphi\big( 
        P_I( X_T ) 
      \big)
    \big]
  \right|
\leq
  \| \varphi \|_{ 
    \operatorname{Lip}^2( H, \R )
  }
  \,
    \max\!\left\{ 
      1 ,
      \sup\nolimits_{ t \in [0,T] }
      \ES\big[
        \| X_t \|^3_H
      \big]
    \right\}
\\ & 
  \cdot
  \left[ 
    \frac{ 1 }{ T^{ \rho } }
    +
    \frac{
      T^{ ( 1 - \rho - \vartheta ) }
      \big[ 
        \| F \|_{
          \operatorname{Lip}^0( H, H_{ - \vartheta } )
        }
        +
        \| B \|^{2}_{
          \operatorname{Lip}^0( H,
            HS( U, H_{ - \vartheta / 2 } ) 
          )
        }
      \big]
    }{
      \left( 1 - \rho - \vartheta \right)
    }
  \right]
  \|P_{\set\backslash I}\|_{L(H,H_{-\rho})}
  .
\end{split}
\end{equation}
\end{proposition}

\begin{proof}
Throughout this proof 
let $ \mathbb{U} \subseteq U $ be an orthonormal basis of $ U $
and let 
$ B^b \in C( H, H_{ - \vartheta / 2 } ) $, $ b \in \mathbb{U} $,
be the functions which satisfy for all 
$ b \in \mathbb{U} $, 
$ v \in H $
that
$
  B^b( v ) 
  =
  B( v ) \, b
$.
Next observe that for all $ t \in [0,T] $
it holds $ \P $-a.s.\ that
$
  P_I( X_t )
  = 
    e^{ A t } P_I( \xi )
  + 
    \int_0^t e^{ A ( t - s ) } P_I F( X_s ) \, ds
  + 
    \int_0^t e^{ A ( t - s ) } P_I B( X_s ) \, dW_s 
$.
The mild It\^{o} formula 
in Corollary~2 in Da Prato et al.~\cite{DaPratoJentzenRoeckner2012} 
hence yields that 
%
\begin{equation}
\begin{split}
&
  \ES\big[\varphi( X_T )\big] - \ES\big[\varphi( P_I( X_T ) )\big] 
  = \ES\big[\varphi( e^{AT}\xi )\big] - \ES\big[\varphi( e^{ A T } P_I( \xi ) )\big]
\\
&
  + \int^T_0 \ES\!\left[ \varphi'( e^{ A(T-t) } X_t ) \, e^{ A (T-t) } F(X_t) \right] 
  - 
  \ES\!\left[ \varphi'( e^{ A(T-t) } P_I( X_t ) ) \, e^{ A (T-t) } P_I F( X_t ) \right] dt
\\
&
  + 
  \frac{ 1 }{ 2 }
  \sum\limits_{ b \in \mathbb{U} }
  \int^T_0 
    \ES\!\left[\varphi''( e^{ A ( T - t ) } X_t ) ( e^{ A (T - t) } B^b( X_t ) , e^{ A (T - t) } B^b( X_t ) ) \right] 
  dt
\\
&
  - 
  \frac{ 1 }{ 2 }
  \sum\limits_{ b \in \mathbb{U} }
  \int^T_0 
    \ES\!\left[
      \varphi''( e^{ A (T - t) } P_I( X_t ) ) 
      (
        e^{ A (T - t) } P_I B^b( X_t ) , 
        e^{ A ( T - t ) } P_I B^b( X_t ) 
      )
    \right]
  dt
  .
\end{split}\label{weakreg_1}
\end{equation}
Next observe that Lemma~\ref{lem:aux} implies that
%
\begin{equation}
\label{weakreg_2}
  \left|
    \ES\big[\varphi( e^{ A T } \xi )\big] 
    - 
    \ES\big[\varphi( e^{ A T } P_I( \xi ) )\big]
  \right|
\leq
  \frac{
    | \varphi |_{ \operatorname{Lip}^0( H, \R ) }
    \,
    \ES\big[\| \xi \|_H\big]
    \,
    \|
      P_{ \set \backslash I } 
    \|_{ L( H, H_{ - \rho } ) } 
  }{
    T^\rho
  }
  .
\end{equation}
Inequality~\eqref{weakreg_2}
provides us a bound for the first difference on the right hand 
side of~\eqref{weakreg_1}.
In the next step we bound
the second difference on the right hand side of~\eqref{weakreg_1}.
For this observe that 
for all $ x \in H $, $ t \in [0,T) $ 
it holds that
\begin{equation}
\label{eq:F_term1}
\begin{split}
&
  \left|
    \left[ 
      \varphi'( e^{ A (T - t) } x ) - 
      \varphi'( e^{ A (T - t) } P_I( x ) ) 
    \right] 
    e^{ A (T - t ) } 
    F( x )
  \right|
\\ & \leq
  \frac{
    | \varphi |_{ \operatorname{Lip}^1( H, \R ) }
    \,
    \| 
      P_{ \set \backslash I }
    \|_{
      L( H, H_{ - \rho } ) 
    }
    \,
    \| x \|_H
    \,
    \| F( x ) \|_{ H_{ - \vartheta } }
  }{
    ( T - t )^{ ( \rho + \vartheta ) }
  }
\end{split}
\end{equation}
and
\begin{equation}
\label{eq:F_term2}
\begin{split}
&
  \left|\varphi'( e^{ A (T - t) } P_I( x ) )\big( 
    [ \operatorname{Id}_H - P_I ] \, e^{ A ( T - t ) } F(x)
  \big)\right|
\leq
  \frac{
    \left| \varphi \right|_{
      \operatorname{Lip}^0( H, \R ) 
    }
    \|
      P_{ \set \backslash I }
    \|_{
      L( H, H_{ - \rho } ) 
    }
    \,
    \| F( x ) \|_{ H_{ - \vartheta } }
  }{
    ( T - t )^{ ( \rho + \vartheta ) }
  }
  .
\end{split}
\end{equation}
Combining~\eqref{eq:F_term1} and~\eqref{eq:F_term2}
proves that
\begin{equation}
\begin{split}
&
  \left|
    \int^T_0 
      \ES\!\left[
        \varphi'( e^{A ( T - t ) } X_t ) \,
        e^{ A ( T - t ) } 
        F( X_t )
      \right] 
    dt 
    - 
    \int^T_0 
      \ES\!\left[
        \varphi'( e^{ A ( T - t ) } P_I( X_t ) ) \,
        e^{ A ( T - t ) } 
        P_I F( X_t )
      \right] 
    dt 
  \right|
\\ & \leq 
  \tfrac{
    T^{ (1 - \rho - \vartheta ) }
    \,
    \| P_{ \set \backslash I } \|_{
      L( H , H_{ - \rho } )
    }
    \,
    \sup_{ t \in [0,T] }
    \E\big[
      \| X_t \|_H \,
      \| F( X_t ) \|_{ H_{ - \vartheta } }
      \left|
        \varphi
      \right|_{
        \operatorname{Lip}^1( H, \R ) 
      }
      +
      \| F( X_t ) \|_{ H_{ - \vartheta } }
      \left|
        \varphi
      \right|_{
        \operatorname{Lip}^0( H, \R ) 
      }
    \big]
  }{
    \left( 1 - \rho - \vartheta \right)
  }
\\ & \leq 
  \tfrac{
    T^{ (1 - \rho - \vartheta ) }
    \,
    \| P_{ \set \backslash I } \|_{
      L( H , H_{ - \rho } )
    } \,
      \left\|
        \varphi
      \right\|_{
        \operatorname{Lip}^1( H, \R ) 
      }
    \,
    \sup_{ t \in [0,T] }
    \max\left\{
      \E\,\left[
        \| X_t \|_H \,
        \| F( X_t ) \|_{ H_{ - \vartheta } }
      \right]
      ,
      \E\,\left[
        \| F( X_t ) \|_{ H_{ - \vartheta } }
      \right]
    \right\}
  }{
    \left( 1 - \rho - \vartheta \right)
  }
\\ & \leq 
  \frac{
    T^{ (1 - \rho - \vartheta ) }
    \,
    \| P_{ \set \backslash I } \|_{
      L( H , H_{ - \rho } )
    }
    \,
    \|
      \varphi
    \|_{
      \operatorname{Lip}^1( H, \R ) 
    }
    \,
    \| F \|_{ 
      \operatorname{Lip}^0( H, H_{ - \vartheta } ) 
    }
    \,
    \max\{ 
      1 ,
      \sup_{ t \in [0,T] }
      \E\,[ \| X_t \|^2_H ]
    \}
  }{
    \left( 1 - \rho - \vartheta \right)
  }
  .
\end{split}\label{weakreg_3}
\end{equation}
Inequality~\eqref{weakreg_3}
provides us a bound for the second difference on the right hand 
side of~\eqref{weakreg_1}.
Next we bound the third difference on the right hand side 
of~\eqref{weakreg_1}.
To this end note that
for all $ x \in H $, $ t \in [0,T) $ 
it holds that
\begin{equation}
\label{eq:B_term1}
\begin{split}
&
  \left|
    \smallsum\limits_{ b \in \mathbb{U} }
    \left[ 
      \varphi''( e^{ A ( T - t ) } x ) - 
      \varphi''( e^{ A ( T - t ) } P_I( x ) ) 
    \right] 
    \!
    (
      e^{ A ( T - t ) } 
      B^b( x ) , 
      e^{ A ( T - t ) }
      B^b(x)
    )
  \right|
\\ & \leq
  \frac{
    \left| \varphi \right|_{ 
      \operatorname{Lip}^2( H, \R ) 
    }
    \| B( x ) \|^2_{ HS( U, H_{ - \nicefrac{ \vartheta }{ 2 } } ) }
    \,
    \| x \|_H
    \,
    \|
      P_{ \set \backslash I } 
    \|_{
      L( H, H_{ - \rho } ) 
    }
  }{
    ( T - t )^{ ( \rho + \vartheta ) } 
  }
\end{split}
\end{equation}
and
\begin{equation}
\label{eq:B_term2}
\begin{split}
&
  \left|\smallsum\limits_{ b\in\mathbb{U} }\varphi''( e^{A(T-t)}P_I( x ) )
    (
      [ \operatorname{Id}_H + P_I ] e^{ A (T-t) } B^b(x)
      ,
      [ \operatorname{Id}_H - P_I ] e^{ A (T-t) } B^b(x)
    )
  \right|
\\ & \leq 
  \frac{
    2 \, 
    | \varphi |_{ 
      \operatorname{Lip}^1( H, \R ) 
    }
    \,
    \| B( x ) \|^2_{ HS( U, H_{ - \nicefrac{ \vartheta }{ 2 } } ) }
    \,
    \|
      P_{ \set \backslash I } 
    \|_{
      L( 
        H, H_{ - \rho } 
      ) 
    }
  }{
    ( T - t )^{ ( \rho + \vartheta ) }
  }
  .
\end{split}
\end{equation}
Combining~\eqref{eq:B_term1} and~\eqref{eq:B_term2} proves that
\begin{equation}
\begin{split}
& 
  \Bigg|
  \frac{ 1 }{ 2 }
  \sum\limits_{ b \in \mathbb{U} }
  \int^T_0 
    \ES\!\left[
      \varphi''( e^{ A (T-t) } X_t ) 
      (
        e^{ A (T-t) } B^b( X_t ) ,
        e^{ A (T-t) } B^b( X_t ) 
      )
    \right]
  dt
\\ &
  - \frac{ 1 }{ 2 }
  \sum\limits_{ b \in \mathbb{U} }
  \int^T_0 
  \ES\!\left[
    \varphi''( e^{ A (T-t) } P_I( X_t ) ) 
    (
      e^{ A (T-t) } 
      P_I B^b(X_t) , 
      e^{ A (T-t) } P_I B^b(X_t)
    )
  \right] dt
  \Bigg|
\\ & \leq   
  \tfrac{   
    T^{ (1 - \rho - \vartheta) }
  \,
  \|
    P_{ \set \backslash I } 
  \|_{
    L( H, H_{ - \rho } ) 
  }
  \,
  \| \varphi\|_{
    \operatorname{Lip}^2( H, \R ) 
  }
    \,
  \sup_{ t \in [0,T] }
  \max\!
  \big\{
    \E\big[
      \| X_t \|_H
      \| B( X_t ) \|_{ 
        HS( 
          U, 
          H_{ 
            - \vartheta / 2 
          } 
        ) 
      }^2
    \big]
    ,\,
    \E\big[
      \| B( X_t ) \|_{ 
        HS( 
          U, 
          H_{ 
            - \vartheta / 2 
          } 
        ) 
      }^2
    \big]
  \big\}
  }{
    ( 1 - \rho - \vartheta) 
  }
\\ & \leq   
  \frac{   
    T^{ (1 - \rho - \vartheta) }
  \,
  \|
    P_{ \set \backslash I } 
  \|_{
    L( H, H_{ - \rho } ) 
  }
  \,
  \| \varphi\|_{
    \operatorname{Lip}^2( H, \R ) 
  }
    \,
    \| B \|^{2}_{
      \operatorname{Lip}^0( H,
        HS( U, H_{ - \vartheta / 2 } ) 
      )
    }
    \,
    \max\{ 
      1 ,
      \sup_{ t \in [0,T] }
      \E\big[
        \| X_t \|^3_H
      \big]
    \}
  }{
    ( 1 - \rho - \vartheta) 
  }
  .
\end{split}\label{weakreg_4}
\end{equation}
Combining~\eqref{weakreg_1}, \eqref{weakreg_2}, \eqref{weakreg_3}, and~\eqref{weakreg_4} 
finally proves that 
\begin{equation}
\begin{split}
&
  \left|
    \ES\big[ 
      \varphi( X_T )
    \big]
    -
    \ES\big[ 
      \varphi( P_I( X_T ) )
    \big]
  \right|
\leq
  \| \varphi \|_{ 
    \operatorname{Lip}^2(H,\R) 
  }
  \,
    \max\!\left\{ 
      1 ,
      \sup\nolimits_{ t \in [0,T] }
      \E\,\big[
        \| X_t \|^3_H
      \big]
    \right\}
\\ & 
  \cdot
  \left[ 
    \frac{ 1 }{ T^{ \rho } }
    +
    \frac{
      T^{ ( 1 - \rho - \vartheta ) }
      \big[ 
        \| F \|_{
          \operatorname{Lip}^0( H, H_{ - \vartheta } )
        }
        +
        \| B \|^{{2}}_{
          \operatorname{Lip}^0( H,
            HS( U, H_{ - \vartheta / 2 } ) 
          )
        }
      \big]
    }{
      \left( 1 - \rho - \vartheta \right)
    }
  \right]
  \|P_{\set\backslash I}\|_{L(H,H_{-\rho})}
  .
\end{split}
\end{equation}
This finishes the proof of Proposition \ref{prop:weak_regle0}.
\end{proof}

\section{Weak convergence for Galerkin approximations of SEEs with mollified nonlinearities}
\label{sec:weak_reg}

In this section we establish weak convergence rates for \emph{Galerkin approximations of SEEs
with mollified nonlinearities}; see Corollary~\ref{cor:weak_reg}, Corollary~\ref{cor:weak_reg2}, and Corollary~\ref{cor:mollified.weak.rate.simplified} below.
Roughly speaking, in the framework of Section~\ref{sec:general_setting} we establish in Corollary~\ref{cor:weak_reg} below explicit upper bounds for the weak approximation error 
\begin{equation}
\label{eq:mollified.weak.error}
  \left|
    \ES\big[ 
      \varphi( X^\set_T )
    \big]
    -
    \ES\big[ 
      \varphi( X_T^I )
    \big]
  \right|,
\end{equation}
where $I\subseteq\mathbb{H}$ is a set,
where $\varphi\colon H\to\R$ is a four times continuously Fr\'{e}chet differentiable function with globally bounded derivatives, 
and where 
$X^{\mathbb{H}}\colon[0,T]\times\Omega\to H$ 
and
$X^I\colon[0,T]\times\Omega\to P_I(H)$
are appropriate mild solution processes of the SEEs in~\eqref{eq:Galerkin}.
Here, $X^I\colon[0,T]\times\Omega\to P_I(H)$ is a spectral Galerkin approximation of $X^{\mathbb{H}}\colon[0,T]\times\Omega\to H$.
We prove Corollary~\ref{cor:weak_reg} by using a decomposition of the weak approximation error as in~\eqref{eq:triangle_intro} in Section~\ref{sec:sketch_proof} above. 
Corollary~\ref{cor:weak_reg} is then an immediate consequence of the triangle inequality, of Lemma~\ref{weak_regle0} below, and of Lemma~\ref{weak_regle} below.
In the proof of Corollary~\ref{cor:weak_reg2} we further estimate the right hand side of inequality~\eqref{eq:weak_reg_estimate} in Corollary~\ref{cor:weak_reg} to obtain a more explicit upper bound for~\eqref{eq:mollified.weak.error} and the right hand side of~\eqref{eq:weak_reg_estimate} in Corollary~\ref{cor:weak_reg}, respectively.
Corollary~\ref{cor:weak_reg2}, in particular, enables us to prove 
inequality~\eqref{eq:estimate_no_mollifying} in the introduction.
In Section~\ref{sec:weak_irregular} below we will use Corollary~\ref{cor:weak_reg2} to establish  
weak convergence rates for \emph{Galerkin approximations of SEEs with ``non-mollified'' nonlinearities}.

\subsection{Regularity properties for solutions of infinite dimensional Kolmogorov equations in Hilbert spaces}

\begin{lemma}
	\label{lem:Kolmogorov}
	Assume the setting in Section~\ref{sec:general_setting},
	let 
	$
	\varphi \in C^4_b( H, \R)
	$, 
	$
	F \in C^4_b(H,H)
	$, 
	$ 
	B \in C^4_b(H,HS(U,H))
	$, 
	let 
	$ X^x \colon [0,T] \times \Omega \to H $, 
	$ x \in H $,
	be 
	$
	( \mathcal{F}_t )_{ t \in [0,T] }
	$-predictable stochastic processes
	which satisfy for all 
	$ x \in H $
	that
	$
	\sup_{ t \in [0,T] }
	\ES\big[\| X^x_t \|^4_H\big] 
	$
	$
	< \infty
	$ 
	and which satisfy that
	for all 
	$ x \in H $, 
	$ t \in [0,T] $
	it holds $ \P $-a.s.\ that
	\begin{equation}
	\label{eq:SEE.kolmogorov}
	X^x_t =
	e^{ A t } x
	+
	\int_0^t
	e^{ A (t - s) }
	F( 
	X^x_s 
	)
	\,
	ds
	+
	\int_0^t
	e^{ A (t - s) }
	B( 
	X^x_s 
	)
	\,
	dW_s
	,
	\end{equation}
	and let 
	$\phi\colon[0,T]\times H \to \R$ 
	be the function which satisfies for all 
	$t\in[0,T]$, 
	$x\in H$
	that
	$
	\phi(t,x)
	=
	\ES[\varphi(X^x_t)]
	$.
	Then
	\begin{enumerate}[(i)]
		\item
		\label{item:kolmogorov.diff}
		it holds for all $t\in[0,T]$ that 
		$
		(H \ni x \mapsto \phi(t,x) \in \R)
		\in C^4_b(H,\R)
		$
		and
		\item
		\label{item:c.delta}
		it holds for all 
		$ k \in \{ 1, 2, 3, 4 \} $, 
		$ \delta_1, \dots, \delta_k \in (-\nicefrac{1}{2},0] $
		with
		$
		\sum^k_{i=1} \delta_i
		> -\nicefrac{1}{2}
		$ 
		that 
		\label{item:kolmogorov.c.delta}
		\begin{equation}
		\label{eq:c.delta.constant}
		\begin{split} 
		&
		\sup_{
			t \in (0,T]
		}
		\sup_{ 
			x \in H
		}
		\sup_{ 
			v_1, \dots, v_k \in H \backslash \{ 0 \}
		}
		\left[
		\frac{
			\big|
			( 
			\frac{ 
				\partial^k
			}{
				\partial x^k
			}
			\phi
			)( t, x )( v_1, \dots, v_k )
			\big|
		}{
			t^{ 
				(
				\delta_1 + \ldots + \delta_k
				) 
			}
			\left\| v_1 \right\|_{ H_{ \delta_1 } }
			\cdot
			\ldots
			\cdot
			\left\| v_k \right\|_{ H_{ \delta_k } }
		}
		\right]
		< \infty
		.
		\end{split}
		\end{equation}
	\end{enumerate}
\end{lemma}
\begin{proof}
	Observe that~\eqref{eq:SEE.kolmogorov} together with items~(iii) \& (vii) of Theorem~3.3 in Andersson et al.~\cite{AnderssonHefterJentzenKurniawan2016} 
	(with
	$T=T$,
	$\eta=0$,
	$H=H$,
	$U=U$,
	$V=\R$,
	$W=W$,
	$A=A$,
	$n=4$,
	$\varphi=\varphi$,
	$F=F$,
	$B=B$,
	$k=k$,
	$\delta_1=-\delta_1,\ldots,\delta_k=-\delta_k$,
	$\alpha=0$,
	$\beta=0$
	for 
	$ 
	(\delta_1, \dots, \delta_k) \in 
	\{(x_1,\ldots,x_k)\in(-\nicefrac{1}{2},0]^k\colon
	\sum^k_{i=1} x_i > -\nicefrac{1}{2}
	\} 
	$, 
	$ k \in \{ 1, 2, 3, 4 \} $
	in the notation of Theorem~3.3 in~\cite{AnderssonHefterJentzenKurniawan2016})
	establishes items~\eqref{item:kolmogorov.diff}--\eqref{item:kolmogorov.c.delta} above.
	The proof of Lemma~\ref{lem:Kolmogorov} is thus completed.
\end{proof}

\noindent In the following we add some comments to Lemma~\ref{lem:Kolmogorov}.
Lemma~\ref{lem:Kolmogorov} is used in the proof of Lemma~\ref{weak_regle} below to establish essentially sharp weak convergence rates.
As demonstrated above in the proof of Lemma~\ref{lem:Kolmogorov}, Lemma~\ref{lem:Kolmogorov} is an immediate consequence of Theorem~3.3 in Andersson et al.~\cite{AnderssonHefterJentzenKurniawan2016}.
Theorem~3.3 in Andersson et al.~\cite{AnderssonHefterJentzenKurniawan2016}, in particular, establishes a similar result as Lemma~\ref{lem:Kolmogorov} but under the more general hypothesis that there exists a natural number $n\in\N$ such that $F$ and $B$ are $n$-times continuously Fr\'{e}chet differentiable with globally bounded derivatives.
However, in the proof of Lemma~\ref{weak_regle} below we merely employ estimates of the form~\eqref{eq:c.delta.constant} for the first four derivatives of the generalized solution
$
  \phi(t,x)
  =
  \ES[\varphi(X^x_t)]
$, 
$(t,x)\in[0,T]\times H$, 
of the Kolmogorov equation associated to~\eqref{eq:SEE.kolmogorov} and, therefore, we restrict ourselves in Lemma~\ref{lem:Kolmogorov} above to the case $n=4$.
Results related to~\eqref{eq:c.delta.constant} can, e.g., be found in Debussche~\cite[Lemmas~4.4--4.6]{Debussche2011} and in Wang \& Gan~\cite[Lemma~3.3]{WangGan2013WeakHeatAdditiveNoise}.
In particular, very roughly speaking, Lemmas~4.4--4.5 in~\cite{Debussche2011} establish~\eqref{eq:c.delta.constant} for all 
$\delta_1,\delta_k\in(-\nicefrac{1}{2},0]$, $k\in\{1,2\}$ 
without the constraint that $\delta_1+\delta_2 > - \nicefrac{1}{2}$
but under the additional assumption~\eqref{intro_eq:affine}.
Moreover, very roughly speaking, Lemma~3.3 in~\cite{WangGan2013WeakHeatAdditiveNoise} establishes~\eqref{eq:c.delta.constant} for all 
$\delta_1,\delta_k\in(-1,0]$, $k\in\{1,2\}$ 
with the constraint that $\delta_1+\delta_2 > - 1$ in the case of additive noise.
Note that condition~\eqref{intro_eq:affine} is obviously satisfied in the case of additive noise.
Next we briefly present the idea of the proof of Lemma~\ref{lem:Kolmogorov} above and of items~(iii) \& (vii) of Theorem~3.3 in Andersson et al.~\cite{AnderssonHefterJentzenKurniawan2016}, respectively. We first combine Vitali's convergence theorem with repeated applications of the chain rule from calculus (cf.\ Andersson et al.~\cite[Lemma~2.1, (77), and~(100)]{AnderssonHefterJentzenKurniawan2016}) to obtain explicit formulas for the higher order space derivatives of $\phi$ (cf.~Andersson et al.~\cite[Item~(v) of Theorem~3.3]{AnderssonHefterJentzenKurniawan2016})
in terms of higher order derivatives of the test function $\varphi$ and in terms of higher order derivative processes associated to~\eqref{eq:SEE.kolmogorov}. 
Thereafter, we employ H\"{o}lder's inequality and suitable estimates for the higher order derivative processes associated to~\eqref{eq:SEE.kolmogorov} from Andersson et al.~\cite[Item~(ii) of Theorem~2.1]{AnderssonJentzenKurniawan2016a} (cf.\ Andersson et al.~\cite[(60), (101), and~(103)]{AnderssonHefterJentzenKurniawan2016}).
The next result, Lemma~\ref{lem:gen.H1.solution} below, is an elementary lemma which provides sufficient conditions for mild solutions of SEEs to be strong solutions.

\begin{lemma}
	\label{lem:gen.H1.solution}
	Consider the notation in Section~\ref{sec:notation}, 
	let 
	$ ( H, \langle \cdot, \cdot \rangle_H, \left\| \cdot \right\|_H ) $
	and
	$ ( U, \langle \cdot, \cdot \rangle_U, \left\| \cdot \right\|_U ) $
	be separable $ \R $-Hilbert spaces,
	let $ T \in (0,\infty) $, $ p\in[2,\infty) $, 
	let 
	$ ( \Omega, \mathcal{F}, \P, ( \mathcal{F}_t )_{ t \in [0,T] } ) $
	be a stochastic basis,
	let $ ( W_t )_{ t \in [0,T] } $ be an $ \operatorname{Id}_U $-cylindrical
	$ ( \Omega, \mathcal{F}, \P, ( \mathcal{F}_t )_{ t \in [0,T] } ) $-Wiener process,
	let 
	$ A \colon D(A) \subseteq H \to H $ be a generator of a strongly continuous analytic semigroup with $\operatorname{spectrum}(A)\subseteq\{z\in\mathbb{C}\colon\operatorname{Re}(z)<0\}$,
	let
	$  
	  ( 
	    H_r 
	    ,
	    \langle \cdot, \cdot \rangle_{ H_r }
	    ,$ $
	    \left\| \cdot \right\|_{ H_r } 
	  )
	$,
	$ r \in \R $,
	be a family of interpolation spaces associated to $ - A $,
	let 
	$\xi\in\lpn{p}{\P}{H_1}$, 
	let
	$ X \colon [0,T] \times \Omega \to H $, 
	$ Y \colon [0,T] \times \Omega \to H_1 $, 
	and 
	$ Z \colon [0,T] \times \Omega \to HS(U,H_1) $
	be 
	$
	( \mathcal{F}_t )_{ t \in [0,T] }
	$-predictable stochastic processes
	which satisfy that
	$
	\int^T_0
	\ES\big[
	  \|Y_s\|^p_{H_1}
	  +
	  \|Z_s\|^p_{HS(U,H_1)}
	\big]
	\, ds
	< \infty
	$
	and which satisfy that
	for all $ t \in [0,T] $
	it holds $ \P $-a.s.\ that
	$
	  \int^t_0
	  \|e^{A(t-s)}Y_s\|_H
	  +
	  \|e^{A(t-s)}Z_s\|^2_{HS(U,H)}
	  \, ds
	  < \infty
	$
	and
	\begin{equation}
	\label{eq:mild.strong.SEE}
	X_t =
	e^{ A t } \xi
	+
	\int_0^t
	e^{ A (t - s) }
	Y_s
	\,
	ds
	+
	\int_0^t
	e^{ A (t - s) }
	Z_s
	\,
	dW_s
	.
	\end{equation}
	Then
	\begin{enumerate}[(i)]
		\item
		\label{item:general.H1.mild.integral}
		it holds that
		\begin{equation}
		\label{eq:general.H1.mild.integral}
		\sup_{t\in[0,T]}
		\int^t_0
		\ES\big[
		\|
		e^{A(t-s)}Y_s
		\|^p_{H_1}
		+
		\|e^{A(t-s)}Z_s\|^p_{HS(U,H_1)}
		\big]
		\, ds
		< \infty,
		\end{equation}
		\item
		\label{item:general.H1.reg}
		it holds for all 
		$t\in[0,T]$ 
		that 
		\begin{equation}
		\P\big(X_t\in H_1\big)=1
		,
		\end{equation}
		\item
		\label{item:gen.H1.solution.apriori}
		it holds that 
		\begin{equation}
		\sup_{t\in[0,T]}
		\ES\big[\|X_t \, \mathbbm{1}_{H_1}(X_t)\|^p_{H_1}\big]
		< \infty
		,
		\end{equation}
		\item
		\label{item:gen.H1.space.time.continuous}
		it holds for all $t\in[0,T]$ that 
		\begin{equation}
		  \limsup_{[0,T]\ni s \to t}
		  \|X_s \, \mathbbm{1}_{H_1}(X_s)-X_t \, \mathbbm{1}_{H_1}(X_t)\|_{\lpn{p}{\P}{H_1}}
		  =0,
		\end{equation}
		\item
		\label{item:integral.makes.sense}
		it holds that 
		\begin{equation}
		\!\!\!
		\P\!\left(
		\int^T_0
		\|AX_s\|_{H_{-1}}
		+
		\|A(X_s \mathbbm{1}_{H_1}(X_s))\|_H
		+
		\|Y_s\|_{H_1}
		+
		\|Z_s\|^2_{HS(U,H_1)}
		\, ds
		< \infty
		\right)
		=1,
		\end{equation}
		and
		\item
		\label{item:general.H1.solution}
		for all $t\in[0,T]$ it holds $\P$-a.s.\ that 
		\begin{equation}
		\begin{split}
		  X_t
		  &=\xi
		  +
		  \int^t_0
		  A\big(X_s \mathbbm{1}_{H_1}(X_s)\big) + Y_s
		  \, ds
		  +
		  \int^t_0
		  Z_s
		  \, dW_s
		  \\&=\xi
		  +
		  \int^t_0
		  AX_s + Y_s
		  \, ds
		  +
		  \int^t_0
		  Z_s
		  \, dW_s
		  .
		\end{split}
		\end{equation}
	\end{enumerate}
\end{lemma}
	\begin{proof}
	Throughout this proof 
	let $\tstep{N}{t}\in[0,T]$, $t\in[0,T]$, $N\in\N$, 
	be the real numbers which satisfy for all $N\in\N$, $t\in[0,T]$ that 
	$\tstep{N}{t}=\frac{t}{N}$, 
	let 
	$\floor{\cdot}{N}{t}\colon\R\to\R$, $t\in[0,T]$, $N\in\N$, 
	be the functions which satisfy for all $N\in\N$, $t\in[0,T]$, $s\in\R$
	that
	\begin{equation}
	  \floor{s}{N}{t}=
	  \max((-\infty,s]\cap\{0,-\tstep{N}{t},\tstep{N}{t},-2\tstep{N}{t},2\tstep{N}{t},\ldots\}),
	\end{equation}
	let $\Co\in[0,\infty)$ and $\Cr{r}\in[0,\infty)$, $r\in[0,1]$, be the real numbers which satisfy for all $r\in[0,1]$ that 
	\begin{equation}
	\Co=\sup_{t\in[0,T]}\|e^{At}\|_{L(H)}
	\qquad\text{and}\qquad
	\Cr{r}=\sup_{t\in[0,T]}\|(-tA)^{-r}(e^{At}-\operatorname{Id}_H)\|_{L(H)}
	\end{equation}
	(cf., e.g., \cite[Lemma~11.36]{rr93}), 
	and let 
	$\mathcal{X}\colon[0,T]\times\Omega\to H_1$ be the $
	( \mathcal{F}_t )_{ t \in [0,T] }
	$-predictable stochastic process which satisfies for all $t\in[0,T]$ that
	\begin{equation}
	\label{eq:X.H1.version}
	\mathcal{X}_t=X_t \mathbbm{1}_{H_1}(X_t).
	\end{equation}
	Observe that for all $t\in[0,T]$ it holds that 
		\begin{equation}
		\int^t_0
		\ES\big[
		\|
		e^{A(t-s)}Y_s
		\|^p_{H_1}
		+
		\|e^{A(t-s)}Z_s\|^p_{HS(U,H_1)}
		\big]
		\, ds
		\leq
		|\Co|^p
		\int^t_0
		\ES\big[
		\|
		Y_s
		\|^p_{H_1}
		+
		\|Z_s\|^p_{HS(U,H_1)}
		\big]
		\, ds.
		\end{equation}
	Therefore, we obtain that 
		\begin{equation}
		\begin{split}
		&
		\sup_{t\in[0,T]}
		\int^t_0
		\ES\big[
		\|
		e^{A(t-s)}Y_s
		\|^p_{H_1}
		+
		\|e^{A(t-s)}Z_s\|^p_{HS(U,H_1)}
		\big]
		\, ds
		\\&\leq
		|\Co|^p
		\int^T_0
		\ES\big[
		\|
		Y_s
		\|^p_{H_1}
		+
		\|Z_s\|^p_{HS(U,H_1)}
		\big]
		\, ds
		< \infty.
		\end{split}
		\end{equation}
	This establishes item~\eqref{item:general.H1.mild.integral}.
	Moreover, Jensen's inequality and the assumption that $p\geq 2$ 
	ensure that 
	\begin{equation}
	\begin{split}
	&
	  \int^T_0
	  \ES\big[
	  \|
	  Y_t
	  \|_{H_1}
	  \big]
	  \, dt
	=
	  T \, \bigg[
	  \frac{1}{T}
	  \int^T_0
	  \ES\Big[
	  \big(
	  \|
	  Y_t
	  \|^p_{H_1}
	  \big)^{\nicefrac{1}{p}}
	  \Big]
	  \, dt
	  \bigg]
	\\&\leq
	  T \, 
	  \bigg[
	  \frac{1}{T}
	  \int^T_0
	  \ES\big[
	  \|
	  Y_t
	  \|^p_{H_1}
	  \big]
	  \, dt
	  \bigg]^{\nicefrac{1}{p}}
	=
	T^{(1-\nicefrac{1}{p})}
	  \bigg[
	  \int^T_0
	  \ES\big[
	  \|
	  Y_t
	  \|^p_{H_1}
	  \big]
	  \, dt
	  \bigg]^{\nicefrac{1}{p}}
	\end{split}
	\end{equation}
	and
	\begin{equation}
	\begin{split}
	&
	  \int^T_0
	  \ES\big[
	  \|
	  Z_t
	  \|^2_{HS(U,H_1)}
	  \big]
	  \, dt
	=
	  T \, \bigg[
	  \frac{1}{T}
	  \int^T_0
	  \ES\Big[
	  \big(
	  \|
	  Z_t
	  \|^p_{HS(U,H_1)}
	  \big)^{\nicefrac{2}{p}}
	  \Big]
	  \, dt
	  \bigg]
	\\&\leq
	  T \, 
	  \bigg[
	  \frac{1}{T}
	  \int^T_0
	  \ES\big[
	  \|
	  Z_t
	  \|^p_{HS(U,H_1)}
	  \big]
	  \, dt
	  \bigg]^{\nicefrac{2}{p}}
	=
	T^{(1-\nicefrac{2}{p})}
	  \bigg[
	  \int^T_0
	  \ES\big[
	  \|
	  Z_t
	  \|^p_{HS(U,H_1)}
	  \big]
	  \, dt
	  \bigg]^{\nicefrac{2}{p}}
	  .
	\end{split}
	\end{equation}
	Hence, we obtain that 
	\begin{equation}
	\label{eq:integral.without.exponent}
	  \int^T_0
	  \ES\big[
	  \|Y_t\|_{H_1}
	  +
	  \|
	  Z_t
	  \|^2_{HS(U,H_1)}
	  \big]
	  \, dt
	  < \infty.
	\end{equation}
	This ensures that for all $t\in[0,T]$ it holds that 
	\begin{equation}
\begin{split}
		&
		\int^t_0
		\ES\big[
		\|
		e^{A(t-s)}Y_s
		\|_{H_1}
		+
		\|e^{A(t-s)}Z_s\|^2_{HS(U,H_1)}
		\big]
		\, ds
		\\&\leq
		\int^t_0
		\Co \,
		\ES\big[
		\|
		Y_s
		\|_{H_1}
		\big]
		+
		|\Co|^2\,
		\ES\big[
		\|Z_s\|^2_{HS(U,H_1)}
		\big]
		\, ds
		< \infty.
\end{split}
	\end{equation}
	This implies that for all $t\in[0,T]$ it holds that 
	\begin{equation}
	  \P\!\left(
		\int^t_0
		\|
		e^{A(t-s)}Y_s
		\|_{H_1}
		+
		\|e^{A(t-s)}Z_s\|^2_{HS(U,H_1)}
		\, ds
	  < \infty
	  \right)
	  =1.
	\end{equation}
	 This, \eqref{eq:mild.strong.SEE}, and the assumption that 
	 $\xi\in\lpn{p}{\P}{H_1}$
	 prove item~\eqref{item:general.H1.reg}.
	Item~\eqref{item:general.H1.reg} and~\eqref{eq:mild.strong.SEE} show that for all $t\in[0,T]$ it holds $\P$-a.s.\ that 
	\begin{equation}
	\label{eq:version.mild.strong.SEE}
	\mathcal{X}_t = X_t =
	e^{ A t } \xi
	+
	\int_0^t
	e^{ A (t - s) }
	Y_s
	\,
	ds
	+
	\int_0^t
	e^{ A (t - s) }
	Z_s
	\,
	dW_s
	.
	\end{equation}
	This and the Burkholder-Davis-Gundy type inequality 
	in Lemma~7.7 in Da Prato \& Zabczyk~\cite{dz92} imply that for all 
	$t\in[0,T]$ it holds that 
	\begin{equation}
	\begin{split}
	\|\mathcal{X}_t\|_{\lpn{p}{\P}{H_1}}
	&\leq
	\|e^{At}\xi\|_{\lpn{p}{\P}{H_1}}
	+
	\int^t_0
	\|e^{A(t-s)}Y_s\|_{\lpn{p}{\P}{H_1}}
	\, ds
	\\&\quad+
	\bigg[\frac{p \, (p-1)}{2}
	\int^t_0
	\|e^{A(t-s)}Z_s\|^2_{\lpn{p}{\P}{HS(U,H_1)}}
	\, ds
	\bigg]^{\nicefrac{1}{2}}
	.
	\end{split}
	\end{equation}
	H\"{o}lder's inequality hence shows that for all $t\in[0,T]$ it holds that 
	\begin{equation}
	\begin{split}
	\|\mathcal{X}_t\|_{\lpn{p}{\P}{H_1}}
	&\leq
	\Co \,
	\|\xi\|_{\lpn{p}{\P}{H_1}}
	+
	t^{(1-\nicefrac{1}{p})}
	\bigg[
	\int^t_0
	\ES\big[
	\|e^{A(t-s)}Y_s\|^p_{H_1}
	\big]
	\, ds
	\bigg]^{\nicefrac{1}{p}}
	\\&\quad+
	\bigg\{\frac{p \, (p-1) \, t^{(1-\nicefrac{2}{p})}}{2}
	\bigg[
	\int^t_0
	\ES\big[
	\|e^{A(t-s)}Z_s\|^p_{HS(U,H_1)}
	\big]
	\, ds
	\bigg]^{\nicefrac{2}{p}}
	\bigg\}^{\nicefrac{1}{2}}
	.
	\end{split}
	\end{equation}
	This and item~\eqref{item:general.H1.mild.integral} assure that 
	\begin{equation}
	\begin{split}
	&
	\sup_{t\in[0,T]}
	\|\mathcal{X}_t\|_{\lpn{p}{\P}{H_1}}
	\\&\leq
	\Co \,
	\|\xi\|_{\lpn{p}{\P}{H_1}}
	+
	T^{(1-\nicefrac{1}{p})}
	\bigg[
	\sup_{t\in[0,T]}
	\int^t_0
	\ES\big[
	\|e^{A(t-s)}Y_s\|^p_{H_1}
	\big]
	\, ds
	\bigg]^{\nicefrac{1}{p}}
	\\&\quad+
	\bigg\{\frac{p \, (p-1) \, T^{(1-\nicefrac{2}{p})}}{2}
	\bigg[
	\sup_{t\in[0,T]}
	\int^t_0
	\ES\big[
	\|e^{A(t-s)}Z_s\|^p_{HS(U,H_1)}
	\big]
	\, ds
	\bigg]^{\nicefrac{2}{p}}
	\bigg\}^{\nicefrac{1}{2}}
	< \infty
	.
	\end{split}
	\end{equation}
	This establishes item~\eqref{item:gen.H1.solution.apriori}.
	Next note that for all $s,t\in[0,T]$ with $s\leq t$ it holds $\P$-a.s.\ that 
	\begin{equation}
	\begin{split}
	&
	\int^t_0
	e^{A(t-r)} Y_r
	\, dr
	-
	\int^s_0
	e^{A(s-r)} Y_r
	\, dr
	=
	\int^t_0
	\big(
	e^{A(t-r)}-\mathbbm{1}_{[0,s]}(r) \, e^{A\max\{s-r,0\}}
	\big) \, Y_r
	\, dr.
	\end{split}
	\end{equation}
	This and H\"{o}lder's inequality show that for all
	$t,\tau\in[0,T]$ it holds that 
	\begin{equation}
	\label{eq:drift.int.diff}
	\begin{split}
	&
	\bigg\|
	\int^t_0
	e^{A(t-s)} Y_s
	\, ds
	-
	\int^\tau_0
	e^{A(\tau-s)} Y_s
	\, ds
	\bigg\|_{\lpn{p}{\P}{H_1}}
	\\&\leq
	\int^{\max\{t,\tau\}}_0
	\big\|
	\big(
	e^{A(\max\{t,\tau\}-s)}-\mathbbm{1}_{[0,\min\{t,\tau\}]}(s)\,e^{A\max\{\min\{t,\tau\}-s,0\}}
	\big) \, Y_s
	\big\|_{\lpn{p}{\P}{H_1}}
	\, ds
	\\&\leq
	T^{(1-\nicefrac{1}{p})}
	\bigg[
	\int^{\max\{t,\tau\}}_0
	\ES\Big[
	\big\|
	\big(
	e^{A(\max\{t,\tau\}-s)}
	\\&\quad
	-\mathbbm{1}_{[0,\min\{t,\tau\}]}(s)\,e^{A\max\{\min\{t,\tau\}-s,0\}}
	\big) \, Y_s
	\big\|^p_{H_1}
	\Big]
	\, ds
	\bigg]^{\nicefrac{1}{p}}
	\\&=
	T^{(1-\nicefrac{1}{p})}
	\bigg[
	\int^T_0
	\ES\Big[
	\mathbbm{1}_{[0,\max\{t,\tau\}]}(s) \,
	\big\|
	\big(
	e^{A\max\{\max\{t,\tau\}-s,0\}}
	\\&\quad
	-\mathbbm{1}_{[0,\min\{t,\tau\}]}(s)\,e^{A\max\{\min\{t,\tau\}-s,0\}}
	\big) \, Y_s
	\big\|^p_{H_1}
	\Big]
	\, ds
	\bigg]^{\nicefrac{1}{p}}
	.
	\end{split}
	\end{equation}
	Moreover, observe that for all 
	$s,t,\tau\in[0,T]$ it holds that 
	\begin{equation}
	\label{eq:drift.int.dominated}
	\begin{split}
	&
	\mathbbm{1}_{[0,\max\{t,\tau\}]}(s) \,
	\big\|
	\big(
	e^{A\max\{\max\{t,\tau\}-s,0\}}-\mathbbm{1}_{[0,\min\{t,\tau\}]}(s)\,e^{A\max\{\min\{t,\tau\}-s,0\}}
	\big) \, Y_s
	\big\|_{H_1}
	\\&\leq
	2 \Co \, \|Y_s\|_{H_1}.
	\end{split}
	\end{equation}
	Next note that for all $t\in[0,\infty)$, $v\in H_1$ it holds that 
	\begin{equation}
	\label{eq:semigroup.strong.continuity}
	\limsup_{[0,\infty)\ni s\to t} \|(e^{At}-e^{As})v\|_{H_1}=0.
	\end{equation}
	Combining~\eqref{eq:drift.int.diff}--\eqref{eq:semigroup.strong.continuity} with Lebesgue's theorem of dominated convergence
	and the assumption that 
	$
	\int^T_0
	\ES\big[
	  \|Y_s\|^p_{H_1}
	\big]
	\, ds
	< \infty
	$
	yields that for all $t\in[0,T]$ it holds that 
	\begin{equation}
	\label{eq:SEE.int.cont}
	\begin{split}
	&
	\limsup_{[0,T]\ni \tau \to t}
	\bigg\|
	\int^t_0
	e^{A(t-s)} Y_s
	\, ds
	-
	\int^\tau_0
	e^{A(\tau-s)} Y_s
	\, ds
	\bigg\|_{\lpn{p}{\P}{H_1}}
	=0.
	\end{split}
	\end{equation}
	In the next step note that for all $s,t\in[0,T]$ with $s\leq t$
	it holds $\P$-a.s.\ that 
	\begin{equation}
	\begin{split}
	&
	\int^t_0
	e^{A(t-r)} Z_r
	\, dW_r
	-
	\int^s_0
	e^{A(s-r)} Z_r
	\, dW_r
	\\&=
	\int^t_0
	\big(
	e^{A(t-r)}-\mathbbm{1}_{[0,s]}(r) \, e^{A\max\{s-r,0\}}
	\big) \, Z_r
	\, dW_r.
	\end{split}
	\end{equation}
	This, H\"{o}lder's inequality, and the Burkholder-Davis-Gundy type inequality 
	in Lemma~7.7 in Da Prato \& Zabczyk~\cite{dz92} show that for all
	$t,\tau\in[0,T]$ it holds that 
	\begin{equation}
	\label{eq:diff.int.diff}
	\begin{split}
	&
	\bigg\|
	\int^t_0
	e^{A(t-s)} Z_s
	\, dW_s
	-
	\int^\tau_0
	e^{A(\tau-s)} Z_s
	\, dW_s
	\bigg\|_{\lpn{p}{\P}{H_1}}
	\\&\leq
	\bigg[
	\frac{p \, (p-1)}{2}
	\bigg]^{\nicefrac{1}{2}} \,
	\bigg[
	\int^{\max\{t,\tau\}}_0
	\big\|
	\big(
	e^{A(\max\{t,\tau\}-s)}
	\\&\quad
	-\mathbbm{1}_{[0,\min\{t,\tau\}]}(s)\,e^{A\max\{\min\{t,\tau\}-s,0\}}
	\big) \, Z_s
	\big\|^2_{\lpn{p}{\P}{HS(U,H_1)}}
	\, ds
	\bigg]^{\nicefrac{1}{2}}
	\\&\leq
	\bigg[
	\frac{p \, (p-1) \, T^{(1-\nicefrac{2}{p})}}{2}
	\bigg]^{\nicefrac{1}{2}} \,
	\bigg[
	\int^{\max\{t,\tau\}}_0
	\ES\Big[\big\|
	\big(
	e^{A(\max\{t,\tau\}-s)}
	\\&\quad
	-\mathbbm{1}_{[0,\min\{t,\tau\}]}(s)\,e^{A\max\{\min\{t,\tau\}-s,0\}}
	\big) \, Z_s
	\big\|^p_{HS(U,H_1)}\Big]
	\, ds
	\bigg]^{\nicefrac{1}{p}}
	\\&=
	\bigg[
	\frac{p \, (p-1) \, T^{(1-\nicefrac{2}{p})}}{2}
	\bigg]^{\nicefrac{1}{2}} \,
	\bigg[
	\int^T_0
	\ES\Big[\mathbbm{1}_{[0,\max\{t,\tau\}]}(s)\,\big\|
	\big(
	e^{A\max\{\max\{t,\tau\}-s,0\}}
	\\&\quad
	-\mathbbm{1}_{[0,\min\{t,\tau\}]}(s)\,e^{A\max\{\min\{t,\tau\}-s,0\}}
	\big) \, Z_s
	\big\|^p_{HS(U,H_1)}\Big]
	\, ds
	\bigg]^{\nicefrac{1}{p}}.
	\end{split}
	\end{equation}
	Moreover, observe that for all 
	$s,t,\tau\in[0,T]$ it holds that 
	\begin{equation}
	\label{eq:diff.int.dominated}
	\begin{split}
	&
	\mathbbm{1}_{[0,\max\{t,\tau\}]}(s)\,
	\big\|
	\big(
	e^{A\max\{\max\{t,\tau\}-s,0\}}-\mathbbm{1}_{[0,\min\{t,\tau\}]}(s)\,e^{A\max\{\min\{t,\tau\}-s,0\}}
	\big) \, Z_s
	\big\|_{HS(U,H_1)}
	\\&\leq
	2\Co \, \|Z_s\|_{HS(U,H_1)}.
	\end{split}
	\end{equation}
	In addition, note that~\eqref{eq:semigroup.strong.continuity} and Lebesgue's theorem of dominated convergence ensure that for all 
	$t\in[0,\infty)$, 
	$B\in HS(U,H_1)$
	and all orthonormal bases $\mathbb{U}\subseteq U$ of $U$ it holds that 
	\begin{equation}
	  \limsup_{[0,\infty)\ni s \to t}
	  \|(e^{At}-e^{As})B\|^2_{HS(U,H_1)}
	  =
	  \limsup_{[0,\infty)\ni s \to t}
	  \Bigg[
	  \sum_{u\in\mathbb{U}}
	  \|(e^{At}-e^{As})Bu\|^2_{H_1}
	  \Bigg]
	  =0.
	\end{equation}
	Therefore, we obtain for all $t\in[0,T]$, $B\in HS(U,H_1)$ that 
	\begin{equation}
	  \limsup_{[0,T]\ni s \to t}
	  \|(e^{At}-e^{As})B\|^p_{HS(U,H_1)}
	  =0.
	\end{equation}
	Combining~\eqref{eq:diff.int.diff} with~\eqref{eq:diff.int.dominated}, Lebesgue's theorem of dominated convergence, 
	and the assumption that 
	$
	\int^T_0
	\ES\big[
	  \|Z_s\|^p_{HS(U,H_1)}
	\big]
	\, ds
	< \infty
	$
	hence yields that for all $t\in[0,T]$ it holds that 
	\begin{equation}
	\label{eq:SEE.int.cont.II}
	\begin{split}
	\limsup_{[0,T]\ni \tau \to t}
	\bigg\|
	\int^t_0
	e^{A(t-s)} Z_s
	\, dW_s
	-
	\int^\tau_0
	e^{A(\tau-s)} Z_s
	\, dW_s
	\bigg\|_{\lpn{p}{\P}{H_1}}
	=0.
	\end{split}
	\end{equation}
	In addition, note that~\eqref{eq:semigroup.strong.continuity} and Lebesgue's theorem of dominated convergence ensure that for all $t\in[0,T]$ it holds that 
	\begin{equation}
	\limsup_{[0,T]\ni s \to t}
	\|(e^{At}-e^{As})\xi\|_{\lpn{p}{\P}{H_1}}
	=0.
	\end{equation}
	Combining this, \eqref{eq:SEE.int.cont}, and~\eqref{eq:SEE.int.cont.II} with~\eqref{eq:version.mild.strong.SEE} establishes item~\eqref{item:gen.H1.space.time.continuous}.
	Next note that items~\eqref{item:general.H1.reg}--\eqref{item:gen.H1.solution.apriori} imply that 
	\begin{equation}
	\begin{split}
	  &
	  \int^T_0
	  \ES\big[\|A\mathcal{X}_s\|_H\big]
	  +
	  \ES\big[\|AX_s\|_{H_{-1}}\big]
	  \, ds
	  =
	  \int^T_0
	  \ES\big[\|\mathcal{X}_s\|_{H_1}\big]
	  +
	  \ES\big[\|X_s\|_H\big]
	  \, ds
	  \\&=
	  \int^T_0
	  \ES\big[\|\mathcal{X}_s\|_{H_1}\big]
	  +
	  \ES\big[\|\mathcal{X}_s\|_H\big]
	  \, ds
	  \leq
	  (1+\|A^{-1}\|_{L(H)})
	  \int^T_0
	  \ES\big[\|\mathcal{X}_s\|_{H_1}\big]
	  \, ds
	  \\&\leq
	  T \, (1+\|A^{-1}\|_{L(H)})
	  \sup_{t\in[0,T]}
	  \ES\big[\|\mathcal{X}_t\|_{H_1}\big]
	  < \infty.
	\end{split}
	\end{equation}
	Combining this with~\eqref{eq:integral.without.exponent} yields that 
		\begin{equation}
		\P\!\left(
		\int^T_0
		\|AX_s\|_{H_{-1}} +\|A\mathcal{X}_s\|_H 
		+
		 \|Y_s\|_{H_1}
		+
		\|Z_s\|^2_{HS(U,H_1)}
		\, ds
		< \infty
		\right)
		=1.
		\end{equation}
		This proves item~\eqref{item:integral.makes.sense}.
	It thus remains to establish item~\eqref{item:general.H1.solution}.
	For this 
	let $\mathbb{X}\colon[0,T]\times\Omega\to H$ be a stochastic process which satisfies that for all $t\in[0,T]$ it holds $\P$-a.s.\ that 
	\begin{equation}
	\label{eq:tildeX}
	  \mathbb{X}_t
	  =
	  X_t
	  -\xi
	  -\int^t_0 Y_s \, ds
	  -\int^t_0 Z_s \, dW_s.
	\end{equation}
	Observe that for all $N\in\N$, $t\in(0,T]$ it holds $\P$-a.s.\ that 
	\begin{equation}
	\begin{split}
	\mathbb{X}_t
	&=
	\sum^{N-1}_{n=0}
	\big(\mathbb{X}_{(n+1)\tstep{N}{t}}-\mathbb{X}_{n\tstep{N}{t}}\big)
	=
	\sum^{N-1}_{n=0}
	\int^{(n+1)\tstep{N}{t}}_{n\tstep{N}{t}}
	\bigg(
	\frac{\mathbb{X}_{(n+1)\tstep{N}{t}}-\mathbb{X}_{n\tstep{N}{t}}}{\tstep{N}{t}}
	\bigg)
	\, ds
	\\&=
	\sum^{N-1}_{n=0}
	\int^{(n+1)\tstep{N}{t}}_{n\tstep{N}{t}}
	\bigg(
	\frac{\mathbb{X}_{\floor{s}{N}{t}+\tstep{N}{t}}-\mathbb{X}_{\floor{s}{N}{t}}}{\tstep{N}{t}}
	\bigg)
	\, ds
	\\&=
	\int^t_0
	\bigg(
	\frac{\mathbb{X}_{\floor{s}{N}{t}+\tstep{N}{t}}-\mathbb{X}_{\floor{s}{N}{t}}}{\tstep{N}{t}}
	\bigg)
	\, ds.
	\end{split}
	\end{equation}
	This implies that for all $t\in(0,T]$ it holds that 
	\begin{equation}
	\label{eq:strong.diff}
	\begin{split}
	&
	\bigg\|
	  \mathbb{X}_t-\int^t_0 A\mathcal{X}_s \, ds
	\bigg\|_{\lpn{p}{\P}{H}}
	\\&\leq
	\liminf_{\N\ni N\to\infty}
	\bigg(
	\int^t_0
	\bigg\|
	\frac{\mathbb{X}_{\floor{s}{N}{t}+\tstep{N}{t}}-\mathbb{X}_{\floor{s}{N}{t}}}{\tstep{N}{t}}
	-
	A\mathcal{X}_s
	\bigg\|_{\lpn{p}{\P}{H}}
	\, ds
	\bigg)
	\\&\leq
	\liminf_{\N\ni N\to\infty}
	\bigg(
	\int^t_0
	\bigg\|
	\frac{\mathbb{X}_{\floor{s}{N}{t}+\tstep{N}{t}}-\mathbb{X}_{\floor{s}{N}{t}}}{\tstep{N}{t}}
	-
	A\mathcal{X}_{\floor{s}{N}{t}}
	\bigg\|_{\lpn{p}{\P}{H}}
	\, ds
	\bigg)
	\\&\quad+
	\liminf_{\N\ni N\to\infty}
	\bigg(
	\int^t_0
	\big\|
	A(\mathcal{X}_{\floor{s}{N}{t}}-\mathcal{X}_s)
	\big\|_{\lpn{p}{\P}{H}}
	\, ds
	\bigg)
	.
	\end{split}
	\end{equation}
	Next note that items~\eqref{item:gen.H1.solution.apriori}--\eqref{item:gen.H1.space.time.continuous} and Lebesgue's theorem of dominated convergence ensure that for all $t\in(0,T]$ it holds that
	\begin{equation}
	\label{eq:dom.time.cont}
	\limsup_{\N\ni N\to\infty}
	\bigg(
	\int^t_0
	\big\|
	A(\mathcal{X}_{\floor{s}{N}{t}}-\mathcal{X}_s)
	\big\|_{\lpn{p}{\P}{H}}
	\, ds
	\bigg)
	=0.
	\end{equation}
	Moreover, observe that item~\eqref{item:general.H1.reg} and~\eqref{eq:mild.strong.SEE} imply that for all $s,t\in[0,T]$ with $s\leq t$ it holds $\P$-a.s.\ that 
	\begin{equation}
	\begin{split}
	  X_t
	  &=e^{A(t-s)} X_s
	  +
	  \int^t_s
	  e^{A(t-r)} Y_r
	  \, dr
	  +
	  \int^t_s
	  e^{A(t-r)} Z_r
	  \, dW_r
	  \\&=e^{A(t-s)} \mathcal{X}_s
	  +
	  \int^t_s
	  e^{A(t-r)} Y_r
	  \, dr
	  +
	  \int^t_s
	  e^{A(t-r)} Z_r
	  \, dW_r.
	\end{split}
	\end{equation}
	This and item~\eqref{item:general.H1.reg} show that for all $N\in\N$, 
	$s,t\in(0,T]$ with $s < t$
	it holds $\P$-a.s.\ that 
	\begin{equation}
	\begin{split}
	&
	  \frac{\mathbb{X}_{\floor{s}{N}{t}+\tstep{N}{t}}-\mathbb{X}_{\floor{s}{N}{t}} - \tstep{N}{t} \, A\mathcal{X}_{\floor{s}{N}{t}}}{\tstep{N}{t}}
	\\&=
	  \frac{
	    (e^{A\tstep{N}{t}}-\operatorname{Id}_H-\tstep{N}{t}A) \mathcal{X}_{\floor{s}{N}{t}}
	  }{\tstep{N}{t}}
	  +
	  \int^{\floor{s}{N}{t}+\tstep{N}{t}}_{\floor{s}{N}{t}}
	  \frac{(e^{A(\floor{s}{N}{t}+\tstep{N}{t}-r)}-\operatorname{Id}_H)}{\tstep{N}{t}} \, Y_r
	  \, dr
	\\&\quad+
	  \int^{\floor{s}{N}{t}+\tstep{N}{t}}_{\floor{s}{N}{t}}
	  \frac{(e^{A(\floor{s}{N}{t}+\tstep{N}{t}-r)}-\operatorname{Id}_H)}{\tstep{N}{t}} \, Z_r
	  \, dW_r.
	\end{split}
	\end{equation}
	This yields that for all $N\in\N$, $t\in(0,T]$ it holds that 
	\begin{equation}
	\label{eq:strong.decompose}
	\begin{split}
	&
	\int^t_0
	\bigg\|
	\frac{\mathbb{X}_{\floor{s}{N}{t}+\tstep{N}{t}}-\mathbb{X}_{\floor{s}{N}{t}}}{\tstep{N}{t}}
	-
	A\mathcal{X}_{\floor{s}{N}{t}}
	\bigg\|_{\lpn{p}{\P}{H}}
	\, ds
	\\&\leq
	\int^t_0
	\bigg\|\frac{(e^{A\tstep{N}{t}}-\operatorname{Id}_H-\tstep{N}{t}A) \mathcal{X}_{\floor{s}{N}{t}}}{\tstep{N}{t}}\bigg\|_{\lpn{p}{\P}{H}}
	\, ds
	\\&\quad+
	\int^t_0
	  \int^{\floor{s}{N}{t}+\tstep{N}{t}}_{\floor{s}{N}{t}}
	  \frac{\|(e^{A(\floor{s}{N}{t}+\tstep{N}{t}-r)}-\operatorname{Id}_H)Y_r\|_{\lpn{p}{\P}{H}}}{\tstep{N}{t}}
	\, dr
	\, ds
	\\&\quad+
	\int^t_0
	\bigg\|
	  \int^{\floor{s}{N}{t}+\tstep{N}{t}}_{\floor{s}{N}{t}}
	  \frac{(e^{A(\floor{s}{N}{t}+\tstep{N}{t}-r)}-\operatorname{Id}_H)}{\tstep{N}{t}} \, Z_r
	  \, dW_r
	\bigg\|_{\lpn{p}{\P}{H}}
	\, ds
	.
	\end{split}
	\end{equation}
	Next note that for all $N\in\N$, $s,t\in(0,T]$ with $s\leq t$ it holds that
	\begin{equation}
	\label{eq:fixed.tie.decomposition}
	\begin{split}
	&
	\bigg\|
	\frac{(e^{A\tstep{N}{t}}-\operatorname{Id}_H-\tstep{N}{t}A) \mathcal{X}_{\floor{s}{N}{t}}}{\tstep{N}{t}}
	\bigg\|_{\lpn{p}{\P}{H}}
	\\&\leq
	\bigg\|
	\frac{(e^{A\tstep{N}{t}}-\operatorname{Id}_H-\tstep{N}{t}A) (\mathcal{X}_{\floor{s}{N}{t}}-\mathcal{X}_s)}{\tstep{N}{t}}
	\bigg\|_{\lpn{p}{\P}{H}}
	\\&\quad+
	\bigg\|
	\frac{(e^{A\tstep{N}{t}}-\operatorname{Id}_H-\tstep{N}{t}A) \mathcal{X}_s}{\tstep{N}{t}}
	\bigg\|_{\lpn{p}{\P}{H}}
	\\&\leq
	\bigg\|
	\frac{(e^{A\tstep{N}{t}}-\operatorname{Id}_H) (\mathcal{X}_{\floor{s}{N}{t}}-\mathcal{X}_s)}{\tstep{N}{t}}
	\bigg\|_{\lpn{p}{\P}{H}}
	+
	\|
	\mathcal{X}_{\floor{s}{N}{t}}-\mathcal{X}_s
	\|_{\lpn{p}{\P}{H_1}}
	\\&\quad+
	\bigg\|
	\frac{(e^{A\tstep{N}{t}}-\operatorname{Id}_H-\tstep{N}{t}A) \mathcal{X}_s}{\tstep{N}{t}}
	\bigg\|_{\lpn{p}{\P}{H}}
	\\&\leq
	(\Cr{1}+1) \, 
	\|
	\mathcal{X}_{\floor{s}{N}{t}}-\mathcal{X}_s
	\|_{\lpn{p}{\P}{H_1}}
	+
	\bigg\|
	\frac{(e^{A\tstep{N}{t}}-\operatorname{Id}_H-\tstep{N}{t}A) \mathcal{X}_s}{\tstep{N}{t}}
	\bigg\|_{\lpn{p}{\P}{H}}.
	\end{split}
	\end{equation}
	In addition, observe that the fact that 
	$
	  \forall \, v \in H_1 \colon
	  \limsup_{(0,\infty)\ni h \to 0}
	  \big\|\frac{(e^{Ah}-\operatorname{Id}_H-hA)v}{h}\|_H
	  =0
	$
	assures that for all 
	$s,t\in(0,T]$ with $s\leq t$ it holds that
	\begin{equation}
	\label{eq:initial.term.convergence}
	  \limsup_{\N\ni N\to\infty}
	  \bigg\|
	  \frac{(e^{A\tstep{N}{t}}-\operatorname{Id}_H-\tstep{N}{t}A) \mathcal{X}_s}{\tstep{N}{t}}
	  \bigg\|_H
	  =0.
	\end{equation}
	Next observe that for all $N\in\N$, $s,t\in(0,T]$ with $s\leq t$ it holds that
	\begin{equation}
	\label{eq:initial.term.dominated}
	\begin{split}
	&
		\bigg\|\frac{(e^{A\tstep{N}{t}}-\operatorname{Id}_H-\tstep{N}{t}A) \mathcal{X}_s}{\tstep{N}{t}}\bigg\|_H
	\leq
	(\Cr{1}+1) \, \|\mathcal{X}_s\|_{H_1}.
	\end{split}
	\end{equation}
	This, \eqref{eq:initial.term.convergence}, item~\eqref{item:gen.H1.solution.apriori}, and Lebesgue's theorem of dominated convergence ensure that for all $s,t\in(0,T]$ with $s\leq t$ it holds that 
	\begin{equation}
	\label{eq:convergence.fixed.time}
	\limsup_{\N\ni N\to\infty}
	\bigg\|
	\frac{(e^{A\tstep{N}{t}}-\operatorname{Id}_H-\tstep{N}{t}A) \mathcal{X}_s}{\tstep{N}{t}}
	\bigg\|_{\lpn{p}{\P}{H}}
	=0.
	\end{equation}
	Combining~\eqref{eq:fixed.tie.decomposition} with~\eqref{eq:convergence.fixed.time} and item~\eqref{item:gen.H1.space.time.continuous}
	shows that for all $s,t\in(0,T]$ with $s\leq t$ it holds that 
	\begin{equation}
	\limsup_{\N\ni N\to\infty}
	\bigg\|
	\frac{(e^{A\tstep{N}{t}}-\operatorname{Id}_H-\tstep{N}{t}A) \mathcal{X}_{\floor{s}{N}{t}}}{\tstep{N}{t}}
	\bigg\|_{\lpn{p}{\P}{H}}
	=0.
	\end{equation}
	This, \eqref{eq:initial.term.dominated}, item~\eqref{item:gen.H1.solution.apriori}, and Lebesgue's theorem of dominated convergence yield that for all $t\in(0,T]$ it holds that 
	\begin{equation}
	\label{eq:strong.decompose.initial}
	\begin{split}
	&
	\limsup_{\N\ni N\to\infty}
	\bigg(
	\int^t_0
	  \bigg\|\frac{(e^{A\tstep{N}{t}}-\operatorname{Id}_H-\tstep{N}{t}A) \mathcal{X}_{\floor{s}{N}{t}}}{\tstep{N}{t}}\bigg\|_{\lpn{p}{\P}{H}}
	\, ds
	\bigg)
	=0.
	\end{split}
	\end{equation}
	Furthermore, observe that for all 
	$N\in\N$, $s,t\in(0,T]$ with $s < t$ it holds that 
	\begin{equation}
	\label{eq:drift.diff.int.dominated}
	\begin{split}
	&
	  \int^{\floor{s}{N}{t}+\tstep{N}{t}}_{\floor{s}{N}{t}}
	  \frac{\|(e^{A(\floor{s}{N}{t}+\tstep{N}{t}-r)}-\operatorname{Id}_H)Y_r\|_{\lpn{p}{\P}{H}}}{\tstep{N}{t}}
	\, dr
	\\&\leq
	  \Cr{1}
	  \int^{\floor{s}{N}{t}+\tstep{N}{t}}_{\floor{s}{N}{t}}
	  \frac{(\floor{s}{N}{t}+\tstep{N}{t}-r) \,
	  \|Y_r\|_{\lpn{p}{\P}{H_1}}}{\tstep{N}{t}}
	\, dr
	\\&\leq
	\Cr{1}
	  \int^{\floor{s}{N}{t}+\tstep{N}{t}}_{\floor{s}{N}{t}}
	  \|Y_r\|_{\lpn{p}{\P}{H_1}}
	\, dr
	=
	\Cr{1}
	\int^t_0
	\mathbbm{1}_{(\floor{s}{N}{t},\floor{s}{N}{t}+\tstep{N}{t})}(r) \,
	\|Y_r\|_{\lpn{p}{\P}{H_1}}
	\, dr
	.
	\end{split}
	\end{equation}
	In addition, note that H\"{o}lder's inequality and the assumption that 
	$
	\int^T_0
	\ES\big[
	  \|Y_t\|^p_{H_1}
	  +
	  \|Z_t\|^p_{HS(U,H_1)}
	\big]
	\, dt
	$
	$
	< \infty
	$
	assure that 
	\begin{equation}
	\label{eq:Lp.integral}
	  \int^T_0
	  \|Y_t\|_{\lpn{p}{\P}{H_1}}
	  +
	  \|Z_t\|^2_{\lpn{p}{\P}{HS(U,H_1)}}
	  \, dt
	  < \infty.
	\end{equation}
	This, Lebesgue's theorem of dominated convergence, and~\eqref{eq:drift.diff.int.dominated} imply that for all $t\in(0,T]$ it holds that 
	\begin{equation}
	\label{eq:strong.decompose.drift}
	\begin{split}
	&
	  \limsup_{\N\ni N\to\infty}
	  \bigg(
	  \int^t_0
	  \int^{\floor{s}{N}{t}+\tstep{N}{t}}_{\floor{s}{N}{t}}
	  \frac{\|(e^{A(\floor{s}{N}{t}+\tstep{N}{t}-r)}-\operatorname{Id}_H)Y_r\|_{\lpn{p}{\P}{H}}}{\tstep{N}{t}}
	  \, dr
	  \, ds
	  \bigg)
	\\&\leq
	\Cr{1} \,
	\limsup_{\N\ni N\to\infty}
	\int^t_0
	\int^t_0
	\mathbbm{1}_{(\floor{s}{N}{t},\floor{s}{N}{t}+\tstep{N}{t})}(r) \,
	\|Y_r\|_{\lpn{p}{\P}{H_1}}
	\, dr
	\, ds
	=0.
	\end{split}
	\end{equation}
	Next observe that the Burkholder-Davis-Gundy type inequality 
	in Lemma~7.7 in Da Prato \& Zabczyk~\cite{dz92} shows that for all 
	$N\in\N$, $s,t\in(0,T]$ with $s < t$
	it holds that 
	\begin{equation}
	\begin{split}
	&
	\bigg\|
	  \int^{\floor{s}{N}{t}+\tstep{N}{t}}_{\floor{s}{N}{t}}
	  \frac{(e^{A(\floor{s}{N}{t}+\tstep{N}{t}-r)}-\operatorname{Id}_H)}{\tstep{N}{t}} \, Z_r
	  \, dW_r
	\bigg\|_{\lpn{p}{\P}{H}}
	\\&\leq
	\bigg[
	\frac{p \, (p-1)}{2}
	  \int^{\floor{s}{N}{t}+\tstep{N}{t}}_{\floor{s}{N}{t}}
	  \frac{\|(e^{A(\floor{s}{N}{t}+\tstep{N}{t}-r)}-\operatorname{Id}_H)Z_r\|^2_{\lpn{p}{\P}{HS(U,H)}}}{|\tstep{N}{t}|^2}
	  \, dr
	\bigg]^{\nicefrac{1}{2}}
	\\&\leq
	\bigg[
	\frac{p \, (p-1) \, |\Cr{1}|^2}{2}
	  \int^{\floor{s}{N}{t}+\tstep{N}{t}}_{\floor{s}{N}{t}}
	  \frac{(\floor{s}{N}{t}+\tstep{N}{t}-r)^2 \, \|Z_r\|^2_{\lpn{p}{\P}{HS(U,H_1)}}}{|\tstep{N}{t}|^2}
	  \, dr
	\bigg]^{\nicefrac{1}{2}}
	\\&\leq
	\bigg[
	\frac{p \, (p-1) \, |\Cr{1}|^2}{2}
	  \int^{\floor{s}{N}{t}+\tstep{N}{t}}_{\floor{s}{N}{t}}
	  \|Z_r\|^2_{\lpn{p}{\P}{HS(U,H_1)}}
	  \, dr
	\bigg]^{\nicefrac{1}{2}}
	\\&=
	\bigg[
	\frac{p \, (p-1) \, |\Cr{1}|^2}{2}
	\int^t_0
	\mathbbm{1}_{(\floor{s}{N}{t},\floor{s}{N}{t}+\tstep{N}{t})}(r) \,
	\|Z_r\|^2_{\lpn{p}{\P}{HS(U,H_1)}}
	\, dr
	\bigg]^{\nicefrac{1}{2}}
	.
	\end{split}
	\end{equation}
	H\"{o}lder's inequality hence implies that for all 
	$N\in\N$, $t\in(0,T]$
	it holds that 
	\begin{equation}
	\label{eq:diff.diff.int.dominated}
	\begin{split}
	&
	\int^t_0
	\bigg\|
	\int^{\floor{s}{N}{t}+\tstep{N}{t}}_{\floor{s}{N}{t}}
	\frac{(e^{A(\floor{s}{N}{t}+\tstep{N}{t}-r)}-\operatorname{Id}_H)}{\tstep{N}{t}} \, Z_r
	\, dW_r
	\bigg\|_{\lpn{p}{\P}{H}}
	\, ds
	\\&\leq
	\bigg[\frac{p \, (p-1) \, |\Cr{1}|^2}{2}\bigg]^{\nicefrac{1}{2}}
	\int^t_0
	\bigg[
	\int^t_0
	\mathbbm{1}_{(\floor{s}{N}{t},\floor{s}{N}{t}+\tstep{N}{t})}(r) \,
	\|Z_r\|^2_{\lpn{p}{\P}{HS(U,H_1)}}
	\, dr
	\bigg]^{\nicefrac{1}{2}}
	\, ds
	\\&\leq
	\bigg[\frac{p \, (p-1) \, |\Cr{1}|^2 \, t}{2}
	\int^t_0
	\int^t_0
	\mathbbm{1}_{(\floor{s}{N}{t},\floor{s}{N}{t}+\tstep{N}{t})}(r) \,
	\|Z_r\|^2_{\lpn{p}{\P}{HS(U,H_1)}}
	\, dr
	\, ds
	\bigg]^{\nicefrac{1}{2}}
	.
	\end{split}
	\end{equation}
	This, Lebesgue's theorem of dominated convergence, and~\eqref{eq:Lp.integral} ensure that for all $t\in(0,T]$ it holds that 
	\begin{equation}
	\label{eq:strong.decompose.diff}
	\begin{split}
	&
	  \limsup_{\N\ni N\to\infty}
	\bigg(
	\int^t_0
	\bigg\|
	  \int^{\floor{s}{N}{t}+\tstep{N}{t}}_{\floor{s}{N}{t}}
	  \frac{(e^{A(\floor{s}{N}{t}+\tstep{N}{t}-r)}-\operatorname{Id}_H)}{\tstep{N}{t}} \, Z_r
	  \, dW_r
	\bigg\|_{\lpn{p}{\P}{H}}
	\, ds
	\bigg)
	\\&\leq
	\bigg[\frac{p \, (p-1) \, |\Cr{1}|^2 \, t}{2} \,
	\limsup_{\N\ni N\to\infty}
	\bigg(
	\int^t_0
	\int^t_0
	\mathbbm{1}_{(\floor{s}{N}{t},\floor{s}{N}{t}+\tstep{N}{t})}(r) \,
	\|Z_r\|^2_{\lpn{p}{\P}{HS(U,H_1)}}
	\, dr
	\, ds
	\bigg)
	\bigg]^{\nicefrac{1}{2}}
	\\&=0.
	\end{split}
	\end{equation}
	Putting~\eqref{eq:strong.decompose}, \eqref{eq:strong.decompose.initial}, \eqref{eq:strong.decompose.drift}, and~\eqref{eq:strong.decompose.diff} together yields that for all 
	$t\in(0,T]$ it holds that 
	\begin{equation}
	\limsup_{\N\ni N\to\infty}
	\bigg(
	\int^t_0
	\bigg\|
	\frac{\mathbb{X}_{\floor{s}{N}{t}+\tstep{N}{t}}-\mathbb{X}_{\floor{s}{N}{t}}}{\tstep{N}{t}}
	-
	A\mathcal{X}_{\floor{s}{N}{t}}
	\bigg\|_{\lpn{p}{\P}{H}}
	\, ds
	\bigg)
	=0.
	\end{equation}
	Combining this and~\eqref{eq:dom.time.cont} with~\eqref{eq:strong.diff} shows that for all $t\in(0,T]$ it holds that 
	\begin{equation}
	\bigg\|
	  \mathbb{X}_t-\int^t_0 A\mathcal{X}_s \, ds
	\bigg\|_{\lpn{p}{\P}{H}}
	=0.
	\end{equation}
	This, \eqref{eq:mild.strong.SEE}, and~\eqref{eq:tildeX} imply that for all $t\in[0,T]$ it holds $\P$-a.s.\ that 
		\begin{equation}
		\label{eq:strong.solution.1st.version}
		\begin{split}
		  X_t
		  &=\xi
		  +
		  \int^t_0
		  A\mathcal{X}_s + Y_s
		  \, ds
		  +
		  \int^t_0
		  Z_s
		  \, dW_s
		  .
		\end{split}
		\end{equation}
	Moreover, 	items~\eqref{item:general.H1.reg} \& \eqref{item:integral.makes.sense} imply that for all $t\in[0,T]$
	it holds $\P$-a.s.\ that 
		\begin{equation}
		\begin{split}
		  \xi
		  +
		  \int^t_0
		  A\mathcal{X}_s + Y_s
		  \, ds
		  +
		  \int^t_0
		  Z_s
		  \, dW_s
		  =
		  \xi
		  +
		  \int^t_0
		  AX_s + Y_s
		  \, ds
		  +
		  \int^t_0
		  Z_s
		  \, dW_s
		  .
		\end{split}
		\end{equation}
	This, \eqref{eq:X.H1.version}, and~\eqref{eq:strong.solution.1st.version} establish item~\eqref{item:general.H1.solution}.
	The proof of Lemma~\ref{lem:gen.H1.solution} is thus completed.
	\end{proof}

\noindent The following result, Lemma~\ref{lem:time.derivative} below, can be shown by employing Lemma~\ref{lem:gen.H1.solution} above together with the standard It\^{o} formula in infinite dimensions (cf., e.g., Brze{\'z}niak et al.~\cite[Theorem~2.4]{bvvw08}).

\begin{lemma}
	\label{lem:time.derivative}
	Consider the notation in Section~\ref{sec:notation}, 
	let 
	$ ( H, \langle \cdot, \cdot \rangle_H, \left\| \cdot \right\|_H ) $, 
	$ ( U, \langle \cdot, \cdot \rangle_U, \left\| \cdot \right\|_U ) $, 
	and 
	$ ( V, \langle \cdot, \cdot \rangle_V, \left\| \cdot \right\|_V ) $
	be separable $ \R $-Hilbert spaces,
	let $ T \in (0,\infty) $,
	let 
	$ ( \Omega, \mathcal{F}, \P, 
	$
	$
	( \mathcal{F}_t )_{ t \in [0,T] } ) $
	be a stochastic basis,
	let $ ( W_t )_{ t \in [0,T] } $ be an $ \operatorname{Id}_U $-cylindrical
	$ ( \Omega, \mathcal{F}, \P, ( \mathcal{F}_t )_{ t \in [0,T] } ) $-Wiener process,
	let 
	$ A \colon D(A) \subseteq H \to H $ be a generator of a strongly continuous analytic semigroup with $\operatorname{spectrum}(A)\subseteq\{z\in\mathbb{C}\colon\operatorname{Re}(z)<0\}$,
	let
	$  
	( 
	H_r 
	,
	\langle \cdot, \cdot \rangle_{ H_r }
	,$ $
	\left\| \cdot \right\|_{ H_r } 
	)
	$,
	$ r \in \R $,
	be a family of interpolation spaces associated to $ - A $,
	let 
	$
	F \in \operatorname{Lip}^0(H,H_1)
	$, 
	$ 
	B \in \operatorname{Lip}^0(H,HS(U,H_1))
	$, 
	$
	\varphi \in C^2_b( H, V)
	$, 
	let
	$ X^x \colon 
	[0,T] \times \Omega \to H $, 
	$ x \in H $,
	be 
	$
	( \mathcal{F}_t )_{ t \in [0,T] }
	$-predictable stochastic processes
	which satisfy for all 
	$ x \in H $
	that
	$
	\sup_{ t \in [0,T] }
	\ES\big[\| X^x_t \|^2_H\big] 
	< \infty
	$ 
	and which satisfy that
	for all $ t \in [0,T] $, $ x \in H $
	it holds $ \P $-a.s.\ that
	\begin{equation}
	\label{eq:temporal.SEE}
	X^x_t =
	e^{ A t } x
	+
	\int_0^t
	e^{ A (t - s) }
	F( 
	X^x_s 
	)
	\,
	ds
	+
	\int_0^t
	e^{ A (t - s) }
	B( 
	X^x_s 
	)
	\,
	dW_s
	,
	\end{equation}
	and let 
	$
	u \colon [0,T] \times H \to V  
	$
	be the function which satisfies for all 
	$ t \in [0,T] $, 
	$ x \in H $
	that 
	$
	u( t, x ) = 
	\ES\big[ 
	\varphi( X^x_{ T - t } )
	\big]
	$.
	Then
	\begin{enumerate}[(i)]
		\item 
		\label{item:SEE.H1.regular}
		it holds for all $x\in H_1$, $t\in[0,T]$ that 
		$\P(X^x_t\in H_1)=1$, 
		\item
		\label{item:H1.solution.apriori}
		it holds for all $p\in[2,\infty)$ that 
		\begin{equation}
		\sup_{x\in H_1}
		\sup_{t\in[0,T]}
		\frac{
			\big(\ES\big[\|X^x_t \, \mathbbm{1}_{H_1}(X^x_t)\|^p_{H_1}\big]\big)^{\nicefrac{1}{p}}
		}{\max\{1,\|x\|_{H_1}\}}
		< \infty,
		\end{equation}
		\item
		\label{item:H1.space.time.continuous}
		it holds for all $x\in H_1$, $t\in[0,T]$ that 
		\begin{equation}
		\limsup_{[0,T]\times H_1\ni(s,y)\to(t,x)}
		\ES\big[
		\|X^x_t \, \mathbbm{1}_{H_1}(X^x_t)-X^y_s \, \mathbbm{1}_{H_1}(X^y_s)\|_{H_1}
		\big]
		=0
		,
		\end{equation}
	\item 
	\label{item:time.derivative.exist}
	it holds for all 
	$x\in H_1$ 
	that 
	$
	([0,T] \ni t \mapsto u(t,x) \in V)
	\in C^1([0,T],V)
	$, and
	\item
	\label{item:time.derivative.continuity}
	it holds that
	$
	\big([0,T]\times H_1 \ni(t,x)\mapsto (\tfrac{ \partial }{ \partial t } u)( t, x ) \in V\big)
	\in C([0,T]\times H_1,V)
	$.
	\end{enumerate}
\end{lemma}
\begin{proof}
	Throughout this proof 
	assume w.l.o.g.\ that $H\neq\{0\}$,
	let $ \mathbb{U} \subseteq U $ be an orthonormal basis of $ U $,
	let $\Co\in[0,\infty)$ be the real number given by 
	\begin{equation}
	\Co=\sup_{t\in[0,T]}\|e^{At}\|_{L(H)},
	\end{equation}
	let 
	$
	\mathcal{X}^x \colon [0,T] \times \Omega \to H_1
	$, 
	$ x \in H_1 $, 
	be the 
	$
	( \mathcal{F}_t )_{ t \in [0,T] }
	$-predictable
	stochastic processes which satisfy for all 
	$x\in H_1$,
	$t\in[0,T]$
	that 
	\begin{equation}
	\mathcal{X}^x_t
	=
	X^x_t \, \mathbbm{1}_{H_1}(X^x_t),
	\end{equation}
	and let 
	$\Phi\colon H_1 \to V$
	be the function which satisfies for all 
	$x\in H_1$ that 
	\begin{equation}
	\Phi(x)=
	\varphi'(x)(Ax+F(x))
	+
	\frac{1}{2}\,
	{\smallsum\limits_{b\in\mathbb{U}}}
	\varphi''(x)(B(x)b,B(x)b)
	.
	\end{equation}
	Observe that item~(i) of Corollary~2.10 in Andersson et al.~\cite{AnderssonJentzenKurniawan2016arXiv} 
		(with 
		$H=H$,
		$U=U$,
		$T=T$,
		$\eta=0$,
		$\alpha=0$,
		$\beta=0$,
		$W=W$,
		$A=A$,
		$F=(H \ni x \mapsto F(x)\in H)$, 
		$B=(H \ni x \mapsto (U\ni u \mapsto B(x)u \in H)\in HS(U,H))$ 
		in the notation of Corollary~2.10 in~\cite{AnderssonJentzenKurniawan2016arXiv})
	implies that there exist $(\mathcal{F}_t)_{t\in[0,T]}$-predictable stochastic processes 
	$\mathbb{X}^x\colon[0,T]\times\Omega\to H$, $x\in H$, which satisfy for all $p\in[2,\infty)$, $x\in H$ that 
	$\sup_{t\in[0,T]}\ES\big[\|\mathbb{X}^x_t\|^p_H\big]<\infty$
	and which satisfy that for all $t\in[0,T]$, $x\in H$ it holds $\P$-a.s.\ that 
	\begin{equation}
	\label{eq:higher.moment.SEE}
	\mathbb{X}^x_t =
	e^{ A t } x
	+
	\int_0^t
	e^{ A (t - s) }
	F( 
	\mathbb{X}^x_s 
	)
	\,
	ds
	+
	\int_0^t
	e^{ A (t - s) }
	B( 
	\mathbb{X}^x_s 
	)
	\,
	dW_s
	.
	\end{equation}
	In particular, this implies that for all $x\in H$ it holds that 
	\begin{equation}
	\label{eq:higher.moment.square.integrable}
	\sup_{t\in[0,T]}\ES\big[\|\mathbb{X}^x_t\|^2_H\big]<\infty
	\end{equation}
	Combining this and~\eqref{eq:higher.moment.SEE} with~\eqref{eq:temporal.SEE} shows that for all 
	$x\in H$, $t\in[0,T]$ it holds $\P$-a.s.\ that 
	\begin{equation}
	X^x_t=\mathbb{X}^x_t.
	\end{equation}
	(cf., e.g., Da Prato \& Zabczyk~\cite[Item~(i) of Theorem~7.2]{dz14},
		Da Prato et al.~\cite[Proposition~3]{DaPratoJentzenRoeckner2012}, 
		or Andersson et al.~\cite[Item~(i) of Theorem~2.9]{AnderssonJentzenKurniawan2016arXiv}).
	Therefore, we obtain that for all $p\in[2,\infty)$, $x\in H$, $t\in[0,T]$
	it holds that 
	\begin{equation}
	  \ES\big[\|X^x_t\|^p_H\big]
	  =
	  \ES\big[\|\mathbb{X}^x_t\|^p_H\big]
	  .
	\end{equation}
	This 
	ensures that
	for all 
	$p\in[2,\infty)$, 
	$x\in H$
	it holds that 
	\begin{equation}
	\label{eq:SEE.moments}
	  \sup_{t\in[0,T]}
	  \ES\big[\|X^x_t\|^p_H\big]
	  < \infty.
	\end{equation}
	This, in turn, demonstrates that for all 
	$p\in[2,\infty)$, 
	$x\in H$
	it holds that 
	\begin{equation}
	  \begin{split}
	  &
	  \sup_{t\in[0,T]}
	  \big[
	  \|F(X^x_t)\|_{\lpn{p}{\P}{H_1}}
	  +
	  \|B(X^x_t)\|_{\lpn{p}{\P}{HS(U,H_1)}}
	  \big]
	  \\&\leq
	  \big(
	    \|F\|_{\operatorname{Lip}^0(H,H_1)}
	    +
	    \|B\|_{\operatorname{Lip}^0(H,H_1)}
	  \big)
	  \sup_{t\in[0,T]}
	  \max\big\{
	  1,
	  \|X^x_t\|_{\lpn{p}{\P}{H}}
	  \big\}
	  < \infty.
	  \end{split}
	\end{equation}
	Hence, we obtain that for all 
	$p\in[2,\infty)$, 
	$x\in H$ it holds that 
	\begin{equation}
	\label{eq:strong.integrability}
	  \int^T_0
	  \ES\big[
	  \|F(X^x_t)\|^p_{H_1}+\|B(X^x_t)\|^p_{HS(U,H_1)}
	  \big]
	  \, dt
	  < \infty.
	\end{equation}
	This, \eqref{eq:temporal.SEE}, and items~\eqref{item:general.H1.mild.integral}, \eqref{item:general.H1.reg}, \eqref{item:integral.makes.sense}, and~\eqref{item:general.H1.solution} of Lemma~\ref{lem:gen.H1.solution} imply that 
	\begin{enumerate}[(I)]
		\item 
		it holds for all 
		$p\in[2,\infty)$, 
		$x\in H_1$
		that
		\begin{equation}
		\label{eq:H1.mild.integral.II}
		\begin{split}
		&
		\sup_{t\in[0,T]}
		\int^t_0
		\ES\big[
		\|
		e^{A(t-s)}F(X^x_s)
		\|^p_{H_1}
		+
		\|e^{A(t-s)}B(X^x_s)\|^p_{HS(U,H_1)}
		\big]
		\, ds
		< \infty
		,
		\end{split}
		\end{equation}
		\item\label{item:shrek} it holds for all 
		$x\in H_1$, 
		$t\in[0,T]$
		that 
		\begin{equation}
		\label{eq:H1.modification}
		\P\big(
		X^x_t \in H_1
		\big)
		=
		\P\big(
		\mathcal{X}^x_t = X^x_t
		\big)
		=1,
		\end{equation}
		\item
		it holds for all $x\in H_1$ that 
		\begin{equation}
		\P\!\left(
		\int^T_0
		\|AX^x_s\|_{H_{-1}}
		+
		\|F(X^x_s)\|_{H_1}
		+
		\|B(X^x_s)\|^2_{HS(U,H_1)}
		\, ds
		< \infty
		\right)
		=1,
		\end{equation}
		and
		\item 
		for all 
		$x\in H_1$, $t\in[0,T]$ it holds $\P$-a.s.\ that 
		\begin{equation}
		\label{eq:analytic.strong.solution}
		X^x_t
		=
		x
		+
		\int^t_0
		AX^x_s + F(X^x_s) 
		\, ds
		+
		\int^t_0
		B(X^x_s) 
		\, dW_s
		.
		\end{equation}
	\end{enumerate}
	Observe that item~\eqref{item:shrek} proves item~\eqref{item:SEE.H1.regular}.
	In the next step we combine~\eqref{eq:H1.mild.integral.II}--\eqref{eq:H1.modification} with~\eqref{eq:temporal.SEE} to obtain that for all 
	$x\in H_1$, 
	$t\in[0,T]$
	it holds $\P$-a.s.\ that 
	\begin{equation}
	\label{eq:H1.modification.SEE.I}
	A\mathcal{X}^x_t
	=
	e^{At}Ax
	+
	\int^t_0
	e^{A(t-s)}A F(X^x_s) \, ds
	+
	\int^t_0
	e^{A(t-s)}A B(X^x_s) \, dW_s
	.
	\end{equation}
	The Burkholder-Davis-Gundy type inequality 
	in Lemma~7.7 in Da Prato \& Zabczyk~\cite{dz92} therefore shows
	that for all 
	$p\in[2,\infty)$, 
	$x\in H_1$, $t\in[0,T]$ it holds that 
	\begin{equation}
	\begin{split}
	&
	\|A\mathcal{X}^x_t\|_{\lpn{p}{\P}{H}}
	\\&\leq
	\|e^{At}Ax\|_H
	+
	\int^t_0
	\|e^{A(t-s)}A F(X^x_s)\|_{\lpn{p}{\P}{H}}
	\, ds
	\\&\quad+
	\left[
	\frac{p \, (p-1)}{2}
	\int^t_0
	\|e^{A(t-s)}A B(X^x_s)\|^2_{\lpn{p}{\P}{HS(U,H)}}
	\, ds
	\right]^{\nicefrac{1}{2}}
	\\&\leq
	\Co\,
	\|Ax\|_H
	+
	\int^t_0
	\|e^{A(t-s)}F(X^x_s)\|_{\lpn{p}{\P}{H_1}}
	\, ds
	\\&\quad+
	\left[
	\frac{p \, (p-1)}{2}
	\int^t_0
	\|e^{A(t-s)}B(X^x_s)\|^2_{\lpn{p}{\P}{HS(U,H_1)}}
	\, ds
	\right]^{\nicefrac{1}{2}}
	.
	\end{split}
	\end{equation}
	This, H\"{o}lder's inequality, and~\eqref{eq:H1.mild.integral.II} imply that for all 
	$p\in[2,\infty)$, 
	$x\in H_1$ it holds that 
	\begin{equation}
	\sup_{t\in[0,T]}
	\|A\mathcal{X}^x_t\|_{\lpn{p}{\P}{H}}
	< \infty.
	\end{equation}
	Therefore, we obtain for all $p\in[2,\infty)$, $x\in H$ that 
	\begin{equation}
	\label{eq:H1.modification.apriori}
	\sup_{t\in[0,T]}
	\|A\mathcal{X}^{A^{-1}x}_t\|_{\lpn{p}{\P}{H}}
	< \infty.
	\end{equation}
	Furthermore, observe that~\eqref{eq:H1.mild.integral.II} and~\eqref{eq:H1.modification} yield that for all 
	$x\in H_1$, 
	$t\in[0,T]$ 
	it holds $\P$-a.s.\ that 
	\begin{equation}
	\int^t_0
	\|e^{A(t-s)}AF(\mathcal{X}^x_s)\|_H
	+
	\|e^{A(t-s)}AB(\mathcal{X}^x_s)\|^2_{HS(U,H)}
	\, ds
	< \infty
	\end{equation}
	and
	\begin{multline}
	\int^t_0
	e^{A(t-s)}AF(\mathcal{X}^x_s)
	\, ds
	+
	\int^t_0
	e^{A(t-s)}AB(\mathcal{X}^x_s)
	\, dW_s
	\\=
	\int^t_0
	e^{A(t-s)}AF(X^x_s)
	\, ds
	+
	\int^t_0
	e^{A(t-s)}AB(X^x_s)
	\, dW_s
	.
	\end{multline}
	This and~\eqref{eq:H1.modification.SEE.I} yield that 
	for all 
	$x\in H_1$, 
	$t\in[0,T]$
	it holds $\P$-a.s.\ that 
	\begin{equation}
	\begin{split}
	A\mathcal{X}^x_t
	&=
	e^{At}Ax
	+
	\int^t_0
	e^{A(t-s)}A F(\mathcal{X}^x_s) \, ds
	+
	\int^t_0
	e^{A(t-s)}A B(\mathcal{X}^x_s) \, dW_s
	\\&=
	e^{At}Ax
	+
	\int^t_0
	e^{A(t-s)}A F(A^{-1}(A\mathcal{X}^x_s)) \, ds
	+
	\int^t_0
	e^{A(t-s)}A B(A^{-1}(A\mathcal{X}^x_s)) \, dW_s
	.
	\end{split}
	\end{equation}
	Hence, we obtain that for all $x\in H$, $t\in[0,T]$ it holds 
	$\P$-a.s.\ that 
	\begin{equation}
	\label{eq:H1.modification.SEE.II}
	\begin{split}
	A\mathcal{X}^{A^{-1}x}_t
	&=
	e^{At}x
	+
	\int^t_0
	e^{A(t-s)}A F(A^{-1}(A\mathcal{X}^{A^{-1}x}_s)) \, ds
	\\&\quad+
	\int^t_0
	e^{A(t-s)}A B(A^{-1}(A\mathcal{X}^{A^{-1}x}_s)) \, dW_s
	.
	\end{split}
	\end{equation}
	Moreover, note that for all $x,y\in H$ it holds that 
	\begin{equation}
	\label{eq:AFA.lip}
	\|AF(A^{-1}x)-AF(A^{-1}y)\|_H
	\leq 
	|F|_{\operatorname{Lip}^0(H,H_1)} \, \|A^{-1}\|_{L(H)} \, \|x-y\|_H
	\end{equation}
	and
	\begin{equation}
	\label{eq:ABA.lip}
	\|AB(A^{-1}x)-AB(A^{-1}y)\|_{HS(U,H)}
	\leq 
	|B|_{\operatorname{Lip}^0(H,HS(U,H_1))} \, \|A^{-1}\|_{L(H)} \, \|x-y\|_H
	.
	\end{equation}
	Combining~\eqref{eq:H1.modification.apriori} and~\eqref{eq:H1.modification.SEE.II}--\eqref{eq:ABA.lip} with items~(i)--(iii) of Corollary~2.10 in Andersson et al.~\cite{AnderssonJentzenKurniawan2016arXiv} 
		(with 
		$H=H$,
		$U=U$,
		$T=T$,
		$\eta=0$,
		$\alpha=0$,
		$\beta=0$,
		$W=W$,
		$A=A$,
		$F=(H \ni x \mapsto AF(A^{-1}x)\in H)$, 
		$B=(H \ni x \mapsto (U\ni u \mapsto A[B(A^{-1}x)u] \in H)\in HS(U,H))$
		in the notation of Corollary~2.10 in~\cite{AnderssonJentzenKurniawan2016arXiv})
	shows that for all $p\in[2,\infty)$ it holds that 
	\begin{equation}
	\label{eq:H1.modification.apriori.bound.0}
	\begin{split}
	\sup_{x\in H}
	\sup_{t\in(0,T]}
	\left[
	\frac{\|A\mathcal{X}^{A^{-1}x}_t\|_{\lpn{p}{\P}{H}}}{
		\max\{1,\|x\|_H\}
	}
	\right]
	< \infty
	\end{split}
	\end{equation}
	and 
	\begin{equation}
	\label{eq:H1.space.continuous.0}
	\sup_{\substack{x,y\in H,\\x\neq y}}
	\sup_{t\in(0,T]}
	\left[
	\frac{\|A\mathcal{X}^{A^{-1}x}_t-A\mathcal{X}^{A^{-1}y}_t\|_{\lpn{p}{\P}{H}}}{\|x-y\|_H}
	\right]
	< \infty
	.
	\end{equation}
	Combining~\eqref{eq:H1.modification.apriori.bound.0} and~\eqref{eq:H1.space.continuous.0} demonstrates that for all 
	$p\in[2,\infty)$ it holds that 
	\begin{equation}
	\label{eq:H1.modification.apriori.bound}
	\begin{split}
	\sup_{x\in H_1}
	\sup_{t\in[0,T]}
	\left[
	\frac{\|\mathcal{X}^x_t\|_{\lpn{p}{\P}{H_1}}}{
		\max\{1,\|x\|_{H_1}\}
	}
	\right]
	\leq
	\sup_{x\in H}
	\sup_{t\in[0,T]}
	\left[
	\frac{\|\mathcal{X}^{A^{-1}x}_t\|_{\lpn{p}{\P}{H_1}}}{
		\max\{1,\|x\|_H\}
	}
	\right]
	< \infty
	\end{split}
	\end{equation}
	and 
	\begin{equation}
	\label{eq:H1.space.continuous}
	\!\!\!
	\sup_{\substack{x,y\in H_1,\\x\neq y}}
	\sup_{t\in[0,T]}
	\left[
	\frac{\|\mathcal{X}^{x}_t-\mathcal{X}^{y}_t\|_{\lpn{p}{\P}{H_1}}}{\|x-y\|_{H_1}}
	\right]
	\leq
	\sup_{\substack{x,y\in H,\\x\neq y}}
	\sup_{t\in[0,T]}
	\left[
	\frac{\|\mathcal{X}^{A^{-1}x}_t-\mathcal{X}^{A^{-1}y}_t\|_{\lpn{p}{\P}{H_1}}}{\|x-y\|_H}
	\right]
	< \infty
	.
	\end{equation}
	Note that item~\eqref{item:H1.solution.apriori} follows from~\eqref{eq:H1.modification.apriori.bound}.
	Moreover, observe that~\eqref{eq:H1.modification.apriori}, \eqref{eq:H1.modification.SEE.II}--\eqref{eq:ABA.lip}, and, e.g., 
	Proposition~3 in Da Prato et al.~\cite{DaPratoJentzenRoeckner2012}
	(with 
	$T=T$, 
	$H=H$, 
	$U=U$, 
	$Q=\operatorname{Id}_U$, 
	$W=W$, 
	$A=A$, 
	$\eta=0$, 
	$\alpha=0$, 
	$\gamma=0$, 
	$F=(H\ni v \mapsto AF(A^{-1}v)\in H)$, 
	$\beta=0$, 
	$B=(H\ni v \mapsto (U\ni u \mapsto A[B(A^{-1}v)u] \in H) \in HS(U,H))$, 
	$p=p$, 
	$\xi=(\Omega\ni\omega\mapsto Ax \in H)$
	for
	$x\in H_1$, $p\in[2,\infty)$
	in the notation of Proposition~3 in~\cite{DaPratoJentzenRoeckner2012}) 
	show that for all $p\in[2,\infty)$, $x\in H_1$ it holds that the function
	\begin{equation}
	\label{eq:Z.time.continuous}
	[0,T]\ni t \mapsto A\mathcal{X}^x_t \in \lpn{p}{\P}{H}
	\end{equation}
	is continuous.
	Next note that H\"{o}lder's inequality shows that for all 
	$s,t\in[0,T]$, 
	$x,y\in H_1$
	it holds that 
	\begin{equation}
	\label{eq:H1.modification.triangle}
	\begin{split}
	&
	\ES\big[\|\mathcal{X}^x_s-\mathcal{X}^y_t\|_{H_1}\big]
	\\&\leq
	\ES\big[\|\mathcal{X}^x_s-\mathcal{X}^y_s\|_{H_1}\big]
	+
	\ES\big[\|\mathcal{X}^y_s-\mathcal{X}^y_t\|_{H_1}\big]
	\\&\leq
	\|\mathcal{X}^x_s-\mathcal{X}^y_s\|_{\lpn{2}{\P}{H_1}}
	+
	\|\mathcal{X}^y_s-\mathcal{X}^y_t\|_{\lpn{2}{\P}{H_1}}
	\\&\leq
	\left[
	\sup_{\substack{u,v\in H_1,\\ u\neq v}}
	\sup_{w\in[0,T]}
	\frac{
		\|\mathcal{X}^u_w-\mathcal{X}^v_w\|_{\lpn{2}{\P}{H_1}}
	}{
		\|u-v\|_{H_1}
	}
	\right]
	\|x-y\|_{H_1}
	+\|\mathcal{X}^y_s-\mathcal{X}^y_t\|_{\lpn{2}{\P}{H_1}}
	.
	\end{split}
	\end{equation}
	Combining~\eqref{eq:H1.space.continuous}--\eqref{eq:H1.modification.triangle} yields that for all 
	$t\in[0,T]$, 
	$x\in H_1$
	it holds that 
	\begin{equation}
	\limsup_{[0,T]\times H_1\ni(s,y)\to(t,x)}
	\ES\big[\|\mathcal{X}^x_t-\mathcal{X}^y_s\|_{H_1}\big]
	=
	0.
	\end{equation}
	This establishes item~\eqref{item:H1.space.time.continuous}.
	In the next step we observe that~\eqref{eq:H1.modification.apriori.bound} ensures that for all 
	$p\in[2,\infty)$, 
	$x\in H_1$
	it holds that 
	\begin{equation}
	\label{eq:H1.Z.integrability}
	\begin{split}
	\int^T_0
	\ES\big[
	\|\mathcal{X}^x_t\|^p_{H_1}
	\big]
	\, dt
	< \infty
	.
	\end{split}
	\end{equation}
	Combining~\eqref{eq:strong.integrability}, \eqref{eq:H1.modification}, \eqref{eq:analytic.strong.solution}, and~\eqref{eq:H1.Z.integrability} with
	the standard It\^{o} formula in infinite dimensions (cf., e.g., Brze{\'z}niak et al.~\cite[Theorem~2.4]{bvvw08}) implies that for all 
	$x\in H_1$, $t\in[0,T]$ it holds $\P$-a.s.\ that 
	\begin{equation}
	\label{eq:ito.formula}
	\varphi(\mathcal{X}^x_t)-\varphi(x)
	=
	\int^t_0
	\Phi(\mathcal{X}^x_s)
	\, ds
	+
	\int^t_0
	\varphi'(\mathcal{X}^x_s) B(\mathcal{X}^x_s)
	\, dW_s
	.
	\end{equation}
	In addition, observe that for all $x\in H_1$ it holds that 
	\begin{equation}
	\label{eq:psi.growth}
	\begin{split}
	\|\Phi(x)\|_V
	&\leq
	|\varphi|_{C^1_b(H,V)} \,
	\big(\|Ax\|_H+\|F(x)\|_H\big)
	+
	\frac{1}{2} \, |\varphi|_{C^2_b(H,V)} \,
	\|B(x)\|^2_{HS(U,H)}
	\\&\leq
	|\varphi|_{C^1_b(H,V)} \,
	\big(\|x\|_{H_1}+\|F\|_{\operatorname{Lip}^0(H,H)}  \max\{1,\|x\|_H\}\big)
	\\&\quad+
	\frac{1}{2} \,
	|\varphi|_{C^2_b(H,V)} \,
	\|B\|^2_{\operatorname{Lip}^0(H,HS(U,H))} 
	\max\{1,\|x\|^2_H\}
	.
	\end{split}
	\end{equation}
	This, \eqref{eq:H1.modification}, \eqref{eq:H1.Z.integrability}, and~\eqref{eq:ito.formula} show that for all 
	$x\in H_1$, 
	$t\in[0,T]$
	it holds that 
	$
	\ES\big[\|\varphi(X^x_t)\|_V+\|\varphi(\mathcal{X}^x_t)\|_V
	+
	\int^t_0
	\|\Phi(\mathcal{X}^x_s)\|_V
	+
	\|\varphi'(\mathcal{X}^x_s) B(\mathcal{X}^x_s)\|^2_{HS(U,V)}
	\, ds
	\big]
	< \infty
	$
	and 
	\begin{equation}
	\label{eq:ito.time.derivative}
	\begin{split}
	\ES\big[\varphi(X^x_t)\big]=
	\ES\big[\varphi(\mathcal{X}^x_t)\big]&=
	\varphi(x)
	+
	\int^t_0
	\ES\big[
	\Phi(\mathcal{X}^x_s)
	\big]
	\, ds
	.
	\end{split}
	\end{equation}
	Moreover, item~\eqref{item:H1.space.time.continuous}, the fact that 
	$\Phi\in C(H_1,V)$, and, e.g., Lemma~2.4 in Cox et al.~\cite{CoxJentzenKurniawanPusnik2016} 
	(with
	$(\Omega,\mathcal{F},\nu)=(\Omega,\mathcal{F},\P)$,
	$E=H_1$,
	$\mathcal{E}=V$,
	$\phi=\Phi$,
	$f_n=(\Omega\ni\omega\mapsto \mathcal{X}^{x_n}_{t_n}(\omega)\in H_1)$ 
	for 
	$n\in\N_0$, 
	$((t_m,x_m))_{m\in\N_0}\in\{((s,y)=((s_m,y_m))_{m\in\N_0}\colon\N_0\to[0,T]\times H_1)\colon \limsup_{\N\ni m\to\infty}|s_m-s_0|+\|y_m-y_0\|_{H_1}=0\}$
	in the notation of Lemma~2.4 in~\cite{CoxJentzenKurniawanPusnik2016}) ensure that for all 
	$((t_n,x_n))_{n\in\N_0}\subseteq [0,T]\times H_1$
	with 
	$
	\limsup_{\N\ni n\to\infty}
	|t_n-t_0|+\|x_n-x_0\|_{H_1}=0
	$
	it holds that 
	\begin{equation}
	\label{eq:psi.prob.continuous}
	\limsup_{\N\ni n\to\infty}
	\ES\big[\!\min\{1,\|
	\Phi(\mathcal{X}^{x_n}_{t_n})
	-
	\Phi(\mathcal{X}^{x_0}_{t_0})
	\|_V\}\big]
	=0
	.
	\end{equation}
	Next note that~\eqref{eq:psi.growth} implies that for all 
	$p\in[2,\infty)$, 
	$x\in H_1$, 
	$t\in[0,T]$
	it holds that
	\begin{equation}
	\begin{split}
	&
	\|\Phi(\mathcal{X}^x_t)\|_{\lpn{p}{\P}{V}}
	\\&\leq
	|\varphi|_{C^1_b(H,V)} \,
	\Big(\|\mathcal{X}^x_t\|_{\lpn{p}{\P}{H_1}}+\|F\|_{\operatorname{Lip}^0(H,H)} \, \|\!\max\{1,\|\mathcal{X}^x_t\|_H\}\|_{\lpn{p}{\P}{\R}}\Big)
	\\&\quad+
	\frac{1}{2} \,
	|\varphi|_{C^2_b(H,V)} \,
	\|B\|^2_{\operatorname{Lip}^0(H,HS(U,H))} \,
	\|\!\max\{1,\|\mathcal{X}^x_t\|^2_H\}\|_{\lpn{p}{\P}{\R}}
	.
	\end{split}
	\end{equation}
	Item~\eqref{item:H1.solution.apriori} hence ensures that for all 
	$p\in[2,\infty)$, 
	$x\in H_1$
	it holds that 
	\begin{equation}
	\begin{split}
	&
	\sup_{\substack{y\in H_1,\\ \|x-y\|_{H_1}\leq 1}}
	\sup_{t\in[0,T]}
	\|\Phi(\mathcal{X}^y_t)\|_{\lpn{p}{\P}{V}}
	\leq
	|\varphi|_{C^1_b(H,V)} \,
	\left[
	\sup_{\substack{y\in H_1,\\ \|x-y\|_{H_1}\leq 1}}
	\max\{1,\|y\|_{H_1}\}
	\right]
	\\&\cdot
	\Bigg(
	\bigg[
	\sup_{y\in H_1}
	\sup_{t\in[0,T]}
	\frac{
		\|\mathcal{X}^y_t\|_{\lpn{p}{\P}{H_1}}
	}{
	\max\{1,\|y\|_{H_1}\}
	}
	\bigg]
	+
	\|F\|_{\operatorname{Lip}^0(H,H)} \,
	\bigg[
	\sup_{y\in H_1}
	\sup_{t\in[0,T]}
	\frac{
		\|\!\max\{1,\|\mathcal{X}^y_t\|_H\}\|_{\lpn{p}{\P}{\R}}
	}{
	\max\{1,\|y\|_{H_1}\}
	}
	\bigg]
	\Bigg)
	\\&+
	\frac{1}{2} \,
	|\varphi|_{C^2_b(H,V)} \,
	\|B\|^2_{\operatorname{Lip}^0(H,HS(U,H))} \,
	\left[
	\sup_{\substack{y\in H_1,\\ \|x-y\|_{H_1}\leq 1}}
	\max\{1,\|y\|_{H_1}\}
	\right]^2
	\\&\cdot
	\bigg[
	\sup_{y\in H_1}
	\sup_{t\in[0,T]}
	\frac{
		\|\!\max\{1,\|\mathcal{X}^y_t\|_H\}\|_{\lpn{2p}{\P}{\R}}
	}{
	\max\{1,\|y\|_{H_1}\}
	}
	\bigg]^2
	< \infty.
	\end{split}
	\end{equation}
	Combining this, the fact that $\Phi\in C(H_1,V)$, and~\eqref{eq:psi.prob.continuous} with, e.g., Proposition~4.5 in Hutzenthaler et al.~\cite{HutzenthalerJentzenSalimova2016arXiv} 
	(with
	$I=\{\emptyset\}$, 
	$p=p$, 
	$V=V$, 
	$X^n=\big(\{\emptyset\}\times\Omega\ni(\emptyset,\omega)\mapsto
	\Phi(\mathcal{X}^{x_n}_{s_n}(\omega))
	\in V\big)$, 
	$q=1$
	for 
	$n\in\N_0$, 
	$((s_m,x_m))_{m\in\N_0}\in\{((t,y)=((t_m,y_m))_{m\in\N_0}\colon\N_0\to[0,T]\times H_1)\colon \limsup_{\N\ni m\to\infty}|t_m-t_0|+\|y_m-y_0\|_{H_1}=0\}$, 
	$p\in[2,\infty)$
	in the notation of Proposition~4.5 in~\cite{HutzenthalerJentzenSalimova2016arXiv}) 
	yields that for all 
	$p\in[2,\infty)$, 
	$((s_n,x_n))_{n\in\N_0}\subseteq [0,T]\times H_1$
	with 
	$
	\limsup_{\N\ni n\to\infty}
	|s_n-s_0|+\|x_n-x_0\|_{H_1}=0
	$
	it holds that 
	\begin{equation}
	\limsup_{\N\ni n\to\infty}
	\ES\big[
	\|\Phi(\mathcal{X}^{x_n}_{s_n})
	-
	\Phi(\mathcal{X}^{x_0}_{s_0})
	\|_V\big]
	=0.
	\end{equation}
	This assures that the function
	\begin{equation}
	\label{eq:time.derivative.continuous}
	[0,T]\times H_1\ni (s,x) \mapsto 
	\ES\big[
	\Phi(\mathcal{X}^x_s)
	\big]
	\in V
	\end{equation}
	is continuous.
	In particular, it holds for all $x\in H_1$ that the function
	\begin{equation}
	[0,T]\ni s \mapsto 
	\ES\big[
	\Phi(\mathcal{X}^x_s)
	\big]
	\in V
	\end{equation}
	is continuous.
	This, \eqref{eq:ito.time.derivative}, the fundamental theorem of calculus, and the fact that
	\begin{equation}
	\forall \, t\in[0,T], \, x\in H \colon
	\quad
	\ES\big[\varphi(X^x_t)\big]
	=
	u(T-t,x)
	\end{equation}
	show that for all $x\in H_1$, $t\in[0,T]$ it holds that 
	$
	([0,T] \ni s \mapsto u(s,x) \in V)
	\in C^1([0,T],V)
	$
	and
	\begin{equation}
	  (\tfrac{\partial}{\partial t}u)(t,x)
	  =
	  -\ES\big[
	  \Phi(\mathcal{X}^x_t)
	  \big].
	\end{equation}
	This and~\eqref{eq:time.derivative.continuous} establish items~\eqref{item:time.derivative.exist}--\eqref{item:time.derivative.continuity}. 
	The proof of Lemma~\ref{lem:time.derivative} is thus completed.
\end{proof}

\subsection{Setting}
\label{sec:setting}

Assume the setting in Section~\ref{sec:general_setting},
let $ \mathbb{U} \subseteq U $ be an orthonormal basis of $ U $,
let 
$ \vartheta \in [ 0 , \nicefrac{ 1 }{ 2 } ) $,
$
  F \in C^4_b(H,H_2)
$, 
$ 
  B \in C^4_b(H,HS(U,H_2))
$, 
$
  \varphi \in C^4_b( H, \R)
$, 
$
  \xi\in\lpn{4}{ \P|_{ \mathcal{F}_0 } }{H_2}
$, 
let
$ \varsigma_{ F, B } \in \R $
be given by
$  
  \varsigma_{ F, B } 
  =
    \max\!\big\{
    1 ,
    \|
      F
    \|_{ 
      C_b^3( H, H_{-\vartheta} )
    }^2
    ,
    \|
      B
    \|_{ 
      C_b^3( H, HS( U, H_{-\nicefrac{\vartheta}{2}} ) )
    }^4
    \big\}
$,
let
$ ( F_{ I } )_{ I \in \grid(\set) } \subseteq C( H, H ) $,
$ ( B_I )_{ I \in \grid(\set) } \subseteq C( H, HS(U,H) ) $,
$ ( B^b )_{ b \in \mathbb{U} } \subseteq C( H, H ) $,
and
$ ( B_{ I }^b )_{ I \in \grid(\set), b \in \mathbb{U} } \subseteq C( H, H ) $
satisfy for all 
$ I \in \grid(\set) $, 
$ b \in \mathbb{U} $, 
$ u \in U $, 
$ v \in H $ 
that 
\begin{equation}
  F_{ I }( v ) 
  = 
    F\big( 
      P_I(v) 
    \big) 
  ,
  \quad
  B_I( v ) \, u
  = 
    B\big( 
      P_I(v) 
    \big)
    \,
    u
  ,
    \quad
     B^b( v ) 
    = 
      B( 
        v 
      )
      \,
      b
    ,
  \quad
   B_{ I }^b( v ) 
  = 
    B\big( 
      P_I(v) 
    \big)
    \,
    b 
    ,
\end{equation}
let
$ ( g_r )_{ r \in [0,\infty) } \subseteq C( H , \R ) $
satisfy for all 
$ r \in [0,\infty) $,
$ x \in H $
that 
$
  g_r( x ) = \max\{ 1, \| x \|_{ H }^r \}
$,
let
$ 
  X^I \colon [0,T] \times \Omega \to P_I( H )
$,
$ I \in \mathcal{P}(\set) $,
$
  Y^I \colon [0,T] \times \Omega \rightarrow  H_2
$, 
$ I \in \grid( \set ) $, 
and
$ X^{ \set, x } \colon [0,T] \times \Omega \to H $, 
$ x \in H $,
be 
$
  ( \mathcal{F}_t )_{ t \in [0,T] }
$-predictable stochastic processes
which satisfy for all 
$ I \in \grid(\set) $, $ x \in H $
that
$
  \sup_{ t \in [0,T] }
  \ES\big[
  \| X^I_t \|^4_H 
  +
  \| Y^I_t \|^4_{ H_2 } 
  +
  \| X^{ \set, x }_t \|^4_H 
  \big] < \infty
$ 
and which satisfy that
for all $ t \in [0,T] $, $ I \in \grid(\set) $, $ x \in H $
it holds $ \P $-a.s.\ that
\begin{equation}
\label{eq:Galerkin}
  X_t^I 
  = 
    e^{ A t } P_{ I }( \xi )
  + 
    \int_0^t e^{ A ( t - s ) } P_I F( X^I_s ) \, ds
  + 
    \int_0^t e^{ A ( t - s ) } P_I B( X^I_s ) \, dW_s 
  ,
\end{equation}
\begin{equation}\label{eq:mod_SPDE}
  Y^I_t
  =
  e^{At}
  \xi
  +
  \int^t_0
  e^{ A (t-s) }
  F_I( 
    Y^I_s  
  )
  \, ds
  +
  \int^t_0
  e^{ A (t - s ) }
  B_I( 
    Y^I_s
  ) 
  \, dW_s
  ,
\end{equation}
\begin{equation}
\label{eq:unique.SEE}
  X^{\set, x }_t =
  e^{ A t } x
  +
  \int_0^t
  e^{ A (t - s) }
  F( 
    X^{\set, x }_s 
  )
  \,
  ds
  +
  \int_0^t
  e^{ A (t - s) }
  B( 
    X^{\set, x }_s 
  )
  \,
  dW_s
  ,
\end{equation}
let 
$
  u \colon [0,T] \times H \to \R  
$
be the function which satisfies for all 
$ t \in [0,T] $, 
$ x \in H $
that 
$
  u( t, x ) = 
  \ES\big[ 
    \varphi( X^{\set, x }_{ T - t } )
  \big]
$,
let 
$
  c_{ \delta_1, \dots, \delta_k }
  \in [0,\infty]	
$,
$ \delta_1, \dots, \delta_k \in (-\infty,0] $,
$ k \in \{ 1, 2, 3, 4 \} $, 
be the extended real numbers which satisfy for all 
$ k \in \{ 1, 2, 3, 4 \} $, 
$ \delta_1, \dots, \delta_k \in (-\infty,0] $
that
\begin{equation}
\begin{split} 
&
  c_{ \delta_1, \dots, \delta_k }
  =
  \sup_{
    t \in [0,T)
  }
  \sup_{ 
    x \in H
  }
  \sup_{ 
    v_1, \dots, v_k \in H \backslash \{ 0 \}
  }
  \left[
  \frac{
    \big|
      ( 
        \frac{ 
          \partial^k
        }{
          \partial x^k
        }
        u
      )( t, x )( v_1, \dots, v_k )
    \big|
  }{
    ( T - t )^{ 
      (
        \delta_1 + \ldots + \delta_k
      ) 
    }
    \left\| v_1 \right\|_{ H_{ \delta_1 } }
    \cdot
    \ldots
    \cdot
    \left\| v_k \right\|_{ H_{ \delta_k } }
  }
  \right]
\end{split}
\end{equation}
(cf., e.g., item~\eqref{item:kolmogorov.diff} of Lemma~\ref{lem:Kolmogorov}), 
and let
$ 
  ( K^I_r )_{ 
    r \in (0,4] ,\, I \in \grid(\set) 
  } 
  \subseteq [0,\infty) 
$
satisfy for all 
$
  I \in \grid(\set) 
$, 
$
  r \in (0,4] 
$ 
that 
$
  K^I_r = 
  \sup_{ t \in [0,T] }
  \ES\big[
    g_r( Y_t^I ) 
  \big]
$.

\subsection{Weak convergence results}
\label{sec:weak_regular}

\begin{lemma}
	\label{lem:c.delta.finite}
	Assume the setting in Section~\ref{sec:setting}. Then it holds for all 
	$ k \in \{ 1, 2, 3, 4 \} $, 
	$ \delta_1, \dots, \delta_k \in (-\nicefrac{1}{2},0] $
	with
	$
	\sum^k_{i=1} \delta_i
	> -\nicefrac{1}{2}
	$ 
	that 
	$
	c_{\delta_1,\ldots,\delta_k}
	< \infty
	$.
\end{lemma}
\begin{proof}
	Throughout this proof let 
	$\phi\colon[0,T]\times H\to\R$ 
	be the function which satisfies for all 
	$t\in[0,T]$,
	$x\in H$ 
	that 
	$
	\phi(t,x)
	=\ES[\varphi(X^{\set,x}_t)]
	$.
	Note that for all 
	$t\in[0,T]$,
	$x\in H$ 
	it holds that 
	$
	u(t,x)=\phi(T-t,x)
	$.
	This and item~\eqref{item:kolmogorov.c.delta} of Lemma~\ref{lem:Kolmogorov}
	(with
	$\varphi=\varphi$,
	$F=(H\ni v \mapsto F(v) \in H )$,
	$B=(H\ni v \mapsto (U \ni u \mapsto B(v)u \in H) \in HS(U,H) )$,
	$X^x=X^{\set,x}$, 
	$\phi=\phi$,
	$k=k$,
	$\delta_1=\delta_1,\ldots,\delta_k=\delta_k$
	for 
	$ 
	(\delta_1, \dots, \delta_k) \in 
	\{(x_1,\ldots,x_k)\in(-\nicefrac{1}{2},0]^k\colon
	\sum^k_{i=1} x_i > -\nicefrac{1}{2}
	\} 
	$, 
	$ k \in \{ 1, 2, 3, 4 \} $,
	$x\in H$
	in the notation of Lemma~\ref{lem:Kolmogorov})
	imply that for all 
	$ k \in \{ 1, 2, 3, 4 \} $, 
	$ \delta_1, \dots, \delta_k \in (-\nicefrac{1}{2},0] $
	with
	$
	\sum^k_{i=1} \delta_i
	> -\nicefrac{1}{2}
	$ 
	it holds that 
	\begin{equation}
	\begin{split}
	c_{\delta_1,\ldots,\delta_k}
	&=
	\sup_{
		t \in [0,T)
	}
	\sup_{ 
		x \in H
	}
	\sup_{ 
		v_1, \dots, v_k \in H \backslash \{ 0 \}
	}
	\left[
	\frac{
		\big|
		( 
		\frac{ 
			\partial^k
		}{
			\partial x^k
		}
		\phi
		)( T-t, x )( v_1, \dots, v_k )
		\big|}{
		( T - t )^{ 
			(
			\delta_1 + \ldots + \delta_k
			) 
		}
		\left\| v_1 \right\|_{ H_{ \delta_1 } }
		\cdot
		\ldots
		\cdot
		\left\| v_k \right\|_{ H_{ \delta_k } }}\right]
	\\&=
	\sup_{
		t \in (0,T]
	}
	\sup_{ 
		x \in H
	}
	\sup_{ 
		v_1, \dots, v_k \in H \backslash \{ 0 \}
	}
	\left[
	\frac{
		\big|
		( 
		\frac{ 
			\partial^k
		}{
			\partial x^k
		}
		\phi
		)( t, x )( v_1, \dots, v_k )
		\big|}{
		t^{ 
			(
			\delta_1 + \ldots + \delta_k
			) 
		}
		\left\| v_1 \right\|_{ H_{ \delta_1 } }
		\cdot
		\ldots
		\cdot
		\left\| v_k \right\|_{ H_{ \delta_k } }}\right]
	< \infty.
	\end{split} 
	\end{equation}
	The proof of Lemma~\ref{lem:c.delta.finite} is thus completed.
\end{proof}

\begin{lemma}\label{weak_regle0}
Assume the setting in Section~\ref{sec:setting} and let 
$
  \rho \in [ 0 , 1 - \vartheta ) 
$, 
$
  I \in \grid(\set) 
$. 
Then
\begin{equation}
  \left|
    \ES\big[ 
      \varphi( Y^I_T )
    \big]
    -
    \ES\big[ 
      \varphi( X_T^I )
    \big]
  \right|
\leq
  \left[ 
    \frac{ 1 }{ T^{ \rho } } 
    + 
    \frac{ 
      2 \, T^{ ( 1 - \rho - \vartheta ) } 
    }{ 
      ( 1 - \rho - \vartheta ) 
    }
  \right]
  \| \varphi \|_{ C_b^3( H, \R ) }
  \,
  \varsigma_{ F, B }
  \,
  K^I_3
  \,
  \| P_{ \set \backslash I } \|_{
    L( H, H_{ - \rho } )
  }
  .
\end{equation}
\end{lemma}

\begin{proof}
First of all, note that for all $ t \in [0,T] $ 
it holds $ \P $-a.s.\ that
\begin{equation}
  P_I( Y^I_t ) 
=
  e^{ A t } P_I( \xi ) 
  +
  \int^t_0 
  e^{ A ( t - s ) } 
  P_I 
  F\big( P_I( Y^I_s ) \big)
  \, ds 
  + 
  \int^t_0
  e^{ A (t - s) } 
  P_I
  B\big( P_I( Y^I_s ) \big)
  \, dW_s 
  .
\end{equation}
The fact that mild solutions of~\eqref{eq:Galerkin}
are within a suitable class of solutions unique up to modifications 
(see, e.g., item~(i) of Theorem~7.2 in Da Prato \& Zabczyk~\cite{dz14}
for details)
hence ensures that for all 
$ t \in [0,T] $
it holds $ \P $-a.s.\ that
$ P_I( Y^I_t ) = X^I_t $.
An application of Proposition~\ref{prop:weak_regle0}
hence proves that
\begin{equation}
\begin{split}
&
  \left|
    \ES\big[ 
      \varphi( Y^I_T )
    \big]
    -
    \ES\big[ 
      \varphi\big( 
        X^I_T 
      \big)
    \big]
  \right|
\leq
  \| \varphi \|_{ C_b^3(H,\R) }
  \,
    \max\!\left\{ 
      1 ,
      \sup\nolimits_{ t \in [0,T] }
      \ES\big[
        \| Y^I_t \|^3_H
      \big]
    \right\}
\\ & 
  \cdot
  \left[ 
    \frac{ 1 }{ T^{ \rho } }
    +
    \frac{
      T^{ ( 1 - \rho - \vartheta ) }
      \big[ 
        \| F \|_{
          C^1_b( H, H_{ - \vartheta } )
        }
        +
        \| B \|^{2}_{
          C^1_b( H,
            HS( U, H_{ - \vartheta / 2 } ) 
          )
        }
      \big]
    }{
      \left( 1 - \rho - \vartheta \right)
    }
  \right]
  \|P_{\set\backslash I}\|_{L(H,H_{-\rho})}
  .
\end{split}
\end{equation}
This completes the proof
of Lemma~\ref{weak_regle0}.
\end{proof}

\begin{corollary}
	\label{lem:integrated.version}
	Assume the setting in Section~\ref{sec:setting} and let $I\in\grid(\set)$.
	Then
	\begin{enumerate}[(i)]
		\item
		\label{item:H2.integrated.integrability}
		it holds that
		\begin{equation}
		\int^T_0
		\ES\big[
		\|AY^I_t\|_{H_1} + \|F_I(Y^I_t)\|_{H_1}
		+
		\|B_I(Y^I_t)\|^2_{HS(U,H_1)}\big]
		\, dt
		< \infty
		\end{equation}
		and
		\item
		\label{item:H2.strong.solution}
		for all $t\in[0,T]$ it holds $\P$-a.s.\ that
		\begin{equation}
		Y^I_t=\xi
		+
		\int^t_0
		AY^I_s + F_I(Y^I_s)
		\, ds
		+
		\int^t_0
		B_I(Y^I_s)
		\, dW_s
		.
		\end{equation}
	\end{enumerate}
\end{corollary}
\begin{proof}
	Observe that for all $p\in[1,4]$ it holds that 
	\begin{equation}
	\begin{split}
	&
	\sup_{t\in[0,T]}
	\big[
	\|F_I(Y^I_t)\|_{\lpn{p}{\P}{H_2}}
	+
	\|B_I(Y^I_t)\|_{\lpn{p}{\P}{HS(U,H_2)}}
	\big]
	\\&\leq
	  \sup_{t\in[0,T]}
	  \big[
	    \|F(0)\|_{H_2}
	    +
	    \|F(P_I(Y^I_t))-F(0)\|_{\lpn{p}{\P}{H_2}}
	  \big]
	  \\&\quad+
	  \sup_{t\in[0,T]}
	  \big[
	    \|B(0)\|_{HS(U,H_2)}
	    +
	    \|B(P_I(Y^I_t))-B(0)\|_{\lpn{p}{\P}{HS(U,H_2)}}
	  \big]
	\\&\leq
	  \sup_{t\in[0,T]}
	  \big[
	    \|F(0)\|_{H_2}
	    +
	    |F|_{\operatorname{Lip}^0(H,H_2)} \,
	    \|P_I(Y^I_t)\|_{\lpn{p}{\P}{H}}
	  \big]
	  \\&\quad+
	  \sup_{t\in[0,T]}
	  \big[
	    \|B(0)\|_{HS(U,H_2)}
	    +
	    |B|_{\operatorname{Lip}^0(H,HS(U,H_2))} \,
	    \|P_I(Y^I_t)\|_{\lpn{p}{\P}{H}}
	  \big]
	\\&\leq
	  \sup_{t\in[0,T]}
	  \big[
	    \|F(0)\|_{H_2}
	    +
	    |F|_{\operatorname{Lip}^0(H,H_2)}
	\max\!\big\{
	1,
	\|P_I(Y^I_t)\|_{\lpn{p}{\P}{H}}
	\big\}
	  \big]
	  \\&\quad+
	  \sup_{t\in[0,T]}
	  \big[
	    \|B(0)\|_{HS(U,H_2)}
	    +
	    |B|_{\operatorname{Lip}^0(H,HS(U,H_2))}
	\max\!\big\{
	1,
	\|P_I(Y^I_t)\|_{\lpn{p}{\P}{H}}
	\big\}
	  \big]
	.
	\end{split}
	\end{equation}
	Hence, we obtain that for all $p\in[1,4]$ it holds that
	\begin{equation}
	\label{eq:H2.integrability}
	\begin{split}
	&
	\sup_{t\in[0,T]}
	\big[
	\|F_I(Y^I_t)\|_{\lpn{p}{\P}{H_2}}
	+
	\|B_I(Y^I_t)\|_{\lpn{p}{\P}{HS(U,H_2)}}
	\big]
	\\&\leq
	\sup_{t\in[0,T]}
	\big[
	\|F(0)\|_{H_2}
	+
	|F|_{\operatorname{Lip}^0(H,H_2)}
	\max\!\big\{
	1,
	\|Y^I_t\|_{\lpn{p}{\P}{H}}
	\big\}
	\big]
	\\&\quad+
	\sup_{t\in[0,T]}
	\big[
	\|B(0)\|_{HS(U,H_2)}
	+
	|B|_{\operatorname{Lip}^0(H,HS(U,H_2))}
	\max\!\big\{
	1,
	\|Y^I_t\|_{\lpn{p}{\P}{H}}
	\big\}
	\big]
	\\&\leq
	\sup_{t\in[0,T]}
	\Big(
	\big[
	\|F(0)\|_{H_2}
	+
	|F|_{\operatorname{Lip}^0(H,H_2)}
	\big]
	\max\!\big\{
	1,
	\|Y^I_t\|_{\lpn{p}{\P}{H}}
	\big\}
	\Big)
	\\&\quad+
	\sup_{t\in[0,T]}
	\Big(
	\big[
	\|B(0)\|_{HS(U,H_2)}
	+
	|B|_{\operatorname{Lip}^0(H,HS(U,H_2))}
	\big]
	\max\!\big\{
	1,
	\|Y^I_t\|_{\lpn{p}{\P}{H}}
	\big\}
	\Big)
	\\&=
	\big(
	\|F\|_{\operatorname{Lip}^0(H,H_2)}
	+
	\|B\|_{\operatorname{Lip}^0(H,HS(U,H_2))}
	\big)
	\sup_{t\in[0,T]}
	\max\!\big\{
	1,
	\|Y^I_t\|_{\lpn{p}{\P}{H}}
	\big\}
	< \infty.
	\end{split}
	\end{equation}
	This ensures that for all $p\in[1,4]$ it holds that 
	\begin{equation}
	\label{eq:int.H2}
	\int^T_0
	\ES\big[
	\|F_I(Y^I_t)\|^p_{H_2}
	+
	\|B_I(Y^I_t)\|^p_{HS(U,H_2)}\big]
	\, dt
	< \infty
	.
	\end{equation}
	In addition, note that H\"{o}lder's inequality and the assumption that 
	$
	  \sup_{t\in[0,T]}
	  \ES\big[\|Y^I_t\|^4_{H_2}\big]
	  < \infty
	$
	imply that 
	\begin{equation}
	  \sup_{t\in[0,T]}
	  \ES\big[\|Y^I_t\|_{H_2}\big]
	  < \infty.
	\end{equation}
	Therefore, we obtain that 
	\begin{equation}
	\begin{split}
	\int^T_0
	\ES\big[\|Y^I_t\|_{H_2}\big]
	\, dt
	< \infty
	.
	\end{split}
	\end{equation}
	This and~\eqref{eq:int.H2} prove item~\eqref{item:H2.integrated.integrability}.
	In the next step we combine~\eqref{eq:int.H2} with~\eqref{eq:mod_SPDE} and items~\eqref{item:integral.makes.sense}--\eqref{item:general.H1.solution} of Lemma~\ref{lem:gen.H1.solution} to obtain
	that for all 
	$t\in[0,T]$ it holds $\P$-a.s.\ that 
	\begin{equation}
	\label{eq:integrated.analytic.strong.solution}
	Y^I_t
	=
	\xi
	+
	\int^t_0
	AY^I_s + F_I(Y^I_s) 
	\, ds
	+
	\int^t_0
	B_I(Y^I_s) 
	\, dW_s
	.
	\end{equation}
	This establishes item~\eqref{item:H2.strong.solution}. 
	The proof of Corollary~\ref{lem:integrated.version} is thus completed.
\end{proof}

\begin{lemma}
\label{weak_regle}
Assume the setting in Section~\ref{sec:setting} and let 
$ \rho \in [ 0, 1 - \vartheta ) $, $ I \in \grid( \set ) $.
Then
\begin{multline}
\label{eq:weak_regle}
  \left|
    \ES\big[ 
      \varphi( X^\set_T )
    \big]
    -
    \ES\big[ 
      \varphi( Y_T^I )
    \big]
  \right|
\leq
  \frac{ 
    T^{ (1 - \vartheta - \rho) } 
    \,
    \varsigma_{ F, B }
    \, 
    K^I_4 
  }{
    \left( 1 - \vartheta - \rho \right) 
  }
  \left[ 
    1
    + 
    \frac{ 
      9 \, T^{ (1 - \vartheta) } 
    }{
      2 {\left( 1 - \vartheta \right)}
    } 
  \right]
  \left\|
      P_{ \set \backslash I }
      \right\|_{ L(H,H_{-\rho}) }
\\ 
\cdot
  \big[ c_{-\vartheta} + c_{-\vartheta,0} + c_{-\vartheta,0,0} + c_{-\nicefrac{\vartheta}{2},-\nicefrac{\vartheta}{2}} + c_{-\nicefrac{\vartheta}{2},-\nicefrac{\vartheta}{2},0} + c_{-\nicefrac{\vartheta}{2},-\nicefrac{\vartheta}{2},0,0} \big]
  < \infty
  .
\end{multline}
\end{lemma}

\begin{proof}
Throughout this proof let
$ u_{1,0} \colon [0,T] \times H_1 \to \R $
and
$
  u_{ 0,k } \colon [0,T] \times H \to L^{ (k) }( H, \R )
$,
$k\in\{1,2,3,4\}$, 
be the functions which satisfy 
for all 
$k\in\{1,2,3,4\}$, 
$ t \in [0,T] $, 
$w\in H_1$, 
$ x, v_1, \ldots, v_k \in H $ 
that 
\begin{equation}
  u_{1,0}( t, w )
  =
  (\tfrac{ \partial }{ \partial t } u)( t, w )
  \qquad\text{and}\qquad
  u_{ 0,k }( t, x )( v_1, \ldots, v_k )
  =
  \big( (\tfrac{ \partial^k }{ \partial x^k } u)( t, x ) \big)( v_1, \ldots, v_k )
\end{equation}
(cf., e.g., item~\eqref{item:time.derivative.exist} of Lemma~\ref{lem:time.derivative} and item~\eqref{item:kolmogorov.diff} of Lemma~\ref{lem:Kolmogorov}).
Note that items~\eqref{item:SEE.H1.regular} \& \eqref{item:H1.space.time.continuous} of Lemma~\ref{lem:time.derivative} and the fact that $|\varphi|_{\operatorname{Lip}^0(H_1,\R)}<\infty$ establish that 
\begin{equation}
\label{eq:TPS.continuous}
([0,T]\times H_1 \ni (t,x)\mapsto u(t,x)\in\R)
\in C([0,T]\times H_1,\R)
.
\end{equation}
Moreover, item~\eqref{item:time.derivative.continuity} of Lemma~\ref{lem:time.derivative} proves that
\begin{equation}
\label{eq:TPS.temporal.regularity}
  u_{1,0}\in C([0,T]\times H_1,\R)
  .
\end{equation}
It is well known that $u$ is a strong solution of the Kolmogorov equation associated to
$
  X^{\set,x}\colon[0,T]\times\Omega\rightarrow H
$, 
$ x \in H $.
More precisely, note that for all $t\in(0,T)$, $x\in H_1$ it holds that 
\begin{equation}
\label{eq:kolmogorov.solution}
  u_{1,0}(t,x)=
  -u_{0,1}(t,x)(Ax+F(x))
  -\frac{1}{2}
  \sum_{b\in\mathbb{U}}
  u_{0,2}(t,x)(B^b(x),B^b(x))
\end{equation}
(cf., e.g., Da Prato \& Zabczyk~\cite[Theorem~9.25]{dz14}).
In addition, observe that for all $k \in \{1,2\}$ it holds that the function 
\begin{multline}
\label{eq:TPS.regularity}
  \big(
  [0,T] \times H_1 \ni (t,x) \mapsto
  \big(
    (H_1)^k \ni (h_1,\ldots,h_k) \mapsto u_{0,k}(t,x)(h_1,\ldots,h_k) \in \R
  \big)
\\
  \in L^{(k)}(H_1,\R)
  \big)
\end{multline}
is continuous (cf., e.g., Da Prato \& Zabczyk~\cite[Theorem~9.25]{dz14} and Andersson et al.~\cite[Theorem~3.3]{AnderssonHefterJentzenKurniawan2016}).
Combining~\eqref{eq:TPS.continuous}, \eqref{eq:TPS.temporal.regularity}, and Corollary~\ref{lem:integrated.version} with the standard It\^{o} formula in infinite dimensions (cf., e.g., Brze{\'z}niak et al.~\cite[Theorem~2.4]{bvvw08})
hence proves that it holds $\P$-a.s.\ that 
\begin{equation}
\begin{split}
&
      u( T, Y^I_{T} )
    -
      u( 0, Y^I_0 )
\\&=
  \int_0^{T}
    u_{1,0}( t, Y^I_t )
    +
    u_{0,1}( t, Y^I_t )
    \!\left(
      A
      Y^I_t
      +
      F_I(
        Y^I_t 
      )
    \right)
  dt
+
  \int^T_0
  u_{0,1}(t,Y^I_t) B_I(Y^I_t)
  \, dW_t
\\&\quad+
  \frac{ 1 }{ 2 }
  \int_0^{T}
  \sum_{ b\in\mathbb{U} }
    u_{0,2}( t, Y^I_t )\!\left(
      B^b_I(
        Y^I_t 
      )
      ,
      B^b_I(
        Y^I_t 
      )
    \right)
  dt
  .
\end{split}
\end{equation}
This and~\eqref{eq:kolmogorov.solution} ensure that it holds $\P$-a.s.\ that 
\begin{equation}
\label{eq:ito}
\begin{split}
&
      u( T, Y^I_{T} )
    -
      u( 0, Y^I_0 )
\\ & =
  \int_0^{T}
    u_{0,1}( t, Y^I_t )
    \!
    \left(
      F_I\big(
        Y^I_t 
      \big)
      -
      F\big(
        Y^I_t 
      \big)
    \right)
  dt
+
  \int^T_0
  u_{0,1}(t,Y^I_t) B_I(Y^I_t)
  \, dW_t
\\&\quad+
  \frac{ 1 }{ 2 }
  \int_0^{T}
  \sum_{ b\in\mathbb{U} }
    u_{0,2}( t, Y^I_t )\!\left(
      B_I^b\big(
        Y^I_t
      \big)
      +
      B^b\big(
        Y^I_t 
      \big)
      ,
      B^b_I\big(
        Y^I_t 
      \big)
      -
      B^b\big(
        Y^I_t 
      \big)
    \right)
  dt
  .
\end{split}
\end{equation}
Next note that Lemma~\ref{lem:c.delta.finite} shows that
\begin{equation}
\label{eq:c.0.finite}
  c_0 + c_{0,0} + c_{0,0,0}
  < \infty
  .
\end{equation}
This guarantees that 
\begin{equation}
\begin{split}
  \sup_{t\in[0,T)}
  \sup_{x\in H}
  \frac{\|u_{0,1}(t,x)B_I(x)\|_{HS(U,\R)}}{g_1(x)}
  &\leq
  \sup_{t\in[0,T)}
  \sup_{x\in H}
  \bigg[
  \frac{\|u_{0,1}(t,x)\|_{L(H,\R)} \, \|B_I(x)\|_{HS(U,H)}}{g_1(x)}
  \bigg]
  \\&\leq
  c_0 \,
  \bigg[
  \sup_{t\in[0,T)}
  \sup_{x\in H}
  \frac{\|B_I(x)\|_{HS(U,H)}}{g_1(x)}
  \bigg]
  < \infty.
\end{split}
\end{equation}
This and the fact that $K^I_2<\infty$ assure that 
\begin{equation}
\label{eq:1st.stochastic.makes.sense}
  \ES\bigg[
  \int^T_0
    \|u_{0,1}(t,Y^I_t) B_I(Y^I_t)\|^2_{HS(U,\R)}
  \, dt
  \bigg]
  < \infty.
\end{equation}
Furthermore, note that Markov's property associated to 
$
  X^{\set,x}\colon[0,T]\times\Omega\rightarrow H
$, 
$ x \in H $, 
implies that 
\begin{equation}
\label{eq:markov}
    \ES\big[ 
      \varphi( Y_T^I )
    -
      \varphi( X^\set_T )
    \big]
    =
   \ES\big[ 
      u( T, Y^I_T )
    -
      u( 0, Y^I_0 )
    \big]
\end{equation}
(cf., e.g., Theorem~9.14 in Da Prato \& Zabczyk~\cite{dz14}).
Combining~\eqref{eq:ito} with~\eqref{eq:1st.stochastic.makes.sense} therefore shows that 
$
  \ES\big[\!
  \int^T_0
  u_{0,1}(t,Y^I_t) B_I(Y^I_t)
  \, dW_t
  \big]
  =0
$
and 
\begin{equation}
\label{eq:Kolmogorov}
\begin{split}
&
    \ES\big[ 
      \varphi( Y_T^I )
    -
      \varphi( X^\set_T )
    \big]
    =
    \ES\big[ 
      u( T, Y^I_{T} )
    -
      u( 0, Y^I_0 )
    \big]
\\&=
  \int_0^{T}
  \E\left[
    u_{0,1}( t, Y^I_t )
    \!
    \left(
      F_I\big(
        Y^I_t 
      \big)
      -
      F\big(
        Y^I_t 
      \big)
    \right)
  \right]
  dt
\\ & \quad +
  \frac{ 1 }{ 2 }
  \sum_{ b\in\mathbb{U} }
  \int_0^{T}
  \E\left[
    u_{0,2}( t, Y^I_t )\!\left(
      B_I^b\big(
        Y^I_t
      \big)
      +
      B^b\big(
        Y^I_t 
      \big)
      ,
      B^b_I\big(
        Y^I_t 
      \big)
      -
      B^b\big(
        Y^I_t 
      \big)
    \right)
  \right]
  dt
  .
\end{split}
\end{equation}
The identities in~\eqref{eq:Kolmogorov} provide us an explicit representation for the weak approximation error 
$
    \ES\big[ 
      \varphi( Y^I_T )
    \big]
    -
    \ES\big[ 
      \varphi( X_T^\set )
    \big]
$.
In the following we employ the mild It\^{o} formula 
in Corollary~2 in 
Da Prato et al.~\cite{DaPratoJentzenRoeckner2012}
to estimate the two summands appearing on the right hand side of~\eqref{eq:Kolmogorov}.
We first concentrate on the first summand on the right hand side of~\eqref{eq:Kolmogorov}.
For the application of the mild It\^{o} formula
in~\cite[Corollary~2]{DaPratoJentzenRoeckner2012}
to the first summand on the right hand side of~\eqref{eq:Kolmogorov} (see~\eqref{eq:ito.F} below) it is convenient to introduce suitable auxiliary functions.
More formally, let
$ \tilde{F}_{ t, s } \colon H \to \R $,
$ t \in (s,T] $, $ s \in [0,T) $,
be the functions which satisfy for all 
$ s \in [0,T) $,
$ t \in (s,T] $,
$ x \in H $
that 
\begin{equation}
\label{eq:tildeF_def}
\begin{split}
&
  \tilde{F}_{ t, s }( x ) 
\\ & = 
    u_{0,2}( t, e^{ A ( t - s ) } x )
    \!\left(
      F_I(
        e^{ A ( t - s ) }x
      )
      -
      F(
        e^{ A ( t - s ) } x
      )
      ,
      e^{ A (t - s) }
      F_I( x )
    \right)
\\ & +
    u_{0,1}( t, e^{ A ( t - s ) } x )
    \!\left(
      \left[
        F_I'(
          e^{ A ( t - s ) }x
        )
        -
        F'(
          e^{ A ( t - s ) } x
        )
      \right]
      e^{ A (t - s) }
      F_I( x )
    \right)
\\ & +
    \tfrac{ 1 }{ 2 }
    \smallsum\limits_{ b\in\mathbb{U} }
    u_{0,3}( t, e^{ A ( t - s ) } x )
    \!\left(
      F_I(
        e^{ A ( t - s ) }x
      )
      -
      F(
        e^{ A ( t - s ) } x
      )
      ,
      e^{ A (t - s) }
      B^b_I( x ) 
      ,
      e^{ A (t - s) }
      B^b_I( x )
    \right)
\\ & +
    \smallsum\limits_{ b\in\mathbb{U} }
    u_{0,2}( t, e^{ A ( t - s ) } x )
    \!\left(
      \left[
        F_I'(
          e^{ A ( t - s ) }x
        )
        -
        F'(
          e^{ A ( t - s ) } x
        )
      \right]
      e^{ A (t - s) }
      B^b_I( x )
      ,
      e^{ A (t - s) }
      B^b_I( x ) 
    \right)
\\ & +\tfrac{ 1 }{ 2 }
    \smallsum\limits_{ b\in\mathbb{U} }
    u_{0,1}( t, e^{ A ( t - s ) } x )
    \!\left(
      \left[
        F_I''(
          e^{ A ( t - s ) }x
        )
        -
        F''(
          e^{ A ( t - s ) } x
        )
      \right]
      \!
      \left(
        e^{ A (t - s) }
        B^b_I( x ) 
        ,
        e^{ A (t - s) }
        B^b_I( x )
      \right)
    \right)
\end{split}
\end{equation}
and let 
$ \hat{F}_{ t, s } \colon H \to HS(U,\R) $,
$ t \in (s,T] $, $ s \in [0,T) $,
be the functions which satisfy for all 
$ s \in [0,T) $,
$ t \in (s,T] $,
$ x \in H $
that 
\begin{equation}
\label{eq:hatF_def}
\begin{split}
&
  \hat{F}_{ t, s }( x ) 
\\ & = 
    u_{0,2}( t, e^{ A ( t - s ) } x )
    \!\left(
      F_I(
        e^{ A ( t - s ) }x
      )
      -
      F(
        e^{ A ( t - s ) } x
      )
      ,
      e^{ A (t - s) }
      B_I( x )
    \right)
\\ & +
    u_{0,1}( t, e^{ A ( t - s ) } x )
    \!\left(
      \left[
        F_I'(
          e^{ A ( t - s ) }x
        )
        -
        F'(
          e^{ A ( t - s ) } x
        )
      \right]
      e^{ A (t - s) }
      B_I( x )
    \right)
    .
\end{split}
\end{equation}
An application of 
the mild It\^{o} formula 
in Corollary~2 in 
Da Prato et al.~\cite{DaPratoJentzenRoeckner2012} 
then proves that
for all $ t \in (0,T] $
it holds $\P$-a.s.\ that
\begin{equation}
\label{eq:ito.F}
\begin{split}
&
    u_{0,1}( t, Y^I_t )
    \!\left(
      F_I\big(
        Y^I_t 
      \big)
      -
      F\big(
        Y^I_t 
      \big)
    \right)
\\ & =
    u_{0,1}( t, e^{ A t } \xi )
    \!\left(
      F_I\big(
        e^{ A t } \xi 
      \big)
      -
      F\big(
        e^{ A t } \xi
      \big)
    \right)
  +
  \int_0^t
    \tilde{F}_{ t, s }( 
      Y^I_s
    )
  \, ds
  +
  \int_0^t
    \hat{F}_{ t, s }( 
      Y^I_s
    )
  \, dW_s
  .
\end{split}
\end{equation}
Moreover, observe that~\eqref{eq:c.0.finite}, \eqref{eq:hatF_def}, and Lemma~\ref{lem:aux} ensure that 
\begin{equation}
\label{eq:hatF.growth}
  \sup_{s\in[0,T)}
  \sup_{t\in(s,T)}
  \sup_{x\in H}
  \frac{
    \|\hat{F}_{t,s}(x)\|_{HS(U,\R)}
  }{g_2(x)}
  < \infty.
\end{equation}
This and the fact that $K^4_I<\infty$ assure that for all $t\in(0,T)$ it holds that 
\begin{equation}
  \ES\bigg[
    \int^t_0
    \|
    \hat{F}_{ t, s }( 
      Y^I_s
    )
    \|^2_{HS(U,\R)}
  \, ds
  \bigg]
  < \infty
  .
\end{equation}
This and~\eqref{eq:ito.F} yield that for all $t\in(0,T)$ it holds that 
$
  \ES\big[|\!
  \int^t_0
    \hat{F}_{ t, s }( 
      Y^I_s
    )
  \, dW_s
  |
  \big]
  < \infty
$, 
$
  \ES\big[\!
  \int^t_0
    \hat{F}_{ t, s }( 
      Y^I_s
    )
  \, dW_s
  \big]
  =0
$,
and 
\begin{equation}
\label{eq:mild_Ito_F}
\begin{split}
&
  \E\left[
    u_{0,1}( t, Y^I_t )
    \!\left(
      F_I\big(
        Y^I_t 
      \big)
      -
      F\big(
        Y^I_t 
      \big)
    \right)
  \right]
\\ & =
  \E\left[
    u_{0,1}( t, e^{ A t } \xi )
    \!\left(
      F_I\big(
        e^{ A t } \xi 
      \big)
      -
      F\big(
        e^{ A t } \xi
      \big)
    \right)
  \right]
  +
  \int_0^t
  \E\left[ 
    \tilde{F}_{ t, s }( 
      Y^I_s
    )
  \right]
  ds
  .
\end{split}
\end{equation}
Next we apply the mild It\^{o} formula
in~\cite[Corollary~2]{DaPratoJentzenRoeckner2012} to the second summand on the right hand side of~\eqref{eq:Kolmogorov}.
For this application it is again convenient to introduce suitable functions.
More formally, let
$ \tilde{B}_{ t, s } \colon H \to \R $,
$ t \in (s,T] $, $ s \in [0,T) $,
be the functions which satisfy for all 
$ s \in [0,T) $, $ t \in (s,T] $, $ x \in H $
that
{\allowdisplaybreaks  
\begin{align}
\label{eq:tildeB_def}
&
  \tilde{B}_{ t, s }( x ) \nonumber
\\ & 
  =
  \smallsum\limits_{
    b\in\mathbb{U}
  }
  u_{0,2}( t, e^{ A ( t - s ) } x )
  \Big(
    \left[
      ( B^b_I )'( e^{ A ( t - s ) } x ) 
      + 
      ( B^b )'( e^{ A ( t - s ) } x )
    \right]
    e^{ A ( t - s ) } F_I( x ) 
    ,\nonumber
\\ & 
\qquad\qquad\qquad\qquad\qquad
    B^b_I( e^{ A ( t - s ) } x ) - B^b( e^{ A ( t - s ) } x )
  \Big)\nonumber
\\ & +
  \smallsum\limits_{ b\in\mathbb{U} }
  u_{0,2}( t, e^{ A ( t - s ) } x )
  \Big(
      B^b_I( e^{ A ( t - s ) } x ) 
      + 
      B^b( e^{ A ( t - s ) } x )
    ,\nonumber
\\ & 
\qquad\qquad\qquad\qquad\qquad
    \left[
      ( B^b_I )'( e^{ A ( t - s ) } x ) - ( B^b )'( e^{ A ( t - s ) } x )
    \right]
    e^{ A ( t - s ) } F_I( x ) 
  \Big)\nonumber
\\ & +
  \smallsum\limits_{ b\in\mathbb{U} }
  u_{0,3}( t, e^{ A ( t - s ) } x )
  \Big(
      B^b_I( e^{ A ( t - s ) } x ) 
      + 
      B^b( e^{ A ( t - s ) } x )
    ,\nonumber
\\ & 
\qquad\qquad\qquad\qquad\qquad
      B^b_I( e^{ A ( t - s ) } x ) - B^b( e^{ A ( t - s ) } x )
    ,
      e^{ A ( t - s ) } F_I( x ) 
  \Big)\nonumber
\\ & +
  \smallsum\limits_{ b_1, b_2 \in \mathbb{U} }
  u_{0,2}( t, e^{ A ( t - s ) } x )
  \Big(
    \left[
      ( B_I^{b_2} )'( e^{ A ( t - s ) } x ) 
      + 
      ( B^{b_2} )'( e^{ A ( t - s ) } x )
    \right]
    e^{ A ( t - s ) } B^{b_1}_I( x )
    ,\nonumber
\\ & 
\qquad\qquad
\qquad\qquad
\quad\qquad
    \left[
      ( B^{b_2}_I )'( e^{ A ( t - s ) } x ) - ( B^{b_2} )'( e^{ A ( t - s ) } x )
    \right]
    e^{ A ( t - s ) } B^{b_1}_I( x ) 
  \Big)\nonumber
\\ & +
  \frac{ 1 }{ 2 }
  \smallsum\limits_{ b_1, b_2 \in \mathbb{U} }
  u_{0,2}( t, e^{ A ( t - s ) } x )
  \Big(
    B^{b_2}_I( e^{ A ( t - s ) } x ) - B^{b_2}( e^{ A ( t - s ) } x )
    ,
\\ & 
\qquad\qquad
\quad\qquad
    \left[
      ( B_I^{b_2} )''( e^{ A ( t - s ) } x ) 
      + 
      ( B^{b_2} )''( e^{ A ( t - s ) } x )
    \right]
    \!
    \left(
      e^{ A ( t - s ) } B^{b_1}_I( x )
      ,
      e^{ A ( t - s ) } B^{b_1}_I( x )
    \right)
  \Big)\nonumber
\\ & +
  \frac{ 1 }{ 2 }
  \smallsum\limits_{ b_1,b_2\in\mathbb{U} }
  u_{0,2}( t, e^{ A ( t - s ) } x )
  \Big(
      B_I^{b_2}( e^{ A ( t - s ) } x ) 
      + 
      B^{b_2}( e^{ A ( t - s ) } x )
    ,\nonumber
\\ & 
\qquad\qquad
\quad\qquad
    \left[
      ( B^{b_2}_I )''( e^{ A ( t - s ) } x ) - ( B^{b_2} )''( e^{ A ( t - s ) } x )
    \right]
    \!
    \left(
      e^{ A ( t - s ) } B^{b_1}_I( x ) 
      ,
      e^{ A ( t - s ) } B^{b_1}_I( x ) 
    \right)
  \Big)\nonumber
\\ & +
  \smallsum\limits_{ b_1,b_2\in\mathbb{U} }
  u_{0,3}( t, e^{ A ( t - s ) } x )
  \Big(
    \left[
      ( B^{b_2}_I )'( e^{ A ( t - s ) } x ) 
      + 
      ( B^{b_2} )'( e^{ A ( t - s ) } x )
    \right]
      e^{ A ( t - s ) } B^{b_1}_I( x ) 
    ,\nonumber
\\ & 
\qquad\qquad
\qquad\qquad
\quad\qquad
      B_I^{b_2}( e^{ A ( t - s ) } x ) - B^{b_2}( e^{ A ( t - s ) } x )
    ,
      e^{ A ( t - s ) } B^{b_1}_I( x )
  \Big)\nonumber
\\ & +
  \smallsum\limits_{ b_1,b_2\in\mathbb{U} }
  u_{0,3}( t, e^{ A ( t - s ) } x )
  \Big(
      B_I^{b_2}( e^{ A ( t - s ) } x ) 
      + 
      B^{b_2}( e^{ A ( t - s ) } x )
    ,\nonumber
\\ & 
\qquad\qquad
\qquad\qquad
\qquad
    \left[
      ( B_I^{b_2} )'( e^{ A ( t - s ) } x ) - ( B^{b_2} )'( e^{ A ( t - s ) } x )
    \right]
      e^{ A ( t - s ) } B^{b_1}_I( x ) 
    ,
      e^{ A ( t - s ) } B^{b_1}_I( x ) 
  \Big)\nonumber
\\ & +
  \frac{ 1 }{ 2 }
  \smallsum\limits_{ b_1,b_2\in\mathbb{U} }
  u_{0,4}( t, e^{ A ( t - s ) } x )
  \Big(
      B_I^{b_2}( e^{ A ( t - s ) } x ) 
      + 
      B^{b_2}( e^{ A ( t - s ) } x )
    ,
      B^{b_2}_I( e^{ A ( t - s ) } x ) - B^{b_2}( e^{ A ( t - s ) } x )
    ,\nonumber
\\ & 
\qquad\qquad
\qquad\qquad
\quad\qquad
      e^{ A ( t - s ) } B^{b_1}_I( x ) 
    ,
      e^{ A ( t - s ) } B^{b_1}_I( x ) 
  \Big)\nonumber
\end{align}}and let
$ \hat{B}_{ t, s } \colon H \to HS(U,\R) $,
$ t \in (s,T] $, $ s \in [0,T) $,
be the functions which satisfy for all 
$ s \in [0,T) $, $ t \in (s,T] $, $ x \in H $
that
\begin{equation}
\label{eq:hatB_def}
\begin{split}
&
  \hat{B}_{ t, s }( x ) 
\\ & 
  =
  \smallsum\limits_{
    b\in\mathbb{U}
  }
  u_{0,2}( t, e^{ A ( t - s ) } x )
  \Big(
    \left[
      ( B^b_I )'( e^{ A ( t - s ) } x ) 
      + 
      ( B^b )'( e^{ A ( t - s ) } x )
    \right]
    e^{ A ( t - s ) } B_I( x ) 
    ,
\\ & 
\qquad\qquad\qquad\qquad\qquad
    B^b_I( e^{ A ( t - s ) } x ) - B^b( e^{ A ( t - s ) } x )
  \Big)
\\ & +
  \smallsum\limits_{ b\in\mathbb{U} }
  u_{0,2}( t, e^{ A ( t - s ) } x )
  \Big(
      B^b_I( e^{ A ( t - s ) } x ) 
      + 
      B^b( e^{ A ( t - s ) } x )
    ,
\\ & 
\qquad\qquad\qquad\qquad\qquad
    \left[
      ( B^b_I )'( e^{ A ( t - s ) } x ) - ( B^b )'( e^{ A ( t - s ) } x )
    \right]
    e^{ A ( t - s ) } B_I( x ) 
  \Big)
\\ & +
  \smallsum\limits_{ b\in\mathbb{U} }
  u_{0,3}( t, e^{ A ( t - s ) } x )
  \Big(
      B^b_I( e^{ A ( t - s ) } x ) 
      + 
      B^b( e^{ A ( t - s ) } x )
    ,
\\ & 
\qquad\qquad\qquad\qquad\qquad
      B^b_I( e^{ A ( t - s ) } x ) - B^b( e^{ A ( t - s ) } x )
    ,
      e^{ A ( t - s ) } B_I( x ) 
  \Big)
  .
\end{split}
\end{equation}
An application of the mild It\^{o} formula 
in Corollary~2 in 
Da Prato et al.~\cite{DaPratoJentzenRoeckner2012} 
then proves that
for all $ t \in (0,T] $ 
it holds $\P$-a.s.\ that
\begin{equation}
\label{eq:B.mild}
\begin{split}
&
  \sum_{ b\in\mathbb{U} }
    u_{0,2}( t, Y^I_t )\!\left(
      B^b_I\big(
        Y^I_t
      \big)
      +
      B^b\big(
        Y^I_t 
      \big)
      ,
      B^b_I\big(
        Y^I_t 
      \big)
      -
      B^b\big(
        Y^I_t 
      \big)
    \right)
\\ & =
  \sum_{ b\in\mathbb{U} }
    u_{0,2}( t, e^{ A t } \xi )\!\left(
      B_I^b\big(
        e^{ A t } \xi
      \big)
      +
      B^b\big(
        e^{ A t } \xi 
      \big)
      ,
      B_I^b\big(
        e^{ A t } \xi 
      \big)
      -
      B^b\big(
        e^{ A t } \xi 
      \big)
    \right)
\\&\quad+
  \int_0^t
    \tilde{B}_{ t, s }( 
      Y^I_s
    )
  \, ds
  +
  \int_0^t
    \hat{B}_{ t, s }( 
      Y^I_s
    )
  \, dW_s
  .
\end{split}
\end{equation}
Next observe that there exists a continuous $(\mathcal{F}_t)_{t\in[0,T]}$-adapted stochastic process 
$\mathcal{Y}\colon[0,T]\times\Omega\to H$ 
which satisfies for all $t\in[0,T]$ 
that 
$
  \P(\mathcal{Y}_t=Y^I_t)=1
$
and
\begin{equation}
\label{eq:running.inside}
  \ES\bigg[
  \sup_{s\in[0,T]}
  \|\mathcal{Y}_s\|^4_H
  \bigg]
  < \infty
\end{equation}
(cf., e.g., Da Prato \& Zabczyk~\cite[Theorem~7.2]{dz14}).
This and, e.g., Lemma~3.1 in Jentzen \& Pu\v{s}nik~\cite{JentzenPusnik2016} 
ensure that for all 
$t\in(0,T]$ it holds $\P$-a.s.\ that 
\begin{equation}
    \int^t_0
    \|    \hat{B}_{ t, s }( 
          Y^I_s
        )\|^2_{HS(U,\R)} \,
  ds
  =
    \int^t_0
    \|    \hat{B}_{ t, s }( 
          \mathcal{Y}_s
        )\|^2_{HS(U,\R)} \,
  ds
  .
\end{equation}
Moreover, note that~\eqref{eq:c.0.finite}, \eqref{eq:hatB_def}, and Lemma~\ref{lem:aux} assure that 
\begin{equation}
\label{eq:hatB.growth}
  \sup_{s\in[0,T)}
  \sup_{t\in(s,T)}
  \sup_{x\in H}
  \frac{
    \|\hat{B}_{t,s}(x)\|_{HS(U,\R)}
  }{g_3(x)}
  < \infty.
\end{equation}
Combining~\eqref{eq:running.inside}--\eqref{eq:hatB.growth} yields that for all $t\in(0,T)$ it holds that 
\begin{equation}
\begin{split}
&
  \ES\Bigg[
    \bigg(
    \int^t_0
    \|    \hat{B}_{ t, s }( 
          Y^I_s
        )\|^2_{HS(U,\R)} \,
  ds
  \bigg)^{\nicefrac{1}{2}}
  \Bigg]
  \\&=
  \ES\Bigg[
    \bigg(
    \int^t_0
    \|    \hat{B}_{ t, s }( 
          \mathcal{Y}_s
        )\|^2_{HS(U,\R)} \,
  ds
  \bigg)^{\nicefrac{1}{2}}
  \Bigg]
  \\&\leq
  \bigg[
  \sup_{s\in[0,t)}
  \sup_{x\in H}
  \frac{
    \|\hat{B}_{t,s}(x)\|_{HS(U,\R)}
  }{g_3(x)}
  \bigg]
  \ES\Bigg[
    \bigg(
    \int^t_0
    g_6(\mathcal{Y}_s) \,
  ds
  \bigg)^{\nicefrac{1}{2}}
  \Bigg]
  \\&\leq
  \sqrt{t} \,
  \bigg[
  \sup_{s\in[0,t)}
  \sup_{x\in H}
  \frac{
    \|\hat{B}_{t,s}(x)\|_{HS(U,\R)}
  }{g_3(x)}
  \bigg]
  \ES\bigg[
  \sup_{s\in[0,t]}
    \max\{1,\|\mathcal{Y}_s\|^3_H\}
  \bigg]
  < \infty.
\end{split}
\end{equation}
This shows that for all $t\in(0,T)$ it holds that 
$
  \ES\big[|\!
  \int^t_0
  \hat{B}_{ t, s }( 
            Y^I_s
          )
  \, dW_s
  |\big]
  < \infty
$ and
\begin{equation}
  \ES\bigg[\!
  \int^t_0
  \hat{B}_{ t, s }( 
            Y^I_s
          )
  \, dW_s
  \bigg]
  =0
\end{equation}
(cf., e.g., van Neerven et al.~\cite[Theorem~4.7]{NeervenVeraarWeis2015}).
This and~\eqref{eq:B.mild} ensure that for all $t\in(0,T)$ it holds that 
\begin{equation}
\begin{split}
\label{eq:mild_Ito_B}
&
\sum_{ b\in\mathbb{U} }
\E\left[
u_{0,2}( t, Y^I_t )\!\left(
B^b_I\big(
Y^I_t
\big)
+
B^b\big(
Y^I_t 
\big)
,
B^b_I\big(
Y^I_t 
\big)
-
B^b\big(
Y^I_t 
\big)
\right)
\right]
\\ & =
\sum_{ b\in\mathbb{U} }
\E\left[
u_{0,2}( t, e^{ A t } \xi )\!\left(
B_I^b\big(
e^{ A t } \xi
\big)
+
B^b\big(
e^{ A t } \xi 
\big)
,
B_I^b\big(
e^{ A t } \xi 
\big)
-
B^b\big(
e^{ A t } \xi 
\big)
\right)
\right]
\\&\quad+
\int_0^t
\E\left[ 
\tilde{B}_{ t, s }( 
Y^I_s
)
\right]
ds
.
\end{split}
\end{equation}
Putting~\eqref{eq:mild_Ito_F} and~\eqref{eq:mild_Ito_B}
into~\eqref{eq:Kolmogorov}
proves that
\begin{equation}
\label{eq:Kolmogorov2}
\begin{split}
&
    \ES\big[ 
      \varphi( Y^I_T )
    \big]
    -
    \ES\big[ 
      \varphi( X^\set_T )
    \big]
\\ & =
  \int_0^T
  \E\left[
    u_{0,1}( t, e^{ A t } \xi )
    \!\left(
      F_I\big(
        e^{ A t } \xi 
      \big)
      -
      F\big(
        e^{ A t } \xi
      \big)
    \right)
  \right]
  dt
\\ & +
  \frac{1}{2}
  \smallsum\limits_{ b\in\mathbb{U} }
  \int\limits_0^T
  \E\left[
    u_{0,2}( t, e^{ A t } \xi )\!\left(
      B^b_I\big(
        e^{ A t } \xi
      \big)
      +
      B^b\big(
        e^{ A t } \xi 
      \big)
      ,
      B^b_I\big(
        e^{ A t } \xi 
      \big)
      -
      B^b\big(
        e^{ A t } \xi 
      \big)
    \right)
  \right]
  dt
\\ & +
  \int_0^T
  \int_0^t
  \E\left[ 
    \tilde{F}_{ t, s }( 
      Y^I_s
    )
  \right]
  +
  \frac{ 1 }{ 2 }
  \,
  \E\left[ 
    \tilde{B}_{ t, s }( 
      Y^I_s
    )
  \right]
  ds
  \,
  dt
  .
\end{split}
\end{equation}
In the following we estimate the absolute values of the summands on
the right hand side of~\eqref{eq:Kolmogorov2}.
To this end we observe that Lemma~\ref{lem:c.delta.finite} and the hypothesis that 
$
  \vartheta < \nicefrac{1}{2}
$
demonstrate that $c_{-\vartheta} < \infty$.
This and Lemma~\ref{lem:aux} ensure that
for all 
$ t \in (0,T) $
it holds that
\begin{equation}
\label{eq:initial.value.drift}
\begin{split}
&
  \left|
  u_{0,1}( t, e^{ A t } \xi )
  \!\left(
    F_I( e^{ A t } \xi ) 
    -
    F( e^{ A t } \xi )
  \right)
  \right|
\leq
  \frac{  
    c_{ -\vartheta }
  }{
    {( T - t )^{ \vartheta }}
  }
  \left\|
    F_I( e^{ A t } \xi ) 
    -
    F( e^{ A t } \xi )
  \right\|_{ H_{ -\vartheta } }
\\ & \leq
  \frac{  
    c_{ -\vartheta }
    \left| F \right|_{
      C^1_b( H , H_{ -\vartheta } )
    }
    \left\| 
      P_{ \set\backslash I } 
    \right\|_{ L( H, H_{ - \rho } ) }
    \left\| 
      \xi
    \right\|_{ H }
  }{
    {( T - t )^{ \vartheta }}
    \,
    t^\rho
  }
  .
\end{split}
\end{equation}
Next note that H\"{o}lder's inequality implies that for all 
$x,y\in(0,\infty)$ 
with 
$
  (x-1)(y-1)
  \geq 0
$
and 
$
  x+y > 1
$
it holds that 
\begin{equation}
\label{eq:Beta.ub}
  \int^1_0
  (1-t)^{(x-1)} \,
  t^{(y-1)}
  \, dt
  \leq
  \frac{1}{(x+y-1)}
  .
\end{equation}
This, the fact that 
$
  \ES[
  \| \xi \|_H
  ]\leq
  K^I_1
$, 
and~\eqref{eq:initial.value.drift}
imply that
\begin{equation}
\label{eq:estimate_first_term}
\begin{split}
&
  \left|
  \int_0^T
  \E\left[
    u_{0,1}( t, e^{ A t } \xi )
    \!\left(
      F_I\big(
        e^{ A t } \xi 
      \big)
      -
      F\big(
        e^{ A t } \xi
      \big)
    \right)
  \right]
  dt
  \right|
\\ & \leq
  \frac{ 
    K^I_1
    \,
    c_{ -\vartheta }
    \,
    T^{ ( 1 - \vartheta - \rho ) }
    \left| F \right|_{
      C^1_b( H , H_{ -\vartheta } )
    }
    \left\| 
      P_{ \set\backslash I } 
    \right\|_{ L( H, H_{ - \rho } ) }
  }{
    \left( 1 - \vartheta - \rho \right)
  }
  .
\end{split}
\end{equation}
Inequality~\eqref{eq:estimate_first_term} provides us an estimate for the absolute value of the first summand
on the right hand side of~\eqref{eq:Kolmogorov2}.
In the next step we bound the absolute value of the second summand on the right hand
side~\eqref{eq:Kolmogorov2}. For this we observe that the Cauchy-Schwarz inequality ensures that for all
$ t \in [0,T) $ 
it holds that
\begin{equation}
\begin{split}
&
  \smallsum\limits_{ b\in\mathbb{U} }
  \left|
    u_{0,2}( t, e^{ A t } \xi )\!\left(
      B^b_I\big(
        e^{ A t } \xi
      \big)
      +
      B^b\big(
        e^{ A t } \xi
      \big)
      ,
      B^b_I\big(
        e^{ A t } \xi 
      \big)
      -
      B^b\big(
        e^{ A t } \xi 
      \big)
    \right)
  \right|
\\ & \leq
  \frac{
    c_{
      -\nicefrac{\vartheta}{2} , -\nicefrac{\vartheta}{2}
    }
  }{
    {( T - t )^{ \vartheta }}
  }
  \left[
  \smallsum\limits_{ b\in\mathbb{U} }
  \left\|
      B^b_I\big(
        e^{ A t } \xi
      \big)
      +
      B^b\big(
        e^{ A t } \xi
      \big)
  \right\|_{ 
    H_{ -\nicefrac{\vartheta}{2} }
  }
  \left\|
      B^b_I\big(
        e^{ A t } \xi 
      \big)
      -
      B^b\big(
        e^{ A t } \xi 
      \big)
  \right\|_{
    H_{ -\nicefrac{\vartheta}{2} }
  }
  \right]
\\&\leq
  \frac{
  	c_{
  		-\nicefrac{\vartheta}{2} , -\nicefrac{\vartheta}{2}
  	}
  }{
  {( T - t )^{ \vartheta }}
}
\left[
\smallsum\limits_{ b\in\mathbb{U} }
\left\|
B^b_I\big(
e^{ A t } \xi
\big)
+
B^b\big(
e^{ A t } \xi
\big)
\right\|^2_{ 
	H_{ -\nicefrac{\vartheta}{2} }
}
\right]^{\nicefrac{1}{2}}
\\&\quad\cdot
\left[
\smallsum\limits_{ b\in\mathbb{U} }
\left\|
B^b_I\big(
e^{ A t } \xi 
\big)
-
B^b\big(
e^{ A t } \xi 
\big)
\right\|^2_{
	H_{ -\nicefrac{\vartheta}{2} }
}
\right]^{\nicefrac{1}{2}}
\\ & 
{
=
  \frac{
    c_{
      -\nicefrac{\vartheta}{2} , -\nicefrac{\vartheta}{2}
    }
  }{
    {( T - t )^{ \vartheta }}
  }
  \left\|
      B_I\big(
        e^{ A t } \xi
      \big)
      +
      B\big(
        e^{ A t } \xi
      \big)
  \right\|_{ 
    HS(U,H_{ -\nicefrac{\vartheta}{2} })
  }
  \left\|
      B_I\big(
        e^{ A t } \xi 
      \big)
      -
      B\big(
        e^{ A t } \xi 
      \big)
  \right\|_{
    HS(U,H_{ -\nicefrac{\vartheta}{2} })
  }
  }
  .
\end{split} 
\end{equation}
The fact that 
$
g_1( \xi )
\| \xi \|_H
\leq
g_2( \xi )
$ 
hence shows that for all
$ t \in (0,T) $ 
it holds that
\begin{equation}
\begin{split}
&
\smallsum\limits_{ b\in\mathbb{U} }
\left|
u_{0,2}( t, e^{ A t } \xi )\!\left(
B^b_I\big(
e^{ A t } \xi
\big)
+
B^b\big(
e^{ A t } \xi
\big)
,
B^b_I\big(
e^{ A t } \xi 
\big)
-
B^b\big(
e^{ A t } \xi 
\big)
\right)
\right|
\\ & 
	\leq
	\frac{
		c_{
			-\nicefrac{\vartheta}{2} , -\nicefrac{\vartheta}{2}
		}
	}{
	{( T - t )^{ \vartheta }}
} \,
\big[
  \|B(P_I e^{At} \xi)\|_{HS(U,H_{ -\nicefrac{\vartheta}{2} })}
  +
  \|B(e^{At} \xi)\|_{HS(U,H_{ -\nicefrac{\vartheta}{2} })}
\big]
\\&\quad\cdot
\left\|
B
\right\|_{
	C_b^1( H , HS( U, H_{ -\nicefrac{\vartheta}{2} } ) )
}
\left\|
P_{\mathbb{H}\backslash I}e^{At}\xi
\right\|_H
\\ & 
{
	\leq
	\frac{
		c_{
			-\nicefrac{\vartheta}{2} , -\nicefrac{\vartheta}{2}
		}
	}{
	{( T - t )^{ \vartheta }}
}
\left\|
B
\right\|_{
	C_b^1( H , HS( U, H_{ -\nicefrac{\vartheta}{2} } ) )
}^2
\big(
g_1( P_I e^{At} \xi )
+
g_1( e^{At} \xi )
\big)
\left\|
P_{\mathbb{H}\backslash I}e^{At}\xi
\right\|_H
}
\\ & \leq
\frac{
	2 \,
	c_{
		-\nicefrac{\vartheta}{2} , -\nicefrac{\vartheta}{2}
	}
	\,
	g_2( \xi )
}{
{( T - t )^{ \vartheta }}
\,
t^\rho
}
\left\|
B
\right\|_{
	C_b^1( H , HS( U, H_{ -\nicefrac{\vartheta}{2} } ) )
}^2
\left\|
P_{ \set \backslash I }  
\right\|_{ 
	L( H , H_{ - \rho } ) 
}
.
\end{split} 
\end{equation}
This, the fact that 
$
  \ES[
  g_2( \xi )
  ]\leq
  K^I_2
$, and~\eqref{eq:Beta.ub}
imply that
\begin{equation}
\label{eq:inequality_second_term}
\begin{split}
& 
  \left|
  \smallsum\limits_{ b\in\mathbb{U} }
  \int_0^T
  \E\left[
    u_{0,2}( t, e^{ A t } \xi )\!\left(
      B^b_I\big(
        e^{ A t } \xi
      \big)
      +
      B^b\big(
        e^{ A t } \xi 
      \big)
      ,
      B^b_I\big(
        e^{ A t } \xi 
      \big)
      -
      B^b\big(
        e^{ A t } \xi 
      \big)
    \right)
  \right]
  dt
  \right|
\\ & \leq 
  \frac{
    2 \,
    K^I_2
    \,
      c_{
        -\nicefrac{\vartheta}{2} , -\nicefrac{\vartheta}{2}
      }
    \,
    T^{ 
      ( 1 -\vartheta - \rho ) 
    }
  }{
    \left( 1 -\vartheta - \rho 
    \right) 
  }
  \left\|
    B
  \right\|_{
    C_b^1( H , HS( U, H_{-\nicefrac{\vartheta}{2}} ) )
  }^2
  \left\|
    P_{ \set \backslash I }  
  \right\|_{ 
    L( H , H_{ - \rho } ) 
  }
  .
\end{split}
\end{equation}
Inequality~\eqref{eq:inequality_second_term}
provides us an estimate for the second term
on the right hand side of~\eqref{eq:Kolmogorov2}.
In the next step we bound the absolute value of the term
$
  \int_0^T
  \int_0^t
  \ES\big[ 
    \tilde{F}_{ t, s }( Y^I_s )
  \big] 
  \, ds \, dt
$
on the right hand side of ~\eqref{eq:Kolmogorov2}.
For this we note that~\eqref{eq:tildeF_def} 
shows that for all 
$ s, t \in [0,T) $,
$ x \in H $
with $ t > s $
it holds that
\begin{equation}
\label{eq:F_third_summand_bound}
\begin{split}
&
  \big|
    \tilde{F}_{ t, s }( x ) 
  \big|
\\ & \leq
  \frac{ c_{ -\vartheta , 0 } }{
    {( T - t )^{ \vartheta }}
  }
    \left\|
      F_I(
        e^{ A ( t - s ) }x
      )
      -
      F(
        e^{ A ( t - s ) } x
      )
    \right\|_{ H_{ -\vartheta } }
    \left\|
      e^{ A (t - s) }
      F_I( x )
    \right\|_{ H }
\\ & +
  \frac{
    c_{ -\vartheta }
  }{
    {( T - t )^{ \vartheta }}
  }
    \left\|
      \left[
        F_I'(
          e^{ A ( t - s ) }x
        )
        -
        F'(
          e^{ A ( t - s ) } x
        )
      \right]
      e^{ A (t - s) }
      F_I( x )
    \right\|_{ 
      H_{ -\vartheta }
    }
\\ & +
    \frac{
      c_{ -\vartheta,0,0 }
    }{
      2 \, {( T - t )^{ \vartheta }}
    }
    \left\|
      F_I(
        e^{ A ( t - s ) }x
      )
      -
      F(
        e^{ A ( t - s ) } x
      )
    \right\|_{ H_{ -\vartheta } }
    \left\|
      e^{ A (t - s) }
      B_I( x ) 
    \right\|^2_{ HS( U, H ) }
\\ & +
  \frac{ c_{ -\vartheta, 0 } }{
    {( T - t )^{ \vartheta }}
  }
    \smallsum\limits_{ b\in\mathbb{U} }
    \left[
    \left\|
      \left[
        F_I'(
          e^{ A ( t - s ) }x
        )
        -
        F'(
          e^{ A ( t - s ) } x
        )
      \right]
      e^{ A (t - s) }
      B^b_I( x ) 
    \right\|_{ H_{ -\vartheta } }
    \left\|
      e^{ A (t - s) }
      B^b_I( x ) 
    \right\|_{ H }
    \right]
\\ & +
  \frac{ c_{ -\vartheta } }{
    2 \, {( T - t )^{ \vartheta }}
  }
    \smallsum\limits_{ b\in\mathbb{U} }
    \left\|
      \left[
        F_I''(
          e^{ A ( t - s ) }x
        )
        -
        F''(
          e^{ A ( t - s ) } x
        )
      \right]
      \!
      \left(
        e^{ A (t - s) }
        B^b_I( x ) 
        ,
        e^{ A (t - s) }
        B^b_I( x ) 
      \right)
    \right\|_{
      H_{ -\vartheta }
    }
  .
\end{split}
\end{equation}
Next observe that for all 
$ x, v \in H $,
$ r \in [0, \vartheta] $,
$ s, t \in [0,T] $ 
with
$
  s < t
$
it holds that
\begin{equation}
\begin{split}
&
  \left\|
      \left[
        F_I'(
          e^{ A ( t - s ) }x
        )
        -
        F'(
          e^{ A ( t - s ) } x
        )
      \right]
      e^{ A ( t - s ) } \, v
    \right\|_{ H_{ -\vartheta } }
\\ & =
  \left\|
      \left[
        F'(
          e^{ A ( t - s ) } P_I( x )
        )
        P_I
        -
        F'(
          e^{ A ( t - s ) } x
        )
      \right]
      e^{ A ( t - s ) } \, v
  \right\|_{ H_{ -\vartheta } }
\\ & \leq
  \left\|
      \left[
        F'(
          e^{ A ( t - s ) } P_I( x )
        )
        -
        F'(
          e^{ A ( t - s ) } x
        )
      \right]
      P_I \,
      e^{ A ( t - s ) } \, v
  \right\|_{ H_{ -\vartheta } }
\\ & \quad 
  +
  \left\|
        F'(
          e^{ A ( t - s ) } x
        )
        \,
      P_{ \set \backslash I } \, 
      e^{ A ( t - s ) } \, v
  \right\|_{ H_{ -\vartheta } }
\\ & \leq
  \left|
    F
  \right|_{
    C^2_b( H, H_{-\vartheta} )
  }
  \left\|
    e^{ A ( t - s ) } 
    P_{ \set \backslash I } x
  \right\|_{ H }
  \left\|
    P_I e^{ A ( t - s ) } v
  \right\|_{
    H
  }
  +
  \left\|
    F
  \right\|_{
    C_b^1( H, H_{ -\vartheta } )
  }
  \left\|
    P_{ \set \backslash I } 
    e^{A(t-s)} v
  \right\|_H 
\\ & \leq
  \left[
    g_1( x )
    \left|
      F
    \right|_{
      C^2_b( H, H_{-\vartheta} )
    }
  +
    \left\|
      F
    \right\|_{
      C_b^1( H, H_{ -\vartheta } )
    }
  \right]
  \frac{
    \left\| v \right\|_{ H_{ - r } }
    \left\|
      P_{ \set \backslash I } 
    \right\|_{ L( H, H_{ - \rho } ) }
  }{
    (t-s)^{ \rho + r }
  }
\\ & \leq
  \frac{
    g_1( x )
    \left\|
      F
    \right\|_{
      C^2_b( H, H_{-\vartheta} )
    }
    \left\| v \right\|_{ H_{ - r } }
    \left\|
      P_{ \set \backslash I } 
    \right\|_{ L( H, H_{ - \rho } ) }
  }{
    ( t - s )^{ ( \rho + r ) }
  }
  .
\end{split}
\end{equation}
This and the fact that for all 
$ x \in H $ 
it holds that 
$
  g_1( x )
  \| x \|_H
  \leq
  g_2( x )
$ 
imply that
for all 
$ x \in H $,
$ s, t \in [0,T] $ 
with
$
  s < t
$
it holds that
\begin{equation}
\label{eq:F_prime_estimate1}
\begin{split}
&
  \left\|
      \left[
        F_I'(
          e^{ A ( t - s ) }x
        )
        -
        F'(
          e^{ A ( t - s ) } x
        )
      \right]
      e^{A(t-s)}F_I(x)
    \right\|_{ H_{ -\vartheta } }
\leq
  \tfrac{ 
    g_2( x )
    \,
    \left\|
      F
    \right\|_{
      C_b^2( H, H_{-\vartheta} )
    }^2
    \,
    \|
      P_{ \set \backslash I } 
    \|_{ L( H, H_{ - \rho } ) }
  }{
    \left(
      t - s
    \right)^{ 
      \rho + \vartheta
    }
  }
\end{split}
\end{equation}
\begin{equation}
\label{eq:F_prime_estimate2}
\begin{split}
  \text{and}
  \qquad
&
    \left[
    \smallsum\limits_{ b\in\mathbb{U} }
  \left\|
      \left[
        F_I'(
          e^{ A ( t - s ) }x
        )
        -
        F'(
          e^{ A ( t - s ) } x
        )
      \right]
      e^{A(t-s)}B^b_I(x)
    \right\|_{ H_{ -\vartheta } }^2
    \right]^{\frac{1}{2}}
\\ & \leq
  \frac{ 
    g_2( x )
    \left\|
      F
    \right\|_{
      C_b^2( H, H_{-\vartheta} )
    }
    \left\|
      B
    \right\|_{
      C_b^1( H, HS(U,H_{-\nicefrac{\vartheta}{2}}) )
    }
    \|
      P_{ \set \backslash I } 
    \|_{ L( H, H_{ - \rho } ) }
  }{
    \left(
      t - s
    \right)^{ 
     (\rho + \frac{\vartheta}{2})
    }
  }
  .
\end{split}
\end{equation}
Moreover, note that
for all
$ x \in H $,
$ s, t \in [0,T] $ 
with
$
  s < t 
$
it holds that
\begin{equation}
\label{eq:F_prime_prime}
\begin{split}
&
    \smallsum\limits_{ b\in\mathbb{U} }
    \left\|
      \left[
        F_I''(
          e^{ A ( t - s ) }x
        )
        -
        F''(
          e^{ A ( t - s ) } x
        )
      \right]
      \!
      \left(
        e^{ A (t - s) }
        B^b_I( x )
        ,
        e^{ A (t - s) }
        B^b_I( x ) 
      \right)
    \right\|_{
      H_{ -\vartheta }
    }
\\ & \leq
    \smallsum\limits_{ b\in\mathbb{U} }
    \left\|
        F''(
          e^{ A ( t - s ) } P_I( x )
        )
      \!
      \left(
        \left[
          \operatorname{Id}_H
          +
          P_I
        \right]
        e^{ A (t - s) }
        B^b_I( x )
        ,
        \left[
          \operatorname{Id}_H
          -
          P_I
        \right]
        e^{ A (t - s) }
        B^b_I( x )
      \right)
    \right\|_{
      H_{ -\vartheta }
    }
\\ & \quad +
    \smallsum\limits_{ b\in\mathbb{U} }
    \left\|
      \left[
        F''(
          e^{ A ( t - s ) } P_I( x )
        )
        -
        F''(
          e^{ A ( t - s ) } x
        )
      \right]
      \!
      \left(
        e^{ A (t - s) }
        B^b_I( x )
        ,
        e^{ A (t - s) }
        B^b_I( x )
      \right)
    \right\|_{
      H_{ -\vartheta }
    }
\\ & \leq
  2 
  \left\|
    F
  \right\|_{
    C_b^2( H , H_{-\vartheta} )
  }
      \left\|
        e^{ A (t - s) }
        B_I( x )
      \right\|_{ HS( U, H ) }
      \left\|
        P_{ \set \backslash I }
        e^{ A (t - s) }
        B_I( x )
      \right\|_{
        HS(U,H)
      }
\\ & \quad +
  \frac{
    \left|
      F
    \right|_{
      C^3_b( H , H_{-\vartheta} )
    }
    g_1( x )
      \left\|
        P_{ \set \backslash I 
        }
      \right\|_{
        L( H, H_{ - \rho } )
      }
      \left\|
        e^{ A (t - s) }
        B_I( x )
      \right\|^2_{ HS( U, H) }
  }{
    \left( t - s \right)^{ 
      \rho 
    }
  }
\\ & \leq
  \frac{
    2 
    \left\|
      F
    \right\|_{
      C_b^2( H , H_{-\vartheta} )
    }
    \left\|
      B
    \right\|^2_{
      C_b^1( H, HS( U, H_{-\nicefrac{\vartheta}{2}} ) )
    }
    \|
      P_{ \set \backslash I }
    \|_{
      L( H, H_{ - \rho } )
    }
    \,
    g_2( x )
  }{
    \left( t - s \right)^{ 
      \rho + \vartheta
    }
  }
\\ & \quad +
  \frac{
    \left|
      F
    \right|_{
      C^3_b( H , H_{-\vartheta} )
    }
    \left\|
      B
    \right\|_{
      C_b^1( H , HS( U, H_{-\nicefrac{\vartheta}{2}} ) )
    }^2
      \|
        P_{ \set \backslash I 
        }
      \|_{
        L( H, H_{ - \rho } )
      }
      \,
    g_3( x )
  }{
    \left( t - s \right)^{ 
      \rho + \vartheta
    }
  }
\\ & \leq
  \frac{
    2 
    \left\|
      F
    \right\|_{
      C_b^3( H , H_{-\vartheta} )
    }
    \left\|
      B
    \right\|^2_{
      C_b^1( H, HS( U, H_{-\nicefrac{\vartheta}{2}} ) )
    }
    \|
      P_{ \set \backslash I }
    \|_{
      L( H, H_{ - \rho } )
    }
    \,
    g_3( x )
  }{
    \left( t - s \right)^{ 
      \rho + \vartheta
    }
  }.
\end{split}
\end{equation}
Putting~\eqref{eq:F_prime_estimate1}, \eqref{eq:F_prime_estimate2}, and~\eqref{eq:F_prime_prime} 
into~\eqref{eq:F_third_summand_bound}
proves that
for all
  $ x \in H $,
$ s, t \in [0,T) $
with 
$ t > s $
it holds that
\begin{equation}
\begin{split}
  \big|
    \tilde{F}_{ t, s }( x ) 
  \big|
& \leq
\Big[ 
  c_{ - \vartheta , 0 } 
  \,
  \| F \|_{ C_b^1( H , H_{ - \vartheta } ) }^2 
  \, g_2(x)
  + c_{ - \vartheta } \,
  \| F \|_{ C_b^2( H , H_{ - \vartheta } ) }^2 \, g_2(x) 
\\
&
  + c_{ - \vartheta , 0 , 0 } 
  \,
  \| F \|_{ C_b^1( H, H_{ - \vartheta } ) 
  } 
  \,
  \| B \|_{ 
    C_b^1( H, HS(U, H_{ - \nicefrac{ \vartheta }{ 2 } } ) ) 
  }^2 
  \,
  g_3(x) 
\\
&
  + 
  c_{ - \vartheta , 0 } 
  \,
  \| 
    B
  \|_{
    C_b^1( H, HS(U, H_{ - \nicefrac{ \vartheta }{ 2 } } ) ) 
  }^2 
  \,
  g_3( x ) 
  \,
  \| F \|_{
    C_b^2( H, H_{ - \vartheta } ) 
  } 
\\
& 
  + c_{-\vartheta} \, 
  \| B \|_{
    C_b^1( H , HS(U, H_{ - \nicefrac{ \vartheta }{ 2 } } ) ) 
  }^2 
  \, g_3(x) 
  \,
  \| F \|_{
    C_b^3( H , H_{ - \vartheta } ) 
  } 
\Big]
  \frac{
    \| P_{ \set \backslash I } \|_{ L( H, H_{ - \rho } ) } 
  }{
    {( T - t )^{ \vartheta }}
    \left( t - s \right)^{ ( \rho + \vartheta ) } 
  }
  .
\end{split}
\end{equation}
This implies that for all 
$ t \in (0,T) $,
$ s \in [0,t) $,
$ x \in H $
it holds that
\begin{equation}
\begin{split}
&
  \big|
    \tilde{F}_{ t, s }(x)
  \big|\leq
  \frac{
    2
    \left[ 
      c_{ -\vartheta }
      +
      c_{ -\vartheta, 0 }
      +
      c_{ -\vartheta, 0, 0 }
    \right]
    \varsigma_{ F, B }
    \,
    g_3(x)
    \,
    \|
      P_{ \set \backslash I }
    \|_{ L(H,H_{-\rho}) }
  }{
    {( T - t )^{ \vartheta }}
    \left( t - s \right)^{ 
      \rho+\vartheta
    }
  }
  .
\end{split}
\end{equation}
This, in turn, proves that
\begin{equation}
\label{eq:F_estimate}
\begin{split}
&
  \left|
  \int_0^T
  \int_0^t
  \E\left[ 
    \tilde{F}_{ t, s }( Y^I_s )
  \right]
  ds
  \,
  dt
  \right|
\\ & \leq
  \frac{
    2 \, 
    \varsigma_{ F, B }
    \,
    \|
      P_{ \set \backslash I }
    \|_{ L( H , H_{ - \rho } ) }
    \,
    K^I_3
    \,
    \big[
      c_{ -\vartheta }
      +
      c_{ -\vartheta, 0 }
      +
      c_{ -\vartheta, 0, 0 }
    \big]
  }{
    \left(
      1 - \rho - \vartheta
    \right)
  }
  \int^T_0
  \frac{
    t^{ ( 1 - \rho - \vartheta ) }
  }{
    ( T - t )^\vartheta
  }
  \, dt
\\ & \leq
  \frac{
    2 \, 
    T^{ ( 1 - \rho - \vartheta ) }
    \,
    \varsigma_{ F, B }
    \,
    \|
      P_{ \set \backslash I }
    \|_{ L( H , H_{ - \rho } ) }
    \,
    K^I_3
    \,
    \big[
      c_{ -\vartheta }
      +
      c_{ -\vartheta, 0 }
      +
      c_{ -\vartheta, 0, 0 }
    \big]
  }{
    \left(
      1 - \rho - \vartheta
    \right)
  }
  \int^T_0
  \frac{1}{
    ( T - t )^\vartheta
  }
  \, dt
\\ & =
  \frac{
    2 \, 
    T^{ ( 2 - \rho - 2 \vartheta ) }
    \,
    \varsigma_{ F, B }
    \,
    \|
      P_{ \set \backslash I }
    \|_{ L( H , H_{ - \rho } ) }
    \,
    K^I_3
    \,
    \big[
      c_{ -\vartheta }
      +
      c_{ -\vartheta, 0 }
      +
      c_{ -\vartheta, 0, 0 }
    \big]
  }{
    \left(
      1 - \rho - \vartheta
    \right)
    {\left(
      1 - \vartheta 
    \right)}
  }
  .
\end{split} 
\end{equation}
It thus remains to bound the term
$
  \frac{ 1 }{ 2 }
  \int_0^T
  \int_0^t
  \ES\big[
    \tilde{B}_{ t, s }( Y^I_s )
  \big]
  \, ds \, dt
$
on the right hand side of~\eqref{eq:Kolmogorov2}.
For this we estimate the terms which appear on the right hand side of~\eqref{eq:tildeB_def} (cf.\ \eqref{eq:tilde.B.many.sum} and~\eqref{eq:tilde.B.many.sum.II} below).
We start by presenting a few auxiliary estimates for these terms.
Note that 
for all
$ t \in (0,T] $,
$ s \in [0,t) $,
$ x, v \in H $
it holds that
\begin{equation}
\begin{split}
&
\left[
\smallsum_{ b\in\mathbb{U} }
\left\|
\left[
( B^b_I )'( e^{ A ( t - s ) } x ) 
- 
( B^b )'( e^{ A ( t - s ) } x )
\right]
e^{ A ( t - s ) } \, v
\right\|^2_{ H_{ -\nicefrac{\vartheta}{2} } }
\right]^{ 1 / 2 }
\\ & =
\left\|
\left[
B'( e^{ A ( t - s ) } P_I( x ) ) P_I
-
B'( e^{ A ( t - s ) } x )
\right]
e^{ A ( t - s ) } \, v
\right\|_{ HS( U, H_{ - \vartheta / 2 } ) }
\\ & \leq
\left\|
\left[
B'( e^{ A ( t - s ) } P_I( x ) )
-
B'( e^{ A ( t - s ) } x )
\right]
P_I \, e^{ A ( t - s ) } \, v
\right\|_{ HS( U, H_{ - \vartheta / 2 } ) }
\\ & \quad +
\left\|
B'( e^{ A ( t - s ) } x )
\, P_{\set\backslash I} \, e^{ A ( t - s ) } \, v
\right\|_{ HS( U, H_{ - \vartheta / 2 } ) }
.
\end{split}
\end{equation}
This shows that for all 
$ r \in [0, \vartheta] $,
$ t \in (0,T] $,
$ s \in [0,t) $,
$ x, v \in H $
it holds that
\begin{equation}\label{eq:tildeBpreestimate_1}
\begin{split}
&
\left[
\smallsum_{ b\in\mathbb{U} }
\left\|
\left[
( B^b_I )'( e^{ A ( t - s ) } x ) 
- 
( B^b )'( e^{ A ( t - s ) } x )
\right]
e^{ A ( t - s ) } \, v
\right\|^2_{ H_{ -\nicefrac{\vartheta}{2} } }
\right]^{ 1 / 2 }
\\ & \leq
\left| B \right|_{
	C^{ 2 }_b( H, HS( U, H_{ -\nicefrac{\vartheta}{2} } ) )
}
\left\|
e^{ A ( t - s ) } 
P_{ \set \backslash I } x
\right\|_{ H }
\left\|
P_I e^{ A ( t - s ) } v
\right\|_{
	H
}
\\ & \quad
+
\left\| B \right\|_{
	C_b^1( H, HS( U, H_{ -\nicefrac{\vartheta}{2} } ) )
}
\left\|
P_{ \set \backslash I } 
e^{A(t-s)} v
\right\|_H 
\\ & \leq
\left[
\left| B \right|_{
	C^{ 2 }_b( H, HS( U, H_{ -\nicefrac{\vartheta}{2} } ) )
}
\,
g_1(x)
+
\left\| B \right\|_{
	C_b^1( H, HS( U, H_{ -\nicefrac{\vartheta}{2} } ) )
}
\right]
\frac{
	\left\| v \right\|_{
		H_{ - r }
	}
	\|
	P_{ \set \backslash I }
	\|_{ L(H,H_{-\rho}) }
}{
	\left( t - s \right)^{ 
		( \rho + r )
	}
}
\\ & \leq
\frac{
	g_1( x )
	\left\| B \right\|_{
		C_b^2( H, HS( U, H_{ -\nicefrac{\vartheta}{2} } ) )
	}
	\left\| v \right\|_{ H_{ - r } }
	\left\|
	P_{ \set \backslash I } 
	\right\|_{ L( H, H_{ - \rho } ) }
}{
	( t - s )^{ ( \rho + r ) }
}
.
\end{split}
\end{equation}
Hence, we obtain that for all 
$ t \in (0,T] $,
$ s \in [0,t) $,
$ x, v \in H $
it holds that
\begin{equation}\label{eq:tildeBpreestimate_1.square}
\begin{split}
&
\smallsum_{ b\in\mathbb{U} }
\left\|
\left[
( B^b_I )'( e^{ A ( t - s ) } x ) 
- 
( B^b )'( e^{ A ( t - s ) } x )
\right]
e^{ A ( t - s ) } \, v
\right\|^2_{ H_{ -\nicefrac{\vartheta}{2} } }
\\ & \leq
\frac{
	|g_1( x )|^2
	\left\| B \right\|^2_{
		C_b^2( H, HS( U, H_{ -\nicefrac{\vartheta}{2} } ) )
	}
	\left\| v \right\|^2_{ H_{ -\nicefrac{\vartheta}{2} } }
	\left\|
	P_{ \set \backslash I } 
	\right\|^2_{ L( H, H_{ - \rho } ) }
}{
	( t - s )^{ ( 2\rho + \vartheta ) }
}
.
\end{split}
\end{equation}
We now estimate the first argument of the bilinear operator appearing in the first summand on the right hand side of~\eqref{eq:tildeB_def}.
Observe that for all 
$ r \in [0, \vartheta] $,
$ t \in (0,T] $,
$ s \in [0,t) $,
$ x, v \in H $
it holds that
\begin{equation}\label{eq:tildeBpreestimate_2}
\begin{split}
&
  \left[
  \smallsum_{ b\in\mathbb{U} }
  \left\|
    \left[
      ( B^b_I )'( e^{ A ( t - s ) } x ) 
      + 
      ( B^b )'( e^{ A ( t - s ) } x )
    \right]
    e^{ A ( t - s ) } \, v
  \right\|^2_{ H_{ -\nicefrac{\vartheta}{2} } }
  \right]^{ 1 / 2 }
\\ & =
  \left\|
    \left[
      B'( e^{ A ( t - s ) } P_I( x ) ) P_I
      +
      B'( e^{ A ( t - s ) } x )
    \right]
    e^{ A ( t - s ) } \, v
  \right\|_{ HS( U, H_{ - \vartheta / 2 } ) }
\\ & \leq
  2\left\| B \right\|_{
    C^{ 1 }_b( H, HS( U, H_{ -\nicefrac{\vartheta}{2} } ) )
  }
  \left\|
    e^{A(t-s)} v
  \right\|_H 
\leq
  \frac{
  2\left\| B \right\|_{
    C^{ 1 }_b( H, HS( U, H_{ -\nicefrac{\vartheta}{2} } ) )
  }
    \left\| v \right\|_{ H_{ - r } }
  }{
    ( t - s )^{ r }
  }
  .
\end{split}
\end{equation}
This implies that for all 
$ t \in (0,T] $,
$ s \in [0,t) $,
$ x \in H $
it holds that
\begin{equation}
\label{eq:messy.I}
\begin{split}
&
\left[
\smallsum_{ b\in\mathbb{U} }
\left\|
\left[
( B^b_I )'( e^{ A ( t - s ) } x ) 
+ 
( B^b )'( e^{ A ( t - s ) } x )
\right]
e^{ A ( t - s ) } F_I( x ) 
\right\|^2_{ H_{ -\nicefrac{\vartheta}{2} } }
\right]^{ 1 / 2 }
\\ & \leq
\frac{
	2
	\left\| B \right\|_{
		C_b^{1}( H, HS( U, H_{ -\nicefrac{\vartheta}{2} } ) )
	}
	\left\| F_I(x) \right\|_{ H_{ - \vartheta } }
}{
( t - s )^{ \vartheta }
}
\leq
\frac{
	2\,
	g_{1}( x )
	\left\| B \right\|_{
		C_b^{1}( H, HS( U, H_{ -\nicefrac{\vartheta}{2} } ) )
	}
	\left\| F \right\|_{
		C_b^{1}( H, H_{ -\vartheta } )
	}
}{
( t - s )^{ \vartheta }
}
.
\end{split}
\end{equation}
{Next we use~\eqref{eq:tildeBpreestimate_2} to estimate the first argument of the bilinear operator appearing in the fourth summand on the right hand side of~\eqref{eq:tildeB_def} and the first argument of the bilinear operator appearing in the seventh summand on the right hand side of~\eqref{eq:tildeB_def}.}
More specifically, note that~\eqref{eq:tildeBpreestimate_2} ensures that for all
$ t \in (0,T] $,
$ s \in [0,t) $,
$ x, v \in H $
it holds that
\begin{equation}\label{eq:tildeBpreestimate_2.square}
\begin{split}
&
\smallsum_{ b\in\mathbb{U} }
\left\|
\left[
( B^b_I )'( e^{ A ( t - s ) } x ) 
+ 
( B^b )'( e^{ A ( t - s ) } x )
\right]
e^{ A ( t - s ) } \, v
\right\|^2_{ H_{ -\nicefrac{\vartheta}{2} } }
\\&\leq
\frac{
	4\left\| B \right\|^2_{
		C^{ 1 }_b( H, HS( U, H_{ -\nicefrac{\vartheta}{2} } ) )
	}
	\left\| v \right\|^2_{ H_{ -\nicefrac{\vartheta}{2} } }
}{
( t - s )^\vartheta
}
.
\end{split}
\end{equation}
This shows that for all 
$ t \in (0,T] $, 
$ s \in [0,t) $, 
$x\in H$ 
it holds that 
\begin{equation}
\label{eq:messy.II}
\begin{split}
&
\left[
\smallsum_{ b_1,b_2\in\mathbb{U} }
\left\|
\left[
( B^{b_2}_I )'( e^{ A ( t - s ) } x ) 
+ 
( B^{b_2} )'( e^{ A ( t - s ) } x )
\right]
e^{ A ( t - s ) } B^{b_1}_I( x ) 
\right\|^2_{ H_{ -\nicefrac{\vartheta}{2} } }
\right]^{ 1 / 2 }
\\&=
\left[
\smallsum_{ b_1\in\mathbb{U} }
\left(
\smallsum_{ b_2\in\mathbb{U} }
\left\|
\left[
( B^{b_2}_I )'( e^{ A ( t - s ) } x ) 
+ 
( B^{b_2} )'( e^{ A ( t - s ) } x )
\right]
e^{ A ( t - s ) } B^{b_1}_I( x ) 
\right\|^2_{ H_{ -\nicefrac{\vartheta}{2} } }
\right)
\right]^{ 1 / 2 }
\\&\leq
\left[
{\sum_{ b_1\in\mathbb{U} }}
\left(
\frac{
	4 \left\|B\right\|_{
		C_b^{1}( H, HS( U, H_{ -\nicefrac{\vartheta}{2} } ) )
	}^2
	\left\|B^{b_1}_I( x )\right\|^2_{H_{ -\nicefrac{\vartheta}{2} }}
	}{
	(t-s)^\vartheta
}
\right)
\right]^{ 1 / 2 }
\\ & =
\frac{
	2
	\left\| B \right\|_{
		C_b^{1}( H, HS( U, H_{ -\nicefrac{\vartheta}{2} } ) )
	}
	\left\| B_I(x) \right\|_{ HS( U, H_{ -\nicefrac{\vartheta}{2} } ) }
}{
( t - s )^{ \frac{\vartheta}{2} }
}
\leq
\frac{
	2 \,
	g_1( x )
	\left\| B \right\|_{
		C_b^{1}( H, HS( U, H_{ -\nicefrac{\vartheta}{2} } ) )
	}^2
}{
( t - s )^{ \frac{\vartheta}{2} }
}
.
\end{split}
\end{equation}
To estimate the second argument of the bilinear operator appearing in the second summand on the right hand side of~\eqref{eq:tildeB_def} we employ~\eqref{eq:tildeBpreestimate_1} to obtain that for all 
$ t \in (0,T] $,
$ s \in [0,t) $,
$ x \in H $
it holds that
\begin{equation}\label{eq:tildeBaux_3}
\begin{split}
&
 \left[
 \smallsum_{ b\in\mathbb{U} }
 \left\|
   \left[
     ( B^b_I )'( e^{ A ( t - s ) } x ) 
     - 
     ( B^b )'( e^{ A ( t - s ) } x )
   \right]
   e^{ A ( t - s ) } F_I( x ) 
 \right\|^2_{ H_{ -\nicefrac{\vartheta}{2} } }
 \right]^{ 1 / 2 }
\\ & \leq
 \frac{
   g_1( x )
   \left\| B \right\|_{
     C_b^{2}( H, HS( U, H_{ -\nicefrac{\vartheta}{2} } ) )
   }
   \left\| F_I(x) \right\|_{ H_{ - \vartheta } }
   \left\|
     P_{ \set \backslash I } 
   \right\|_{ L( H, H_{ - \rho } ) }
 }{
   ( t - s )^{ ( \rho + \vartheta ) }
 }
 \\ & \leq
   \frac{
     g_{2}( x )
     \left\| B \right\|_{
       C_b^{2}( H, HS( U, H_{ -\nicefrac{\vartheta}{2} } ) )
     }
     \left\| F \right\|_{
       C_b^{1}( H, H_{-\vartheta} )
     }
     \left\|
       P_{ \set \backslash I } 
     \right\|_{ L( H, H_{ - \rho } ) }
   }{
     ( t - s )^{ ( \rho + \vartheta ) }
   }
.
\end{split}
\end{equation}
{In addition, we use~\eqref{eq:tildeBpreestimate_1.square} to estimate the second argument of the bilinear operator appearing in the fourth summand on the right hand side of~\eqref{eq:tildeB_def} and the second argument of the trilinear operator appearing in the eighth summand on the right hand side of~\eqref{eq:tildeB_def}.}
More specifically, observe that~\eqref{eq:tildeBpreestimate_1.square}
shows that for all 
$ t \in (0,T] $, 
$ s \in [0,t) $, 
$x\in H$ 
it holds that 
\begin{equation}
\begin{split}
&
 \left[
 \smallsum_{ b_1,b_2\in\mathbb{U} }
 \left\|
   \left[
     ( B^{b_2}_I )'( e^{ A ( t - s ) } x ) 
     - 
     ( B^{b_2} )'( e^{ A ( t - s ) } x )
   \right]
   e^{ A ( t - s ) } B^{b_1}_I( x ) 
 \right\|^2_{ H_{ -\nicefrac{\vartheta}{2} } }
 \right]^{ 1 / 2 }
\\&=
 \left[
 \smallsum_{ b_1\in\mathbb{U} }
 \left(
 \smallsum_{ b_2\in\mathbb{U} }
 \left\|
 \left[
 ( B^{b_2}_I )'( e^{ A ( t - s ) } x ) 
 - 
 ( B^{b_2} )'( e^{ A ( t - s ) } x )
 \right]
 e^{ A ( t - s ) } B^{b_1}_I( x ) 
 \right\|^2_{ H_{ -\nicefrac{\vartheta}{2} } }
 \right)
 \right]^{ 1 / 2 }
\\&\leq
 \left[
 \sum_{ b_1\in\mathbb{U} }
 \left(
 \frac{
   |g_1( x )|^2
   \left\| B \right\|^2_{
   	C_b^{2}( H, HS( U, H_{ -\nicefrac{\vartheta}{2} } ) )
   }
   \left\| B^{b_1}_I(x) \right\|^2_{ H_{ -\nicefrac{\vartheta}{2} } }
   \left\|
   P_{ \set \backslash I } 
   \right\|^2_{ L( H, H_{ - \rho } ) } 	
 	}{
 	(t-s)^{(2\rho+\vartheta)}
 	}
 \right)
 \right]^{ 1 / 2 }
\\ & =
 \frac{
   g_1( x )
   \left\| B \right\|_{
     C_b^{2}( H, HS( U, H_{ -\nicefrac{\vartheta}{2} } ) )
   }
   \left\| B_I(x) \right\|_{ HS( U, H_{ -\nicefrac{\vartheta}{2} } ) }
   \left\|
     P_{ \set \backslash I } 
   \right\|_{ L( H, H_{ - \rho } ) }
 }{
   ( t - s )^{ ( \rho + \frac{\vartheta}{2} ) }
 }
\\ & \leq
 \frac{
   g_{2}( x )
   \left\| B \right\|_{
     C_b^{2}( H, HS( U, H_{ -\nicefrac{\vartheta}{2} } ) )
   }
   \left\| B \right\|_{
     C_b^{1}( H, HS( U, H_{ -\nicefrac{\vartheta}{2} } ) )
   }
   \left\|
     P_{ \set \backslash I } 
   \right\|_{ L( H, H_{ - \rho } ) }
 }{
   ( t - s )^{ ( \rho + \frac{\vartheta}{2} ) }
 }
 .
\end{split}
\end{equation}
We next estimate the second argument of the bilinear operator appearing in the fifth summand on the right hand side of~\eqref{eq:tildeB_def} and the second argument of the bilinear operator appearing in the sixth summand on the right hand side of~\eqref{eq:tildeB_def}.
Observe that for all 
$ x \in H $, $ t \in (0,T] $, $ s \in [0,t) $
it holds that
\begin{equation}
\begin{split}
  &\smallsum\limits_{ b_1 \in \mathbb{U} }
  \left[
  \smallsum\limits_{ b_2 \in \mathbb{U} }
    \left\|
    \left[
      ( B_I^{b_2} )''( e^{ A ( t - s ) } x ) 
      + 
      ( B^{b_2} )''( e^{ A ( t - s ) } x )
    \right]
    \!
    \left(
      e^{ A ( t - s ) } B^{b_1}_I( x )
      ,
      e^{ A ( t - s ) } B^{b_1}_I( x )
    \right)
  \right\|^2_{H_{-\nicefrac{\vartheta}{2}}}
  \right]^{ 1 / 2 }
  \\ & =
  \smallsum\limits_{ b_1 \in \mathbb{U} }
  \left\|
    \left[
      ( B_I )''( e^{ A ( t - s ) } x ) 
      + 
      B''( e^{ A ( t - s ) } x )
    \right]
    \!
    \left(
      e^{ A ( t - s ) } B^{b_1}_I( x )
      ,
      e^{ A ( t - s ) } B^{b_1}_I( x )
    \right)
  \right\|_{HS(U,H_{-\nicefrac{\vartheta}{2}})}
  \\ & \leq
  \smallsum\limits_{ b_1 \in \mathbb{U} }
  \left\|
      B''( e^{ A ( t - s ) } P_I( x ) ) 
    \left(
      P_I e^{ A ( t - s ) } B^{b_1}_I( x )
      ,
      P_I e^{ A ( t - s ) } B^{b_1}_I( x )
    \right)
  \right\|_{HS(U,H_{-\nicefrac{\vartheta}{2}})}
  \\ & \quad +
  \smallsum\limits_{ b_1 \in \mathbb{U} }
  \left\|
      B''( e^{ A ( t - s ) } x )
    \left(
      e^{ A ( t - s ) } B^{b_1}_I( x )
      ,
      e^{ A ( t - s ) } B^{b_1}_I( x )
    \right)
  \right\|_{HS(U,H_{-\nicefrac{\vartheta}{2}})}
\\ & \leq
  \big\|
    B''( e^{ A ( t - s ) } P_I( x ) ) 
  \big\|_{
    L^{(2)}( H, HS( U, H_{ -\nicefrac{\vartheta}{2} } ) )
  }
  \smallsum\limits_{ b_1 \in \mathbb{U} }
  \big\|
    P_I e^{ A ( t - s ) } B^{ b_1 }_I( x )
  \big\|^2_H
\\ & \quad +
  \big\|
    B''( e^{ A ( t - s ) } x ) 
  \big\|_{
    L^{(2)}( H, HS( U, H_{ -\nicefrac{\vartheta}{2} } ) )
  }
  \smallsum\limits_{ b_1 \in \mathbb{U} }
  \big\|
    e^{ A ( t - s ) } 
    B^{ b_1 }_I( x )
  \big\|^2_H
\\ & \leq
  \left\|
    B
  \right\|_{
    C_b^2( H, HS( U, H_{ -\nicefrac{\vartheta}{2} } ) )
  }
  \left[
    \big\|
      P_I e^{ A ( t - s ) } B_I( x )
    \big\|^2_{ HS( U, H ) }
    +
    \big\|
      e^{ A ( t - s ) } B_I( x )
    \big\|^2_{ HS( U, H ) }
  \right]
\\ & \leq
  \frac{
    2
    \left\|
      B
    \right\|_{
      C_b^2( H, HS( U, H_{ -\nicefrac{\vartheta}{2} } ) )
    }
    \left\|
      B
    \right\|_{
      C_b^1( H, HS( U, H_{-\nicefrac{\vartheta}{2}} ) )
    }^2
    g_2(x) 
  }{
    \left( t - s \right)^{ \vartheta }
  }
  ,
\end{split}
\end{equation}
and
\begin{equation}
\label{eq:tildeB_put_end}
\begin{split}
  &\smallsum\limits_{ b_1 \in \mathbb{U} }
  \left[
  \smallsum\limits_{ b_2 \in \mathbb{U} }
    \left\|
    \left[
      ( B_I^{b_2} )''( e^{ A ( t - s ) } x ) 
      - 
      ( B^{b_2} )''( e^{ A ( t - s ) } x )
    \right]
    \!
    \left(
      e^{ A ( t - s ) } B^{b_1}_I( x )
      ,
      e^{ A ( t - s ) } B^{b_1}_I( x )
    \right)
  \right\|^2_{H_{-\nicefrac{\vartheta}{2}}}
  \right]^{ 1 / 2 }
  \\ & =
  \smallsum\limits_{ b_1 \in \mathbb{U} }
  \left\|
    \left[
      ( B_I )''( e^{ A ( t - s ) } x ) 
      - 
      B''( e^{ A ( t - s ) } x )
    \right]
    \!
    \left(
      e^{ A ( t - s ) } B^{b_1}_I( x )
      ,
      e^{ A ( t - s ) } B^{b_1}_I( x )
    \right)
  \right\|_{HS(U,H_{-\nicefrac{\vartheta}{2}})}
  \\ & \leq
  \smallsum\limits_{ b_1 \in \mathbb{U} }
  \left\|
    \left[
      B''( e^{ A ( t - s ) } P_I( x ) ) 
      - 
      B''( e^{ A ( t - s ) } x )
    \right]
    \!
    \left(
      e^{ A ( t - s ) } B^{b_1}_I( x )
      ,
      e^{ A ( t - s ) } B^{b_1}_I( x )
    \right)
  \right\|_{HS(U,H_{-\nicefrac{\vartheta}{2}})}
  \\ & +
  \smallsum\limits_{ b_1 \in \mathbb{U} }
  \left\|
    B''( e^{ A ( t - s ) } P_I(x) )
    \left(
      ( \operatorname{Id}_H - P_I ) 
      e^{ A ( t - s ) } B^{b_1}_I( x ) 
      ,
      ( \operatorname{Id}_H + P_I ) 
      e^{ A ( t - s ) } B^{b_1}_I( x ) 
    \right)
  \right\|_{HS(U,H_{-\nicefrac{\vartheta}{2}})}
  \\ & \leq
  \big\|
      B''( e^{ A ( t - s ) } P_I( x ) ) 
      -
      B''( e^{ A ( t - s ) } x )
  \big\|_{
    L^{ (2) }( 
      H, 
      HS( 
        U, 
        H_{ - \nicefrac{ \vartheta }{ 2 } } 
      ) 
    )
  }
  \smallsum\limits_{ b_1 \in \mathbb{U} }
  \big\|
    e^{ A ( t - s ) } 
    B^{ b_1 }_I( x )
  \big\|^2_H
\\ & +
  \big\|
    B''( e^{ A ( t - s ) } P_I(x) )
  \big\|_{
    L^{ (2) }(
      H,
      HS( 
        U, 
        H_{
          - \nicefrac{ \vartheta }{ 2 } 
        }
      )
    )
  }
\\ & \cdot
  \smallsum\limits_{ b_1 \in \mathbb{U} }
  \Big[
  \big\|
    ( \operatorname{Id}_H - P_I ) 
    \,
    e^{ A ( t - s ) } B^{b_1}_I( x ) 
  \big\|_H
  \,
  \big\|
    ( \operatorname{Id}_H + P_I )
    \,
    e^{ A ( t - s ) } B^{ b_1 }_I( x )
  \big\|_H
  \Big]
  \\ & \leq
  \big\|
      B''( e^{ A ( t - s ) } P_I( x ) ) 
      -
      B''( e^{ A ( t - s ) } x )
  \big\|_{
    L^{ (2) }( 
      H,
      HS( U, H_{ - \nicefrac{ \vartheta }{ 2 } } )
    )
  }
  \,
  \big\|
    e^{ A ( t - s ) } B_I( x )
  \big\|^2_{
    HS( U, H) 
  }
\\ & +
  \big\|
    B''( e^{ A ( t - s ) } P_I(x) )
  \big\|_{
    L^{(2)}(H,HS(U,H_{-\nicefrac{\vartheta}{2}}))
  }
\\ & \cdot
  \big\|
    ( \operatorname{Id}_H - P_I ) 
    \,
    e^{ A ( t - s ) } B_I( x ) 
  \big\|_{HS(U,H)}
  \,
  \big\|
    ( \operatorname{Id}_H + P_I )
    \,
    e^{ A ( t - s ) } B_I( x )
  \big\|_{HS(U,H)}
\\ & \leq
\frac{
    \left|
      B
    \right|_{
      C^3_b( H, HS( U, H_{ -\nicefrac{\vartheta}{2} } ) )
    }
        \left\|
          B
        \right\|_{
          C_b^1( H, HS( U, H_{-\nicefrac{\vartheta}{2}} ) )
        }^2
    \left\|
      P_{\set\backslash I}
    \right\|_{ L( H, H_{-\rho} ) }
    g_3(x)
}{
  \left( t - s \right)^{ ( \rho + \vartheta ) }
}
\\ & +
\frac{
  2
    \left\|
      B
    \right\|_{
      C_b^2( H, HS( U, H_{ -\nicefrac{\vartheta}{2} } ) )
    }
        \left\|
          B
        \right\|_{
          C_b^1( H, HS( U, H_{-\nicefrac{\vartheta}{2}} ) )
        }^2
    \left\|
      P_{\set\backslash I}
    \right\|_{ L( H, H_{-\rho} ) }
  g_2(x)
}{
  \left( t - s \right)^{ ( \rho + \vartheta ) }
}
\\ & \leq
\frac{
  2
    \left\|
      B
    \right\|_{
      C_b^3( H, HS( U, H_{ -\nicefrac{\vartheta}{2} } ) )
    }
        \left\|
          B
        \right\|_{
          C_b^1( H, HS( U, H_{-\nicefrac{\vartheta}{2}} ) )
        }^2
    \left\|
      P_{\set\backslash I}
    \right\|_{ L( H, H_{-\rho} ) }
  g_3(x)
}{
  \left( t - s \right)^{ ( \rho + \vartheta ) }
}
  .
\end{split}
\end{equation}
Putting~\eqref{eq:messy.I} and~\eqref{eq:messy.II}--\eqref{eq:tildeB_put_end}
into~\eqref{eq:tildeB_def}
shows that for all
$ t \in (0,T) $,
$ s \in [0,t) $,
$ x \in H $
it holds that
\begin{equation}
\label{eq:tilde.B.many.sum}
\begin{split}
&
  \big|
    \tilde{B}_{ t, s }(x)
  \big| 
\\&\leq\Big[ 2c_{-\nicefrac{\vartheta}{2},-\nicefrac{\vartheta}{2}} \left\| F \right\|_{C_b^1(H,H_{-\vartheta})} \left| B \right|_{C^1_b(H,HS(U,H_{-\nicefrac{\vartheta}{2}}))} \left\| B \right\|_{C_b^1(H,HS(U,H_{-\nicefrac{\vartheta}{2}}))} g_2(x) \\
&+ 2c_{-\nicefrac{\vartheta}{2},-\nicefrac{\vartheta}{2}} \left\| F \right\|_{C_b^1(H,H_{-\vartheta})} \left\| B \right\|_{C_b^1(H,HS(U,H_{-\nicefrac{\vartheta}{2}}))} \left\| B \right\|_{C_b^2(H,HS(U,H_{-\nicefrac{\vartheta}{2}}))} g_3(x)\\
&+ 2c_{-\nicefrac{\vartheta}{2},-\nicefrac{\vartheta}{2},0} \left\| F \right\|_{C_b^1(H,H_{-\vartheta})} \left| B \right|_{C^1_b(H,HS(U,H_{-\nicefrac{\vartheta}{2}}))} \left\| B \right\|_{C_b^1(H,HS(U,H_{-\nicefrac{\vartheta}{2}}))} g_3(x) \\
&+ 2c_{-\nicefrac{\vartheta}{2},-\nicefrac{\vartheta}{2}} \left\| B \right\|_{C_b^1(H,HS(U,H_{{-\nicefrac{\vartheta}{2}}}))}^3 \left\| B \right\|_{C_b^2(H,HS(U,H_{-\nicefrac{\vartheta}{2}}))} g_3(x)\\
&+ c_{-\nicefrac{\vartheta}{2},-\nicefrac{\vartheta}{2}} \left| B \right|_{C^1_b(H,HS(U,H_{-\nicefrac{\vartheta}{2}}))} \left\| B \right\|_{C_b^1(H,HS(U,H_{{-\nicefrac{\vartheta}{2}}}))}^2 \left\| B \right\|_{C^2_b(H,HS(U,H_{-\nicefrac{\vartheta}{2}}))} g_3(x) \\
&+ 2c_{-\nicefrac{\vartheta}{2},-\nicefrac{\vartheta}{2}} \left\| B \right\|^3_{C_b^1(H,HS(U,H_{-\nicefrac{\vartheta}{2}}))} \left\| B \right\|_{C_b^3(H,HS(U,H_{-\nicefrac{\vartheta}{2}}))} g_4(x)\\
&+ 2c_{-\nicefrac{\vartheta}{2},-\nicefrac{\vartheta}{2},0} \left| B \right|_{C^1_b(H,HS(U,H_{-\nicefrac{\vartheta}{2}}))} \left\| B \right\|^3_{C_b^1(H,HS(U,H_{-\nicefrac{\vartheta}{2}}))} g_3(x) \\
&+ 2c_{-\nicefrac{\vartheta}{2},-\nicefrac{\vartheta}{2},0} \left\| B \right\|_{C_b^1(H,HS(U,H_{{-\nicefrac{\vartheta}{2}}}))}^3 \left\| B \right\|_{C_b^2(H,HS(U,H_{-\nicefrac{\vartheta}{2}}))} g_4(x) \\
&+ c_{-\nicefrac{\vartheta}{2},-\nicefrac{\vartheta}{2},0,0} \left| B \right|_{C^1_b(H,HS(U,H_{-\nicefrac{\vartheta}{2}}))} \left\| B \right\|^3_{C_b^1(H,HS(U,H_{-\nicefrac{\vartheta}{2}}))} g_4(x) \Big]
  \frac{\left\| P_{\set\backslash I} \right\|_{L(H,H_{-\rho})}}{{( T - t )^{ \vartheta }} \, ( t - s )^{ ( \rho + \vartheta ) }}
.
\end{split}
\end{equation}
This implies that for all 
$ t \in (0,T) $,
$ s \in [0,t) $,
$ x \in H $
it holds that
\begin{equation}
\label{eq:tilde.B.many.sum.II}
\begin{split}
&
  \big|
    \tilde{B}_{ t, s }(x)
  \big|\leq
  \frac{
    9
    \left[ 
      c_{ -\nicefrac{\vartheta}{2}, -\nicefrac{\vartheta}{2} }
      +
      c_{ -\nicefrac{\vartheta}{2}, -\nicefrac{\vartheta}{2}, 0 }
      +
      c_{ -\nicefrac{\vartheta}{2}, -\nicefrac{\vartheta}{2}, 0, 0 }
    \right]
    \varsigma_{ F, B }
    \,
    g_4(x)
    \left\|
      P_{ \set \backslash I }
    \right\|_{ L(H,H_{-\rho}) }
  }{{( T - t )^{ \vartheta }}
    \left( t - s \right)^{ 
      \rho + \vartheta
    }
  }
  .
\end{split}
\end{equation}
This proves that
\begin{equation}
\label{eq:B_estimate}
\begin{split}
&  \frac{ 1 }{ 2 }
  \left|
  \int_0^T
  \int_0^t
  \E\left[ 
    \tilde{B}_{ t, s }( Y^I_s )
  \right]
  ds
  \,
  dt
  \right|
\\& \leq
  \frac{
    9 \,
    T^{ ( 2 - \rho - 2 \vartheta ) }
    \,
    \varsigma_{ F, B }
    \left\|
      P_{ \set \backslash I }
    \right\|_{ L(H,H_{-\rho}) }
  }{
    2\left(
      1 - \rho - \vartheta
    \right)
    {\left(
      1 - \vartheta
    \right)}
  }
    \,
    \big[
      c_{ -\nicefrac{\vartheta}{2}, -\nicefrac{\vartheta}{2} }
      +
      c_{ -\nicefrac{\vartheta}{2}, -\nicefrac{\vartheta}{2}, 0 }
      +
      c_{ -\nicefrac{\vartheta}{2}, -\nicefrac{\vartheta}{2}, 0, 0 }
    \big]
    \,
    K^I_4
  .
\end{split} 
\end{equation}
Putting~\eqref{eq:estimate_first_term}, \eqref{eq:inequality_second_term}, \eqref{eq:F_estimate}, and~\eqref{eq:B_estimate} 
into~\eqref{eq:Kolmogorov2} finally yields
\begin{equation}
\begin{split}
&
  \left|
    \ES\big[ 
      \varphi( X^\set_T )
    \big]
    -
    \ES\big[ 
      \varphi( Y_T^I )
    \big]
  \right|
\\&\leq
  \big[ 
    c_{-\vartheta} + c_{-\vartheta,0} + c_{-\vartheta,0,0} + c_{-\nicefrac{\vartheta}{2},-\nicefrac{\vartheta}{2}} + c_{-\nicefrac{\vartheta}{2},-\nicefrac{\vartheta}{2},0} + c_{-\nicefrac{\vartheta}{2},-\nicefrac{\vartheta}{2},0,0} 
  \big]
\\ & \quad
  \cdot 
  \frac{ 
    T^{ ( 1 - \vartheta - \rho ) }
  }{
    ( 1 - \vartheta - \rho ) 
  }
  \left[ 
    1 +
    \frac{ 
      9 \, T^{ (1 - \vartheta) } 
    }{
      2
      {(1-\vartheta)}
    } 
  \right]
  \varsigma_{ F, B }\, K^I_4 \left\|
        P_{ \set \backslash I }
      \right\|_{ L(H,H_{-\rho}) }
  .
\end{split}
\end{equation}
This establishes the first inequality of~\eqref{eq:weak_regle}.
The second inequality of~\eqref{eq:weak_regle} follows from Lemma~\ref{lem:c.delta.finite} and the assumption that $\vartheta<\nicefrac{1}{2}$.
The proof of Lemma~\ref{weak_regle} is thus completed.
\end{proof}

\noindent The next result, Corollary~\ref{cor:weak_reg},
is an immediate consequence of Lemma~\ref{weak_regle0}
and Lemma~\ref{weak_regle} above.

\begin{corollary}\label{cor:weak_reg}
Assume the setting in Section~\ref{sec:setting}
and let $ \rho \in [0, 1 - \vartheta ) $, $ I \in \grid(\set) $.
Then 
\begin{equation}
\label{eq:weak_reg_estimate}
\begin{split}
&
  \left|
    \ES\big[ 
      \varphi( X^\set_T )
    \big]
    -
    \ES\big[ 
      \varphi( X_T^I )
    \big]
  \right|
\leq
  \tfrac{ 9 }{ 2 \, T^{ \rho } }
  \left[
    1
    + 
    \tfrac{ 
      T^{ ( 1 - \vartheta ) } 
    }{
      \left( 1 - \vartheta - \rho \right)
    } 
  \right]^2
  \|
    P_{ \set \backslash I }
  \|_{ L( H, H_{ - \rho } ) }
    \,
    \varsigma_{ F, B }
    \, 
    K^I_4 
\\ &
\cdot
  \big[ 
    \| \varphi \|_{ C_b^3( H, \R ) }
    +
    c_{ - \vartheta } 
    + 
    c_{ - \vartheta, 0 } 
    + 
    c_{ - \vartheta,0,0 } 
    + 
    c_{ - \nicefrac{ \vartheta }{ 2 } , - \nicefrac{ \vartheta }{ 2 } } 
    + 
    c_{ - \nicefrac{ \vartheta }{ 2 } , - \nicefrac{ \vartheta }{ 2 } , 0 } 
    + 
    c_{ - \nicefrac{ \vartheta }{ 2 } , - \nicefrac{ \vartheta }{ 2 } , 0 , 0 } 
  \big]
  < \infty
  .
\end{split}
\end{equation}
\end{corollary}

\noindent In the proof of the next result,
Corollary~\ref{cor:weak_reg2} below, 
we employ an upper bound result for the real numbers 
$ K^I_4 $, $ I \in \mathcal{P}( \set ) $, 
to obtain a further upper bound for the real numbers
$
  \left|
    \ES\big[ 
      \varphi( X^\set_T )
    \big]
    -
    \ES\big[ 
      \varphi( X_T^I )
    \big]
  \right|
$, $ I \in \mathcal{P}( \set ) $.
For the formulation of Corollary~\ref{cor:weak_reg2} 
we recall that for all 
$ x \in [ 0, \infty ) $, 
$ \theta \in [ 0, 1 ) $ 
it holds that 
$
  \mathcal{E}_{ (1 - \theta) }( x )
  =
  \big[
    \sum^{ \infty }_{ n = 0 }
    \frac{ 
        x^{ 2 n }
        \,
        \Gamma( 1 - \theta )^n
    }{
      \Gamma( n ( 1 - \theta ) + 1 ) 
    }
  \big]^{
    \nicefrac{ 1 }{ 2 }
  }
$ 
(see Section~\ref{sec:notation}).

\begin{corollary}\label{cor:weak_reg2}
Assume the setting in Section~\ref{sec:setting}.
Then it holds for every 
$ \theta \in [0,1) $, 
$ \rho \in [ 0, 1 - \vartheta ) $, 
$ I \in \grid( \set ) $ 
that
\begin{equation}
\label{eq:weak_reg_estimate_K}
\begin{split}
&  \left|
    \ES\big[ 
      \varphi( X^\set_T )
    \big]
    -
    \ES\big[ 
      \varphi( X_T^I )
    \big]
  \right|
\leq
  \frac{ 
    18 
  }{ T^{ \rho } } 
  \left[ 
    1 +
    \tfrac{
      T^{ (1 - \vartheta) } 
    }{
      \left( 1 - \vartheta - \rho \right)
    } 
  \right]^2 
  \E\,\big[\!
    \max\{ 1 , \| \xi \|_H^4 \}
  \big]
  \left\|
    P_{ \set \backslash I }
  \right\|_{ 
    L( H , H_{ - \rho } ) 
  }
\\ & \cdot
  \,
  \varsigma_{ F, B }
  \left|
    \mathcal{E}_{ ( 1 - \theta ) 
    }\!\left[ 
      \tfrac{
        T^{ 1 - \theta }
        \sqrt{ 2 }
        \,
        \| F \|_{ 
          C_b^1( H, H_{ - \theta } ) 
        }
      }{
        \sqrt{
          1 - \theta 
        }
      } 
      + 
      \sqrt{
        12 \,
        T^{ 1 - \theta }
      }
      \,
      \| B \|_{
        C_b^1( H, HS( U, H_{ - \nicefrac{ \theta }{ 2 } } )
        )
      }
    \right]
  \right|^4
\\ & \cdot
  \big[
    \|
      \varphi
    \|_{
      C_b^3( H, \R )
    } 
    + 
    c_{ - \vartheta } 
    + 
    c_{ - \vartheta , 0 } 
    + 
    c_{ - \vartheta , 0 , 0 } 
    + 
    c_{ - \nicefrac{ \vartheta }{ 2 } , - \nicefrac{ \vartheta }{ 2 } } 
    + 
    c_{ - \nicefrac{ \vartheta }{ 2 } , - \nicefrac{ \vartheta }{ 2 } , 0 } 
    + 
    c_{ - \nicefrac{ \vartheta }{ 2 } , - \nicefrac{ \vartheta }{ 2 } , 0 , 0 } 
  \big] 
  < \infty
  .
\end{split}
\end{equation}
\end{corollary}
\begin{proof}
Note that, e.g., Proposition~3.4 in Cox et al.~\cite{CoxHutzenthalerJentzenWelti2017}
(with
$H=H$,
$U=U$,
$\mathbb{H}=\mathbb{H}$,
$\lambda=\lambda$,
$A=A$,
$T=T$,
$p=4$,
$\gamma=0$,
$\eta=\theta$,
$F=(H\ni v\mapsto F_I(v)\in H_{-\theta})$,
$B=(H\ni v \mapsto (U\ni u \mapsto B_I(v)u \in H_{-\nicefrac{\theta}{2}})\in HS(U,H_{-\nicefrac{\theta}{2}}))$,
$W=W$,
$X=([0,T]\times\Omega\ni(t,\omega)\mapsto Y^I_t(\omega) \in H)$ 
for 
$I\in\grid(\set)$,
$\theta\in[0,1)$
in the notation of Proposition~3.4 in~\cite{CoxHutzenthalerJentzenWelti2017})
ensures that for all 
$\theta\in[0,1)$, 
$I\in\grid(\set)$
it holds that 
\begin{equation}
\begin{split}
  |K^I_4|^{\nicefrac{1}{4}}
&=
  \sup_{t\in[0,T]}
  \|\!\max\{1,\|Y^I_t\|_H\}\|_{\lpn{4}{\P}{\R}}
  \leq
  \sqrt{2} \,
  \|\!\max\{1,\|\xi\|_H\}\|_{\lpn{4}{\P}{\R}}
\\&\quad\cdot
    \mathcal{E}_{ ( 1 - \theta ) 
    }\!\left[ 
      \tfrac{
        T^{ 1 - \theta }
        \sqrt{ 2 }
        \,
        \| F \|_{ 
          C_b^1( H, H_{ - \theta } ) 
        }
      }{
        \sqrt{
          1 - \theta 
        }
      } 
      + 
      \sqrt{
        12 \,
        T^{ 1 - \theta }
      }
      \,
      \| B \|_{
        C_b^1( H, HS( U, H_{ - \nicefrac{ \theta }{ 2 } } )
        )
      }
    \right]
    < \infty.
\end{split}
\end{equation}
This and~\eqref{eq:weak_reg_estimate} establish~\eqref{eq:weak_reg_estimate_K}.
The proof of Corollary~\ref{cor:weak_reg2} is thus completed.
\end{proof}

\noindent The next result, Corollary~\ref{cor:mollified.weak.rate.simplified} below, is an immediate consequence of Corollary~\ref{cor:weak_reg2} and the fact that 
$
  \forall \, I \in \grid(\set)\backslash\{\set\}, \, \rho\in[0,\infty)
  \colon
  \|P_{\set\backslash I}\|_{L(H,H_{-\rho})}
  =
    \left[
\inf_{
	b \in \set \backslash I
}
|\lambda_b|
\right]^{ 
	-\rho
}
$.

\begin{corollary}
\label{cor:mollified.weak.rate.simplified}
Assume the setting in Section~\ref{sec:setting}.
Then for every 
$ \varepsilon \in (0,\infty) $
there exists a real number $C\in\R$ 
such that for all 
$ I \in \grid( \set )\backslash\{\set\} $
it holds that
\begin{equation}
\begin{split}
&  \left|
    \ES\big[ 
      \varphi( X^\set_T )
    \big]
    -
    \ES\big[ 
      \varphi( X_T^I )
    \big]
  \right|
\leq
  C \cdot
    \left[
\inf_{
	b \in \set \backslash I
}
|\lambda_b|
\right]^{\varepsilon-(1-\vartheta)}
  .
\end{split}
\end{equation}
\end{corollary}


\section{Strong convergence of mollified solutions for SEEs}
\label{sec:strong_convergence}

In this section an elementary strong convergence result,
see Proposition~\ref{strong} below,
is established.
More specifically, in the framework of Section~\ref{sec:general_setting} we establish in Proposition~\ref{strong} below an explicit upper bound for the strong approximation error
\begin{equation}
\label{eq:strong.error}
  \ES\big[\| 
    X^0_T - X^{ \kappa }_T 
  \|^p_H\big]
\end{equation}
for $\kappa\in[0,\infty)$, $p\in[2,\infty)$ 
where $X^\kappa\colon[0,T]\times\Omega\to H$, $\kappa\in[0,\infty)$, are appropriate mild solution processes of~\eqref{eq:mollified}.
We have for every $\kappa\in(0,\infty)$ that $X^\kappa\colon[0,T]\times\Omega\to H$ is a suitably mollified version of $X^0\colon[0,T]\times\Omega\to H$ with the mollification parameter $\kappa\in(0,\infty)$.
We prove Proposition~\ref{strong} below by applying the strong perturbation estimate in Andersson et al.~\cite[Proposition~2.7]{AnderssonJentzenKurniawan2016arXiv}.
Proposition~2.7 in~\cite{AnderssonJentzenKurniawan2016arXiv}, in turn, is established by using a generalized Gronwall inequality (see Henry~\cite[Exercise~3 in Chapter~7]{h81}).
Proposition~\ref{strong} below, in particular, allows us to prove 
estimate~\eqref{eq:intro_estimate3} in the introduction.
In Section~\ref{sec:weak_irregular} below 
we will use Proposition~\ref{strong} 
in conjunction with Corollary~\ref{cor:weak_reg2} in Section~\ref{sec:weak_reg}
to establish weak convergence rates for 
Galerkin approximations of SEEs.
In Corollary~\ref{cor:strong.convergence.simplified} we further simplify the explicit bound obtained in Proposition~\ref{strong} below.

\subsection{Setting}
\label{secstrong:setting}

Assume the setting in Section~\ref{sec:general_setting}
and let 
$ p \in [2,\infty) $, $ \vartheta \in [0,1) $,
$
  F \in 
  \operatorname{Lip}^0( H , H_{ - \vartheta } ) 
$, 
$
  B \in 
  \operatorname{Lip}^0( H , HS( U , H_{ - \nicefrac{ \vartheta }{ 2 } } ) )
$, 
$
  \xi \in \lpn{p}{ \P|_{ \mathcal{F}_0 } }{H}
$.

The above assumptions ensure 
that there exist up-to-modifications unique 
$ ( \mathcal{F}_t )_{ t \in [0,T] }$-predictable stochastic processes
$ 
  X^{ \kappa } \colon [0,T] \times \Omega \to H
$,
$ \kappa \in [0,\infty) $,
which satisfy that
for all $ \kappa \in [0,\infty) $
it holds that
$
  \sup_{ t \in [0,T] }
  \ES\big[
  \| X^{ \kappa }_t \|^p_H
  \big]
  < \infty
$ 
and which satisfy that
for all $ t \in [0,T] $, $ \kappa \in [0,\infty) $
it holds $ \P $-a.s.\ that
\begin{equation}
\label{eq:mollified}
  X_t^{ \kappa }
  = 
    e^{ A t } \xi
  + 
    \int_0^t e^{ A ( \kappa + t - s ) } F( X^{ \kappa }_s ) \, ds
  + 
    \int_0^t e^{ A ( \kappa + t - s ) } B( X^{ \kappa }_s ) \, dW_s 
    .
\end{equation}

\subsection{A strong convergence result}
\label{secstrong:strong}

\begin{proposition}
\label{strong}
Assume the setting in Section~\ref{secstrong:setting} 
and let $ \kappa \in [0,\infty) $, $ \rho \in [ 0, \frac{1 - \vartheta}{2} ) $.
Then
\begin{equation}
\label{eq:strong.bound.mol}
\begin{split}
&  
  \left\| 
    X^0_T - X^{ \kappa }_T 
  \right\|_{
    \lpn{p}{\P}{H}
  }
\leq
  \|\!\max\{1,\|\xi\|_H\}\|_{\lpn{p}{\P}{\R}}
\\ & \quad 
  \cdot 
  2 
  \, 
  \kappa^\rho
  \left|
    \mathcal{E}_{ (1 - \vartheta) }\!\left[
      \tfrac{ 
        T^{ (1 - \vartheta) }  
        \sqrt{2} 
        \,
        \| F \|_{ 
          \operatorname{Lip}^0( H, H_{ - \vartheta } ) 
        }
      }{
        \sqrt{ 1 - \vartheta } 
      }
      +
      \sqrt{
        T^{ 
          (1 - \vartheta) 
        }
        \,
        p 
        \, (p-1)
      }
      \,
      \| B \|_{
        \operatorname{Lip}^0(
          H,
          HS( U, H_{ - \nicefrac{ \vartheta }{ 2 } } )
        )
      }
    \right]\right|^2
\\ & \quad 
  \cdot 
  \left[
    \tfrac{
      T^{ ( 1 - \rho - \vartheta ) } 
    }{
      ( 1 - \rho - \vartheta ) 
    }
    \| F \|_{ 
      \operatorname{Lip}^0( H, H_{ - \vartheta } ) 
    } 
    + 
    \tfrac{
      \sqrt{ p \, ( p - 1 ) \, T^{ ( 1 - 2 \rho - \vartheta ) } } 
    }{
      \sqrt{ 2 - 4 \rho - 2 \vartheta } 
    }
    \| B 
    \|_{
      \operatorname{Lip}^0( 
        H, HS( U, H_{ - \nicefrac{ \vartheta }{ 2 } } )
      )
    }
  \right]
  .
\end{split}
\end{equation}
\end{proposition}

\begin{proof}
%
%
%
First of all, observe that 
Proposition~2.7 in~\cite{AnderssonJentzenKurniawan2016arXiv} 
(with
$
H=H
$, 
$
U=U
$, 
$ T = T $, 
$ \eta = 0 $, 
$ p = p $, 
$ \alpha = \vartheta $, 
$
\hat{\alpha} = 0
$,
$ \beta = \nicefrac{\vartheta}{2} $, 
$
\hat{\beta} 
= 
0
$,
$L_0=|F|_{\operatorname{Lip}^0(H,H_{-\vartheta})}$, 
$
\hat{L}_0
=
\|F(0)\|_{H_{-\vartheta}}
$, 
$L_1=|B|_{\operatorname{Lip}^0(H,HS(U,H_{-\nicefrac{\vartheta}{2}}))}+\varepsilon$, 
$
\hat{L}_1
=
\|B(0)\|_{HS(U,H_{-\nicefrac{\vartheta}{2}})}
$, 
$
W=W
$, 
$ A=A $,
$
\mathbf{F} = 
\big(
[0,T]\times\Omega\times H \ni (t,\omega,v) \mapsto
F(v) \in H_{-\vartheta}
\big)
$, 
$
\mathbf{B} = 
\big(
[0,T]\times\Omega\times H \ni (t,\omega,v) \mapsto
(U\ni u \mapsto B(v)u \in H_{-\nicefrac{\vartheta}{2}}) \in HS(U,H_{-\nicefrac{\vartheta}{2}})
\big)
$,  
$ \delta = 0 $, 
$ Y^1 = X^0 $, $ Y^2 = X^\kappa $,
$ \lambda = 0 $  
for 
$\varepsilon\in(0,\infty)$
in the notation of Proposition~2.7 in~\cite{AnderssonJentzenKurniawan2016arXiv})
shows that for all $\varepsilon\in(0,\infty)$ it holds that
\begin{equation}
\begin{split}
& 
\label{eq:strong_1}
  \left\| 
    X_T^0 - X^{ \kappa }_T
  \right\|_{
    \lpn{p}{\P}{H}
  }
\\ & \leq
  \sqrt{ 2 } \cdot
  \mathcal{E}_{ (1 - \vartheta) 
  }\!\Bigg[
    \tfrac{
      T^{ (1 - \vartheta) }
      \sqrt{ 2 } 
      \, 
      | F |_{
        \operatorname{Lip}^0( H, H_{ - \vartheta } ) 
      }
    }{
      \sqrt{ 1 - \vartheta } 
    }
    +
    \sqrt{
      T^{ (1 - \vartheta) }
      \,
      p \left( p - 1 \right)
    }
    \,
    \big(
    | B |_{
      \operatorname{Lip}^0( 
        H, 
        HS( U, H_{ - \nicefrac{ \vartheta }{ 2 } } ) 
      ) 
    }
	+\varepsilon
	\big)
  \Bigg]
\\ & \cdot 
    \sup_{
      t \in [0,T]
    }
    \left\|
      \int^t_0
      e^{ A (t - s) }
      \left(  
        \operatorname{Id}_H - e^{ A \kappa } 
      \right)
      F(X^{ \kappa }_s)
      \,
      ds
    + 
    \int^t_0e^{A(t-s)}
    \left(
      \operatorname{Id}_H - 
      e^{ A \kappa }
    \right)
    B(X^{ \kappa }_s) \, dW_s
  \right\|_{
    \lpn{p}{\P}{H}
  }
  .
\end{split}
\end{equation}
In the next step
we observe that the Burkholder-Davis-Gundy type inequality 
in Lemma~7.7 in Da Prato \& Zabczyk~\cite{dz92}, Lemma~\ref{lem:aux}, and the fact that for all 
$
  r \in [0,1]
$
it holds that 
$
  \sup_{ t \in (0,\infty) }
  \|
    ( - t A )^{ - r }
    ( 
      \operatorname{Id}_H - e^{ A t } 
    )
  \|_{ L(H) }
  \leq 1
$
imply that for all 
$ 
  t \in [0,T] 
$, 
$
  r \in [ 0, 1 - \vartheta ) 
$ 
it holds that
\begin{equation}
\begin{split}
&
  \left\|
    \int^t_0
    e^{ A (t-s) }
    \left(
      \operatorname{Id}_H - e^{ A \kappa } 
    \right)
    F( X^{ \kappa }_s ) 
    \, ds
  \right\|_{\lpn{p}{\P}{H}}
\\ & \leq 
  \frac{
    T^{ (1 - r - \vartheta) } 
  }{
    \left( 1 - r - \vartheta \right)
  }
  \left[
    \sup_{ s \in [0,T] }
    \left\| 
      F( X^{ \kappa }_s ) 
    \right\|_{ \lpn{p}{\P}{H_{-\vartheta}} }
  \right]
  \kappa^r
\\ & \leq
  \frac{
    T^{ (1 - r - \vartheta) } 
  }{
    \left( 1 - r - \vartheta \right)
  }
  \left\|
    F
  \right\|_{ \operatorname{Lip}^0( H, H_{ -\vartheta } ) }
  \max\!
  \left\{
  1,
  \sup_{ s \in [0,T] }
      \| X^{ \kappa }_s \|_{ \lpn{p}{\P}{H} }
  \right\}
  \kappa^r
\end{split}\label{eq:strong_2}
\end{equation}
and 
that for all 
$ 
  t \in [0,T] 
$, 
$
  r \in [ 0, \frac{ 1 - \vartheta }{ 2 } ) 
$ 
it holds that
\begin{equation}
\begin{split}
&
  \left\|
    \int^t_0
    e^{ A (t-s) }
    \left(
      \operatorname{Id}_H
      -
      e^{ A \kappa } 
    \right)
    B(X^{ \kappa }_s) \, dW_s
  \right\|_{
    \lpn{p}{\P}{H}
  }
\\ & \leq
  \sqrt{
    \frac{ p \, ( p - 1 ) }{ 2 }
  }
  \frac{
    \sqrt{ T^{ (1 - 2 r - \vartheta) } } 
  }{
    \sqrt{ 1 - 2 r - \vartheta } 
  }
  \left[
    \sup_{ s \in [0,T] }
    \left\| 
      B( X^{ \kappa }_s )
    \right\|_{ 
      \lpn{p}{\P}{ 
        HS( 
          U,
          H_{ - \nicefrac{ \vartheta }{ 2 } } 
        ) 
      } 
    }
  \right]
  \kappa^r
\\ & \leq
  \sqrt{
    \frac{ p \, ( p - 1 ) }{ 2 }
  }
  \frac{ 
    \sqrt{ T^{ (1 - 2 r - \vartheta) } }
  }{
    \sqrt{ 1 - 2 r - \vartheta } 
  }
  \left\|
    B
  \right\|_{ 
    \operatorname{Lip}^0( 
      H, 
      HS( 
        U,
        H_{ - \nicefrac{ \vartheta }{ 2 } } 
      ) 
    ) 
  }
  \max\!
  \left\{
  1,
  \sup_{ s \in [0,T] }
      \| X^{ \kappa }_s \|_{ \lpn{p}{\P}{H} }
  \right\}
  \kappa^r
  .
\end{split}\label{eq:strong_3}
\end{equation}
Putting~\eqref{eq:strong_2} and~\eqref{eq:strong_3} 
into~\eqref{eq:strong_1} yields that for all 
$
  r \in [ 0, \frac{ 1 - \vartheta }{ 2 } )
$ 
it holds that
\begin{equation}
\begin{split}
&
  \left\| 
    X_T^0 - X^{ \kappa }_T
  \right\|_{
    \lpn{p}{\P}{H}
  }
\leq 
  \sqrt{2}
  \,
  \kappa^r
  \max\!
  \left\{
    1,
    \sup_{ t \in [0,T] }
    \| X^{ \kappa }_t \|_{ 
      \lpn{p}{\P}{H} 
    }
  \right\}
\\ & \cdot
  \bigg(
  \inf_{\varepsilon\in(0,\infty)}
  \mathcal{E}_{(1-\vartheta)}\!\left[
    \tfrac{ 
      T^{ (1 - \vartheta) }
      \sqrt{2} 
      \,
      | F |_{
        \operatorname{Lip}^0( H, H_{ - \vartheta } ) 
      }
    }{
      \sqrt{ 1 - \vartheta } 
    }
    +
    \sqrt{
      T^{ (1 - \vartheta) } \, p \, ( p - 1 )
    }
    \,
    | B |_{
      \operatorname{Lip}^0(
        H,
        HS( U, H_{ - \nicefrac{ \vartheta }{ 2 } } )
      )
    }
+\varepsilon
  \right]
  \bigg)
\\
&\cdot
  \left[
    \tfrac{
      T^{ ( 1 - r - \vartheta ) } 
    }{
      \left( 1 - r - \vartheta \right)
    }
    \| F \|_{
      \operatorname{Lip}^0( H, H_{ - \vartheta } ) 
    }
    +
    \tfrac{
      \sqrt{ p \, \left( p - 1 \right) \, T^{ (1 - 2 r - \vartheta) } }
    }{
      \sqrt{ 2 - 4 r - 2 \vartheta }
    }
    \| B \|_{
      \operatorname{Lip}^0( 
        H, 
        HS(U, H_{ - \nicefrac{ \vartheta }{ 2 } } 
        )
      )
    }
  \right]
  \\&\leq
  \sqrt{2}
\,
\kappa^r
\bigg[
\sup_{ t \in [0,T] }
\|\! \max\{1,\|X^{ \kappa }_t\|_H\} \|_{ 
	\lpn{p}{\P}{\R} 
}
\bigg]
\\ & \cdot
\mathcal{E}_{(1-\vartheta)}\!\left[
\tfrac{ 
	T^{ (1 - \vartheta) }
	\sqrt{2} 
	\,
	| F |_{
		\operatorname{Lip}^0( H, H_{ - \vartheta } ) 
	}
}{
	\sqrt{ 1 - \vartheta } 
}
+
\sqrt{
	T^{ (1 - \vartheta) } \, p \, ( p - 1 )
}
\,
| B |_{
	\operatorname{Lip}^0(
	H,
	HS( U, H_{ - \nicefrac{ \vartheta }{ 2 } } )
	)
}
\right]
\\
&\cdot
\left[
\tfrac{
	T^{ ( 1 - r - \vartheta ) } 
}{
	\left( 1 - r - \vartheta \right)
}
\| F \|_{
	\operatorname{Lip}^0( H, H_{ - \vartheta } ) 
}
+
\tfrac{
	\sqrt{ p \, \left( p - 1 \right) \, T^{ (1 - 2 r - \vartheta) } }
}{
	\sqrt{ 2 - 4 r - 2 \vartheta }
}
\| B \|_{
	\operatorname{Lip}^0( 
	H, 
	HS(U, H_{ - \nicefrac{ \vartheta }{ 2 } } 
	)
	)
}
\right]  
  .
\end{split}\label{eq:strong_4}
\end{equation}
In addition, e.g., Proposition~3.4 in Cox et al.~\cite{CoxHutzenthalerJentzenWelti2017}
(with
$H=H$,
$U=U$,
$\mathbb{H}=\mathbb{H}$,
$\lambda=\lambda$,
$A=A$,
$T=T$,
$p=p$,
$\gamma=0$,
$\eta=\vartheta$,
$F=(H\ni v \mapsto e^{A\kappa} F(v) \in H_{-\vartheta})$,
$B=(H\ni v \mapsto ( U \ni u \mapsto e^{A\kappa} B(v)u \in H_{-\nicefrac{\vartheta}{2}} ) \in HS(U,H_{-\nicefrac{\vartheta}{2}}))$,
$W=W$,
$X=X^\kappa$ 
in the notation of Proposition~3.4 in~\cite{CoxHutzenthalerJentzenWelti2017})
shows that 
\begin{equation}
\begin{split}
&
  \sup_{ t \in [0,T] }
  \|\! \max\{1,\|X^{ \kappa }_t\|_H\} \|_{ 
	\lpn{p}{\P}{\R} 
  }
  \leq
  \sqrt{2} \,
   \|\!\max\{1,\|\xi\|_H\}\|_{\lpn{p}{\P}{\R}}
  \\ &
  \cdot 
  \mathcal{E}_{ (1 - \vartheta) }\!\left[
  \tfrac{ 
  	T^{ (1 - \vartheta) }  
  	\sqrt{2} 
  	\,
  	\| F \|_{ 
  		\operatorname{Lip}^0( H, H_{ - \vartheta } ) 
  	}
  }{
  	\sqrt{ 1 - \vartheta } 
  }
  +
  \sqrt{
  	T^{ 
  		(1 - \vartheta) 
  	}
  	\,
  	p 
  	\, (p-1)
  }
  \,
  \| B \|_{
  	\operatorname{Lip}^0(
  	H,
  	HS( U, H_{ - \nicefrac{ \vartheta }{ 2 } } )
  	)
  }
  \right]
  .
\end{split}
\end{equation} 
Combining this with~\eqref{eq:strong_4} establishes~\eqref{eq:strong.bound.mol}. The proof of Proposition \ref{strong} is thus completed.
\end{proof}

\noindent The next result, Corollary~\ref{cor:strong.convergence.simplified} below, is an immediate consequence of Proposition~\ref{strong} above.

\begin{corollary}
\label{cor:strong.convergence.simplified}
Assume the setting in Section~\ref{secstrong:setting}. 
Then for every $\varepsilon\in(0,\frac{1 - \vartheta}{2}]$ 
there exists a real number $C\in\R$
such that for all 
$ \kappa \in [0,\infty) $
it holds that 
\begin{equation}
\begin{split}
&  
  \left\| 
    X^0_T - X^{ \kappa }_T 
  \right\|_{
    \lpn{p}{\P}{H}
  }
\leq
  C
  \cdot 
  \kappa^{(\frac{1 - \vartheta}{2}-\varepsilon)}
  .
\end{split}
\end{equation}
\end{corollary}

\section{Weak convergence for Galerkin approximations of SEEs}
\label{sec:weak_irregular}

In this section the weak convergence results in Proposition~\ref{weak_irreg}, Corollary~\ref{weak_irreg_B}, and Corollary~\ref{cor:weak.convergence.rate.simplified} are proved.
Roughly speaking, in the framework of Section~\ref{sec:general_setting} we establish in Proposition~\ref{weak_irreg} below a rate for the weak approximation error
\begin{equation}
\label{eq:weak.error}
  \left|
    \ES\big[ 
      \varphi( X^\set_T )
    \big]
    -
    \ES\big[ 
      \varphi( X_T^I )
    \big]
  \right|
  ,
\end{equation}
where $I\subseteq\mathbb{H}$ is a set,
where $\varphi\colon H\to\R$ is a four times continuously Fr\'{e}chet differentiable function with globally bounded derivatives, 
and where 
$X^{\mathbb{H}}\colon[0,T]\times\Omega\to H$ 
and
$X^I\colon[0,T]\times\Omega\to P_I(H)$
are appropriate mild solution processes of the SEEs in~\eqref{irregeq:Galerkin}.
Here, $X^I\colon[0,T]\times\Omega\to P_I(H)$ is a spectral Galerkin approximation of $X^{\mathbb{H}}\colon[0,T]\times\Omega\to H$.
We note that the drift nonlinearity $F$ and the diffusion nonlinearity $B$ of the SEE~\eqref{irregeq:Galerkin} are not mollified and may take values in negative interpolation spaces.
The proof of Proposition~\ref{weak_irreg} uses both Corollary~\ref{cor:weak_reg2}
in Section~\ref{sec:weak_reg} and Proposition~\ref{strong}
in Section~\ref{sec:strong_convergence}.
Corollary~\ref{weak_irreg_B} follows from an application of Proposition~\ref{weak_irreg}. Moreover, Corollary~\ref{cor:weak.convergence.rate.simplified} is an immediate consequence of Corollary~\ref{weak_irreg_B}.

\subsection{Setting}
\label{irregsec:setting}

Assume the setting in Section~\ref{sec:general_setting},
let 
$
  \varphi\in C^4_b(H,\R)
$, 
$
  \theta \in [ 0 , 1 ) 
$, 
$
  F \in C^4_b( H , H_{ - \theta } )
$, 
$
  B \in C^4_b(H,HS(U,H_{ -\nicefrac{\theta}{2} }))
$, 
$
  \xi\in\lpn{4}{ \P|_{ \mathcal{F}_0 } }{H}
$, 
let
$ \varsigma_{ F, B } \in \R $
be the real number given by
$
  \varsigma_{ F, B } 
  =
    \max\!\big\{
    1 ,
    \|
      F
    \|_{ 
      C_b^3( H, H_{-\theta} )
    }^2
    ,
    \|
      B
    \|_{ 
      C_b^3( H, HS( U, H_{-\nicefrac{\theta}{2}} ) )
    }^4
    \big\}  
$,
let 
$ 
  X^I \colon [0,T] \times \Omega \to P_I( H )
$,
$ I \in \mathcal{P}(\set) $,
and
$
  X^{ \set, \kappa, x } \colon [0,T] \times \Omega \to H 
$, 
$ \kappa \in [0,\infty) $,
$ x \in H $,
be up-to-modifications unique 
$
  ( \mathcal{F}_t )_{ t \in [0,T] }
$-predictable stochastic processes
which satisfy for all 
$ I \in \grid(\set) $, 
$ \kappa \in [0,\infty) $, 
$ x \in H $
that
$
  \sup_{ t \in [0,T] }
  \ES\big[
  \| X^I_t \|^4_H
  +
  \| X^{ \set, \kappa, x }_t \|^4_H
  \big] < \infty
$ 
and which satisfy that 
for all 
$ I \in \grid(\set) $, 
$ \kappa \in [0,\infty) $, 
$ x \in H $, 
$ t \in [0,T] $ 
it holds $ \P $-a.s.\ that
\begin{equation}
\label{irregeq:Galerkin}
  X_t^I 
  = 
    e^{ A t } P_{ I }( \xi )
  + 
    \int_0^t e^{ A ( t - s ) } P_{ I } F( X^I_s ) \, ds
  + 
    \int_0^t e^{ A ( t - s ) } P_{ I } B( X^I_s ) \, dW_s 
    ,
\end{equation}
\begin{equation}
\label{eq:unique.mollified.SEE}
\begin{split}
&
  X^{ \set, \kappa, x }_t =
  e^{ A t } x
+
  \int_0^t
  e^{ A( \kappa + t - s) }
  F( 	
    X^{ \set, \kappa, x }_s 
  )
  \,
  ds
+
  \int_0^t
  e^{ A( \kappa + t - s) }
  B( 
    X^{ \set, \kappa, x }_s 
  )
  \,
  dW_s
  ,
\end{split}
\end{equation}
let 
{$
  u^{ ( \kappa ) } 
  \colon [0,T] \times H \to \R
$}, 
$ 
  \kappa \in (0,\infty)
$,  
be the functions which satisfy 
for all 
$ \kappa \in (0,\infty) $, 
$ t \in [0,T] $, 
$ x \in H $
that 
$
  u^{ ( \kappa ) }( t, x ) 
  = 
  \ES\big[ 
    \varphi( X^{ \set, \kappa, x }_{ T - t } )
  \big]
$, 
and let 
$
  c^{ ( \kappa ) }_{ \delta_1, \dots, \delta_k }
  \in [0,\infty]	
$,
$ \delta_1, \dots, \delta_k \in (-\infty,0] $,
$ k \in \{ 1, 2, 3, 4 \} $, 
$ \kappa \in (0,\infty) $,
be the extended real numbers which satisfy for all 
$ \kappa \in (0,\infty) $,
$ k \in \{ 1, 2, 3, 4 \} $, 
$ \delta_1, \dots, \delta_k \in (-\infty,0] $
that 
\begin{equation}
\label{eq:c.kappa.delta}
\begin{split} 
&
  c^{ ( \kappa ) }_{ \delta_1, \dots, \delta_k }
  =
  \sup_{
    t \in [0,T)
  }
  \sup_{ 
    x \in H
  }
  \sup_{ 
    v_1, \dots, v_k \in H \backslash \{ 0 \}
  }
  \left[
  \frac{
    \big|
      ( 
        \frac{ 
          \partial^k
        }{
          \partial x^k
        }
        u^{ ( \kappa ) }
      )( t, x )( v_1, \dots, v_k )
    \big|
  }{
    ( T - t )^{ 
      (
        \delta_1 + \ldots + \delta_k
      ) 
    }
    \left\| v_1 \right\|_{ H_{ \delta_1 } }
    \cdot
    \ldots
    \cdot
    \left\| v_k \right\|_{ H_{ \delta_k } }
  }
  \right]
\end{split}
\end{equation}
(cf., e.g., item~\eqref{item:kolmogorov.diff} of Lemma~\ref{lem:Kolmogorov}).

\subsection{A weak convergence result}
\label{irregsec:weak_irreg}

The following result, Lemma~\ref{lem:c.kappa.delta.finite}, is a slightly modified version of Lemma~\ref{lem:c.delta.finite}. 
Lemma~\ref{lem:c.kappa.delta.finite} provides an a priori estimate for the quantities in~\eqref{eq:c.kappa.delta} which is uniform in the mollification parameter $\kappa\in(0,T]$.

\begin{lemma}
	\label{lem:c.kappa.delta.finite}
	Assume the setting in Section~\ref{irregsec:setting}. Then it holds for all 
	$ k \in \{ 1, 2, 3, 4 \} $, 
	$ \delta_1, \dots, \delta_k \in (-\nicefrac{1}{2},0] $
	with
	$
	\sum^k_{i=1} \delta_i
	> -\nicefrac{1}{2}
	$ 
	that 
	$
	\sup_{\kappa\in(0,T]}
	c^{(\kappa)}_{\delta_1,\ldots,\delta_k}
	< \infty
	$.
\end{lemma}
\begin{proof}
	Throughout this proof let 
	$\phi_\kappa\colon[0,T]\times H\to\R$, 
	$\kappa\in(0,T]$, 
	be the functions which satisfy for all 
	$\kappa\in(0,T]$, 
	$t\in[0,T]$,
	$x\in H$ 
	that 
	$
	\phi_\kappa(t,x)
	=\ES[\varphi(X^{\set,\kappa,x}_t)]
	$.
	Note that for all 
	$\kappa\in(0,T]$, 
	$t\in[0,T]$,
	$x\in H$ 
	it holds that 
	$
	u^{(\kappa)}(t,x)=\phi_\kappa(T-t,x)
	$.
	This, \eqref{eq:unique.mollified.SEE}, and~\cite[Item~(iv) of Corollary~4.2]{AnderssonHefterJentzenKurniawan2016} 
	(with
	$T=T$,
	$\eta=0$,
	$H=H$,
	$U=U$,
	$V=\R$,
	$W=W$,
	$A=A$,
	$n=4$,
	$\alpha=\theta$,
	$\beta=\nicefrac{\theta}{2}$,	
	$F=F$,
	$B=B$,
	$\varphi=\varphi$,
	$k=k$,
	$\delta_1=-\delta_1,\ldots,\delta_k=-\delta_k$,
	for 
	$ 
	(\delta_1, \dots, \delta_k) \in 
	\{(x_1,\ldots,x_k)\in(-\nicefrac{1}{2},0]^k\colon
	\sum^k_{i=1} x_i > -\nicefrac{1}{2}
	\} 
	$, 
	$ k \in \{ 1, 2, 3, 4 \} $
	in the notation of Corollary~4.2 in~\cite{AnderssonHefterJentzenKurniawan2016})
	imply that for all 
	$ k \in \{ 1, 2, 3, 4 \} $, 
	$ \delta_1, \dots, \delta_k \in (-\nicefrac{1}{2},0] $
	with
	$
	\sum^k_{i=1} \delta_i
	> -\nicefrac{1}{2}
	$ 
	it holds that 
	\begin{equation}
	\begin{split}
	&
	\sup_{\kappa\in(0,T]}
	c^{(\kappa)}_{\delta_1,\ldots,\delta_k}
	\\&=
	\sup_{\kappa\in(0,T]}
	\sup_{
		t \in [0,T)
	}
	\sup_{ 
		x \in H
	}
	\sup_{ 
		v_1, \dots, v_k \in H \backslash \{ 0 \}
	}
	\left[
	\frac{
		\big|
		( 
		\frac{ 
			\partial^k
		}{
			\partial x^k
		}
		u^{(\kappa)}
		)( t, x )( v_1, \dots, v_k )
		\big|}{
		( T - t )^{ 
			(
			\delta_1 + \ldots + \delta_k
			) 
		}
		\left\| v_1 \right\|_{ H_{ \delta_1 } }
		\cdot
		\ldots
		\cdot
		\left\| v_k \right\|_{ H_{ \delta_k } }}\right]
	\\&=
	\sup_{\kappa\in(0,T]}
	\sup_{
		t \in [0,T)
	}
	\sup_{ 
		x \in H
	}
	\sup_{ 
		v_1, \dots, v_k \in H \backslash \{ 0 \}
	}
	\left[
	\frac{
		\big|
		( 
		\frac{ 
			\partial^k
		}{
			\partial x^k
		}
		\phi_\kappa
		)( T-t, x )( v_1, \dots, v_k )
		\big|}{
		( T - t )^{ 
			(
			\delta_1 + \ldots + \delta_k
			) 
		}
		\left\| v_1 \right\|_{ H_{ \delta_1 } }
		\cdot
		\ldots
		\cdot
		\left\| v_k \right\|_{ H_{ \delta_k } }}\right]
	\\&=
	\sup_{\kappa,t\in(0,T]}
	\sup_{ 
		x \in H
	}
	\sup_{ 
		v_1, \dots, v_k \in H \backslash \{ 0 \}
	}
	\left[
	\frac{
		\big|
		( 
		\frac{ 
			\partial^k
		}{
			\partial x^k
		}
		\phi_\kappa
		)( t, x )( v_1, \dots, v_k )
		\big|}{
		t^{ 
			(
			\delta_1 + \ldots + \delta_k
			) 
		}
		\left\| v_1 \right\|_{ H_{ \delta_1 } }
		\cdot
		\ldots
		\cdot
		\left\| v_k \right\|_{ H_{ \delta_k } }}\right]
	< \infty.
	\end{split} 
	\end{equation}
	The proof of Lemma~\ref{lem:c.kappa.delta.finite} is thus completed.
\end{proof}

\begin{proposition}
\label{weak_irreg}
Assume the setting in Section~\ref{irregsec:setting} and let 
$
  I \in \grid(\set)
$,
$ 
  \vartheta \in [ 0, \nicefrac{ 1 }{ 2 } ) \cap [ 0, \theta ]  
$.
Then it holds for all 
$ 
  r \in [ 0 , 1 - \vartheta ) 
$,
$ 
  \rho \in ( 0 , 1 - \theta ) 
$
that
\begin{align}
\label{eq:weak_main_estimate}
&
  \big|
    \ES\big[ 
      \varphi( X^{ \set }_T )
    \big]
    -
    \ES\big[ 
      \varphi( X_T^{ I } )
    \big]
  \big|
\\&\nonumber\leq
{
  \tfrac{
    18
  }{
      |\!\min\{1,T\}|^{
            2 ( \theta - \vartheta ) 
          }
  }
  \,
    \max\left\{
    \|
      \operatorname{Id}_H
    \|_{ 
      L( H,H_{-1}) 
    },
\tfrac{1}{
    \|
      \operatorname{Id}_H
    \|_{ 
      L( H,H_{-1}) 
    }}
    \right\}
}
  \,
  \E\,\big[\!
    \max\{ 1 , \| \xi \|^4_H \}
  \big]
    \,
    \|
      P_{ \set \backslash I }
    \|^{ 
      \frac{ r \rho }{ \rho + 4 ( \theta - \vartheta ) } 
    }_{ 
      L( H, H_{ - 1 } ) 
    }
\\ & \cdot
\nonumber
  \Bigg\{
   \bigg[
     \tfrac{ 
       T^{ 
         (1 - \nicefrac{ \rho }{ 2 } - \theta )
       } 
       \,
       \|
         F
       \|_{
         C_b^1( H , H_{ - \theta } ) 
       } 
     }{ 
       (1 - \nicefrac{ \rho }{ 2 } - \theta) 
     }
   + 
   \tfrac{
     \sqrt{
       T^{ (1 - \rho - \theta) } 
     }
     \,
     \|
       B
     \|_{ 
       C_b^1( H, HS( U, H_{ -\nicefrac{ \theta }{ 2 } } ) ) 
     }
   }{
     \sqrt{ 1 - \rho - \theta }
   }
 \bigg]
   \| \varphi \|_{
     C_b^1( H , \R )
   }
   +
  \tfrac{ \varsigma_{F,B} }{ T^r }
  \!
  \left[ 
    1 + 
    \tfrac{ 
      T^{ (1 - \vartheta) } 
    }{
      (1 - \vartheta - r) 
    } 
  \right]^2
\\ & \cdot
\nonumber
  \bigg[ 
    \| \varphi \|_{ C_b^3( H, \R) }
    +
    \sup_{ \kappa \in (0,{T]} }
    \big[
    c^{ ( \kappa ) }_{ - \vartheta } 
    + 
    c^{ ( \kappa ) }_{ - \vartheta , 0 } 
    + 
    c^{ ( \kappa ) }_{ - \vartheta , 0 , 0 } 
    + 
    c^{ ( \kappa ) }_{ - \nicefrac{ \vartheta }{ 2 } , - \nicefrac{ \vartheta }{ 2 } } 
    + 
    c^{ ( \kappa ) }_{ - \nicefrac{ \vartheta }{ 2 } , - \nicefrac{ \vartheta }{ 2 } , 0 } 
    + 
    c^{ ( \kappa ) }_{ - \nicefrac{ \vartheta }{ 2 } , - \nicefrac{ \vartheta }{ 2 }, 0, 0 } 
    \big]
  \bigg]
  \Bigg\}
\\ & \cdot
\nonumber
  \left|
  \mathcal{E}_{ ( 1 - \theta ) }
  \!\left[ 
      \tfrac{ 
        T^{ (1 - \theta) } \sqrt{2} \,
        \| 
          F
        \|_{
          C_b^1( H , H_{ - \theta } ) 
        }
      }{
        \sqrt{ 1 - \theta } 
      }
      +
      \sqrt{ 12 \, T^{ (1 - \theta) } }
      \,
      \| 
        B
      \|_{
        C_b^1( 
          H , 
          HS( U, H_{ - \nicefrac{ \theta}{ 2 } } )
        )
      }
  \right]
  \right|^4
  < \infty
  .
\end{align}
\end{proposition}

\begin{proof}
Throughout this proof assume w.l.o.g.\ that 
$I\neq\set$. 
We intend to prove Proposition~\ref{weak_irreg} through an
application of Corollary~\ref{cor:weak_reg2}.
Corollary~\ref{cor:weak_reg2} assumes that the initial random variable 
of the considered SEE takes values in $  H_2 \subseteq H $.
In Section~\ref{irregsec:setting} above we, however, merely assume that
the initial random variable $ \xi $ takes values in $ H $.
To overcome this difficulty,
we mollify the initial random variable in an appropriate sense
so that the assumptions of Corollary~\ref{cor:weak_reg2} are met 
and Corollary~\ref{cor:weak_reg2} can be applied.
More formally, note that there exist
up-to-modifications unique 
$
  ( \mathcal{F}_t )_{ t \in [0,T] }
$-predictable stochastic processes
$ 
  \hat{X}^{ J , \kappa, \delta } 
  \colon [0,T] \times \Omega \to H
$,
$ \kappa, \delta \in [0,\infty) $,
$ J \in \grid(\set) $, 
such that for all 
$ J \in \grid(\set) $, 
$ 
  \kappa, \delta \in [0,\infty) 
$ 
it holds that 
$
  \sup_{ t \in [0,T] }
  \ES\big[
  \| 
    \hat{X}^{ J, \kappa , \delta }_t
  \|^4_H
  \big]
  < \infty
$ 
and such that for all 
$ J \in \grid(\set) $, 
$ \kappa, \delta \in [0,\infty) $, $ t \in [0,T] $ it holds $ \P $-a.s.\ that
\begin{equation}
\label{irregeq:mollified}
  \hat{X}_t^{ J, \kappa, \delta } 
  = 
    e^{ A( \delta + t ) } P_J(\xi)
  + 
    \int_0^t e^{ A( \kappa + t - s ) } 
  P_J F( \hat{X}_s^{ J, \kappa, \delta } ) \, ds
  +
    \int_0^t e^{ A( \kappa + t - s ) } 
  P_J B( \hat{X}_s^{ J, \kappa, \delta } ) \, dW_s 
    .
\end{equation}
In the next step we 
observe that the triangle inequality 
ensures that for all $ \kappa, \delta \in (0,\infty) $ 
it holds that
\begin{equation}
\label{irreg_1}
\begin{split}
&
  \big|
    \ES\big[ 
      \varphi( \hat{X}^{ \set, 0, \delta }_T )
    \big]
    -
    \ES\big[ 
      \varphi( \hat{X}_T^{ I, 0, \delta } )
    \big]
  \big|
  \leq 
  \big|
    \ES\big[ 
      \varphi( \hat{X}^{ \set, 0, \delta }_T )
    \big]
    -
    \ES\big[ 
      \varphi( \hat{X}_T^{ \set, \kappa, \delta } )
    \big]
  \big|
\\ & +
  \big|
    \ES\big[ 
      \varphi( \hat{X}_T^{ \set, \kappa, \delta } )
    \big]
    -
    \ES\big[ 
      \varphi( \hat{X}_T^{ I, \kappa, \delta } )
    \big]
  \big|
  +
  \big|
    \ES\big[ 
      \varphi( \hat{X}_T^{ I, \kappa, \delta } )
    \big]
    -
    \ES\big[ 
      \varphi( \hat{X}_T^{ I, 0, \delta } )
    \big]
  \big|
  .
\end{split}
\end{equation}
%
%
In the following we bound the three summands on the right hand side of~\eqref{irreg_1}.
For the first and third summands on the right hand side of~\eqref{irreg_1}
we observe that Proposition~\ref{strong} shows 
that for all 
$ \kappa, \delta \in (0,\infty) $, 
$ \rho \in [ 0, 1 - \theta ) $ 
it holds that
\begin{equation}
\begin{split}
&
  \big|
    \ES\big[ 
      \varphi( \hat{X}^{ \set, 0, \delta }_T )
    \big]
    -
    \ES\big[ 
      \varphi( \hat{X}_T^{ \set, \kappa, \delta } )
    \big]
  \big|
  +
  \big|
    \ES\big[ 
      \varphi( \hat{X}_T^{I,\kappa,\delta} )
    \big]
    -
    \ES\big[ 
      \varphi( \hat{X}_T^{ I, 0, \delta } )
    \big]
  \big|
\\ & \leq
  4
  \,
  \|
    \varphi
  \|_{
    C_b^1( H, \R ) 
  }
  \left|
    \mathcal{E}_{ ( 1 - \theta ) }\!\left[
      \tfrac{ 
        T^{ (1 - \theta) } \sqrt{2} \,
        \|  F \|_{
          C_b^1( H , H_{ - \theta } ) 
        }
      }{
        \sqrt{ 1 - \theta } 
      }
      +
      \sqrt{ 2 \, T^{ (1 - \theta) } }
      \,
      \| B \|_{
        C_b^1( 
          H , 
          HS( U, H_{ - \nicefrac{ \theta}{ 2 } } )
        )
      }
    \right]
    \right|^2
\\ & \cdot 
  \left[
    \tfrac{ 
      T^{ (1 - \nicefrac{ \rho }{ 2 } - \theta) } 
    }{
      (1 - \nicefrac{ \rho }{ 2 } - \theta) 
    }
    \| 
       F 
    \|_{
      C_b^1( H, H_{ - \theta } ) 
    }
    +
    \tfrac{
      \sqrt{
        T^{ (1 - \rho - \theta) } 
      }
    }{
      \sqrt{ 1 - \rho - \theta } 
    }
    \| B \|_{
      C_b^1( H , HS( U , H_{ - \nicefrac{ \theta }{ 2 } } ) ) 
    }
  \right]
  \|\!\max\{1,\|\xi\|_H\}\|_{\lpn{2}{\P}{\R}} \,
  \kappa^{ \frac{ \rho }{ 2 } }
  .
\end{split}\label{irreg_4}
\end{equation}
Next we bound the second summand on the right hand side 
of~\eqref{irreg_1}. For this we note that 
for all $ \kappa \in (0,\infty) $ it holds that
\begin{equation}
\begin{split}
    \max\!\big\{
    1 ,
    \|
      e^{ \kappa  A }  F( \cdot )
    \|_{ 
      C_b^3( H, H_{-\vartheta} )
    }^2
    ,
    \|
      e^{ \kappa  A } B( \cdot )
    \|_{ 
      C_b^3( H, HS( U, H_{-\nicefrac{\vartheta}{2}} ) )
    }^4
    \big\}
  \leq
  \varsigma_{F,B}
  \,
  \max\!\big\{ 
    1 ,
    \kappa^{ - 2 ( \theta - \vartheta ) }
  \big\}
  .
\end{split}
\label{irreg_2}
\end{equation}
This and Corollary~\ref{cor:weak_reg2}
show that for all 
$ 
  \kappa, \delta \in (0,\infty) 
$, 
$ 
  r \in [ 0, 1 - \vartheta ) 
$ 
it holds that
\begin{equation}
\label{irreg_5}
\begin{split}
&
  \big|
    \ES\big[ 
      \varphi( \hat{X}_T^{ \set, \kappa, \delta } )
    \big]
    -
    \ES\big[ 
      \varphi( \hat{X}_T^{ I, \kappa, \delta } )
    \big]
  \big|
\\ & \leq
  \left[  
    1 + 
    \tfrac{ 
      T^{ (1 - \vartheta) } 
    }{ (1 - \vartheta - r) } 
  \right]^2
  \left|\mathcal{E}_{(1-\theta)}\!\left[ 
      \tfrac{ 
        T^{ (1 - \theta) } \sqrt{2} \,
        \|  F \|_{
          C_b^1( H , H_{ - \theta } ) 
        }
      }{
        \sqrt{ 1 - \theta } 
      }
      +
      \sqrt{ 12 \, T^{ (1 - \theta) } }
      \,
      \| B \|_{
        C_b^1( 
          H , 
          HS( U, H_{ - \nicefrac{ \theta}{ 2 } } )
        )
      }
  \right]\right|^4\\
  &\quad\cdot\left[ 
    \left\|\varphi\right\|_{C_b^3(H,\R)}
    +
    c^{(\kappa)}_{-\vartheta} + c^{(\kappa)}_{-\vartheta,0} + c^{(\kappa)}_{-\vartheta,0,0} + c^{(\kappa)}_{-\nicefrac{\vartheta}{2},-\nicefrac{\vartheta}{2}} + c^{(\kappa)}_{-\nicefrac{\vartheta}{2},-\nicefrac{\vartheta}{2},0} + c^{(\kappa)}_{-\nicefrac{\vartheta}{2},-\nicefrac{\vartheta}{2},0,0} \right] \\
  &\quad\cdot 
  \frac{ 18 \, \varsigma_{F,B} }{ T^r }
  \,
  \E\,\big[\!
    \max\{ 1 , \| \xi \|_H^4 \}
  \big]
  \max\!\big\{ 
    1 ,
    \kappa^{ - 2 ( \theta - \vartheta ) }
  \big\}
  \left\|
    P_{ \set \backslash I }
  \right\|_{ L( H, H_{ - r } ) 
  }
  .
\end{split}
\end{equation}
%
In the next step 
we plug~\eqref{irreg_4} 
and~\eqref{irreg_5} into~\eqref{irreg_1}
and we use the fact that 
$ 
  \forall \, r \in (0,\infty) 
  \colon
$
$
  \| P_I \|_{
    L( H, H_{ - r } )
  }
  =
  \| P_I \|_{ L( H, H_{ - 1 } ) }^r
$
to obtain that for all $ \kappa, \delta \in (0,\infty) $, 
$ r \in [ 0, 1 - \vartheta ) $, 
$ \rho \in [ 0, 1 - \theta ) $ 
it holds that
\begin{equation}
\label{irreg_6}
\begin{split}
&
  \big|
    \ES\big[ 
      \varphi( \hat{X}^{ \set, 0, \delta }_T )
    \big]
    -
    \ES\big[ 
      \varphi( \hat{X}_T^{ I, 0, \delta } )
    \big]
  \big|
\leq
  {\max\!\left\{  
    4 \,
    \kappa^{
      \frac{ \rho }{ 2 } 
    }
    ,
    18 
    \max\!\big\{ 
      1 ,
      \kappa^{ - 2 ( \theta - \vartheta ) }
    \big\}
    \|
      P_{ \set \backslash I }
    \|^{ r }_{ 
      L( H, H_{ - 1 } ) 
    }
  \right\}}
\\ & \cdot
  \bigg\{
   \left[
     \tfrac{ 
       T^{ 
         (1 - \nicefrac{ \rho }{ 2 } - \theta) } 
       }{ 
         (1 - \nicefrac{ \rho }{ 2 } - \theta) 
       }
     \|
        F
     \|_{
       C_b^1( H , H_{ - \theta } ) 
     } 
   + 
   \tfrac{
     \sqrt{
       T^{ (1 - \rho - \theta) } 
     }
   }{
     \sqrt{ 1 - \rho - \theta }
   }
   \|
     B
   \|_{ 
     C_b^1( H, HS( U, H_{ -\nicefrac{ \theta }{ 2 } } ) ) 
   }
 \right]
   \| \varphi \|_{C_b^1(H,\R)}
   +
  \tfrac{ \varsigma_{F,B} }{ T^r }
  \left[ 
    1 + 
    \tfrac{ 
      T^{ (1 - \vartheta) } 
    }{
      (1 - \vartheta - r) 
    } 
  \right]^2
  \\ & \cdot
  \left[ 
    \left\| \varphi \right\|_{ C_b^3( H, \R) }
    +
    c^{ ( \kappa ) }_{ - \vartheta } 
    + 
    c^{ ( \kappa ) }_{ - \vartheta , 0 } 
    + 
    c^{ ( \kappa ) }_{ - \vartheta , 0 , 0 } 
    + 
    c^{ ( \kappa ) }_{ - \nicefrac{ \vartheta }{ 2 } , - \nicefrac{ \vartheta }{ 2 } } 
    + 
    c^{ ( \kappa ) }_{ - \nicefrac{ \vartheta }{ 2 } , - \nicefrac{ \vartheta }{ 2 } , 0 } 
    + 
    c^{ ( \kappa ) }_{ - \nicefrac{ \vartheta }{ 2 } , - \nicefrac{ \vartheta }{ 2 }, 0, 0 } 
  \right]
  \bigg\}
\\ & \cdot
  \E\,\big[\!
    \max\{ 1 , \| \xi \|_H^4 \}
  \big]
  \left|
  \mathcal{E}_{ ( 1 - \theta ) }
  \!\left[ 
      \tfrac{ 
        T^{ (1 - \theta) } \sqrt{2} \,
        \|  F \|_{
          C_b^1( H , H_{ - \theta } ) 
        }
      }{
        \sqrt{ 1 - \theta } 
      }
      +
      \sqrt{ 12 \, T^{ (1 - \theta) } }
      \,
      \| B \|_{
        C_b^1( 
          H , 
          HS( U, H_{ - \nicefrac{ \theta}{ 2 } } )
        )
      }
  \right]
  \right|^4
  .
\end{split}
\end{equation}
Next we use the fact that 
$
  \|P_{\set\backslash I}\|_{L(H,H_{-1})}
  \leq
  \|\operatorname{Id}_H\|_{L(H,H_{-1})}
$
to obtain that 
for all 
$ r \in [0,\infty) $,
$
  \rho \in ( 0, 1 - \theta )
$ 
it holds that
{\allowdisplaybreaks

\begin{align}
\label{eq:eps_minimize}
&
  \inf_{ \kappa \in (0,T] }
  \max\left\{  
    4 \,
    \kappa^{
      \frac{ \rho }{ 2 } 
    }
    ,
    18 
    \max\!\big\{ 
      1 ,
      \kappa^{ - 2 ( \theta - \vartheta ) }
    \big\}
    \|
      P_{ \set \backslash I }
    \|^{ r }_{ 
      L( H, H_{ - 1 } ) 
    }
  \right\}
\\ & \leq
\nonumber
    \max\left\{
    4 
  \left[
    \min\{1,T\}
    \left|
    \tfrac{
    \|
      P_{ \set \backslash I }
    \|_{ 
      L( H, H_{ - 1 } ) 
    }}{
    \|
      \operatorname{Id}_H
    \|_{ 
      L( H, H_{ - 1 } ) 
    }
    }
    \right|^{ 
          \frac{ 2 r }{ 
            ( \rho + 4 ( \theta - \vartheta ) )
          }
        }
  \right]^{
      \frac{ \rho }{ 2 } 
    }\right.,
\\&\nonumber\quad\left.
    18 
    \max\!\left\{ 
      1 ,
  \left[
    \min\{1,T\}
    \left|
\tfrac{
      \|
      P_{ \set \backslash I }
    \|_{ 
          L( H, H_{ - 1 } ) 
        }
        }{
    \|\operatorname{Id}_H\|_{L( H, H_{ -1 } )}
    }
    \right|^{ 
      \frac{ 2 r }{ 
        ( \rho + 4 ( \theta - \vartheta ) )
      }
    }
  \right]^{
      - 2 ( \theta - \vartheta ) 
    }  
    \right\}
    \|
      P_{ \set \backslash I }
    \|^{ r }_{ 
      L( H, H_{ - 1 } ) 
    }
    \right\}
\\ & =
\nonumber
    \max\left\{
    4 
  \left[
    \min\{1,T\}
    \left|
    \tfrac{
    \|
      P_{ \set \backslash I }
    \|_{ 
      L( H, H_{ - 1 } ) 
    }}{
    \|
      \operatorname{Id}_H
    \|_{ 
      L( H, H_{ - 1 } ) 
    }
    }
    \right|^{ 
          \frac{ 2 r }{ 
            ( \rho + 4 ( \theta - \vartheta ) )
          }
        }
  \right]^{
      \frac{ \rho }{ 2 } 
    },
    \right.
\\&\nonumber\quad\left.
    18 
    \|
      P_{ \set \backslash I }
    \|^{ r }_{ 
      L( H, H_{ - 1 } ) 
    }
  \left[
    \min\{1,T\}
    \left|
\tfrac{
      \|
      P_{ \set \backslash I }
    \|_{ 
          L( H, H_{ - 1 } ) 
        }
        }{
    \|\operatorname{Id}_H\|_{L( H, H_{ -1 } )}
    }
    \right|^{ 
      \frac{ 2 r }{ 
        ( \rho + 4 ( \theta - \vartheta ) )
      }
    }
  \right]^{
      - 2 ( \theta - \vartheta ) 
    }
    \right\}
\\ & =
\nonumber
    \max\left\{ 
    \tfrac{
    4\,|\!\min\{1,T\}|^{\frac{\rho}{2}}
    }{
    \|
      \operatorname{Id}_H
    \|^{ 
              \frac{ r \rho }{ 
                ( \rho + 4 ( \theta - \vartheta ) )
              }
            }_{ 
      L( H, H_{ - 1 } ) 
    }
    },
\tfrac{
    18\,\|\operatorname{Id}_H\|^{ 
          \frac{ 4 r(\theta-\vartheta) }{ 
            ( \rho + 4 ( \theta - \vartheta ) )
          }
        }_{L( H, H_{ -1 } )}
        }{
      |\!\min\{1,T\}|^{
            2 ( \theta - \vartheta ) 
          }
    }
    \right\}
    \,
    \|
      P_{ \set \backslash I }
    \|^{ 
      \frac{ r \rho }{ \rho + 4 ( \theta - \vartheta ) } 
    }_{ 
      L( H, H_{ - 1 } ) 
    }
\\ & \leq
\nonumber
    18\,\max\left\{ 
    \tfrac{
    1
    }{
    \min\{1,\|\operatorname{Id}_H\|^r_{L( H, H_{ -1 } )}\}
    },
\tfrac{
    \max\{1,\|\operatorname{Id}_H\|^r_{L( H, H_{ -1 } )}\}
        }{
      |\!\min\{1,T\}|^{
            2 ( \theta - \vartheta ) 
          }
    }
    \right\}
    \,
    \|
      P_{ \set \backslash I }
    \|^{ 
      \frac{ r \rho }{ \rho + 4 ( \theta - \vartheta ) } 
    }_{ 
      L( H, H_{ - 1 } ) 
    }
  \\&\nonumber\leq
  \tfrac{
    18
  }{
      |\!\min\{1,T\}|^{
            2 ( \theta - \vartheta ) 
          }
  }
  \,
    \max\left\{
    \|
      \operatorname{Id}_H
    \|^r_{ 
      L( H,H_{-1}) 
    },
\tfrac{1}{
    \|
      \operatorname{Id}_H
    \|^r_{ 
      L( H,H_{-1}) 
    }}
    \right\}
    \|
      P_{ \set \backslash I }
    \|^{ 
      \frac{ r \rho }{ \rho + 4 ( \theta - \vartheta ) } 
    }_{ 
      L( H, H_{ - 1 } ) 
    }
    .
\end{align}}Putting~\eqref{eq:eps_minimize} into~\eqref{irreg_6}
implies that for all $ \delta \in (0,\infty) $, 
$ r \in [ 0, 1 - \vartheta ) $, 
$ \rho \in ( 0, 1 - \theta ) $ 
it holds that
\begin{equation}
\label{irreg_6b}
\begin{split}
&
  \big|
    \ES\big[ 
      \varphi( \hat{X}^{ \set, 0, \delta }_T )
    \big]
    -
    \ES\big[ 
      \varphi( \hat{X}_T^{ I, 0, \delta } )
    \big]
  \big|
\\&\leq
{
  \tfrac{
    18
  }{
      |\!\min\{1,T\}|^{
            2 ( \theta - \vartheta ) 
          }
  }
  \,
    \max\left\{
    \|
      \operatorname{Id}_H
    \|^r_{ 
      L( H,H_{-1}) 
    },
\tfrac{1}{
    \|
      \operatorname{Id}_H
    \|^r_{ 
      L( H,H_{-1}) 
    }}
    \right\}
}
    \|
      P_{ \set \backslash I }
    \|^{ 
      \frac{ r \rho }{ \rho + 4 ( \theta - \vartheta ) } 
    }_{ 
      L( H, H_{ - 1 } ) 
    }
\\ & \cdot
  \Bigg\{
   \left[
     \tfrac{ 
       T^{ 
         (1 - \nicefrac{ \rho }{ 2 } - \theta) } 
       }{ 
         (1 - \nicefrac{ \rho }{ 2 } - \theta) 
       }
     \|
        F
     \|_{
       C_b^1( H , H_{ - \theta } ) 
     } 
   + 
   \tfrac{
     \sqrt{
       T^{ (1 - \rho - \theta) } 
     }
   }{
     \sqrt{ 1 - \rho - \theta }
   }
   \|
     B
   \|_{ 
     C_b^1( H, HS( U, H_{ -\nicefrac{ \theta }{ 2 } } ) ) 
   }
 \right]
   \| \varphi \|_{C_b^1(H,\R)}
   +
  \tfrac{ \varsigma_{F,B} }{ T^r }
  \left[ 
    1 + 
    \tfrac{ 
      T^{ (1 - \vartheta) } 
    }{
      (1 - \vartheta - r) 
    } 
  \right]^2
  \\ & \cdot
  \bigg[ 
    \left\| \varphi \right\|_{ C_b^3( H, \R) }
    +
    \sup_{ \kappa \in (0,{T]} }
    \big[
    c^{ ( \kappa ) }_{ - \vartheta } 
    + 
    c^{ ( \kappa ) }_{ - \vartheta , 0 } 
    + 
    c^{ ( \kappa ) }_{ - \vartheta , 0 , 0 } 
    + 
    c^{ ( \kappa ) }_{ - \nicefrac{ \vartheta }{ 2 } , - \nicefrac{ \vartheta }{ 2 } } 
    + 
    c^{ ( \kappa ) }_{ - \nicefrac{ \vartheta }{ 2 } , - \nicefrac{ \vartheta }{ 2 } , 0 } 
    + 
    c^{ ( \kappa ) }_{ - \nicefrac{ \vartheta }{ 2 } , - \nicefrac{ \vartheta }{ 2 }, 0, 0 } 
    \big]
  \bigg]
  \Bigg\}
\\ & \cdot
  \E\,\big[\!
    \max\{ 1 , \| \xi \|_H^4 \}
  \big]
  \left|
  \mathcal{E}_{ ( 1 - \theta ) }
  \!\left[ 
      \tfrac{ 
        T^{ (1 - \theta) } \sqrt{2} \,
        \|  F \|_{
          C_b^1( H , H_{ - \theta } ) 
        }
      }{
        \sqrt{ 1 - \theta } 
      }
      +
      \sqrt{ 12 \, T^{ (1 - \theta) } }
      \,
      \| B \|_{
        C_b^1( 
          H , 
          HS( U, H_{ - \nicefrac{ \theta}{ 2 } } )
        )
      }
  \right]
  \right|^4
  .
\end{split}
\end{equation}
Moreover, we note that~\eqref{irregeq:Galerkin}, \eqref{irregeq:mollified}, and
Corollary~2.8 in~\cite{AnderssonJentzenKurniawan2016arXiv} 
(with
$
H=H
$, 
$
U=U
$, 
$ T = T $, 
$ \eta = 0 $, 
$ p = 2 $, 
$ \alpha = \theta $, 
$
\hat{\alpha} = 0
$,
$ \beta = \nicefrac{\theta}{2} $, 
$
\hat{\beta} 
= 
0
$,
$L_0=|F|_{\operatorname{Lip}^0(H,H_{-\theta})}$, 
$
\hat{L}_0
=
\|F(0)\|_{H_{-\theta}}
$, 
$L_1=|B|_{\operatorname{Lip}^0(H,HS(U,H_{-\nicefrac{\theta}{2}}))}$, 
$
\hat{L}_1
=
\|B(0)\|_{HS(U,H_{-\nicefrac{\theta}{2}})}
$, 
$
W=W
$, 
$ A=A $,
$
\mathbf{F} = 
\big(
[0,T]\times\Omega\times H \ni (t,\omega,v) \mapsto
P_J F(v) \in H_{-\theta}
\big)
$, 
$
\mathbf{B} = 
\big(
[0,T]\times\Omega\times H \ni (t,\omega,v) \mapsto
(U\ni u \mapsto P_J B(v)u \in H_{-\nicefrac{\theta}{2}}) \in HS(U,H_{-\nicefrac{\vartheta}{2}})
\big)
$,  
$ \delta = 0 $, 
$ X^1 = \hat{X}^{ J, 0, \delta } $, 
$ X^2 = \hat{X}^{ J, 0, 0 } $,
$ \lambda = 0 $  
for 
$ J \in \mathcal{P}( \set ) $
in the notation of Corollary~2.8 in~\cite{AnderssonJentzenKurniawan2016arXiv})
ensure that
for all $ J \in \mathcal{P}( \set ) $
it holds that
$
\lim_{(0,\infty)\ni\delta\rightarrow 0}
    \ES\big[ 
      \varphi( \hat{X}_T^{ J, 0, \delta } )
    \big]
    =
    \ES\big[ 
      \varphi( \hat{X}_T^{ J, 0, 0 } )
    \big]
    =
    \ES\big[ 
      \varphi( X^J_T )
    \big]
$.
Combining this with inequality~\eqref{irreg_6b} 
proves the first inequality in~\eqref{eq:weak_main_estimate}.
The second inequality in~\eqref{eq:weak_main_estimate} 
follows from Lemma~\ref{lem:c.kappa.delta.finite}.
The proof of Proposition~\ref{weak_irreg} is thus completed.
\end{proof}

\noindent In a number of cases the difference $ \theta - \vartheta \geq 0 $
can be chosen to be an arbitrarily small positive real number 
(cf.\ Theorem~\ref{intro:theorem} above).
In the next result, 
Corollary~\ref{weak_irreg_B},
we further estimate the right hand side
of \eqref{eq:weak_main_estimate}.

\begin{corollary}
\label{weak_irreg_B}
Assume the setting in Section~\ref{irregsec:setting}
and let 
$
  I \in \grid(\set) \backslash\{\set\}
$, 
$ 
  \vartheta \in [ 0, \nicefrac{ 1 }{ 2 } ) \cap [ 0, \theta ]  
$.
Then it holds for all 
$ 
  \rho \in ( 0 , 1 - \theta ) 
  \cap 
  \big( 
    4 ( \theta - \vartheta ), \infty
  \big)
$
that
{\allowdisplaybreaks
\begin{align}
\label{eq:weak_main_estimate_B}
&
  \big|
    \ES\big[ 
      \varphi( X^{ \set }_T )
    \big]
    -
    \ES\big[ 
      \varphi( X_T^{ I } )
    \big]
  \big|
\\&\nonumber\leq
  \tfrac{
    18
  }{
      |\!\min\{1,T\}|^{
            2 ( \theta - \vartheta ) 
          }
  }
  \,
    \max\left\{
    \|
      \operatorname{Id}_H
    \|_{ 
      L( H,H_{-2}) 
    },
\tfrac{1}{
    \|
      \operatorname{Id}_H
    \|_{ 
      L( H,H_{-2}) 
    }}
    \right\}
    \left[
      \inf_{
        b \in \set \backslash I
      }
        |\lambda_b|
    \right]^{ 
      - \left(
        \rho - 4 ( \theta - \vartheta )
      \right)
    }
\\ & \cdot
\nonumber
  \Bigg\{
   \bigg[
     \tfrac{ 
       T^{ 
         (1 - \nicefrac{ \rho }{ 2 } - \theta) 
       } 
       \,
       \|
         F
       \|_{
         C_b^1( H , H_{ - \theta } ) 
       } 
     }{ 
       (1 - \nicefrac{ \rho }{ 2 } - \theta) 
     }
   + 
   \tfrac{
     \sqrt{
       T^{ (1 - \rho - \theta) } 
     }
     \,
     \|
       B
     \|_{ 
       C_b^1( H, HS( U, H_{ -\nicefrac{ \theta }{ 2 } } ) ) 
     }
   }{
     \sqrt{ 1 - \rho - \theta }
   }
 \bigg]
   \| \varphi \|_{
     C_b^1( H , \R )
   }
   +
  \tfrac{ \varsigma_{F,B} }{ T^\rho }
  \!
  \left[ 
    1 + 
    \tfrac{ 
      T^{ (1 - \vartheta) } 
    }{
      (1 - \vartheta - \rho) 
    } 
  \right]^2
\\ & \cdot
\nonumber
  \bigg[ 
    \| \varphi \|_{ C_b^3( H, \R) }
    +
    \sup_{ \kappa \in (0,{T]} }
    \big[
    c^{ ( \kappa ) }_{ - \vartheta } 
    + 
    c^{ ( \kappa ) }_{ - \vartheta , 0 } 
    + 
    c^{ ( \kappa ) }_{ - \vartheta , 0 , 0 } 
    + 
    c^{ ( \kappa ) }_{ - \nicefrac{ \vartheta }{ 2 } , - \nicefrac{ \vartheta }{ 2 } } 
    + 
    c^{ ( \kappa ) }_{ - \nicefrac{ \vartheta }{ 2 } , - \nicefrac{ \vartheta }{ 2 } , 0 } 
    + 
    c^{ ( \kappa ) }_{ - \nicefrac{ \vartheta }{ 2 } , - \nicefrac{ \vartheta }{ 2 }, 0, 0 } 
    \big]
  \bigg]
  \Bigg\}
\\ & \cdot
\nonumber
    \ES\big[\!\max\{
      1,\| \xi \|^4_H 
    \}\big]
  \left|
  \mathcal{E}_{ ( 1 - \theta ) }
  \!\left[ 
      \tfrac{ 
        T^{ (1 - \theta) } \sqrt{2} \,
        \| 
          F
        \|_{
          C_b^1( H , H_{ - \theta } ) 
        }
      }{
        \sqrt{ 1 - \theta } 
      }
      +
      \sqrt{ 12 \, T^{ (1 - \theta) } }
      \,
      \| 
        B
      \|_{
        C_b^1( 
          H , 
          HS( U, H_{ - \nicefrac{ \theta}{ 2 } } )
        )
      }
  \right]
  \right|^4
  < \infty
  .
\end{align}
}
\end{corollary}

\begin{proof}
First of all, we choose $ r = \rho $ in~\eqref{eq:weak_main_estimate}
in Proposition~\ref{weak_irreg} above to obtain that for all 
$ 
  \rho \in ( 0 , 1 - \theta ) 
$
it holds that
\begin{align}
\label{eq:consequence_prop_weak}
&
  \big|
    \ES\big[ 
      \varphi( X^{ \set }_T )
    \big]
    -
    \ES\big[ 
      \varphi( X_T^{ I } )
    \big]
  \big|
\\&\nonumber\leq
  \tfrac{
    18
  }{
      |\!\min\{1,T\}|^{
            2 ( \theta - \vartheta ) 
          }
  }
  \,
    \max\left\{
    \|
      \operatorname{Id}_H
    \|_{ 
      L( H,H_{-1}) 
    },
\tfrac{1}{
    \|
      \operatorname{Id}_H
    \|_{ 
      L( H,H_{-1}) 
    }}
    \right\}
    \ES\big[\!\max\{
      1,\| \xi \|^4_H 
    \}\big]
    \,
    \|
      P_{ \set \backslash I }
    \|^{ 
      \frac{ \rho^2 }{ \rho + 4 ( \theta - \vartheta ) } 
    }_{ 
      L( H, H_{ - 1 } ) 
    }
\\ & \cdot
\nonumber
  \Bigg\{
   \bigg[
     \tfrac{ 
       T^{ 
         (1 - \nicefrac{ \rho }{ 2 } - \theta) 
       } 
       \,
       \|
         F
       \|_{
         C_b^1( H , H_{ - \theta } ) 
       } 
     }{ 
       (1 - \nicefrac{ \rho }{ 2 } - \theta) 
     }
   + 
   \tfrac{
     \sqrt{
       T^{ (1 - \rho - \theta) } 
     }
     \,
     \|
       B
     \|_{ 
       C_b^1( H, HS( U, H_{ -\nicefrac{ \theta }{ 2 } } ) ) 
     }
   }{
     \sqrt{ 1 - \rho - \theta }
   }
 \bigg]
   \| \varphi \|_{
     C_b^1( H , \R )
   }
   +
  \tfrac{ \varsigma_{F,B} }{ T^{ \rho } }
  \!
  \left[ 
    1 + 
    \tfrac{ 
      T^{ (1 - \vartheta) } 
    }{
      (1 - \vartheta - \rho) 
    } 
  \right]^2
\\ & \cdot
\nonumber
  \bigg[ 
    \| \varphi \|_{ C_b^3( H, \R) }
    +
    \sup_{ \kappa \in (0,{T]} }
    \big[
    c^{ ( \kappa ) }_{ - \vartheta } 
    + 
    c^{ ( \kappa ) }_{ - \vartheta , 0 } 
    + 
    c^{ ( \kappa ) }_{ - \vartheta , 0 , 0 } 
    + 
    c^{ ( \kappa ) }_{ - \nicefrac{ \vartheta }{ 2 } , - \nicefrac{ \vartheta }{ 2 } } 
    + 
    c^{ ( \kappa ) }_{ - \nicefrac{ \vartheta }{ 2 } , - \nicefrac{ \vartheta }{ 2 } , 0 } 
    + 
    c^{ ( \kappa ) }_{ - \nicefrac{ \vartheta }{ 2 } , - \nicefrac{ \vartheta }{ 2 }, 0, 0 } 
    \big]
  \bigg]
  \Bigg\}
\\ & \cdot
\nonumber
  \left|
  \mathcal{E}_{ ( 1 - \theta ) }
  \!\left[ 
      \tfrac{ 
        T^{ (1 - \theta) } \sqrt{2} \,
        \| 
          F
        \|_{
          C_b^1( H , H_{ - \theta } ) 
        }
      }{
        \sqrt{ 1 - \theta } 
      }
      +
      \sqrt{ 12 \, T^{ (1 - \theta) } }
      \,
      \| 
        B
      \|_{
        C_b^1( 
          H , 
          HS( U, H_{ - \nicefrac{ \theta}{ 2 } } )
        )
      }
  \right]
  \right|^4
  < \infty
  .
\end{align}
Next we note that for all 
$ 
  \rho \in 
  \big( 
    0, 1 - \theta 
  \big) 
  \cap 
  \big( 
    4 ( \theta - \vartheta ) , \infty 
  \big) 
$
it holds that
\begin{align}
&
    \|
      P_{ \set \backslash I }
    \|^{ 
      \frac{ \rho^2 }{ \rho + 4 ( \theta - \vartheta ) } 
    }_{ 
      L( H, H_{ - 1 } ) 
    }
=
    \|
      P_{ \set \backslash I }
    \|^{ 
      \rho 
      \left[
        \frac{ 1 }{ 1 + 4 ( \theta - \vartheta ) / \rho } 
        -
        1 
        + \frac{ 4 ( \theta - \vartheta ) }{ \rho }
      \right]
    }_{ 
      L( H, H_{ - 1 } ) 
    }
    \,
    \|
      P_{ \set \backslash I }
    \|^{ 
      \rho 
      \left[
        1 - \frac{ 4 ( \theta - \vartheta ) }{ \rho }
      \right]
    }_{ 
      L( H, H_{ - 1 } ) 
    }
\\ &\nonumber \leq
  \big|\!\max\{
    1 ,
    \|
      P_{ \set \backslash I }
    \|_{ 
      L( H, H_{ - 1 } ) 
    }
  \}\big|^{
    \rho 
    \left[
      \frac{ 1 }{ 1 + 4 ( \theta - \vartheta ) / \rho } 
      -
      1 
      + \frac{ 4 ( \theta - \vartheta ) }{ \rho }
    \right]
  }
    \,
    \|
      P_{ \set \backslash I }
    \|^{ 
      \left(
        \rho - 4 ( \theta - \vartheta ) 
      \right)
    }_{ 
      L( H, H_{ - 1 } ) 
    }
\\ &\nonumber \leq
\max\{
    1 ,
    \|
      P_{ \set \backslash I }
    \|^\rho_{ 
      L( H, H_{ - 1 } ) 
    }
  \}
    \,
    \|
      P_{ \set \backslash I }
    \|^{ 
      \left(
        \rho - 4 ( \theta - \vartheta ) 
      \right)
    }_{ 
      L( H, H_{ - 1 } ) 
    }
\leq
  \max\!\big\{
    1 
    ,
    \| 
      \operatorname{Id}_H
      \!
    \|_{ L( H, H_{ - 1 } ) 
    }
  \big\}
    \,
    \|
      P_{ \set \backslash I }
    \|^{ 
      \left(
        \rho - 4 ( \theta - \vartheta ) 
      \right)
    }_{ 
      L( H, H_{ - 1 } ) 
    }
\end{align}
Combining this
with~\eqref{eq:consequence_prop_weak},
the fact that 
$
    \| 
      \operatorname{Id}_H
      \!
    \|_{ L( H, H_{ - 1 } ) 
    }^2
  =
    \| 
      \operatorname{Id}_H
      \!
    \|_{ L( H, H_{ - 2 } ) 
    }
$, 
and the fact that
$
    \|
      P_{ \set \backslash I }
    \|^{ 
      \left(
        \rho - 4 ( \theta - \vartheta ) 
      \right)
    }_{ 
      L( H, H_{ - 1 } ) 
    }
  =
    \left[
      \inf_{
        b \in \set \backslash I
      }
      |\lambda_b|
    \right]^{ 
      - \left(
        \rho - 4 ( \theta - \vartheta )
      \right)
    }
$
completes the proof 
of Corollary~\ref{weak_irreg_B}.
\end{proof}
\noindent The next result, Corollary~\ref{cor:weak.convergence.rate.simplified} below, is an immediate consequence of Corollary~\ref{weak_irreg_B} above.
\begin{corollary}
\label{cor:weak.convergence.rate.simplified}
Assume the setting in Section~\ref{irregsec:setting} and assume that $\theta < \nicefrac{3}{5}$.
Then for every 
$ 
  \varepsilon\in(0,\infty)
$
there exists a real number $C\in\R$
such that for all $I\in\grid(\set)\backslash\{\set\}$ it holds that 
\begin{equation}
\label{eq:weak.convergence.rate.simplified}
  \big|
    \ES\big[ 
      \varphi( X^{ \set }_T )
    \big]
    -
    \ES\big[ 
      \varphi( X_T^{ I } )
    \big]
  \big|
  \leq
  C \cdot
    \left[
      \inf_{
        b \in \set \backslash I
      }
        |\lambda_b|
    \right]^{ 
     \varepsilon - \left(
        1-\theta-4\max\{\theta-\nicefrac{1}{2},0\}
      \right)
    }
    .
\end{equation}
\end{corollary}
\begin{proof}
Applying Corollary~\ref{weak_irreg_B}
(with 
$I=I$, 
$\vartheta=\min\{\theta,\frac{1}{2}\}-\frac{\varepsilon}{8}\mathbbm{1}_{[\nicefrac{1}{2},1)}(\theta)$, 
$\rho=1-\theta-\frac{\varepsilon}{2}(2-\mathbbm{1}_{[\nicefrac{1}{2},1)}(\theta))$
for $\varepsilon\in(0,1-\theta-4\max\{\theta-\frac{1}{2},0\})$,
$I\in\grid(\set)\backslash\{\set\}$
in the notation of Corollary~\ref{weak_irreg_B})
yields that for all 
$\varepsilon\in(0,1-\theta-4\max\{\theta-\frac{1}{2},0\})$
there exists a real number $C\in\R$ such that for all 
$I\in\grid(\set)\backslash\{\set\}$
it holds that 
\begin{equation}
\big|
\ES\big[ 
\varphi( X^{ \set }_T )
\big]
-
\ES\big[ 
\varphi( X_T^{ I } )
\big]
\big|
\leq
C \cdot
\left[
\inf_{
	b \in \set \backslash I
}
|\lambda_b|
\right]^{ 
	\varepsilon - \left(
	1-\theta-4\max\{\theta-\nicefrac{1}{2},0\}
	\right)
}
.
\end{equation}
This implies that for all 
$ \varepsilon \in (0,\infty) $, 
$\epsilon\in(0,\min\{1-\theta-4\max\{\theta-\frac{1}{2},0\},\varepsilon\})$
there exists a real number 
$C\in[0,\infty)$
such that for all $I\in\grid(\set)\backslash\{\set\}$ it holds that
\begin{equation}
\begin{split}
\big|
\ES\big[ 
\varphi( X^{ \set }_T )
\big]
-
\ES\big[ 
\varphi( X_T^{ I } )
\big]
\big|
&\leq
C \cdot
\left[
\inf_{
	b \in \set \backslash I
}
|\lambda_b|
\right]^{ 
	\epsilon - \left(
	1-\theta-4\max\{\theta-\nicefrac{1}{2},0\}
	\right)
}
\\&=
C \cdot
\left[
\inf_{
	b \in \set \backslash I
}
|\lambda_b|
\right]^{ 
	(\epsilon-\varepsilon)
}
\cdot
\left[
\inf_{
	b \in \set \backslash I
}
|\lambda_b|
\right]^{ 
	\varepsilon - \left(
	1-\theta-4\max\{\theta-\nicefrac{1}{2},0\}
	\right)
}
\\&=
\frac{C}{\left[
\inf_{
	b \in \set \backslash I
}
|\lambda_b|
\right]^{ 
	(\varepsilon-\epsilon)
}}
\cdot
\left[
\inf_{
	b \in \set \backslash I
}
|\lambda_b|
\right]^{ 
	\varepsilon - \left(
	1-\theta-4\max\{\theta-\nicefrac{1}{2},0\}
	\right)
}
\\&\leq
\left[
\frac{C}{\left[
	\inf_{
		b \in \set
	}
	|\lambda_b|
	\right]^{ 
		(\varepsilon-\epsilon)
}}
\right]
\cdot\left[
\inf_{
	b \in \set \backslash I
}
|\lambda_b|
\right]^{ 
	\varepsilon - \left(
	1-\theta-4\max\{\theta-\nicefrac{1}{2},0\}
	\right)
}
.
\end{split}
\end{equation}
This and the assumption that 
$ 
\sup_{ b \in \set } \lambda_b < 0
$
establish~\eqref{eq:weak.convergence.rate.simplified}.
The proof of Corollary~\ref{cor:weak.convergence.rate.simplified} is thus completed.
\end{proof}
\section{Weak convergence rates for SEEs}
\label{sec:sharp.weak.rates}
In this section our main weak convergence result is established, see Corollary~\ref{cor:sharp.weak.rates} below. Corollary~\ref{cor:sharp.weak.rates} follows from an application of Corollary~\ref{cor:weak.convergence.rate.simplified}.
Theorem~\ref{intro:theorem} in the introductory section is an immediate consequence of Corollary~\ref{cor:sharp.weak.rates}.

\begin{corollary}
\label{cor:sharp.weak.rates}
Let 
$ ( H, \left< \cdot, \cdot \right>_H, \left\| \cdot \right\|_H ) $
and
$ ( U, \left< \cdot, \cdot \right>_U, 
$
$
\left\| \cdot \right\|_U ) $
be separable $ \R $-Hilbert spaces, let 
$ T \in (0,\infty) $, 
$\iota\in[0,\nicefrac{1}{4}]$,
$\gamma\in[0,\nicefrac{1}{2}]$,
let $ ( \Omega, \mathcal{F}, \P, 
$
$
( \mathcal{F}_t )_{ t \in [0,T] } ) $
be a stochastic basis,
let $ ( W_t )_{ t \in [0,T] } $ be an $ \operatorname{Id}_U $-cylindrical
$ ( \Omega, \mathcal{F}, \P, ( \mathcal{F}_t )_{ t \in [0,T] } ) $-Wiener process,
let $ (e_n)_{ n \in \N } \subseteq H $ be an orthonormal basis of $ H $,
let $ ( \lambda_n )_{ n \in \N } \subseteq ( 0, \infty ) $ 
be an increasing sequence,
let 
$ A \colon D(A) \subseteq H \to H $ be a closed linear operator 
which satisfies 
$
D( A ) 
= 
\{ 
v \in H 
\colon 
\sum_{ n\in\N }
\left| 
\lambda_n
\langle e_n, v \rangle_H 
\right|^2
< \infty
\}
$
and
$ 
\forall \, n \in \N \colon
A e_n = 
-\lambda_n
e_n
$,
let
$  
( 
H_r 
,
\langle \cdot, \cdot \rangle_{ H_r }
,
\left\| \cdot \right\|_{ H_r } 
)
$,
$ r \in \R $,
be a family of interpolation spaces associated to $ - A $, 
let 
$\xi\in H_\iota$,
$\varphi\in C(H_\iota,\R)$, 
$F\in C(H_\iota,H_{\iota-1})$, 
$B\in C(H_\iota,HS(U,H_{\iota-\nicefrac{1}{2}}))$
satisfy for all 
$\varepsilon\in(0,\infty)$
that
$
  (H_\iota \ni v \mapsto \varphi(v) \in \R)
$, 
$
  (H_\iota \ni v \mapsto F(v) \in H_{\iota-\gamma-\varepsilon})
$, 
and 
$
  (H_\iota \ni v \mapsto (U \ni u \mapsto B(v)u \in H_{\iota-\nicefrac{\gamma}{2}-\varepsilon}) \in HS(U,H_{\iota-\nicefrac{\gamma}{2}-\varepsilon}))
$ 
are four times continuously Fr\'{e}chet differentiable with globally bounded derivatives,
let
$
( P_N )_{ N \in \N } \subseteq L( H_{ - 1 } )
$ 
satisfy for all 
$ N \in \N $, $ v \in H $ that 
$
P_N(v) = \sum^N_{n=1} \left< e_n, v \right>_H e_n
$, 
and let 
$ 
X \colon [0,T] \times \Omega \to H_\iota
$
and 
$ 
X^N \colon [0,T] \times \Omega \to P_N( H )
$,
$ N \in \N $,
be continuous 
$
( \mathcal{F}_t )_{ t \in [0,T] }
$-adapted stochastic processes
which satisfy that for all 
$ N\in\N $, 
$ t \in [0,T] $
it holds $\P$-a.s.\ that
\begin{equation}
\label{eq:lifting.SEE}
X_t 
= 
e^{ A t } \xi
+ 
\int_0^t e^{ A ( t - s ) } F( X_s ) \, ds
+ 
\int_0^t e^{ A ( t - s ) } B( X_s ) \, dW_s 
\end{equation}
and
\begin{equation}
\label{eq:lifting.galerkin}
X_t^N 
= 
e^{ A t } P_N( \xi )
+ 
\int_0^t e^{ A ( t - s ) } P_N F( X^N_s ) \, ds
+ 
\int_0^t e^{ A ( t - s ) } P_N B( X^N_s ) \, dW_s 
.
\end{equation}
Then for every $\varepsilon\in(0,\infty)$ 
there exists a real number $C\in\R$ 
such that for all $N\in\N$ 
it holds that 
\begin{equation}
\label{eq:weak.rate}
\left|
\ES\big[ 
\varphi( X_T )
\big]
-
\ES\big[ 
\varphi( X^N_T )
\big]
\right|
\leq
C \cdot
( \lambda_N )^{
	- ( 1 - \gamma - \varepsilon )
}
.
\end{equation}
\end{corollary}
\begin{proof}
	Throughout this proof 
	let
	$\tilde{A}\colon D(\tilde{A})
	\subseteq H_\iota \to H_\iota$ 
	be the linear operator which satisfies 
	$
	D( \tilde{A} ) 
	= 
	\{ 
	v \in H_\iota 
	\colon 
	\sum_{ n\in\N }
	\left| 
	\lambda_n
	\langle (-A)^{-\iota} e_n, v \rangle_{H_\iota} 
	\right|^2
	< \infty
	\}
	$
	and
	$ 
	\forall \, v \in D(\tilde{A}) \colon
	\tilde{A} v = 
	\sum_{ n\in\N }
	-\lambda_n
	\langle (-A)^{-\iota} e_n, v \rangle_{H_\iota}
	(-A)^{-\iota} e_n
	$, 
	and let 
	$  
	( 
	\tilde{H}_r 
	,
	\langle \cdot, \cdot \rangle_{ \tilde{H}_r }
	,
	\left\| \cdot \right\|_{ \tilde{H}_r } 
	)
	$,
	$ r \in \R $,
	be $\R$-Hilbert spaces which satisfy for all 
	$r\in\R$
	that 
	\begin{equation}
	\label{eq:transformed.domain}
	( 
	\tilde{H}_r 
	,
	\langle \cdot, \cdot \rangle_{ \tilde{H}_r }
	,
	\left\| \cdot \right\|_{ \tilde{H}_r } 
	)
	=
	( 
	H_{\iota+r} 
	,
	\langle \cdot, \cdot \rangle_{ H_{\iota+r} }
	,
	\left\| \cdot \right\|_{ H_{\iota+r} } 
	)
	.
	\end{equation}
	Observe that 
	$  
	( 
	\tilde{H}_r 
	,
	\langle \cdot, \cdot \rangle_{ \tilde{H}_r }
	,
	\left\| \cdot \right\|_{ \tilde{H}_r } 
	)
	$,
	$ r \in \R $,
	is a family of interpolation spaces associated to $-\tilde{A}$ 
	(cf., e.g., \cite[Section~3.7]{sy02}).
	Moreover, note that for all 
	$v\in H_{\iota-1}$, 
	$t\in[0,T]$
	it holds that 
	\begin{equation}
	e^{At} v=
	e^{\tilde{A} t}v.
	\end{equation}
	This, \eqref{eq:lifting.SEE}, and~\eqref{eq:lifting.galerkin} imply that for all 
	$ N \in \N $, 
	$t\in[0,T]$ 
	it holds $\P$-a.s.\ that 
	\begin{equation}
	X_t 
	= 
	e^{ \tilde{A} t } \xi
	+ 
	\int_0^t e^{ \tilde{A} ( t - s ) } F( X_s ) \, ds
	+ 
	\int_0^t e^{ \tilde{A} ( t - s ) } B( X_s ) \, dW_s 
	\end{equation}
	and
	\begin{equation}
	X_t^N 
	= 
	e^{ \tilde{A} t } P_N( \xi )
	+ 
	\int_0^t e^{ \tilde{A} ( t - s ) } P_N F( X^N_s ) \, ds
	+ 
	\int_0^t e^{ \tilde{A} ( t - s ) } P_N B( X^N_s ) \, dW_s 
	.
	\end{equation}
	Corollary~\ref{cor:weak.convergence.rate.simplified} 
	(with
	$H=H_\iota$,
	$U=U$,
	$T=T$,
	$W=W$,
	$\set=\{(-A)^{-\iota} e_n\colon n\in\N\}$,
	$\lambda((-A)^{-\iota} e_N)=-\lambda_N$,
	$A=\tilde{A}$,
	$H_r=\tilde{H}_r$,
	$P_{\{(-A)^{-\iota} e_1,\ldots,(-A)^{-\iota} e_N\}}(w)=P_N(w)$,
	$\varphi=\varphi$,
	$\theta=\gamma+\nicefrac{\varepsilon}{10}$,
	$F=(\tilde{H}_0\ni v\mapsto F(v) \in \tilde{H}_{-\gamma-\nicefrac{\varepsilon}{10}})$,
	$B=(\tilde{H}_0\ni v\mapsto (U\ni u\mapsto B(v)u \in \tilde{H}_{-\nicefrac{\gamma}{2}-\nicefrac{\varepsilon}{20}})\in HS(U,\tilde{H}_{-\nicefrac{\gamma}{2}-\nicefrac{\varepsilon}{20}}))$,
	$\xi=(\Omega\ni\omega\mapsto\xi\in \tilde{H}_0)$,
	$X^{\set}=X$,
	$X^{\{(-A)^{-\iota} e_1,\ldots,(-A)^{-\iota} e_N\}}=X^N$,
	$\varepsilon=\frac{9\varepsilon}{10}-\frac{2\varepsilon}{5}\mathbbm{1}_{\{\nicefrac{1}{2}\}}(\gamma)$
	for 
	$\varepsilon\in(0,5-10\gamma+\mathbbm{1}_{\{\nicefrac{1}{2}\}}(\gamma))$,
	$w\in H_{\iota-1}$, 
	$r\in\R$,
	$N\in\N$
	in the notation of Corollary~\ref{cor:weak.convergence.rate.simplified})
	therefore ensures that for all 
	$\varepsilon\in(0,5-10\gamma+\mathbbm{1}_{\{\nicefrac{1}{2}\}}(\gamma))$
	there exists a real number 
	$C\in[0,\infty)$
	such that for all $N\in\N$ it holds that
	\begin{equation}
	\begin{split}
	\big|
	\ES\big[
	\varphi(X_T)
	\big]
	-
	\ES\big[
	\varphi(X^N_T)
	\big]
	\big|
	&\leq
	C\cdot\left[\inf_{\{e_n\colon n\in\{N+1,N+2,\ldots\}\}}\lambda_n\right]^{-(1-\gamma-\varepsilon)}
	\\&\leq
	C\cdot(\lambda_N)^{-(1-\gamma-\varepsilon)}.
	\end{split}
	\end{equation}
	This and the fact that $(\lambda_n)_{n\in\N}\subseteq(0,\infty)$ is an increasing sequence imply that for all 
	$ \varepsilon \in (0,\infty) $, 
	$\epsilon\in(0,\min\{5-10\gamma+\mathbbm{1}_{\{\nicefrac{1}{2}\}}(\gamma),\varepsilon\})$
	there exists a real number 
	$C\in[0,\infty)$
	such that for all $N\in\N$ it holds that
	\begin{equation}
	\begin{split}
	\big|
	\ES\big[
	\varphi(X_T)
	\big]
	-
	\ES\big[
	\varphi(X^N_T)
	\big]
	\big|
	&\leq
	C\cdot(\lambda_N)^{-(1-\gamma-\epsilon)}
	\\&=
	C\cdot(\lambda_N)^{-(\varepsilon-\epsilon)}\cdot(\lambda_N)^{-(1-\gamma-\varepsilon)}
	\\&=
	\frac{C}{
		(\lambda_N)^{(\varepsilon-\epsilon)}
	}
	\cdot(\lambda_N)^{-(1-\gamma-\varepsilon)}
	\\&\leq
	\left[
	\frac{C}{
		(\lambda_1)^{(\varepsilon-\epsilon)}
	}
	\right]
	\cdot(\lambda_N)^{-(1-\gamma-\varepsilon)}
	.
	\end{split}
	\end{equation}
	The proof of Corollary~\ref{cor:sharp.weak.rates} is thus completed. 
\end{proof}

\section{Lower bounds for the weak error of Galerkin approximations for SEEs}
\label{sec:lower_bound}

In this section a few specific lower bounds for weak approximation errors
are established in the case of concrete examples {of SEEs}. 
More precisely, in this section we provide lower bounds for weak approximation errors for spatial spectral Galerkin approximations of linear stochastic heat equations with vanishing drift nonlinearities 
$F=0$.
Lower bounds for strong 
approximation errors for example SEEs and whole classes of SEEs can be found in 
\cite{dg01,mr07a,mrw07,mrw08}. In the case of finite dimensional stochastic
ordinary differential equations, lower bounds for both strong and weak approximation errors
have been established; for details see, e.g., the references in the overview article 
M\"{u}ller-Gronbach \& Ritter~\cite{mr08}.

\subsection{Setting}
\label{sec:setting_lower_bound}

Assume the setting in Section~\ref{sec:general_setting},
assume that 
$ 
  ( H, \left< \cdot , \cdot \right>_H , \left\| \cdot \right\|_H ) 
  = 
  ( U, \left< \cdot , \cdot \right>_U , \left\| \cdot \right\|_U ) 
$,
let $ \beta \in [ 0, \nicefrac{ 1 }{ 2 } ) $,
let $ \mu \colon \set \rightarrow \R $ be a function
such that 
$
  \sum_{ b \in \set }
  | \mu_b |^2
  \,
  | \lambda_b |^{ - 2 \beta }
  < \infty
$, 
and
let 
$
  B \in 
    HS( 
      H, 
      H_{ - \beta }
    ) 
$
satisfy that
for all $ v \in H $ 
it holds that
$
  B v = 
  \sum_{ b \in \set }
  \mu_b
  \left< b, v \right>_H
  b
$.

The above assumptions 
ensure 
that there exist up-to-modifications unique 
$
  ( \mathcal{F}_t )_{ t \in [0,T] }
$-predictable stochastic processes 
$ 
  X^I \colon [0,T] \times \Omega \to H, 
  I \in \grid(\set), 
$
such that for all 
$ p \in (0,\infty) $, $ I \in \grid(\set) $ 
it holds that
$
  \sup_{ t \in [0,T] }
  \ES\big[
  \| X^I_t \|^p_H
  \big]
  < \infty
$ 
and such that
for all $ I \in \grid(\set) $, $ t \in [0,T] $
it holds $ \P $-a.s.\ that
$
  X^I_t
  = 
    \int_0^t e^{ A ( t - s ) } P_I B \, dW_s 
  .
$

\subsection{Lower bounds for the weak error}

\begin{lemma}
\label{lemma:standard_result}
Assume the setting in Section~\ref{sec:setting_lower_bound} 
and let $ I \in \grid(\set) $, $ b \in \set $, $ t \in [0,T] $.
Then
$
  \operatorname{Var}\!\left(
  \left<
    b, X^I_t
  \right>_H
  \right)
=
  \frac{
    \mathbbm{1}_{ I }( b )
    \,
    | \mu_b |^2
    \,
    (
      e^{ 2 \lambda_b t }
      - 1
    )
  }{ 
    2 \lambda_b 
  }
$.
\end{lemma}

\begin{proof}
Observe that 
it holds $ \P $-a.s.\ that
\begin{equation}
\begin{split}
&
  \left< 
    b,
    X^I_t
  \right>_H
=
  \left< 
    b,
    \int_0^t
    e^{ A ( t - s ) } P_I B \, dW_s
  \right>_H
=
  \int_0^t
  \left< 
    P_I \,
    e^{ A ( t - s ) } 
    \, b,
    B \, dW_s
  \right>_H
\\ & =
  \mathbbm{1}_I( b )
  \int_0^t
  e^{ \lambda_b \left( t - s \right) }
  \left< 
    b,
    B \, dW_s
  \right>_H
=
  \mathbbm{1}_I( b )
  \,
  \mu_b
  \int_0^t
  e^{ \lambda_b \left( t - s \right) }
  \left< 
    b,
    dW_s
  \right>_H
  .
\end{split}
\end{equation}
This and It\^{o}'s isometry yield that 
\begin{equation}
  \operatorname{Var}\!\left(
    \left<
      b, X^I_t
    \right>_H
  \right)
=
  \mathbbm{1}_{
    I
  }( b )
  \,
  | \mu_b |^2
  \,
  \int^t_0
  e^{ 2 \lambda_b ( t - s ) }
  \, ds
=
  \frac{
    \mathbbm{1}_{
      I
    }( b )
    \,
    | \mu_b |^2
    \left(
      e^{ 2 \lambda_b t }
      - 1
    \right)
  }{ 
    2 \lambda_b 
  }
  .
\end{equation}
The proof of Lemma~\ref{lemma:standard_result}
is thus completed.
\end{proof}

The next elementary result, Lemma~\ref{lem:norm_square},
is an immediate consequence of Lemma~\ref{lemma:standard_result}
above.

\begin{lemma}
\label{lem:norm_square}
Assume the setting in Section~\ref{sec:setting_lower_bound},
let $ I \in \grid(\set) $,
and let 
$ \varphi \colon H \to [0,\infty) $
fulfill that for all $ x \in H $ it holds that
$ \varphi( x ) = \| x \|^2_H $.
Then
$ \varphi \in C^{ \infty }( H, [0,\infty) ) $
and
\begin{equation}
\begin{split}
&
  \ES\big[
    \varphi(
      X^\set_T
    )
  \big]
  -
  \ES\big[
    \varphi(
      X^I_T
    )
  \big]
  =
  \ES\big[
    \|
      X^{ \set \backslash I }_T
    \|^2_H
  \big]
\geq
  \left[
  \tfrac{
    1 - 
    e^{
      - 2 T \inf_{ b \in \set } | \lambda_b |
    }
  }{
    2
  }
  \right]
  \left[
    \smallsum_{ b \in \set \backslash I }
    \tfrac{
      | \mu_b |^2
    }{ 
      \left| \lambda_b \right| 
    }
  \right]
  .
\end{split}
\end{equation}
\end{lemma}

Lemma~\ref{lem:norm_square} establishes
a lower bound in the case of 
the squared norm as the test function.
The next result, Lemma~\ref{lem:lower_bound_exp},
establishes a similar lower bound for a test function 
in $ C^4_b( H, \R ) $.

\begin{lemma}
\label{lem:lower_bound_exp}
Assume the setting in Section~\ref{sec:setting_lower_bound}, 
let $ I \in \mathcal{P}( \set ) $ and let $ \varphi \colon H \to \R $
be given by $ \varphi( v ) = \exp\!\left( - \| v \|^2_H \right) $ for all $ v \in H $.
Then $ \varphi \in C^4_b( H, \R ) $
and
\begin{equation*}
\begin{split}
&
  \ES\!\left[ 
    \varphi( X^I_T )
  \right]
  -
  \ES\!\left[ 
    \varphi( X^{ \set }_T )
  \right]
  \geq
  \tfrac{ 
    \ES\left[ \varphi( X_T^{ \set } ) \right]
    \,
    \ES\left[
      \| X^{ \set \backslash I }_T \|^2_H
    \right]
  }{ 
    2 
    \,
    (
      1 
      +
      \ES[
        \| X^{ \set \backslash I }_T \|^2_H
      ]
    )^{ 3 / 2 }
  }
\geq
  \tfrac{ 
    \ES\left[ 
      \varphi( X_T^{ \set } ) 
    \right]
    \,
    \left(
      1 - 
      e^{
        - 2 T \inf_{ b \in \set } | \lambda_b |
      }
    \right)
  }{ 
    4 
    \,
    (
      1 
      +
      \ES[
        \| X^{ \set \backslash I }_T \|^2_H
      ]
    )^{ 3 / 2 }
  }
  \left[
    \smallsum\limits_{ b \in \set \backslash I }
    \tfrac{
      | \mu_b |^2
    }{ 
      \left| \lambda_b \right| 
    }
  \right]
  .
\end{split}
\end{equation*}
\end{lemma}

\begin{proof}
First of all, observe that for all 
$ x, u_1, u_2, u_3, u_4 \in H $ 
it holds that
\begin{align}
\label{eq:norm_exponent_derivatives_first}
  \varphi^{(1)}( x )( u_1 )
  & =
  - 2 \,
  \varphi( x )
  \left< x, u_1 \right>_H
  ,
\\
  \varphi^{(2)}( x )( u_1, u_2 )
  & =
  - 2
  \left[
  \varphi^{(1)}( x )( u_2 )
  \left< x, u_1 \right>_H
  +
  \varphi( x )
  \left< u_2, u_1 \right>_H
  \right]
\nonumber
\\ & =
\label{eq:norm_exponent_derivatives_second}
 - 2 \, \varphi( x )
 \left[
   \left< u_1, u_2 \right>_H
   - 2
   \left< x, u_1 \right>_H
   \left< x, u_2 \right>_H
 \right]
  ,
\\ 
  \varphi^{(3)}( x )( u_1, u_2, u_3 )
  & =
  - 2 \,
  \big[
  \varphi^{(1)}( x )
    \big(
    \left< u_3, u_1 \right>_H
    u_2
    +
    \left< u_2, u_1 \right>_H
    u_3
    \big)
    \nonumber
\\ & \quad +
\label{eq:norm_exponent_derivatives_third_b}
  \varphi^{(2)}( x )( u_2, u_3 )
  \left< x, u_1 \right>_H
  \big]
  ,
\\ 
  \varphi^{(4)}( x )( u_1, u_2, u_3, u_4 )
  & =
  - 2
  \big[
  \varphi^{(3)}( x )( u_2, u_3, u_4 )
  \left< x, u_1 \right>_H
  +
  \varphi^{(2)}( x )( u_2, u_3 )
  \left< u_4, u_1 \right>_H
  \nonumber
\\ & \quad
\label{eq:norm_exponent_derivatives_fourth_b}
  +
  \varphi^{(2)}( x )
    \big(
    \left< u_3, u_1 \right>_H
    u_2
    +
    \left< u_2, u_1 \right>_H
    u_3
    ,
    u_4
    \big)
  \big]
  .
\end{align}
Identity~\eqref{eq:norm_exponent_derivatives_first} and the fact that 
for all $ r \in [0,\infty) $ it holds that
$
  \sup_{ x \in H }
  \left(
    \left[ 1 + \| x \|^r_H \right]
    \varphi(x)
  \right)
  <
  \infty
$ 
show that for all 
$ r \in [ 0, \infty ) $ 
it holds that 
$
  \sup_{ x \in H }
  \left(
    \left[ 1 + \| x \|^r_H \right]
    \left[
      \varphi(x)
      +
      \|
        \varphi^{ (1) }(x)
      \|_{ L( H, \R ) }
    \right]
  \right)
  <
  \infty
$.
This and identity~\eqref{eq:norm_exponent_derivatives_second} imply that for all 
$ r \in [ 0, \infty ) $ 
it holds that 
\begin{equation}
  \sup_{ x \in H }
  \left(
  \left[ 
    1 + \| x \|^r_H 
  \right]
  \left[
    \smallsum_{ k = 1 }^2
    \|
      \varphi^{ (k) }(x)
    \|_{ L^{ (k) }( H, \R ) }
  \right]
  \right)
  <
  \infty
  .
\end{equation}
This and~\eqref{eq:norm_exponent_derivatives_third_b} 
yield that for all 
$ r \in [ 0, \infty ) $ 
it holds that 
\begin{equation}
  \sup_{ x \in H }
  \left(
    \left[ 
      1 + \| x \|^r_H 
    \right]
    \left[
      \smallsum_{ k = 1 }^3
      \|
        \varphi^{ (k) }(x)
      \|_{ L^{ (k) }( H, \R ) }
    \right]
  \right)
  <
  \infty
  .
\end{equation}
This and~\eqref{eq:norm_exponent_derivatives_fourth_b} 
prove that 
$ \varphi \in C^4_b( H, \R ) $.
Next observe that for all $ \sigma \in \R $ it holds that
\begin{equation}
\begin{split}
&
  \int_{ \R }
  \exp\!\left(
    - \left[ \sigma x \right]^2
  \right)
  \cdot
  \frac{ 1 }{ \sqrt{ 2 \pi } }
  \exp\!\left(
    - \frac{ x^2 }{ 2 }
  \right)
  dx
  =
  \int_{ \R }
  \frac{ 1 }{ \sqrt{ 2 \pi } }
  \exp\!\left(
    -     
    \frac{ x^2 }{ 2 }
    \left[ 1 + 2 \sigma^2 \right] 
  \right)
  dx
\\ & =
  \frac{ 1 }{ 
    \left[ 1 + 2 \sigma^2 \right]^{ \nicefrac{ 1 }{ 2 } }
  }
  \int_{ \R }
  \frac{ 1 }{ \sqrt{ 2 \pi } }
  \exp\!\left(
    -     
    \frac{ x^2 }{ 2 }
  \right)
  dx
  =
  \frac{ 1 }{ 
    \sqrt{ 
      1 + 2 \sigma^2 
    }
  }
  .
\end{split}
\end{equation}
This and Lemma~\ref{lemma:standard_result} imply that
\begin{equation}
\label{eq:phi_lower_bound_1}
\begin{split}
&
  \ES\!\left[ 
    \varphi( X^I_T )
  \right]
  -
  \ES\!\left[ 
    \varphi( X^{ \set }_T )
  \right]
\\ & =
  \prod_{ b \in I }
  \left[
    1 
    +
    \tfrac{
      | \mu_b |^2
    }{ 
      \lambda_b 
    }  
    \left(
      e^{ 2 \lambda_b T }
      - 1
    \right)
  \right]^{ - \nicefrac{ 1 }{ 2 } }
  -
  \prod_{ b \in \set }
  \left[
    1 
    +
    \tfrac{
      | \mu_b |^2
    }{ 
      \lambda_b 
    }  
    \left(
      e^{ 2 \lambda_b T }
      - 1
    \right)
  \right]^{ - \nicefrac{ 1 }{ 2 } }
\\ & =
  \prod_{ b \in I }
  \left[
    1 
    +
    \tfrac{
      | \mu_b |^2
    }{ 
      \lambda_b 
    }  
    \left(
      e^{ 2 \lambda_b T }
      - 1
    \right)
  \right]^{ - \nicefrac{ 1 }{ 2 } }
  \left[
  1
  -
  \prod_{ b \in \set \backslash I }
  \left[
    1 
    +
    \tfrac{
      | \mu_b |^2
    }{ 
      \lambda_b 
    }  
    \left(
      e^{ 2 \lambda_b T }
      - 1
    \right)
  \right]^{ - \nicefrac{ 1 }{ 2 } }
  \right]
\\ & \geq
  \prod_{ b \in \set }
  \left[
    1 
    +
    \tfrac{
      | \mu_b |^2
    }{ 
      \lambda_b 
    }  
    \left(
      e^{ 2 \lambda_b T }
      - 1
    \right)
  \right]^{ - \nicefrac{ 1 }{ 2 } }
  \left[
  1
  -
  \left[
    \prod_{ b \in \set \backslash I }
    \left[
      1 
      +
      \tfrac{
        | \mu_b |^2
      }{ 
        \lambda_b 
      }  
      \left(
        e^{ 2 \lambda_b T }
        - 1
      \right)
    \right]
    \right]^{ \! - \nicefrac{ 1 }{ 2 } }
  \right]
\\ & \geq
  \E\left[ \varphi( X_T^{ \set } ) \right]
  \left[
  1
  -
  \left[
      1 
      +
    \sum_{ b \in \set \backslash I }
      \tfrac{
        | \mu_b |^2
      }{ 
        \lambda_b 
      }  
      \left(
        e^{ 2 \lambda_b T }
        - 1
      \right)
    \right]^{ \! - \nicefrac{ 1 }{ 2 } }
  \right]
  .
\end{split}
\end{equation}
In the next step we note that the fundamental theorem 
of calculus ensures that for all
$ x \in [0,\infty) $
it holds that
$
  1 - 
  \left[ 
    1 + x
  \right]^{ - 1 / 2 }
  =
  \frac{ 1 }{ 2 }
  \int_0^x
  \left[
    1 + y
  \right]^{ - 3 / 2 }
  dy
  \geq
  \frac{ 1 }{ 2 }
  x
  \left[
    1 + x
  \right]^{ - 3 / 2 }
$.
Combining this with~\eqref{eq:phi_lower_bound_1}
and Lemma~\ref{lemma:standard_result}
proves that
\begin{equation}
\label{eq:phi_lower_bound_2}
\begin{split}
&
  \ES\!\left[ 
    \varphi( X^I_T )
  \right]
  -
  \ES\!\left[ 
    \varphi( X^{ \set }_T )
  \right]
\\ & \geq
  \frac{ 
    \E\left[ \varphi( X_T^{ \set } ) \right]
  }{ 2 }
  \left[
    \sum_{ b \in \set \backslash I }
      \tfrac{
        | \mu_b |^2
      }{ 
        \lambda_b 
      }  
      \left(
        e^{ 2 \lambda_b T }
        - 1
      \right)
  \right]
  \left[
    1 
    +
    \sum_{ b \in \set \backslash I }
      \tfrac{
        | \mu_b |^2
      }{ 
        \lambda_b 
      }  
      \left(
        e^{ 2 \lambda_b T }
        - 1
      \right)
  \right]^{ \! - \nicefrac{ 3 }{ 2 } }
\\ & 
\geq
  \frac{ 
    \E\left[ \varphi( X_T^{ \set } ) \right]
    \ES\big[
      \| X^{ \set \backslash I }_T \|^2_H
    \big]
  }{ 
    2 
    \,
    \big(
      1 
      +
      \ES\big[
        \| X^{ \set \backslash I }_T \|^2_H
      \big]
    \big)^{ \nicefrac{ 3 }{ 2 } }
  }
  .
\end{split}
\end{equation}
This and Lemma~\ref{lem:norm_square}
complete the proof of Lemma~\ref{lem:lower_bound_exp}.
\end{proof}

\begin{proposition}[A more concrete lower bound for the weak error]
\label{prop:lower_bound_exp}
Assume the setting in Section~\ref{sec:setting_lower_bound}, 
let $ b \colon \N \to \set $ be a bijective function,
let 
$ I \in \mathcal{P}( \set ) $, $ N \in \N $,
$ c, \rho \in ( 0, \infty ) $, 
$ 
  \delta \in ( - \infty , \frac{ 1 }{ 2 } - \frac{ 1 }{ 2 \rho } )
$
satisfy for all $ n \in \N $ that
$
  \lambda_{ b_n }
  =
  - c \, n^{ \rho }
$ 
and 
$
  \mu_{ b_n }
  =
  | \lambda_{ b_n } |^{ \delta }
$, 
and let $ \varphi \colon H \to \R $ satisfy for all $ v \in H $ that
$ \varphi( v ) = \exp\!\left( - \| v \|^2_H \right) $.
Then 
$ 
  \varphi \in C^4_b( H , \R ) 
$,
$
  B 
  \in
  \cap_{
    r \in 
    ( 
      - \infty , 
      -
      \frac{ 1 }{ 2 }
      \left[ 
        \nicefrac{ 1 }{ \rho } + 2 \delta 
      \right]
    )
  }
  HS( 
    H, H_r 
  )
$,
and 
\begin{equation}
\begin{split}
&
  \ES\big[ 
    \varphi( X^{ \{ b_1, \dots, b_N \} }_T )
  \big]
  -
  \ES\!\left[ 
    \varphi( X^{ \set }_T )
  \right]
\geq
  \tfrac{ 
    \ES\left[ 
      \varphi( X_T^{ \set } ) 
    \right]
    \,
    \left(
      1 - 
      e^{
        - 2 T c
      }
    \right)
    \,
    |
      \lambda_{ 
        b_N 
      }
    |^{
      -	
      (
        1 
        -
        [
          \nicefrac{ 1 }{ \rho }
          +
          2 \delta 
        ]
      )
    }
  }{ 
    2^{ ( 1 - 2 \delta \rho + \rho ) }
    \,
    c^{
      \nicefrac{ 1 }{ \rho }
    }
    \,
    \left(
      \rho - 2 \delta \rho
      +
      c^{ ( 2 \delta - 1 ) }
      \,
      \left(
        \rho - 2 \delta \rho
        -
        1
      \right)^{ - 1 / 3 }
    \right)^{ 3 / 2 }
  }
  .
\end{split}
\end{equation}
\end{proposition}

\begin{proof}
First of all, observe that
Lemma~\ref{lem:lower_bound_exp}
ensures that $ \varphi \in C^4_b( H, \R ) $
and
\begin{equation}
\label{eq:weak_estimate_lower_bound}
\begin{split}
&
  \ES\big[ 
    \varphi( X^{ \{ b_1, \dots, b_N \} }_T )
  \big]
  -
  \ES\!\left[ 
    \varphi( X^{ \set }_T )
  \right]
\geq
  \tfrac{ 
    \ES\left[ 
      \varphi( X_T^{ \set } ) 
    \right]
    \,
    \left(
      1 - 
      e^{
        - 2 T c
      }
    \right)
    \,
    c^{ ( 2 \delta - 1 ) }
  }{ 
    4 
    \,
    \left(
      1 
      +
      c^{ ( 2 \delta - 1 ) }
      \sum_{ n = N + 1 }^{ \infty }
      n^{ \rho ( 2 \delta - 1 ) }
    \right)^{ 3 / 2 }
  }
  \left[
    \smallsum\limits_{ n = N + 1 }^{ \infty }
    n^{ \rho ( 2 \delta - 1 ) }
  \right]
  .
\end{split}
\end{equation}
Next note that the assumption that
$
  \delta < \frac{ 1 }{ 2 } - \frac{ 1 }{ 2 \rho }
$
ensures that
$
  \rho \left( 2 \delta - 1 \right) < - 1
$.
This, in turn, implies that
\begin{equation}
\label{eq:sum_approximation_1}
\begin{split}
&
    \sum_{ n = N + 1 }^\infty
    n^{
      \rho \, ( 2 \delta - 1 )
    }
  =
    \sum_{ n = N + 1 }^\infty
    \int_{ n }^{ n + 1 }
    \frac{
      1
    }{
      n^{
        \rho \, ( 1 - 2 \delta )
      }
    }
    \,
    dx
  \geq
    \sum_{ n = N + 1 }^\infty
    \int_{ n }^{ n + 1 }
    \frac{
      1
    }{
      x^{
        \rho \, ( 1 - 2 \delta )
      }
    }
    \,
    dx
\\ &
  =
    \int^\infty_{ N + 1 }
      x^{
        \rho
        \,
        ( 2 \delta - 1 )
      }
    \, dx
  =
  \frac{ 
    -
    ( N + 1 )^{
      \left[
        1 +
        \rho
        \left( 
          2 \delta - 1 
        \right)
      \right]
    }
  }{
    \left[
      1 +
      \rho
      \left( 2 \delta - 1 \right)
    \right]
  }
\geq
  \frac{ 
    \left(
      2 N
    \right)^{
      \rho
      \,
      \left(
        \nicefrac{ 1 }{ \rho } +
          2 \delta - 1 
      \right)
    }
  }{
    \left[
      \rho
      \left( 1 - 2 \delta \right)
      - 1
    \right]
  }
\\ & =
  \frac{
    \left[ 
      \nicefrac{ 2^{ \rho } }{ c }
    \right]^{
      \left(
        \nicefrac{ 1 }{ \rho } + 2 \delta - 1
      \right)
    }
    \left|
      \lambda_{ 
        b_N 
      }
    \right|^{
      (
        \nicefrac{ 1 }{ \rho }
        +
        2 \delta - 1
      )
    }
  }{
    \left[
      \rho
      \left( 
        1 - 2 \delta
      \right)
      -
      1
    \right]
  }
  .
\end{split}
\end{equation}
Putting this into~\eqref{eq:weak_estimate_lower_bound}
proves that
\begin{equation}
\begin{split}
&
  \ES\big[ 
    \varphi( X^{ \{ b_1, \dots, b_N \} }_T )
  \big]
  -
  \ES\!\left[ 
    \varphi( X^{ \set }_T )
  \right]
\geq
  \tfrac{ 
    \ES\left[ 
      \varphi( X_T^{ \set } ) 
    \right]
    \,
    \left(
      1 - 
      e^{
        - 2 T c
      }
    \right)
    \,
    2^{
      \left(
        1 + 2 \delta \rho - \rho
      \right)
    }
    \,
    |
      \lambda_{ 
        b_N 
      }
    |^{
      (
        \nicefrac{ 1 }{ \rho }
        +
        2 \delta - 1
      )
    }
  }{ 
    4 
    \,
    c^{
      \nicefrac{ 1 }{ \rho }
    }
    \,
    \left(
      \rho - 2 \delta \rho
      -
      1
    \right)
    \,
    \left(
      1 
      +
      c^{ ( 2 \delta - 1 ) }
      \sum_{ n = N + 1 }^{ \infty }
      n^{ \rho ( 2 \delta - 1 ) }
    \right)^{ 3 / 2 }
  }
  .
\end{split}
\end{equation}
This and the fact that
\begin{equation}
\begin{split}
&
    \sum_{ n = N + 1 }^\infty
    n^{
      \rho \, ( 2 \delta - 1 )
    }
  =
    \sum_{ n = N + 1 }^\infty
    \int_{ n - 1 }^{ n }
    \frac{
      1
    }{
      n^{
        \rho \, ( 1 - 2 \delta )
      }
    }
    \,
    dx
  \leq
    \sum_{ n = N + 1 }^\infty
    \int_{ n - 1 }^n
    \frac{
      1
    }{
      x^{
        \rho \, ( 1 - 2 \delta )
      }
    }
    \,
    dx
\\ &
  =
    \int^\infty_N
      x^{
        \rho
        \,
        ( 2 \delta - 1 )
      }
    \, dx
  =
  \frac{ 
    -
    N^{
      \left[
        1 +
        \rho
        \left( 
          2 \delta - 1 
        \right)
      \right]
    }
  }{
    \left[
      1 +
      \rho
      \left( 2 \delta - 1 \right)
    \right]
  }
  =
  \frac{ 
    N^{
      \left(
        1 +
        2 \delta \rho
        -
        \rho
      \right)
    }
  }{
    \left(
      \rho - 2 \delta \rho - 1
    \right)
  }
  \leq
  \frac{ 
    1
  }{
    \left(
      \rho - 2 \delta \rho - 1
    \right)
  }
\end{split}
\end{equation}
complete the proof
of Proposition~\ref{prop:lower_bound_exp}.
\end{proof}

In the next result,
Corollary~\ref{prop:lower_bound_exp},
we specialize
Proposition~\ref{prop:lower_bound_exp}
to the case where
$ \rho = 2 $, $ c = \pi^2 $
(we think of $ A $ being, e.g., the Laplacian with Dirichlet boundary conditions
on $ H = L^2( (0,1) ; \R ) $)
and 
$ \delta \in ( - \infty, \nicefrac{ 1 }{ 4 } ) $.

\begin{corollary}
\label{cor:lower_bound}
Assume the setting in Section~\ref{sec:setting_lower_bound}, 
let $ b \colon \N \to \set $ be a bijective function,
let 
$ I \in \mathcal{P}( \set ) $, $ N \in \N $,
$ 
  \delta \in ( - \infty , \nicefrac{ 1 }{ 4 } )
$
satisfy that for all $ n \in \N $
it holds that
$
  \lambda_{ b_n }
  =
  - \pi^2 \, n^2
$ 
and 
$
  \mu_{ b_n }
  =
  | \lambda_{ b_n } |^{ \delta }
$, 
and let $ \varphi \colon H \to \R $
satisfy for all $ v \in H $ that $ \varphi( v ) = \exp\!\left( - \| v \|^2_H \right) $.
Then 
$ 
  \varphi \in C^4_b( H , \R ) 
$,
$
  B 
  \in
  \cap_{
    r \in 
    ( 
      - \infty , 
      -
      \frac{ 1 }{ 2 }
      \left[ 
        \nicefrac{ 1 }{ 2 } + 2 \delta 
      \right]
    )
  }
  HS( 
    H, H_r 
  )
$,
and 
\begin{equation}
\begin{split}
&
  \ES\big[ 
    \varphi( X^{ \{ b_1, \dots, b_N \} }_T )
  \big]
  -
  \ES\!\left[ 
    \varphi( X^{ \set }_T )
  \right]
\geq
  \left[
  \tfrac{ 
    \ES\left[ 
      \varphi( X_T^{ \set } ) 
    \right]
    \,
    2^{ ( 4 \delta - 5 ) }
    \,
    (
      1 - 
      e^{
        - T
      }
    )
  }{ 
    \left(
      2 - 4 \delta
      +
      2^{ ( 7 \delta - 7 ) }
      \,
      \left(
        1 - 4 \delta 
      \right)^{ - 1 / 3 }
    \right)^{ 3 / 2 }
  }
  \right]
    \left|
      \lambda_{ 
        b_N 
      }
    \right|^{
      -	
      (
        1 
        -
        [
          \nicefrac{ 1 }{ 2 }
          +
          2 \delta 
        ]
      )
    }
  .
\end{split}
\end{equation}
\end{corollary}

\subsection*{Acknowledgements}
Jarred Foster is gratefully acknowledged for a few helpful suggestions regarding English wording. 
Michael Schatz, Xiaojie Wang, and an anonymous referee are gratefully acknowledged for bringing a few typos into our notice.
This project has been supported through the SNSF-Research project 200021\_156603 
``Numerical approximations of nonlinear stochastic ordinary and partial differential equations".

\bibliographystyle{acm}
  \bibliography{Bib/bibfile}
\end{document}